\documentclass[11pt,reqno]{amsart}
\usepackage{amsmath,amssymb,amsthm,amsfonts,mathrsfs,latexsym,times,color,bm}
\usepackage{amsmath,amsfonts}
\usepackage{amsthm}
\usepackage{graphicx}
\usepackage{color}
\usepackage{multicol}

\usepackage{hyperref}

\newtheorem{thm}{Theorem}[section]
\newtheorem{definition}{Definition}[section]

\newtheorem{Lemma}[thm]{Lemma}
\newtheorem{remark}{Remark}[section]
\newtheorem{theorem}[thm]{Theorem}
\newtheorem{proposition}[thm]{Proposition}

\newtheorem{corollary}[thm]{Corollary}

\numberwithin{equation}{section}

\newcommand{\beq}{\begin{equation}}
\newcommand{\eeq}{\end{equation}}
\newcommand{\ben}{\begin{eqnarray}}
\newcommand{\een}{\end{eqnarray}}
\newcommand{\beno}{\begin{eqnarray*}}
\newcommand{\eeno}{\end{eqnarray*}}

\newcommand{\bu}{\mathbf{u}}

\newcommand{\bv}{\mathbf{v}}

\newcommand{\Bm}{\mathbf{m}}

\newcommand{\bQ}{\mathbf{Q}}
\newcommand{\bw}{\mathbf{w}}
\newcommand{\bW}{\mathbf{W}}

\newcommand{\bU}{\mathbf{U}}

\numberwithin{equation}{section}

\begin{document}
\begin{sloppypar}
\title[ relativistic Euler equations in 2D]{Low Regularity Well-Posedness of Cauchy Problem for Two-Dimensional Relativistic Euler Equation}


\author{Huali Zhang}
\address{Hunan University, School of Mathematics, Lushan South Road in Yuelu District, Changsha, 410882, People's Republic of China.}

\email{hualizhang@hnu.edu.cn}

\date{\today}


\begin{abstract}

In this article, we initiate the study of the Cauchy problem for the two-dimensional relativistic Euler equations in a low-regularity setting. By introducing good variables--a rescaled velocity, logarithmic enthalpy, and an appropriately defined vorticity, we reformulate the equations into a coupled wave-transport system. The nice structures, combined with new elliptic estimates and the Smith-Tataru approach \cite{ST}, enable us to extend ideas from the quasilinear wave equations to the two-dimensional relativistic Euler equations.

First, we prove the existence and uniqueness of solutions when the initial logarithmic enthalpy $h_0$, rescaled velocity $\bv_0$, and vorticity $\bw_0$ satisfy $(h_0, \bv_0, \bw_0, \nabla \bw_0) \in H^{\frac{7}{4}+}(\mathbb{R}^2) \times H^{\frac{7}{4}+}(\mathbb{R}^2) \times H^{\frac32+}(\mathbb{R}^2) \times L^8(\mathbb{R}^2)$. If the initial vorticity $\bw_0\in H^{\frac32}(\mathbb{R}^2)$, then a strong Strichartz estimate holds on a time interval that depends on the frequency size. Inspired by the semiclassical analysis approach of Bahouri-Chemin \cite{BC1} and Ai-Ifrim-Tataru \cite{AIT}, we further derive a loss of derivatives in the Strichartz estimate when summing over the number of these intervals. This yields a relaxed well-posedness result when $(h_0, \bv_0, \bw_0, \nabla \bw_0) \in H^{\frac{7}{4}+}(\mathbb{R}^2) \times H^{\frac{7}{4}+}(\mathbb{R}^2) \times H^{\frac32}(\mathbb{R}^2) \times L^8(\mathbb{R}^2)$. Both results are valid for the general state function $p(\varrho)=\varrho^A$ ($A \geq 1$).

Secondly, in the special case where $p(\varrho)=\varrho$, the acoustic metric reduces to the standard flat Minkowski metric. By applying Strichartz estimates to the wave part and combining with new elliptic estimate, we establish the well-posedness of solutions when $(h_0, \mathbf{v}_0, \mathbf{w}_0) \in H^{\frac{7}{4}+}(\mathbb{R}^2) \times H^{\frac{7}{4}+}(\mathbb{R}^2) \times H^{1+}(\mathbb{R}^2)$. The regularity exponents for the log-enthalpy and rescaled velocity correspond to those in Smith and Tataru \cite{ST}, while the vorticity regularity corresponds to Bourgain and Li \cite{BL}.

Moreover, if the stiff flow is irrotational, we can prove the local well-posedness for $(h_0, \mathbf{v}_0) \in H^{1+}(\mathbb{R}^2)$, and global well-posedness for small initial data $(h_0, \bv_0) \in \dot{B}^{1}_{2,1}(\mathbb{R}^2)$.
\end{abstract}
	
\maketitle
\section{Introduction}
\subsection{Background}
In fluid mechanics and astrophysics, the relativistic Euler equations are a generalization of the Euler equations that account for the effects of general relativity. This paper addresses the Cauchy problem for the two-dimensional relativistic Euler equations with low regularity data. To formulate the system, we use the notations $\varrho$ and $\bu=(u^0,u^1,u^2)^{\textrm{T}}$ to represent the energy density and the relativistic velocity, respectively. Moreover, the velocity is assumed to be a forward time-like vector field, normalized by
\begin{equation}\label{muu}
	u^\alpha u_\alpha =-1,
\end{equation}
where $u_{\alpha }= m_{\alpha \beta} u^\beta$, and $\Bm=(m_{\alpha \beta})_{3\times 3}$ is the flat standard Minkowski metric. Also, we use the Greek indices $\alpha,\beta,\gamma,\cdots$ take on the values $0,1,2$, while the Latin spatial indices $a,b,c,\cdots,i,j,k,\cdots$ take on the values $1,2$. Repeated indices are summed over (from $0$ to $2$ if they are Greek, and from $1$ to $2$ if they are Latin). Greek and Latin indices are lowered and raised with the Minkowski metric $\Bm$ and its inverse $\Bm^{-1}$. We will use this convention throughout the paper unless otherwise indicated.

Let $\mathcal{Q}$ be the stress energy tensor, which is defined by
\begin{equation*}
	\mathcal{Q}^{\alpha \beta}=(p+\varrho)u^\alpha u^\beta+ p m^{\alpha \beta}.
\end{equation*}
Then the motion of the relativistic Euler equations can be described by 
\begin{equation*}\label{mo1}
	\partial_{\alpha} \mathcal{Q}^{\alpha \beta}=0.
\end{equation*}
Projecting the above equation into the subspace
parallel and orthogonal to $u^\alpha$ yields (c.f. Disconzi-Ifrim-Tataru \cite{DIT} or Oliynyk \cite{Todd1,Todd2})
\begin{equation}\label{OREE}
	\begin{cases}
		u^\kappa\partial_\kappa \varrho+(p+\varrho) \partial_\kappa u^\kappa=0, \qquad t>0, \ x\in \mathbb{R}^2,
		\\
		(p+\varrho)u^\kappa\partial_\kappa u^\alpha + (m^{\alpha \kappa}+u^{\alpha}u^{\kappa}) \partial_\kappa p=0.
	\end{cases}
\end{equation}
An appropriate equation of state (please c.f. Choquet-Bruhat \cite{Bru}) is of the form
\begin{equation}\label{mo3}
	p=p(\varrho)=\varrho^{A}.
\end{equation}
where $A$ is the adiabatic constant. If $ A > 1 $, the fluid is called polytropic. When $ A = 1 $, it is referred to as a stiff fluid, in which the speed of sound equals the speed of light. Both in mathematics and physics, the system \eqref{OREE}-\eqref{mo3} has attracted significant attention. In our paper, we aim to study low-regularity solutions to the Cauchy problem for the two-dimensional relativistic Euler equations. Next, let us review some historical results related to the compressible or relativistic Euler equations.
\subsection{Historical results}
The local well-posedness theory for the compressible Euler equations began with Majda's work \cite{M}, where the initial velocity, density, and entropy belong to the Sobolev space $ H^{s}(\mathbb{R}^n)$ with $s > 1 + \frac{n}{2}$. In the special case of isentropic and irrotational, the compressible Euler equations can be written as a quasilinear wave system. A general quasilinear wave equation takes the form
\begin{equation}\label{qwe}
	\begin{cases}
			&\square_{h(\phi)} \phi=q(d \phi, d \phi), \qquad t>0, x\in \mathbb{R}^n,
			\\
			& (\phi, \partial_t \phi)|_{t=0}=(\phi_0, \phi_1)\in H^s(\mathbb{R}^n) \times H^{s-1}(\mathbb{R}^n),
		\end{cases}
\end{equation}
where $\phi$ is a scalar function, $h(\phi)$ is a Lorentzian metric depending on $\phi$, $d=(\partial_t, \partial_1, \partial_2, \cdots, \partial_n)$, and $q$ is quadratic in $d \phi$. The sharp local well-posedness was obtained by Smith and Tataru \cite{ST} for $s>\frac74 (n=2)$ and $s>2(n=3)$ ( c.f. Lindblad \cite{L} and Ohlman \cite{Oman} for ill-posedness). We also highlight significant related progress \cite{BC1, BC2, HKM, Geba, Kap, KR2, MSS, Sm, SS, ST0, T1, T2, T3, WQSharp} on the low regularity problems involving quasilinear wave equations. There are also two other important quasilinear wave models, namely, Einstein vacuum and time-like minimal surface equations. Regarding Einstein vacuum equations, the results depend on the choice of gauge. In the Yang-Mills gauge, Klainerman-Rodnianski-Szeftel \cite{KR1} solved the well-known $L^2$ curvature conjecture. In wave gauge, the problem is well-posed in $H^{2+}$ and ill-posed in $H^{2}$, cf. Klainerman-Rodnianski \cite{KR} and Ettinger-Lindblad \cite{EL}. In the CMC gauge, the system is well-posed in $H^{2+}$, please refer Andersson-Moncrief \cite{AM} and Wang \cite{WQRough}. For the time-like minimal surface equations, which satisfy null conditions, a celebrated work was obtained by Ai-Ifrim-Tataru \cite{AIT}, namely by lowering $\frac18$ derivatives in two space dimensions and by $\frac14$ derivatives in higher dimensions. For radial initial data, Wang and Zhou \cite{WZ} proved the critical global well-posedness in two dimensions. Very recently, By finding a good gauge and applying a parametric representation, Moschidis and Rodnianski \cite{MR} reduced the additional $\frac{1}{12}$-order regularity requirement for time-like surface equations in three dimensions.

In the general case, the compressible Euler equations can be written as a coupled wave-transport system (c.f. Luk-Speck \cite{LS2}). Based on this formulation, Disconzi-Luo-Mazzone-Speck \cite{DLS} proved the existence and uniqueness of solutions of compressible Euler equations when the initial velocity $\bu_0$, density $\rho_0$, vorticity $\bw_0$, and entropy $S_0$ satisfy $(\bu_0,\rho_0,\bw_0,S_0)\in H^{2+}(\mathbb{R}^3) \times H^{2+}(\mathbb{R}^3) \times H^{2+}(\mathbb{R}^3) \times H^{3+}(\mathbb{R}^3)$ with H\"older condition $\mathrm{curl} \bw_0, \Delta S_0\in C^{\sigma}, 0<\sigma<1$. Meanwhile, Wang \cite{WQEuler} established the existence and uniqueness of solutions of 3D isentropic compressible Euler equations if $(\bu_0,\rho_0,\bw_0) \in H^{s}(\mathbb{R}^3) \times H^{s}(\mathbb{R}^3) \times H^{s_0}(\mathbb{R}^3)$, where $2<s_0<s$. The work \cite{DLS} and \cite{WQEuler} both are based on vector-fields approach. By applying Smith-Tataru's method \cite{ST}, together with semi-classical analysis and Tao's frequency envelope, Andersson and Zhang \cite{AZ} proved a complete local well-posedness of 3D compressible Euler equations with relaxed assumptions $(\bu_0,\rho_0,\bw_0)\in H^{2+}(\mathbb{R}^3) \times H^{2+}(\mathbb{R}^3) \times H^{2}(\mathbb{R}^3)$ or $(\bu_0,\rho_0,\bw_0,S_0)\in H^{\frac52}(\mathbb{R}^3) \times H^{\frac52}(\mathbb{R}^3) \times H^{\frac32+}(\mathbb{R}^3)\times H^{\frac52+}(\mathbb{R}^3)$. In two dimensions, Zhang \cite{Z1,Z2,Z3} recently proved the well-posedness of solutions of isentropic compressible Euler equations when $(\bu_0, \rho_0, w_0, \nabla w_0) \in H^{\frac74+}(\mathbb{R}^2) \times H^{\frac74+}(\mathbb{R}^2)  \times H^{\frac32}(\mathbb{R}^2)  \times L^8(\mathbb{R}^2) $. In the reverse direction, the ill-posedness results were established by An-Chen-Yin \cite{ACY, ACY1} if $(\bu_0, \rho_0) \in \dot{H}^{2}(\mathbb{R}^3)$ or $(\bu_0, \rho_0) \in \dot{H}^{\frac74}(\mathbb{R}^2)$ with smooth vorticity.

Let us return to the study of the relativistic Euler equations. The early work on the Cauchy problem is attributed to Bruhat \cite{BF} (see also Friedrich \cite{HF}). Makino-Ukai \cite{MU1, MU2} established the local well-posedness of the Cauchy problem for the three-dimensional relativistic Euler equations with initial data in $H^s(\mathbb{R}^3)$ for $s\geq 3$. For an overview of the standard theory, we refer to the relevant chapter in Choquet-Bruhat's book \cite{Bru}. With vacuum, the pioneering result was obtained by Rendall \cite{Ren} and Guo-Tahvildar-Zadeh \cite{GT}. By introducing a new symmetrization, LeFloch-Ukai \cite{LUa} established a local-in-time existence result for solutions in $n$ spatial dimensions when the initial data belong to $H^s(\mathbb{R}^n)$ with $s>1+\frac{n}{2}$. Meanwhile, there are several interesting results on free boundary problems by Avadanei \cite{Av}, Jang-LeFloch-Masmoudi \cite{JLM}, Had$\check{\mathrm{z}}$i$\acute{\mathrm{c}}$-Shkoller-Speck \cite{HSS}, Disconzi-Ifrim-Tataru \cite{DIT}, and Miao-Shahshahani \cite{MS} and so on. Recently, Disconzi and Speck \cite{DS} proposed a new formulation of wave-transport equations for the three-dimensional relativistic Euler equations, in which the inherent structure plays a crucial role in addressing low-regularity problems. Based on \cite{DLS} and \cite{DS}, Yu \cite{Yu} proved the existence and uniqueness of rough solutions of 3D relativistic Euler equations when the initial enthalpy $h_0$, velocity $\bu_0$, vorticity $\bw_0$, and entropy $S_0$ satisfy $(h_0,\bu_0,\bw_0,S_0)\in H^{2+}(\mathbb{R}^3) \times H^{2+}(\mathbb{R}^3) \times H^{2+}(\mathbb{R}^3) \times H^{3+}(\mathbb{R}^3)$, and $\mathrm{vort} \bw_0, \Delta S_0 \in C^\sigma (0<\sigma<1)$. Very recently, Zhang \cite{Zhang} proved the complete local well-posedness of the 3D relativistic Euler equations with initial data $(h_0,\bu_0,\bw_0) \in H^{2+}(\mathbb{R}^3) \times H^{2+}(\mathbb{R}^3) \times H^{2}(\mathbb{R}^3)$. This result also resolves the ``open problem D" proposed by Disconzi \cite{Dis}. 

\subsection{Motivation}
Although there are several low regularity results \cite{DS,Yu,Zhang} for the 3D relativistic Euler equations, the Cauchy problem for the 2D case remains unknown. In three dimensions, the analysis crucially relies on the $\mathbf{u}$-vorticity formulation, which effectively arises in the wave component of velocity and has nice transport dynamics. However, this approach does not directly generalize to two dimensions, as it may not yield effective transport structures. On the other hand, low regularity well-posedness relies on Strichartz estimates, which differ between two and three dimensions. Strongly motivated by these considerations, we study low regularity solutions to the Cauchy problem for the 2D relativistic Euler equations.

To overcome the associated difficulties, we draw inspiration from \cite{DIT} and \cite{DS}, where a rescaled velocity and logarithmic enthalpy are introduced, respectively. Combining their idea, we define new variables, including rescaled velocity and logarithmic enthalpy, which reveal a coupled wave-transport structure and allow for a proper definition of vorticity in the 2D relativistic Euler equations. Under this reformulation, we apply the method of Smith and Tataru \cite{ST} (originally introduced for quasilinear wave equations) to establish the first low regularity result for the Cauchy problem of 2D relativistic Euler equations. Our main results(see section \ref{subs} below) can be summarized as follows:


\begin{itemize}
	\item  in the general polytropic case, i.e. $p(\varrho)=\varrho^A$ with the constant $A\geq 1$, the acoustic metric depends on log-enthalpy $h$ and rescaled velocity $\bv$, making the system \eqref{REEf} genuinely quasilinear. Using the energy method, the key step is to bound the Strichartz estimate $\|dh, d\bv\|_{L^4_t L^\infty_x}$. Following \cite{ST}, we derive Strichartz estimates for linear wave equations equipped with the acoustic metric by using wave packets along characteristic hypersurfaces. By carefully analyzing the characteristic hypersurfaces, a regularity exponent of $\frac{3}{2}+$ is sufficient. Indeed, the vorticity and acoustic metric determines the geometry of characteristic hypersurfaces. Consequently, the existence and uniqueness of \eqref{REEf} holds if the initial data \((h_0, \bv_0, \bw_0, \nabla \bw_0)\) belong to the spaces $H^{\frac{7}{4}+}(\mathbb{R}^2) \times H^{\frac74+}(\mathbb{R}^2) \times H^{\frac32+}(\mathbb{R}^2) \times L^8(\mathbb{R}^2)$. By applying the Strichartz estimates again and combining them with semiclassical analysis, the existence and uniqueness of \eqref{REEf} also holds if the initial data $(h_0,\bv_0,\bw_0,\nabla \bw_0) \in H^{\frac74+}(\mathbb{R}^2) \times H^{\frac74+}(\mathbb{R}^2) \times H^{\frac32}(\mathbb{R}^2) \times L^8 (\mathbb{R}^2)$. This result can be seen as an extension of the author's result \cite{Z3} from the Newtonian compressible Euler equations to the relativistic case.
	\item in the stiff case, i.e. $p(\varrho)=\varrho$, the acoustic metric is a flat Minkowski metric. Applying the Strichartz estimates in \cite{KT}, new elliptic estimate and bootstrap arguments, we can prove the local well-posedness of \eqref{REEf} if the initial data $(h_0,\bv_0,\bw_0) \in H^{\frac74+}(\mathbb{R}^2) \times H^{\frac74+}(\mathbb{R}^2) \times H^{1+}(\mathbb{R}^2)$. The regularity exponents for the log-enthalpy and rescaled velocity correspond to those in \cite{ST}, while the vorticity regularity corresponds to \cite{BL}.
	\item 
	in the stiff case, for irrotational flow, the system becomes a semi-linear wave equations satisfying ``wave map" null conditions. Therefore, referring to the classical results in \cite{KM1} and \cite{T}, the local well-posedness of \eqref{REEf} holds if $(h_0, \bv_0)\in H^{1+}(\mathbb{R}^2) $ and the small, global well-posedness holds if $(h_0, \mathbf{v}_0) \in \dot{B}_{2,1}^{1}(\mathbb{R}^2)$.
\end{itemize}

We next define several good variables as follows. 
\subsection{The good variables}
Let us introduce some variables--acoustic speed, number density, enthalpy, and vorticity.
\begin{definition}\cite{DS}
	We denote the acoustic speed $c_s$
	\begin{equation}\label{mo4}
		c_s=\sqrt{\frac{dp}{d\varrho}},
	\end{equation}
	the number density $q$
	\begin{equation}\label{mo5}
		p+\varrho=q\frac{d\varrho}{dq},
	\end{equation}
	and the enthalpy per particle $H$
	\begin{equation}\label{mo6}
		H=\frac{p+\varrho}{q}.
	\end{equation}
\end{definition}
Due to the work of Disconiz, Ifrim, and Tataru (\cite{DIT}, page 132, equation (1.9)), the rescaled velocity exhibits a better structure than the original velocity. Following this idea, we introduce the rescaled velocity as follows.
\begin{definition}
	We denote the logarithmic enthalpy $h$
	\begin{equation}\label{mo7}
		h=\log H,
	\end{equation}
	and the rescaled velocity
	\begin{equation}\label{mo8}
		\bv=\mathrm{e}^h \bu.
	\end{equation}
\end{definition}
\begin{remark}
\begin{enumerate}
	\item 
To ensure that the system \eqref{OREE} is hyperbolic, we assume $0 < c_s \leq 1$ throughout the paper.
\item Due to \eqref{muu} and \eqref{mo8}, we have
	\begin{equation}\label{rsv}
		\mathrm{e}^{-2h}v^\alpha v_\alpha=-1.
	\end{equation}
	If we set $\bv=(v^0,\mathring{\bv})$, then 
	\begin{equation*}\label{rsv0}
		v^0=\sqrt{\mathrm{e}^{2h}+|\mathring{\bv}|^2 }.
	\end{equation*}
	Under the settings \eqref{mo4}, \eqref{mo5}, \eqref{mo6}, and \eqref{mo7}, then $\varrho, p$, and $c_s$ are functions of $h$. To be simple, we record
	\begin{equation*}\label{und}
		\varrho=\varrho(h), \quad p=p(h), \quad c_s=c_s(h).
	\end{equation*}
\end{enumerate}
\end{remark}

In terms of these variables, the relativistic Euler equations \eqref{OREE} becomes (see Lemma \ref{tra} below)
\begin{align}\label{REEf}
	\begin{cases}
		& (1-c^2_s)v^\kappa \partial_\kappa h + c^2_s \partial_\kappa v^\kappa=0, 
\\
		& v^\kappa \partial_\kappa v^\alpha + \mathrm{e}^{2h} \partial^\alpha h=0.
	\end{cases}
\end{align}
We will work with this system for the rest of the paper. To study the low regularity solutions, we need a precise wave-transport structure for \eqref{REEf}. Therefore, we define the metric $g$ and the vorticity $\bw$ as follows.
\begin{definition}\label{Def}
Define the function $	\Theta=	\Theta(h,\bv)$ as 
\begin{equation}\label{theta}
	\Theta=\frac{1}{c^2_s - \mathrm{e}^{-2h}(c^2_s - 1) (v^0)^2}.
\end{equation}
Define the acoustic metric $g=(g^{\alpha \beta})_{3\times3}$ by
\begin{equation}\label{met1}
\begin{split}
	g^{\alpha \beta}= &  	\Theta \left\{   c^2_s m^{\alpha \beta } + \mathrm{e}^{-2h}(c^2_s-1) v^\alpha v^\beta  \right\} .
\end{split}	
\end{equation}
Define the vorticity $\bw=(w^0,w^1,w^2)$ as
\begin{equation}\label{Vor}
	w^\alpha= \epsilon^{\alpha \beta \gamma} \partial_\beta  v_\gamma .
\end{equation}
and $\bW=(W^0,W^1,W^2)$ as
\begin{equation}\label{EW2}
	W^\alpha=\epsilon^{\alpha \beta \gamma} \partial_\beta  w_\gamma + (1-c^{-2}_s) \epsilon^{\alpha \beta \gamma} w_\gamma \partial_\beta h.
\end{equation}
\end{definition}
\begin{remark}\label{rem}
	 If $p(\varrho)=\varrho$, then $c_s\equiv 1$. In this case, using \eqref{theta} and \eqref{met1}, the acoustic metric $g$ is a flat Minkowski metric and
	\begin{equation*}
		g=\Bm .
	\end{equation*}
\end{remark}
Inspired by Wang's paper \cite{WQEuler}, we also introduce a decomposition of rescaled velocity.
\begin{definition}
	Let $h$ and $\bv$ be the logarithmic enthalpy and rescaled velocity. Let $\bw$ and $\bW$ be defined in \eqref{Vor}- \eqref{EW2}. Let $\mathbf{Id}$ be the identity operator. Define the vector $\bv_{-}=(v_{-}^0,v_{-}^1,v_{-}^2)^{\mathrm{T}}$ by
	\begin{equation}\label{De0}
		\begin{split}
			( \mathbf{Id}-\mathbf{P} ) v_{-}^\alpha& = \epsilon^{\alpha \beta \gamma} \partial_\beta  w_\gamma, 
		\end{split}
	\end{equation}
	where
	\begin{equation}\label{De1}
		\mathbf{P}= (m^{\beta \gamma}+2\mathrm{e}^{-2h}v^{\beta}v^{\gamma}) \partial^2_{\beta \gamma}.
	\end{equation}
Meanwhile, we define the vector $\bv_{+}=(v_{+}^0,v_{+}^1,v_{+}^2)^{\mathrm{T}}$ by
	\begin{equation}\label{De}
		v^\alpha_{+}=v^\alpha-v^\alpha_{-}.
	\end{equation}
\end{definition}
\begin{remark}
From \eqref{De1}, we can verify that $\mathbf{{P}}$ is a space-time elliptic operator throughout the entire space-time domain; for details, see Lemma \ref{Ees} below. From \eqref{De1}, we can verify that $\mathbf{{P}}$ is a space-time elliptic operator throughout the entire space-time domain; for details, see Lemma \ref{Ees} below. 

Therefore, the operator $\mathbf{P}$ differs from the spatial elliptic operator defined in \cite{WQEuler}.
\end{remark}

\subsection{Wave-transport reduction}
\begin{Lemma}\label{Wrs}
	Let $(h,\bv)$ be a solution of \eqref{REEf}. Let $g$ and $\bw$ be defined in \eqref{met1} and \eqref{Vor}. Then the system \eqref{REEf} can be written as
		\begin{equation}\label{wrt}
		\begin{cases}
			& \square_g h = \mathcal{D},
			\\
			& \square_g v^\alpha = -c^2_s  \Theta  \epsilon^{\alpha \beta \gamma} \partial_\beta  w_\gamma + Q^\alpha,
			\\
			& v^\kappa \partial_\kappa w^\alpha= w^\kappa \partial^\alpha v_\kappa- w^\alpha \partial_\kappa v^\kappa,
		\end{cases}
	\end{equation}
	where $\square_g=g^{\alpha \beta } \partial^2_{\alpha \beta} $, $\mathcal{D}$ and $\bQ=(Q^0,Q^1,Q^2)$ are quadratic terms
	\begin{equation}\label{wrt0}
		\begin{split}
			\mathcal{D}=& - 2 \mathrm{e}^{-2h} \Theta c^{-1}_s c'_s v^\beta v^\kappa \partial_\beta h \partial_\kappa h
			-\mathrm{e}^{-2h} \Theta c^2_s \partial_\kappa v^\beta \partial_\beta v^\kappa- \Theta  (1+ c^2_s) \partial_\kappa h \partial^\kappa h,
		\end{split}
	\end{equation}
	and
	\begin{equation}\label{wrt1}
		\begin{split}
			Q^\alpha =& -  \mathrm{e}^{-2h} (c^2_s-1)\Theta v^\beta \partial_\beta v^\kappa \partial_\kappa v^\alpha 
			- 2 (c^2_s-1)\Theta v^\beta \partial_\beta h \partial^\alpha h -2\Theta c_s c'_s  \partial^\alpha h    \partial_\kappa v^\kappa  	
			\\
			& +  (c^2_s-1) \Theta\partial^\alpha v^\beta \partial_\beta h + 2\Theta c_s c'_s  v^\beta \partial^\alpha h \partial_\beta h.
		\end{split}
	\end{equation}
\end{Lemma}
\begin{remark}
	For the proof of Lemma \ref{Wrs}, please refer to Lemma \ref{DW0} and \ref{DW1} below.
\end{remark}

\subsection{Notations}
\begin{itemize}
	\item  Define the operators $\nabla=(\partial_{1}, \partial_{2})^\mathrm{T}$, $d=(\partial_t, \partial_{1}, \partial_{2})^\mathrm{T}$, $	\Delta=\partial^2_{1}+\partial^2_{2}$, $\Lambda_x=(-\Delta)^{\frac{1}{2}}$, $\square=-\partial^2_{tt}+\Delta$, $\square_g=g^{\alpha \beta } \partial^2_{\alpha \beta} $, and
	\begin{equation}\label{opt}
		\mathbf{T}=(v^0)^{-1}v^\kappa \partial_\kappa=\partial_t+ (v^0)^{-1} v^i \partial_i.
	\end{equation} 
	\item Let $\left< \xi \right>=(1+|\xi|^2)^{\frac{1}{2}}, \ \xi \in \mathbb{R}^2$. Denote by $\left< \nabla \right>$ the corresponding Bessel potential multiplier. 
\item Let $\zeta$ be a smooth function with support in the shell $\{ \xi\in \mathbb{R}^2: \frac{1}{2} \leq |\xi| \leq 2 \}$. Here, $\xi$ denotes the variable of the spatial Fourier transform. Let $\zeta$ also satisfy the
condition $\sum_{k \in \mathbb{Z}} \zeta(2^{k}\xi)=1$. Let ${P}_j$ be the Littlewood-Paley operator with frequency $2^j, j \in \mathbb{Z}$ (cf. Bahouri-Chemin-Danchin \cite{BCD}, page 78),
\begin{equation*}\label{Dej}
	{P}_j f = \int_{\mathbb{R}^2} \mathrm{e}^{-\mathrm{i}x\cdot \xi} \zeta(2^{-j}\xi) \hat{f}(\xi)d\xi.
\end{equation*}
For $f \in H^s(\mathbb{R}^2)$, let
\begin{align*}
	\|f\|_{H^s}= \|f\|_{L^2(\mathbb{R}^2)}+\|f\|_{\dot{H}^s(\mathbb{R}^2)},
\end{align*}
with the homogeneous norm $\|f\|^2_{\dot{H}^s} = {\sum_{j \in \mathbb{Z}}} 2^{2js}\|{P}_j f\|^2_{L^2(\mathbb{R}^2)} $. 
To avoid confusion, when the function $f$ is related to both time and space variables, we use the notation $\|f\|_{H_x^s}=\|f(t,\cdot)\|_{H^s(\mathbb{R}^2)}$.
\item Specially, $\bv \in H^s_x$ means $\bv-(1,0,0)^{\mathrm{T}} \in H^s_x$.
\item For $a\in \mathbb{R}$, we define
\begin{equation*}
 \|f\|_{\dot{B}^a_{\infty,2}}=\left( \sum_{j \in \mathbb{Z}} 2^{2ja} \|P_j f\|^2_{L^\infty_x(\mathbb{R}^2)} \right)^{\frac12} . 
\end{equation*}
\item  Assume there exists two positive constants $C_0$ and $c_0$ such that
\begin{equation}\label{HEw}
	|h_0,{\bv}_0 | \leq C_0,  \qquad 0<c_0<c_s|_{t=0}<1,
\end{equation}
where ${\bv}_0=(v^0_0, v^1_0, v^2_0)$. 
\item We use three small parameters
\begin{equation}\label{a0}
	0 < 	\epsilon_3 \ll \epsilon_2 \ll \epsilon_1 \ll \epsilon_0 \ll 1.
\end{equation}
\item 
In the following, constants $C$ depending only on $C_0, c_0$ are called universal. Unless otherwise stated, all constants that appear are universal in this sense.
The notation $X \lesssim Y$ means $X \leq CY$, where $C$ is a universal constant, possibly depending on $C_0, c_0$. Similarly, we write  $X \simeq Y$ when $C_1 Y \leq X \leq C_2Y$, with $C_1$ and $C_2$ universal constants, and $X \ll Y$ when $X \leq CY$ for a sufficiently large constant $C$. The universal constant may change from line to line.
\end{itemize}

\section{Main results}\label{subs} 
We present four theorems concerning the well-posedness of Cauchy problem \eqref{wrt}.
\subsection{Statement of results.} Our first result is stated as follows.
\begin{theorem}\label{dingli}
	Assume $\frac74<s_0\leq s\leq \frac{15}{8}$.  Consider the Cauchy problem \eqref{wrt}. Let $\bw$ be defined in \eqref{Vor}. For any given initial data $(h_0, \bv_0, \bw_0)$ satisfies \eqref{HEw}, and for any $M_0>0$ such that
	\begin{equation}\label{chuzhi1}
		\| (h_0,\bv_0)\|_{H^{s}}  + \| \bw_0 \|_{ H^{s_0-\frac14} } +  \|\nabla \bw_0\|_{ L^8 }
		\leq M_0,
	\end{equation}
	there exist positive constants $T>0$ and $M_1>0$ ($T$ and ${M}_1$ depends on $C_0, c_0, s, s_0, {M}_0$) 
	such that \eqref{wrt} has a unique solution $(h, \bv) \in C([0,T],H_x^s) \cap C^1([0,T],H_x^{s-1})$,
	$\bw \in C([0,T],H_x^{s_0}) \cap C^1([0,T],H_x^{s_0-1})$. Moreover,
	\begin{enumerate}
		\item \label{point:1}
		the solution $(h, \bv, \bw)$ satisfy the energy estimate
		\begin{equation}\label{A02}
			\begin{split}
				&\|( h, \bv )\|_{L^\infty_{[0,T]} H_x^{s}}+ \|\bw\|_{L^\infty_{[0,T]} H_x^{s_0-\frac14}}
				 + \|\nabla \bw \|_{ L^\infty_{[0,T]} L_x^8 } \leq M_1,
			\end{split}
		\end{equation}
		and
		\begin{equation*}
			\|h,\bv\|_{L^\infty_{ [0,{T}]} L^\infty_x } \leq 1+C_0.  
		\end{equation*}
		\item \label{point:2}
		the solution also satisfies the Strichartz estimate
		\begin{equation}\label{SSr}
			 \|dh, d\bv_{+}, d\bv\|_{L^4_{[0,T]}L_x^\infty}  \leq M_1 ,
		\end{equation}
		where $\bv_+$ is denoted in \eqref{De}.
		\item \label{point:3}
		for any $1 \leq r \leq s+1$, and for each $t_0 \in [0,T)$, the Cauchy problem of the linear equation
		\begin{equation}\label{linear}
			\begin{cases}
				& \square_g f=0, \qquad (t,x) \in (t_0,T]\times \mathbb{R}^2,
				\\
				&(f,\partial_t f)|_{t=t_0}=(f_0,f_1) \in H_x^r(\mathbb{R}^2) \times H_x^{r-1}(\mathbb{R}^2),
			\end{cases}
		\end{equation}
		admits a solution $f \in C([0,T],H_x^r) \times C^1([0,T],H_x^{r-1})$. The following estimate holds:
		\begin{equation*}
			\| f\|_{L_{[0,T]}^\infty H_x^r}+ \|\partial_t f\|_{L_{[0,T]}^\infty H_x^{r-1}} \leq  C_{M_0}( \|f_0\|_{H_x^r}+ \|f_1\|_{H_x^{r-1}} ).
		\end{equation*}
		Additionally, the following estimates hold, provided $k<r-\frac34$,
		\begin{equation}\label{SE1}
			\| \left<\nabla \right>^k f\|_{L^4_{[0,T]}L^\infty_x} \leq  C_{M_0}( \|f_0\|_{H_x^r}+ \|f_1\|_{H_x^{r-1}} ),
		\end{equation}
		and the same estimates hold with $\left< \nabla \right>^k$ replaced by $\left< \nabla \right>^{k-1}d$. Here, $C_{M_0}$ is a constant depending on $C_0, c_0, s, s_0$ and $M_0$.
	\end{enumerate}
\end{theorem}
\begin{remark}
\begin{enumerate}
		\item The initial condition $(h_0, \bv_0)\in H^{\frac74+}$ ensures the Strichartz admissibility for 2D wave system. The vorticity $ \bw_0\in H^{\frac32+}$ controls the geometry condition of null hypersurface in section \ref{secp4}. Meanwhile, the additional condition $\nabla \bw_0 \in L^8$ is imposed specifically to close the energy estimates for vorticity due to commutator estimate.
		\item Theorem \ref{dingli} is valid for the general state $p(\varrho)=\varrho^A$ ($A\geq 1$).
\end{enumerate}
\end{remark}
Our second result is stated below.
\begin{theorem}\label{dingli2}
	Assume $\frac74< s\leq \frac{15}{8}$.  Consider the Cauchy problem \eqref{wrt}. Let $\bw$ be defined in \eqref{Vor}. For any given initial data $(h_0, \bv_0, \bw_0)$ satisfies \eqref{HEw}, and for any $M_0>0$ such that
	\begin{equation}\label{czB}
		\| (h_0,\bv_0)\|_{H^{s}}  + \| \bw_0 \|_{ H^{\frac32} } +  \|\nabla \bw_0\|_{ L^8 }
		\leq M_0,
	\end{equation}
	there exist positive constants $T^*>0$ and $M_2>0$ ($T^*$ and ${M}_2$ depends on $C_0, c_0, s, {M}_0$) such that \eqref{wrt} 
	has a unique solution $(h, \bv) \in C([0,T^*],H_x^s) \cap C^1([0,T^*],H_x^{s-1})$,
	$\bw \in C([0,T^*],H_x^{\frac32}) \cap C^1([0,T^*],H_x^{\frac12})$. Moreover,
	the following statements hold:
	\begin{enumerate}
		\item \label{poi1}
		The solutions  $(h, \bv, \bw)$ satisfy the energy estimate
		\begin{equation*}\label{B02}
			\begin{split}
				&\|( h, \bv )\|_{L^\infty_{[0,T^*]} H_x^{s}}+ \|\bw\|_{L^\infty_{[0,T^*]} H_x^{\frac32}}
				+ \|\nabla \bw \|_{ L^\infty_{[0,T^*]} L_x^8 } \leq M_2,
			\end{split}
		\end{equation*}
		and
		\begin{equation*}
			\|h,\bv\|_{L^\infty_{ [0,{T^*}]} L^\infty_x } \leq 1+C_0.  
		\end{equation*}
		\item \label{poi2}
		The solution also satisfies the Strichartz estimate
		\begin{equation*}\label{BSr}
			\|dh, d\bv_{+}, d\bv\|_{L^4_{[0,T^*]}L_x^\infty}  \leq M_2 ,
		\end{equation*}
		where $\bv_+$ is denoted in \eqref{De}.
		\item \label{poi3}
		For any $ s-\frac34 \leq  r \leq \frac{11}{4}$, and for each $t_0 \in [0,T)$, the Cauchy problem of the linear equation
	\eqref{linear} admits a solution $f \in C([0,T^*],H_x^r) \times C^1([0,T^*],H_x^{r-1})$. The following estimate holds:
		\begin{equation*}
			\| f\|_{L_{[0,T^*]}^\infty H_x^r}+ \|\partial_t f\|_{L_{[0,T^*]}^\infty H_x^{r-1}} \leq  {M_3}( \|f_0\|_{H_x^r}+ \|f_1\|_{H_x^{r-1}} ).
		\end{equation*}
		Additionally, the following estimates hold, provided $k<r-(s-1)$,
		\begin{equation*}\label{BE1}
			\| \left<\nabla \right>^k f\|_{L^4_{[0,T^*]}L^\infty_x} \leq  {M_3}( \|f_0\|_{H_x^r}+ \|f_1\|_{H_x^{r-1}} ),
		\end{equation*}
		and the same estimates hold with $\left< \nabla \right>^k$ replaced by $\left< \nabla \right>^{k-1}d$. Here, ${M_3}$ is a positive constant depending on $C_0, c_0, s$ and $M_0$.
	\end{enumerate}
\end{theorem}
\begin{remark}
	\begin{enumerate}
		\item Compared with Theorem \ref{dingli}, Theorem \ref{dingli2} lowers the regularity exponent of the vorticity from $\frac32+$ to exactly $\frac{3}{2}$. This reduction is non-trivial because, if we apply Smith-Tataru's approach directly, we cannot obtain the Strichartz estimates for velocity and density when $w_0\in H^{\frac32}$. 
		
		When $(h_0, \bv_0, \bw_0, \nabla \bw_0)\in H^{\frac74+}(\mathbb{R}^2)\times H^{\frac74+}(\mathbb{R}^2) \times H^{\frac32}(\mathbb{R}^2) \times L^{8}(\mathbb{R}^2)$, we first use the frequency truncation to get a sequence of initial data belonging in $ H^{\frac74+}(\mathbb{R}^2)\times H^{\frac74+}(\mathbb{R}^2) \times H^{\frac32+}(\mathbb{R}^2) \times L^{8}(\mathbb{R}^2)$. Based on Theorem \ref{dingli}, a sequence of solutions on a short time-interval, and these intervals depends on the size of frequency.
		By extending the solutions from these short time intervals to a uniformly regular time-interval, a Strichartz estimate with a loss of derivatives can be obtained. This leads us to prove the local existence and uniqueness of solutions to the 2D compressible Euler equations for initial data $(h_0, \bv_0, \bw_0, \nabla \bw_0)\in H^{\frac74+}(\mathbb{R}^2)\times H^{\frac74+}(\mathbb{R}^2) \times H^{\frac32}(\mathbb{R}^2) \times L^{8}(\mathbb{R}^2)$.
		\item Theorem \ref{dingli2} is also valid for the general state $p(\varrho)=\varrho^A$ ($A\geq 1$).
	\end{enumerate}
\end{remark}
Very interestingly, in the special case where $p(\varrho) = \varrho$, namely the stiff fluid, the system \eqref{wrt} becomes (see Lemma \ref{Wrq} below)
\begin{equation}\label{wrtw}
	\begin{cases}
		& \square h = \mathrm{e}^{-2h} w^\beta w_\beta - 2 \partial_\beta h \partial^\beta h- \mathrm{e}^{-2h} \partial_\beta v^\alpha \partial^\beta v_\alpha,
		\\
		& \square v^\alpha = -  \epsilon^{\alpha \beta \gamma}\partial_\beta w_\gamma ,
		\\
		& v^\kappa \partial_\kappa w^\alpha= w^\kappa \partial^\alpha v_\kappa.
	\end{cases}
\end{equation}
Our third result is as follows.
\begin{theorem}\label{thm3}
	Assume $0< \epsilon' < \epsilon\leq \frac18$. Let the state function $p(\varrho)=\varrho$ and $\bw$ be defined in \eqref{Vor}. Consider the Cauchy problem \eqref{wrtw} with the initial data 
	\begin{equation*}
		(h,\bv,\bw)|_{t=0}=(h_0, \bv_0, \bw_0).
	\end{equation*} 
	If $(h_0, \bv_0, \bw_0)\in H^{\frac74+\epsilon}(\mathbb{R}^2) \times H^{\frac74+\epsilon}(\mathbb{R}^2) \times H^{1+\epsilon'}(\mathbb{R}^2)$, the Cauchy problem \eqref{wrtw} is locally well-posed and $(dh,d\bv)\in L^4_t L^\infty_x \cap L^4_t \dot{B}^{\epsilon'}_{\infty,2}$.
\end{theorem}
\begin{remark}
	\begin{enumerate}
		\item In the case of a stiff fluid, the acoustic metric is the standard flat Minkowski metric. Therefore, the system \eqref{wrtw} is a semilinear wave-transport system, which is different from the quasilinear wave-transport system \eqref{wrt}. Therefore, we can directly obtain the Strichartz estimates for $h$ and $\bv_{+}$. Consequently, compared to theorems \ref{dingli} and \ref{dingli2}, the required regularity of the vorticity can be reduced from $\frac{3}{2}+$ to $1+$.
		\item 
		In Theorem \ref{thm3}, the regularity of enthalpy and velocity corresponds to the sharp results of Smith and Tataru \cite{ST}, and the regularity of vorticity, c.f. Bourgain and Li \cite{BL}, is also sharp.
		\item For the compressible Euler equations, there is no such good structure \eqref{wrtw} in the particular state $p(\varrho)=\varrho$.
	\end{enumerate}
\end{remark}
Furthermore, when the stiff fluid is irrotational, namely, $\bw \equiv 0$, the system \eqref{wrt} becomes
\begin{equation}\label{wrtq}
	\begin{cases}
		& \square h = - 2 \partial_\beta h \partial^\beta h- \mathrm{e}^{-2h} \partial_\beta v^\alpha \partial^\beta v_\alpha,
		\\
		& \square v^\alpha = 0.
	\end{cases}
\end{equation}
We find that there are ``wave-map" null forms in the equation \eqref{wrtq}. Therefore, we can obtain the following result: 
\begin{theorem}\label{thm4}
	Assume $0< \epsilon\leq \frac18$. Let the state function $p(\varrho)=\varrho$ and $\bw \equiv 0$.  Consider the Cauchy problem \eqref{wrtq} with the initial data 
	\begin{equation*}
		(h,\bv)|_{t=0}=(h_0, \bv_0).
	\end{equation*}
	If $(h_0, \bv_0)\in H^{1+\epsilon}(\mathbb{R}^2) $, the Cauchy problem \eqref{wrtq} is locally well-posed. Moreover, if $(h_0, \mathbf{v}_0) \in \dot{B}_{2,1}^{1}(\mathbb{R}^2)$ and is sufficiently small, the Cauchy problem \eqref{wrtq} is globally well-posed.
\end{theorem}
\begin{remark}
	 For the system \eqref{wrtq} satisfying ``wave-map" null forms, therefore, we can derive Theorem \ref{thm4} directly from the works of Klainerman-Machedon \cite{KM1}, Klainerman-Selberg \cite{KS}, and Tataru \cite{T}.
\end{remark}
\subsection{Acknowledgments} We thank Daniel Tataru for very helpful discussions, especially for introducing the rescaled velocity. We express our gratitude to Lars Andersson and Changhua Wei for reminding us about the stiff case. The author is supported by the Natural Science Foundation of Hunan Province, China (Grant No. 2025JJ40003) and the Fundamental Research Funds for the Central Universities (Grant No. 531118010867).

\subsection{Organization of the rest paper} 
In Section \ref{nee}, we present preliminary work, including the derivation of new structures and energy estimates. In Section \ref{ptr}, we provide a proof of Theorem \ref{dingli}, assuming Proposition \ref{p1}. In Section \ref{sec4.3}, we prove Proposition \ref{p1} under the assumptions of Propositions \ref{r2} and \ref{r3}. In Section \ref{secp4}, we present a proof of Proposition \ref{r2}. In Section \ref{sec7}, we prove Proposition \ref{r3} by assuming Proposition \ref{r5}. After that, in Section \ref{Ap}, we prove Proposition \ref{r5}. In Section \ref{Sec8}, we provide a detailed proof of Theorem \ref{dingli2}. In Section \ref{sec10}, we present proofs of Theorems \ref{thm3} and \ref{thm4}.

\section{New structures and energy estimates}\label{nee} 
In this section, we introduce several structures of the system \eqref{REEf} and prove some energy estimates.
\subsection{Hyperbolic reduction}
\begin{Lemma}[\cite{Bru}, Theorem 13.1, page 283]\label{QH}
	Let $(\varrho, \bu)$ be a solution of \eqref{OREE}. Let $h$ and $\bv$ be defined in \eqref{mo7} and \eqref{mo8}. Denote $\bU=(p(h),\mathrm{e}^{-h}v^1,\mathrm{e}^{-h}v^2)^{\mathrm{T}}$. Then System \eqref{OREE} can be reduced to the following symmetric hyperbolic equation:
	\begin{equation}\label{QHl}
		A^0(\bU) \partial_t \bU + A^i(\bU) \partial_i \bU =0,
	\end{equation}
	where
	\begin{equation*}
		A^0=
		\left(
		\begin{array}{cccc}
			(\varrho+p)^{-2}\varrho'(p)\mathrm{e}^{-h}v^0 & -(\varrho+p)^{-1}\frac{v^1}{v_0} & -(\varrho+p)^{-1}\frac{v^2}{v_0}   \\
			-(\varrho+p)^{-1}\frac{v^1}{v_0} & \mathrm{e}^{-h}v^0(1+\frac{v^1v_1}{v^0v_0}) & \mathrm{e}^{-h}v^0\frac{v^1 v_2}{v^0v_0}  \\
			-(\varrho+p)^{-1}\frac{v^2}{v_0} & \mathrm{e}^{-h}v^0\frac{v^1 v_2}{v^0v_0} & \mathrm{e}^{-h} v^0(1+\frac{v^2 v_2}{v^0v_0})  \\
		\end{array}
		\right ),
	\end{equation*}
	\begin{equation*}
		A^1=
		\left(
		\begin{array}{cccc}
			(\varrho+p)^{-2}\varrho'(p)\mathrm{e}^{-h} v^1 & -(\varrho+p)^{-1} & 0    \\
			-(\varrho+p)^{-1} & \mathrm{e}^{-h} v^1(1+\frac{v^1v_1}{v^0v_0}) & \mathrm{e}^{-h} v^1\frac{v^1 v_2}{v^0v_0}  \\
			0 & \mathrm{e}^{-h}v^1\frac{v^1 v_2}{v^0v_0} & \mathrm{e}^{-h} v^1(1+\frac{v^2 v_2}{v^0v_0}) 
		\end{array}
		\right ),
	\end{equation*}
	\begin{equation*}
		A^2=
		\left(
		\begin{array}{cccc}
			(\varrho+p)^{-2}\varrho'(p)\mathrm{e}^{-h} v^2 & 0 & -(\varrho+p)^{-1}   \\
			0 & \mathrm{e}^{-h} v^2(1+\frac{v^1v_1}{v^0v_0}) & \mathrm{e}^{-h}v^2\frac{v^1 v_2}{v^0v_0}  \\
			-(\varrho+p)^{-1} & \mathrm{e}^{-h} v^2\frac{v^1 v_2}{v^0v_0} & \mathrm{e}^{-h} v^2(1+\frac{v^2 v_2}{v^0v_0}) 
		\end{array}
		\right ).
	\end{equation*}
\end{Lemma}
\begin{remark}
	In the book \cite{Bru} (Theorem 13.1, page 283), the hyperbolic system is formulated in terms of the variables $\varrho$ and $\bu$. In our paper, we transform these variables $(\varrho, \bu)$ to $(h, \bv)$. Moreover, by applying equations \eqref{mo3}, \eqref{mo5}, \eqref{mo6}, and \eqref{mo7}, the functions $\varrho$, $p$ and $\varrho'(p)$ are all depend on $h$. For other results on hyperbolic reduction, we also refer to \cite{LQ} and \cite{LU}.
\end{remark}
\subsection{Wave-transport structures}
We begin by deriving \eqref{REEf}. 
\begin{Lemma}\label{tra}
	Under the settings \eqref{mo4}-\eqref{mo7}, the system \eqref{OREE} can be written as \eqref{REEf}.
\end{Lemma}
\begin{proof}
	By using \eqref{mo6} and \eqref{mo7}, we get
	\begin{equation}\label{mo10}
		\begin{split}
			p+\varrho=q\mathrm{e}^h.
		\end{split}
	\end{equation}
	Operating derivatives to $\varrho$ for \eqref{mo10}, we can see
	\begin{equation}\label{mo11}
		\begin{split}
			p'(\varrho)+1=\mathrm{e}^h \frac{dq}{d\varrho}+q\mathrm{e}^h \frac{dh}{d\varrho}.
		\end{split}
	\end{equation}
	Inserting \eqref{mo4}, \eqref{mo5} to \eqref{mo11}, it yields
	\begin{equation}\label{mo12}
		\begin{split}
			c^2_s+1=\mathrm{e}^h \frac{q}{p+\varrho}+q\mathrm{e}^h \frac{dh}{d\varrho}.
		\end{split}
	\end{equation}
	Using \eqref{mo10}, then \eqref{mo12} tells us that
	\begin{equation}\label{mo13}
		\begin{split}
			c^2_s=q\mathrm{e}^h \frac{dh}{d\varrho}.
		\end{split}
	\end{equation}
	From \eqref{mo13}, we find
	\begin{equation}\label{mo14}
		\begin{split}
			\partial_\kappa \varrho=c^{-2}_sq\mathrm{e}^h\partial_\kappa h.
		\end{split}
	\end{equation}
	Using \eqref{mo10} and \eqref{mo8}, we can derive that
	\begin{equation}\label{mo15}
		\begin{split}
			(p+\varrho) \partial_\kappa u^\kappa  = & q\mathrm{e}^h \partial_\kappa u^\kappa
			\\
			=& q ( \partial_\kappa v^\kappa -v^\kappa \partial_\kappa h ).
		\end{split}
	\end{equation}
	From \eqref{mo14}, we obtain
	\begin{equation}\label{mo16}
		\begin{split}
			u^\kappa \partial_\kappa \varrho=& c^{-2}_sq\mathrm{e}^h u^\kappa \partial_\kappa h
			\\
			=& c^{-2}_sq  v^\kappa \partial_\kappa h.
		\end{split}
	\end{equation}
	Adding \eqref{mo15} and \eqref{mo16} yields 
	\begin{equation}\label{mo17}
		v^\kappa \partial_\kappa h (c^{-2}_s -1) + \partial_\kappa v^\kappa=0.
	\end{equation}
	Multiplying \eqref{mo17} with $c^2_s$, we therefore get
		\begin{equation*}\label{mo18}
		(1-c^{2}_s ) v^\kappa \partial_\kappa h  + c^{2}_s \partial_\kappa v^\kappa=0.
	\end{equation*}
Similarly, we can calculate
	\begin{equation}\label{mo19}
		\begin{split}
			(p+\varrho) \partial_\kappa u^\kappa  = & q\mathrm{e}^h \partial_\kappa u^\kappa,
		\\
		=& q\mathrm{e}^h ( v^\kappa \partial_\kappa v^\alpha - v^\alpha v^\kappa \partial_\kappa h),
		\end{split}
	\end{equation}
and
	\begin{equation}\label{mo20}
		\begin{split}
		(m^{\alpha \kappa} + u^\alpha u^\kappa)	\partial_\kappa p=& (m^{\alpha \kappa} + u^\alpha u^\kappa) c^2_s \partial_\kappa \varrho 
		\\
		=& (m^{\alpha \kappa} + \mathrm{e}^{-2h}v^\alpha v^\kappa) q\mathrm{e}^h\partial_\kappa h.
		\end{split}
	\end{equation}
	Adding \eqref{mo19} and \eqref{mo20}, we have shown 
	\begin{equation}\label{mo21}
		v^\kappa \partial_\kappa v^\alpha + \mathrm{e}^{-2h} \partial^\alpha h=0.
	\end{equation}
Combining \eqref{mo17} and \eqref{mo21} yields the system \eqref{REEf}.
\end{proof}

In the following lemma, we derive the new wave structure for the enthalpy and the rescaled velocity.
\begin{Lemma}\label{DW0} Let $(h,\bv)$ be a solution of \eqref{REEf}. Then we have
	\begin{equation*}\label{dw0}
		\begin{cases}
			& \square_g h \  =  \mathcal{D},
			\\
			& \square_g v^\alpha =  - c^2_s  \Theta \epsilon^{\alpha \beta \gamma} \partial_\beta w_\gamma + Q^\alpha,
		\end{cases}
	\end{equation*}
	where $\mathcal{D}$ and $\bQ=(Q^0,Q^1,Q^2)$ are quadratic terms defined as in \eqref{wrt0} and \eqref{wrt1}.
\end{Lemma}

\begin{proof}
	Operating the operator $v^\beta \partial_\beta$ on the first equation of \eqref{REEf}, we have
	\begin{equation}\label{a00}
		v^\beta \partial_\beta \left\{  (1-c^2_s)v^\kappa \partial_\kappa h   + c^2_s  \partial_\kappa v^\kappa \right\} =0.
	\end{equation}
	A direct calculation on \eqref{a00} tells us
	\begin{equation}\label{a02}
		\begin{split}
			& (1-c^2_s)v^\beta v^\kappa \partial^2_{\beta \kappa} h+ \underbrace{c^2_s v^\beta \partial_\beta ( \partial_\kappa u^\kappa)}_{\equiv R_1} + (1-c^2_s)v^\beta \partial_\beta v^\kappa \partial_\kappa h
			\\
			& \quad - 2c_s c'_s v^\beta v^\kappa \partial_\beta h \partial_\kappa h+ 2c_s c'_s v^\beta \partial_\beta h \partial_\kappa v^\kappa  =0.
		\end{split}
	\end{equation}
	For $R_1$, using \eqref{REEf}, so we can compute it by
	\begin{equation}\label{a03}
		\begin{split}
			R_1=& c^2_s \partial_\kappa ( v^\beta \partial_\beta v^\kappa  ) - c^2_s \partial_\kappa v^\beta \partial_\beta v^\kappa
			\\
			=& -c^2_s \partial_\kappa (  \mathrm{e}^{2h} \partial^\kappa  h  ) - c^2_s \partial_\kappa v^\beta \partial_\beta v^\kappa
			\\
			=& -\mathrm{e}^{2h}c^2_s  m^{\beta \kappa} \partial^2_{ \beta \kappa } h -c^2_s \partial_\kappa v^\beta \partial_\beta v^\kappa- 2 \mathrm{e}^{2h} c^2_s \partial_\kappa h \partial^\kappa h.
		\end{split}
	\end{equation}
	Substituting \eqref{a03} to \eqref{a02}, and using \eqref{REEf} again, we obtain
	\begin{equation}\label{a04}
		\begin{split}
			& \left\{  \mathrm{e}^{2h} c^2_s m^{\alpha \beta } + (c^2_s-1) v^\alpha v^\beta  \right\} \partial^2_{ \alpha \beta }  h
			\\
			=& (1-c^2_s)v^\beta \partial_\beta v^\kappa \partial_\kappa h - 2c_s c'_s v^\beta v^\kappa \partial_\beta h \partial_\kappa h+ 2c_s c'_s v^\beta \partial_\beta h \partial_\kappa v^\kappa
			\\
			&  -c^2_s \partial_\kappa v^\beta \partial_\beta v^\kappa- 2 \mathrm{e}^{2h} c^2_s \partial_\kappa h \partial^\kappa h
			\\
			=&  - \mathrm{e}^{2h} (1+ c^2_s) \partial_\kappa h \partial^\kappa h- 2c^{-1}_s c'_s v^\beta v^\kappa \partial_\beta h \partial_\kappa h
			-c^2_s \partial_\kappa v^\beta \partial_\beta v^\kappa.
		\end{split}
	\end{equation}
	To assure $g^{00}=-1$, multiplying \eqref{a04} with $\mathrm{e}^{-2h}\Theta$, we can see
	\begin{equation}\label{a05}
		\begin{split}
			\square_g h 
			=& - 2\mathrm{e}^{-2h} \Theta c^{-1}_s c'_s v^\beta v^\kappa \partial_\beta h \partial_\kappa h
			-\mathrm{e}^{-2h} \Theta c^2_s \partial_\kappa v^\beta \partial_\beta v^\kappa- \Theta  (1+ c^2_s) \partial_\kappa h \partial^\kappa h.
		\end{split}
	\end{equation}
	Operating $(c^2_s-1)v^\beta \partial_\beta$ on the second equation in \eqref{REEf}, we get
	\begin{equation}\label{a06}
		\begin{split}
			(c^2_s-1) v^\beta v^\kappa \partial^2_{\beta \kappa} v^\alpha + \underbrace{ \mathrm{e}^{2h} (c^2_s-1) v^\beta \partial_\beta\partial^\alpha h}_{\equiv R_2}
			+ (c^2_s-1) v^\beta \partial_\beta v^\kappa \partial_\kappa v^\alpha + 2(c^2_s-1) \mathrm{e}^{2h} v^\beta \partial_\beta h \partial^\alpha h =0.
		\end{split}
	\end{equation}
	Due to \eqref{REEf}, we deduce that
	\begin{equation}\label{a07}
		\begin{split}
			R_2 = &  \mathrm{e}^{2h} \partial^\alpha \left\{ (c^2_s-1) v^\beta \partial_\beta h \right\}  
			- \mathrm{e}^{2h} (c^2_s-1) \partial^\alpha v^\beta \partial_\beta h -2c_s c'_s \mathrm{e}^{2h} v^\beta \partial^\alpha h \partial_\beta h
			\\
			= &  \mathrm{e}^{2h} \partial^\alpha (  c^2_s  \partial_\kappa v^\kappa ) 
			- \mathrm{e}^{2h} (c^2_s-1) \partial^\alpha v^\beta \partial_\beta h -2c_s c'_s \mathrm{e}^{2h} v^\beta \partial^\alpha h \partial_\beta h
			\\
			=& c^2_s \mathrm{e}^{2h} \partial_\kappa ( \partial^\alpha   v^\kappa  )+  2c_s c'_s \mathrm{e}^{2h} \partial^\alpha h    \partial_\kappa v^\kappa  	- \mathrm{e}^{2h} (c^2_s-1) \partial^\alpha v^\beta \partial_\beta h -2c_s c'_s \mathrm{e}^{2h} v^\beta \partial^\alpha h \partial_\beta h
		\end{split}
	\end{equation}
	We also note
	\begin{equation}\label{a08}
		\begin{split}
			c^2_s \mathrm{e}^{2h} \partial_\kappa ( \partial^\alpha   v^\kappa  )=& c^2_s \mathrm{e}^{2h} \partial_\kappa ( \partial^\alpha   v^\kappa -  \partial^\kappa   v^\alpha)+ c^2_s \mathrm{e}^{2h} m^{\beta \kappa} \partial^2_{\beta \kappa} v^\alpha
			\\
			= & c^2_s \mathrm{e}^{2h} \epsilon^{\alpha \beta \gamma}\partial_\beta w_\gamma+ c^2_s \mathrm{e}^{2h} m^{\beta \kappa} \partial^2_{\beta \kappa} v^\alpha
		\end{split}
	\end{equation}
	Summing up our outcome \eqref{a06}, \eqref{a07}, \eqref{a08}, we have
	\begin{equation}\label{a09}
		\begin{split}
			\left\{   \mathrm{e}^{2h} c^2_s  m^{\beta \kappa} + (c^2_s-1) v^\beta v^\kappa  \right\} \partial^2_{\beta \kappa} v^\alpha=& -c^2_s \mathrm{e}^{2h} \epsilon^{\alpha \beta \gamma}\partial_\beta w_\gamma - (c^2_s-1) v^\beta \partial_\beta v^\kappa \partial_\kappa v^\alpha 
			\\
			& - 2(c^2_s-1) \mathrm{e}^{2h} v^\beta \partial_\beta h \partial^\alpha h -2c_s c'_s \mathrm{e}^{2h} \partial^\alpha h    \partial_\kappa v^\kappa  	
			\\
			& + \mathrm{e}^{2h} (c^2_s-1) \partial^\alpha v^\beta \partial_\beta h + 2c_s c'_s \mathrm{e}^{2h} v^\beta \partial^\alpha h \partial_\beta h.
		\end{split}
	\end{equation}
	Multiplying \eqref{a09} with $\mathrm{e}^{-2h}\Theta$, then we have
	\begin{equation}\label{a10}
		\begin{split}
			\square_{{g}} v^\alpha=& -\Theta c^2_s  \epsilon^{\alpha \beta \gamma}\partial_\beta w_\gamma - \mathrm{e}^{-2h} \Theta (c^2_s-1) v^\beta \partial_\beta v^\kappa \partial_\kappa v^\alpha 
			\\
			& - 2\Theta (c^2_s-1)  v^\beta \partial_\beta h \partial^\alpha h -2\Theta c_s c'_s  \partial^\alpha h    \partial_\kappa v^\kappa  	
			\\
			& + \Theta (c^2_s-1) \partial^\alpha v^\beta \partial_\beta h + 2\Theta c_s c'_s  v^\beta \partial^\alpha h \partial_\beta h.
		\end{split}
	\end{equation}
	Combining \eqref{a05} and \eqref{a10}, this concludes the proof of this lemma.
\end{proof}
Next, we derive a good wave equation for $\bv_{+}$.
\begin{Lemma}\label{Wag}
Let $(h,\bv)$ be a solution of \eqref{REEf}. Let $\bv_{+}$ and $\bv_{-}$ be defined in \eqref{De}. Then we have
\begin{equation}\label{wag0}
\begin{split}
		 		\square_{{g}} \bv_{+}
		 =& \Theta\mathrm{e}^{-2h}(v^0)^2  (1-3 c^2_s)    \mathbf{T} \mathbf{T} \bv_{-}- c^2_s  \Theta \bv_{-} + \bQ,
\end{split}
\end{equation}
where $\Theta$, $\bQ$, and $\mathbf{T}$ are stated as in \eqref{theta}, \eqref{wrt1}, and \eqref{opt}.
\end{Lemma}
\begin{proof}
	According to the expression of $\square_{{g}}$, we find that
	\begin{equation}\label{A0}
		\begin{split}
				\square_g v^\alpha=& g^{\beta \gamma}\partial^2_{\beta \gamma} v^\alpha
				 =\Theta \left\{ c^2_s m^{\beta \gamma}+ \mathrm{e}^{-2h}(c^2_s -1)v^\beta v^\gamma \right\} \partial^2_{\beta \gamma}v^\alpha.
		\end{split}
	\end{equation}
	For $\bv={\bv}_{+}+\bv_{-}$, the equation \eqref{A0} becomes
		\begin{equation}\label{A1a}
		\begin{split}
			\square_g v^\alpha
			=&  \Theta \left\{ c^2_s m^{\beta \gamma}+ \mathrm{e}^{-2h}(c^2_s -1)v^\beta v^\gamma \right\} \partial^2_{\beta \gamma}v_{+}^\alpha
			\\
			& + \Theta \left\{ c^2_s m^{\beta \gamma}+ \mathrm{e}^{-2h}(c^2_s -1)v^\beta v^\gamma \right\} \partial^2_{\beta \gamma}v_{-}^\alpha.
		\end{split}
	\end{equation}
Due to Lemma \ref{Wrs}, we also get
\begin{equation}\label{A2}
	\square_g v^\alpha = -c^2_s  \Theta \epsilon^{\alpha \beta \gamma} \partial_\beta  w_\gamma + Q^\alpha.
\end{equation}
Combining \eqref{A1a} and \eqref{A2} yield
\begin{equation*}
	\begin{split}
   \Theta \left\{  c^2_s m^{\beta \gamma}+ \mathrm{e}^{-2h}(c^2_s -1)v^\beta v^\gamma \right\} \partial^2_{\beta \gamma}v_{+}^\alpha
	=
		&  - \Theta \left\{ c^2_s m^{\beta \gamma}+ \mathrm{e}^{-2h}(c^2_s -1)v^\beta v^\gamma  \right\} \partial^2_{\beta \gamma}v_{-}^\alpha
		\\
		& -c^2_s  \Theta \epsilon^{\alpha \beta \gamma} \partial_\beta  w_\gamma + Q^\alpha.
	\end{split}
\end{equation*}
Using \eqref{De0} and \eqref{De1}, we further get
\begin{equation*}
	\begin{split}
		& \Theta \left\{ c^2_s m^{\beta \gamma}+ \mathrm{e}^{-2h}(c^2_s -1)v^\beta v^\gamma  \right\} \partial^2_{\beta \gamma}v_{+}^\alpha
		\\
		=
		&  - \Theta \left\{  c^2_s m^{\beta \gamma}+ \mathrm{e}^{-2h}(c^2_s -1)v^\beta v^\gamma  \right\}  \partial^2_{\beta \gamma}v_{-}^\alpha
		 - c^2_s  \Theta ( \mathbf{{Id}} - \mathbf{{P}} ) v_{-}^\alpha + Q^\alpha
		 \\
		 =& - \Theta \left\{ c^2_s m^{\beta \gamma}+ \mathrm{e}^{-2h}(c^2_s -1)v^\beta v^\gamma  \right\} \partial^2_{\beta \gamma}v_{-}^\alpha - c^2_s  \Theta \left\{  v^\alpha_{-}  -  (m^{\beta\gamma} + 2 \mathrm{e}^{-2h}v^\beta v^\gamma) \partial^2_{\beta \gamma} v^\alpha_{-} \right\} + Q^\alpha
		 \\
		 =& \Theta(1-3 c^2_s)   \mathrm{e}^{-2h} v^\beta v^\gamma\partial^2_{\beta \gamma}v^\alpha_{-}- c^2_s  \Theta v^\alpha_{-} + Q^\alpha
		 \\
		 =& \Theta\mathrm{e}^{-2h}(1-3 c^2_s)    (v^0)^2 \mathbf{T} \mathbf{T} v^\alpha_{-}- c^2_s  \Theta v^\alpha_{-} + Q^\alpha.
	\end{split}
\end{equation*}
Thus, we have proved \eqref{wag0}. Therefore, we have finished the proof of this lemma.
\end{proof}
\begin{remark}
By comparing equations \eqref{wrt} and \eqref{wag0}, we see that $\bv_+$ satisfies a better wave equation than $\bv$. Moreover, the equation \eqref{wag0} plays a key role in establishing Strichartz estimates for $\bv_{+}$.
\end{remark}
Finally, we will derive two good transport formulations for vorticity.
\begin{Lemma}\label{DW1}
	Let $(h,\bv)$ be a solution of \eqref{REEf}. Let $\bw$ and $\bW$ be defined in \eqref{Vor} and \eqref{EW2}. Then we have the following two transport equations
	\begin{equation}\label{EW0}
		v^\kappa \partial_\kappa w^\alpha= w^\kappa \partial^\alpha v_\kappa- w^\alpha \partial^\kappa v_\kappa,
	\end{equation}
	and
	\begin{equation}\label{EW1}
		\begin{split}
				v^\kappa \partial_\kappa    W^\alpha   
			=& -\epsilon^{\alpha \beta \gamma} \partial_\beta v^\kappa \partial_\kappa w_\gamma- \epsilon^{\alpha \beta \gamma} \partial_\beta w_\gamma  \partial^\kappa v_\kappa 
			+ \epsilon^{\alpha \beta \gamma}\ \partial_\beta w^\kappa \partial_\gamma v_\kappa
		\\
		&	+ \epsilon^{\alpha \beta \gamma} w_\gamma \partial_\beta \left\{ ( c^{-2}_s-1) v^\kappa \right\} \partial_\kappa h 
			 - \epsilon^{\alpha \beta \gamma} ( c^{-2}_s-1) v^\kappa \partial_\kappa w_\gamma \partial_\beta h 
		\\
		&	+2 \epsilon^{\alpha \beta \gamma}  c^{-3}_s c'_s v^\kappa w_\gamma \partial_\kappa h \partial_\beta h.
		\end{split}
	\end{equation}
\end{Lemma}
\begin{proof}
	Operating $\epsilon^{\alpha \beta \gamma} \partial_\beta$ on the second equation of \eqref{REEf}, we have
	\begin{equation*}
		\epsilon^{\alpha \beta \gamma} \partial_\beta \left( v^\kappa \partial_\kappa v_\gamma + \mathrm{e}^{2h} \partial_\gamma h \right)=0.
	\end{equation*}
	By chain's rule, we further get
	\begin{equation}\label{a11}
		v^\kappa \partial_\kappa \left(	\epsilon^{\alpha \beta \gamma} \partial_\beta  v_\gamma \right) + \epsilon^{\alpha \beta \gamma} \partial_\beta v^\kappa \partial_\kappa v_\gamma+ \underbrace{\mathrm{e}^{2h} \epsilon^{\alpha \beta \gamma}
			\partial_\beta h \partial_\gamma h }_{\equiv 0} + 
		\underbrace{  \mathrm{e}^{2h} \epsilon^{\alpha \beta \gamma} \partial^2_{\beta \gamma} h }_{\equiv 0}=0.
	\end{equation}
	We also calculate
	\begin{equation}\label{a12}
		\begin{split}
			\epsilon^{\alpha \beta \gamma} \partial_\beta v^\kappa \partial_\kappa v_\gamma=& \epsilon^{\alpha \beta \gamma} 
			( \partial_\beta v_\kappa - \partial_\kappa v_\beta) \partial^\kappa v_\gamma + \underbrace{ \epsilon^{\alpha \beta \gamma} \partial^\kappa v_\beta \partial_\kappa v_\gamma }_{\equiv 0}
			\\
			=& \epsilon^{\alpha \beta \gamma} \epsilon_{\mu \beta \kappa} w^\mu \partial^\kappa v_\gamma
			\\
			=&  - w^\kappa \partial^\alpha v_\kappa+ w^\alpha \partial^\kappa v_\kappa.
		\end{split}
	\end{equation}
	By using \eqref{a11} and \eqref{a12}, we therefore get \eqref{EW0}. Operating $\epsilon^{\alpha \beta \gamma}\partial_\beta$ on \eqref{EW0}, we show that
	\begin{equation*}
		\epsilon^{\alpha \beta \gamma}\partial_\beta ( v^\kappa \partial_\kappa w_\gamma )= \epsilon^{\alpha \beta \gamma}\partial_\beta ( w^\kappa \partial_\gamma v_\kappa)- \epsilon^{\alpha \beta \gamma}\partial_\beta ( w_\gamma \partial^\kappa v_\kappa).
	\end{equation*}
	By chain's rule, it yields
	\begin{equation*}
		\begin{split}
			v^\kappa \partial_\kappa ( \epsilon^{\alpha \beta \gamma}\partial_\beta  w_\gamma ) =& -\epsilon^{\alpha \beta \gamma} \partial_\beta v^\kappa \partial_\kappa w_\gamma- \epsilon^{\alpha \beta \gamma} \partial_\beta w_\gamma \partial^\kappa v_\kappa - \epsilon^{\alpha \beta \gamma} w_\gamma \partial_\beta \partial^\kappa v_\kappa
			\\
			&+ \epsilon^{\alpha \beta \gamma}\ \partial_\beta w^\kappa \partial_\gamma v_\kappa + \underbrace{  \epsilon^{\alpha \beta \gamma} w^\kappa \partial^2_{\beta \gamma} v_\kappa }_{\equiv 0}.
		\end{split}
	\end{equation*}
	Therefore, we obtain
	\begin{equation}\label{a14}
		\begin{split}
			v^\kappa \partial_\kappa  ( \epsilon^{\alpha \beta \gamma}\partial_\beta  w_\gamma )  =& -\epsilon^{\alpha \beta \gamma} \partial_\beta v^\kappa \partial_\kappa w_\gamma-  \epsilon^{\alpha \beta \gamma}\partial_\beta  w_\gamma  \partial^\kappa v_\kappa - \epsilon^{\alpha \beta \gamma} w_\gamma \partial_\beta \partial^\kappa v_\kappa
			+ \epsilon^{\alpha \beta \gamma}\ \partial_\beta w^\kappa \partial_\gamma v_\kappa .
		\end{split}
	\end{equation}
	Taking advantage of \eqref{REEf}, we can calculate
	\begin{equation}\label{a15}
		\begin{split}
			- \epsilon^{\alpha \beta \gamma} w_\gamma \partial_\beta \partial^\kappa v_\kappa= & \epsilon^{\alpha \beta \gamma} w_\gamma \partial_\beta \left\{ ( c^{-2}_s-1) v^\kappa \partial_\kappa h  \right\}
			\\
			=& \epsilon^{\alpha \beta \gamma} w_\gamma \partial_\beta \left\{ ( c^{-2}_s-1) v^\kappa \right\} \partial_\kappa h
			+ \epsilon^{\alpha \beta \gamma} ( c^{-2}_s-1) w_\gamma v^\kappa \partial_\kappa (\partial_\beta h)
			\\
			=& \epsilon^{\alpha \beta \gamma} w_\gamma \partial_\beta \left\{ ( c^{-2}_s-1) v^\kappa \right\} \partial_\kappa h
			+ v^\kappa \partial_\kappa \left\{ \epsilon^{\alpha \beta \gamma} ( c^{-2}_s-1) w_\gamma  \partial_\beta h  \right\}  
			\\
			& - \epsilon^{\alpha \beta \gamma} ( c^{-2}_s-1) v^\kappa \partial_\kappa w_\gamma \partial_\beta h 
			- \epsilon^{\alpha \beta \gamma} w_\gamma (-2c^{-3}_s c'_s) v^\kappa \partial_\kappa h \partial_\beta h.
		\end{split}
	\end{equation}
	By \eqref{a14} and \eqref{a15}, we therefore deduce
	\begin{equation*}
		\begin{split}
				v^\kappa \partial_\kappa   W^\alpha 
			=& -\epsilon^{\alpha \beta \gamma} \partial_\beta v^\kappa \partial_\kappa w_\gamma-  \epsilon^{\alpha \beta \gamma}\partial_\beta  w_\gamma   \partial^\kappa v_\kappa 
			+ \epsilon^{\alpha \beta \gamma}\ \partial_\beta w^\kappa \partial_\gamma v_\kappa
		 + \epsilon^{\alpha \beta \gamma} w_\gamma \partial_\beta \left\{ ( c^{-2}_s-1) v^\kappa \right\} \partial_\kappa h 
			\\
			& - \epsilon^{\alpha \beta \gamma} ( c^{-2}_s-1) v^\kappa \partial_\kappa w_\gamma \partial_\beta h 
			+2 \epsilon^{\alpha \beta \gamma}  c^{-3}_s c'_s v^\kappa w_\gamma \partial_\kappa h \partial_\beta h .
		\end{split}
	\end{equation*} 
	Therefore, we have finished the proof of \eqref{EW1}. This concludes the proof of this lemma.
\end{proof}
\subsection{Commutator estimates and product estimates}
Let us start by introducing a classical commutator estimate, which have been proved in the references \cite{KP,LD,ST}.
\begin{Lemma}\label{jiaohuan}\cite{KP}
	Let $a \geq 0$. Then for any scalar function $f_1, f_2$, we have
	\begin{equation*}\label{200}
		\|\Lambda_x^a(f_1f_2)-(\Lambda_x^a f_1)f_2\|_{L^2_x(\mathbb{R}^n)} 
		\lesssim  \|\Lambda_x^{a-1}f_1\|_{L^2_x(\mathbb{R}^n)}\|\nabla f_2\|_{L^\infty_x(\mathbb{R}^n)}+ \|f_1\|_{L^p_x(\mathbb{R}^n)}\|\Lambda_x^a f_2\|_{L^q_x(\mathbb{R}^n)},
	\end{equation*}
	where $\frac{1}{p}+\frac{1}{q}=\frac{1}{2}$.
\end{Lemma}
\begin{Lemma}[\cite{LD}, Theorem 1.2]\label{jiaohuan3}
	Let $0<a<1$. Let $\mathbf{{P}}$ be stated as in \eqref{De1}. Then we have
	\begin{equation*}
		\|[ \Lambda_x^a, \mathbf{{P}}] f \|_{L^2_x(\mathbb{R}^n)} \lesssim \|\Lambda_x^a \bv \|_{L^\infty_x(\mathbb{R}^n)} \|d^2 f \|_{L^2_x(\mathbb{R}^n)} .
	\end{equation*}
\end{Lemma}
Next, let us introduce some product estimates.
\begin{Lemma}\label{jiaohuan0}\cite{KP}
	Let $F(f)$ be a smooth function of $u$, $F(0)=0$ and $f \in L^\infty_x$. For any $a \geq 0$, we have
	\begin{equation*}\label{201}
		\|F(f)\|_{H^a} \lesssim  \|f\|_{H^{a}}(1+ \|f\|_{L^\infty_x}).
	\end{equation*}
\end{Lemma}

\begin{Lemma}\label{pps}\cite{ST}
	Suppose that $0 \leq r, r' < \frac{n}{2}$ and $r+r' > \frac{n}{2}$. Then
	\begin{equation*}\label{20000}
		\|hf\|_{H^{r+r'-\frac{n}{2}}(\mathbb{R}^n)} \leq C_{r,r'} \|h\|_{H^{r}(\mathbb{R}^n)}\|f\|_{H^{r'}(\mathbb{R}^n)}.
	\end{equation*}
	Moreover, if $-r \leq r' \leq r$ and $r>\frac{n}{2}$, then the following estimate holds:
	\begin{equation*}\label{20001}
		\|hf\|_{H^{r'}(\mathbb{R}^n)} \leq C_{r,r'} \|h\|_{H^{r}(\mathbb{R}^n)}\|f\|_{H^{r'}(\mathbb{R}^n)}.
	\end{equation*}
\end{Lemma}
We next prove a product estimate.
\begin{Lemma}\label{pro}
		Let $0 < a <1 $. Then we have
		\begin{equation*}
			\|\Lambda_x^a(f_1f_2)\|_{L^{2}_x(\mathbb{R}^2)} \lesssim \|f_1\|_{\dot{B}^{a}_{\infty,2}(\mathbb{R}^2)}\|f_2\|_{L^2_x(\mathbb{R}^2)}+ \|f_1\|_{L^\infty_x(\mathbb{R}^2)}\|f_2\|_{\dot{H}^a_x(\mathbb{R}^2)}.
		\end{equation*}
\end{Lemma}
\begin{proof}
Due to Littlewood-Paley decomposition, we have
\begin{equation*}
	\begin{split}
		P_j (f_1f_2)= & P_j \big( \sum_{|k-j|\leq 2} S_{k-1} f_1 P_k f_2 \big) + P_j \big( \sum_{|k-j|\leq 2} S_{k-1} f_2 P_k f_1 \big)+ P_j \big( \sum_{k\geq j-1} P_{k} f_1 P_k f_2 \big),
	\end{split}
\end{equation*}
where $ S_{k-1}= \sum_{j'\leq k-1} P_{j'}$. Due to H\"older's inequality, it follows
\begin{equation}\label{ppb}
	\begin{split}
	  \sum_{j\in \mathbb{Z}} 2^{2ja} \| P_j \big( \sum_{|k-j|\leq 2} S_{k-1} f_1 P_k f_2 \big)  \|^2_{L^2_x} 
	\lesssim  & \sum_{j\in \mathbb{Z}} 2^{2ja} \|P_j f_2\|^2_{L^{2}_x(\mathbb{R}^2)} \| f_1 \|^2_{L^\infty_x} 
	\\
	\lesssim & \|f_2\|^2_{\dot{H}^a_x(\mathbb{R}^2)} \|f_1\|^2_{L^\infty_x(\mathbb{R}^2)}.
	\end{split}
\end{equation}
Similarly, we can obtain
\begin{equation}\label{ppc}
	\begin{split}
	 \sum_{j\in \mathbb{Z}} 2^{2ja} \| P_j \big( \sum_{|k-j|\leq 2} S_{k-1} f_2 P_k f_1  \big) \|^2_{L^{2}_x(\mathbb{R}^2)}
		\lesssim & \sum_{j\in \mathbb{Z}} 2^{2ja} \|P_j f_1 \|^2_{L^{\infty}_x(\mathbb{R}^2)} \| f_2 \|^2_{L^{2}_x(\mathbb{R}^2)} 
			\\
		\lesssim & \|f_1\|^2_{\dot{B}^{a}_{\infty,2}(\mathbb{R}^2)}\|f_2\|^2_{L^2_x(\mathbb{R}^2)}.
	\end{split}
\end{equation}
For the rest term, owing to H\"older's inequality, it yields
\begin{equation}\label{ppd}
	\begin{split}
	\sum_{j\in \mathbb{Z}} 2^{2ja} \|  P_j \big( \sum_{k\geq j-1} P_{k} f_1 P_k f_2 \big) \|^2_{L^{2}_x(\mathbb{R}^2)} 
		\lesssim &   \sum_{k \in \mathbb{Z}} 2^{2ka}\| P_{k} f_1 \|^2_{L^\infty_x} \|P_k f_2 \|^2_{L^{2}_x(\mathbb{R}^2)}  \sum_{j\leq k} 2^{2(j-k)a}
		\\
		\lesssim & \|f_1\|^2_{\dot{B}^{a}_{\infty,2}(\mathbb{R}^2)}\|f_2\|^2_{L^2_x(\mathbb{R}^2)}.
	\end{split}
\end{equation}
Combing \eqref{ppb}, \eqref{ppc}, and \eqref{ppd}, we have proved this lemma.
\end{proof}
\subsection{Elliptic operator and estimates}
In \eqref{De1}, we define a new second-order operator $\mathbf{P}= (m^{\beta \gamma}+2\mathrm{e}^{-2h}v^{\beta}v^{\gamma}) 
\partial^2_{\beta \gamma}$. In the following, we will verify that $\mathbf{P}$ is a uniformly elliptic operator.
\begin{Lemma}\label{Ees}
	Let $(h,\bv)$ be a solution of \eqref{REEf} and $(h,\bv)\in C([0,T]\times \mathbb{R}^{2})$, and assume that \eqref{HEw} holds. 
	Let $\mathbf{P}$ be stated as in \eqref{De1}. Then $\mathbf{P}$ is a uniformly elliptic operator on $[0,T]\times \mathbb{R}^{2}$. 
	Furthermore, for $F \in L^2([0,T];L_x^{2})$, consider the elliptic equation:
	\begin{equation*}\label{ell}
	(\mathbf{Id}-\mathbf{P}) f =F.
	\end{equation*}
	Then there is a unique solution $f\in H^2_{t,x}( [0,T]\times \mathbb{R}^{2} )$ and 
	\begin{equation}\label{appl}
		\| f \|_{H^2_{t,x}( [0,T]\times \mathbb{R}^{2} ) } \lesssim \|F\|_{L_{[0,T]}^2 L_x^{2}},
	\end{equation}
	where $$ H^2_{t,x}( [0,T]\times \mathbb{R}^{2} )=\left\{ f(t,x): \partial^{a}_t \partial_x^b f\in L^2( [0,T]\times \mathbb{R}^{2} ) , a\geq 0,b\geq 0, a+b \leq 2 \right\} , $$
	and
	$$ \| f \|_{H^2_{t,x}( [0,T]\times \mathbb{R}^{2} ) } =\sum_{a \geq 0, b \geq 0, a+b\leq 2} \| \partial^{a}_t \partial_x^b f \|_{ L^2_ { [0,T] } L^2_x  } . $$
\end{Lemma}
\begin{proof}
	We first calculate the matrix $\mathrm{P}=(\mathrm{P}^{\beta \gamma})_{3\times 3}$, where $\mathrm{P}^{\beta \gamma}= m^{\beta \gamma}+2\mathrm{e}^{-2h}v^{\beta}v^{\gamma}$. Let us pick three leading principal minors
	\begin{equation*}
		\begin{split}
			p_1= \begin{matrix}
				-1+2\mathrm{e}^{-2h}(v^0)^2
			\end{matrix}  , \qquad p_2=\left | \begin{matrix}
				-1+2\mathrm{e}^{-2h}(v^0)^2 &2\mathrm{e}^{-2h}v^0v^1    \\
				2\mathrm{e}^{-2h}v^0v^1 &1+2\mathrm{e}^{-2h}(v^1)^2   \\
			\end{matrix} \right |,
		\end{split}
	\end{equation*}
	and
	\begin{equation*}
		\begin{split}
			p_3= \left | \begin{matrix}
				-1+2\mathrm{e}^{-2h}(v^0)^2 &2\mathrm{e}^{-2h}v^0v^1   & 2\mathrm{e}^{-2h}v^0v^2 \\
				2\mathrm{e}^{-2h}v^0v^1 &1+2\mathrm{e}^{-2h}(v^1)^2 & 2\mathrm{e}^{-2h}v^1v^2  \\
				2\mathrm{e}^{-2h}v^0v^2 & 2\mathrm{e}^{-2h}v^1v^2 &1+2\mathrm{e}^{-2h}(v^2)^2 \\
			\end{matrix} \right |  .
		\end{split}
	\end{equation*}
	Using $\mathrm{e}^{-2h}m_{\alpha \beta} v^\alpha v^\beta=-1$, it follows
	\begin{equation}\label{apps}
		(v^0)^2= \mathrm{e}^{2h} +(v^1)^2 +(v^2)^2 .
	\end{equation}
	For $p_1$, using \eqref{apps}, we get
	\begin{equation}\label{app2}
		\begin{split}
			p_1=-1+2\mathrm{e}^{-2h}(v^0)^2    \geq 1.
		\end{split}
	\end{equation}
	Due to \eqref{apps}, a direct calculation yields
	\begin{equation}\label{app3}
		\begin{split}
			p_2=& -1+2\mathrm{e}^{-2h} \{(v^0)^2 - (v^1)^2\} \geq 1,
			\\
			p_3=& -1+2\mathrm{e}^{-2h} \{(v^0)^2 - (v^1)^2- (v^2)^2 \} \geq 1.
		\end{split}
	\end{equation}
	By applying \eqref{app2} and \eqref{app3}, $\mathbf{P}$ is a positive definite matrix. Therefore, $\mathbf{P}$ is a uniformly elliptic operator.

Since we are only considering the Cauchy problem, no boundary conditions are added to solve equation \eqref{ell}. Therefore, we extend \(h\), \(\mathbf{v}\), \(F\), and $\mathbf{{P}}$ from $[0,T]\times \mathbb{R}^2$ to $\mathbb{R}^{1+2}$. Let
	\begin{equation*}
		\bar{h}(t,x)=\begin{cases}
			& 0, \qquad \qquad \qquad  \quad    t\in (-\infty,-T),
			\\
			& (1+\frac{t}{T})h(-t,x),  \ \quad t\in [-T,0),
			\\
			& h(t,x), \qquad \qquad \ \ \ t\in [0,T],
			\\
			& (2- \frac{t}{T})h(2T-t,x) , \quad t\in (T,2T),
			\\
			& 0, \quad \qquad \qquad \ \ \ \ \ \ \ t\in (2T, +\infty).
		\end{cases}
	\end{equation*}
	and
	\begin{equation*}
	\bar{\bv}(t,x)=\begin{cases}
		& 0, \qquad \qquad \qquad  \quad    t\in (-\infty,-T),
		\\
		& (1+\frac{t}{T})\bv(-t,x),  \ \quad t\in [-T,0),
		\\
		& \bv(t,x), \qquad \qquad \ \ \ t\in [0,T],
		\\
		& (2- \frac{t}{T})\bv(2T-t,x) , \quad t\in (T,2T),
		\\
		& 0, \quad \qquad \qquad \ \ \ \ \ \ \ t\in (2T, +\infty).
	\end{cases}
\end{equation*}
and
	\begin{equation}\label{lit}
	\bar{F}(t,x)=\begin{cases}
		& 0, \qquad \qquad \qquad  \quad    t\in (-\infty,-T),
		\\
		& (1+\frac{t}{T})F(-t,x),  \ \quad t\in [-T,0),
		\\
		& F(t,x), \qquad \qquad \ \ \ t\in [0,T],
		\\
		& (2- \frac{t}{T})F(2T-t,x) , \quad t\in (T,2T),
		\\
		& 0, \quad \qquad \qquad \ \ \ \ \ \ \ t\in (2T, +\infty).
	\end{cases}
\end{equation}
	We also set
	\begin{equation*}
	\mathbf{	\bar{P}} = ( m^{\alpha \beta} + 2 \mathrm{e}^{-2\bar{h}} \bar{v}^\alpha  \bar{v}^\beta) \partial^2_{\alpha \beta}. 
	\end{equation*}
Consider the problem 
\begin{equation}\label{ell2}
		(\mathbf{Id}-\mathbf{\bar{P}}) \bar{f} = \bar{F}.
\end{equation}
Note $\mathbf{\bar{P}}|_{[0,T]\times \mathbb{R}^2}=\mathbf{{P}}$ and
\begin{equation*}\label{lim}
	(\bar{h}, \bar{\bv}, \bar{F})|_{[0,T]\times \mathbb{R}^2} =({h}, {\bv}, {F}).
\end{equation*}
According to the uniqueness of solutions, we therefore obtain
\begin{equation}\label{lim2}
	\bar{f}|_{ [0,T]\times \mathbb{R}^2 } =f.
\end{equation}
For \eqref{ell2}, using \cite{GTe} (Lemma 9.17, page 255), we derive
	\begin{equation}\label{lim4}
		\begin{split}
			\| \bar{f} \|_{H^2_{t,x}(\mathbb{R}^{1+2})} \leq & \| \bar{F} \|_{L^2_{t,x}(\mathbb{R}^{1+2}) } .
		\end{split}
	\end{equation}
Due to \eqref{lit}, we can calculate
	\begin{equation}\label{lim6}
	\begin{split}
	\| \bar{F} \|^2_{L^2_{t,x}(\mathbb{R}^{1+2}) } =& \int^0_{-T} \|(1+\frac{t}{T}) F(-t,x)\|^2_{L^2_x} dt + \int_0^{T} \| F(t,x)\|^2_{L^2_x} dt 
	\\
	& + \int^{2T}_{T} \|(2-\frac{t}{T}) F(2T-t,x)\|^2_{L^2_x} dt 
	\\
	\lesssim & \int^0_{-T} \| F(-t,x)\|^2_{L^2_x} dt + \int_0^{T} \| F(t,x)\|^2_{L^2_x} dt 
	 + \int^{2T}_{T} \| F(2T-t,x)\|^2_{L^2_x} dt 
	\\
 \lesssim 	& \| {F} \|^2_{L^2_{[0,T]} L^2_x }. 
	\end{split}
\end{equation}
By using \eqref{lim2}, we have
\begin{equation}\label{lim8}
	\| {f} \|_{H^2_{t,x}([0,T]\times \mathbb{R}^{2})} \leq \| \bar{f} \|_{H^2_{t,x}(\mathbb{R}^{1+2})} .
\end{equation}
Combining \eqref{lim4}, \eqref{lim6}, and \eqref{lim8}, we obtain the estimate \eqref{appl}. Therefore, we complete the proof of this lemma.
\end{proof}

Based on the above lemma \ref{Ees}, we can obtain the following result:
\begin{Lemma}\label{yuy}
	Assume $s \in (\frac74,\frac{15}{8}]$. Let $(h,\bv,\bw)$ be a solution of \eqref{wrt} on $[0,T]\times \mathbb{R}^2$. Let $\bv_{-}$ and $\mathbf{T}$ be defined as in \eqref{De0} and \eqref{opt}. Then the function $\bv_{-}$ satisfies the following estimates
	\begin{equation}\label{Dsb}
			\| \partial_t \bv_{-}  \|_{L^2_t H_x^{1+a}}+ \| \bv_{-} \|_{L^2_t H_x^{2+a}}  \lesssim (1+\| h,\bv \|_{L^\infty_t H^{s}_x}) \| \bw \|_{L^2_t H_x^{1+a}}, \quad 0\leq a < s-1, 
	\end{equation}	
	and
	\begin{equation}\label{etaq}
			\| \partial_t \mathbf{T} \bv_{-} \|_{L^2_tH_x^{s-1}} + 	\| \mathbf{T} \bv_{-} \|_{L^2_tH_x^s} \lesssim \| \bw\|_{L^2_t H_x^{1+a}} (1+\| h,\bv \|^2_{L^\infty_t H_x^{s}}), \quad a > 0. 
	\end{equation}
\end{Lemma}
\begin{proof}
	Noting that \eqref{De0}, and using Lemma \ref{Ees}, it yields
	\begin{equation}\label{Ds}
		\begin{split}
			\| \bv_{-}  \|_{L^2_t H_x^{2}} \lesssim & \| d\bw \|_{L^2_t L_x^{2}}
			\lesssim \| \bw \|_{L^2_t H_x^{1}}.
		\end{split}
	\end{equation}
	Operating $\Lambda_x^a$ on \eqref{De0}, we deduce
	\begin{equation*}\label{Dsa}
		\begin{split}
			(	\mathbf{{Id}}-\mathbf{{P}} )  \Lambda_x^a v^\alpha_{-}=\Lambda_x^a \left(  \epsilon^{\alpha \beta \gamma} \partial_\beta w_\gamma \right) + [\Lambda_x^a, \mathbf{{P}} ] v^\alpha_{-} .
		\end{split}
	\end{equation*}
	Using \ref{Ees}, for $0\leq a <s-1$ and $s>\frac74$, we can derive that
	\begin{equation}\label{Dsbb}
		\begin{split}
			\| \partial_t \Lambda_x^a \bv_{-}  \|_{L^2_t H_x^{1}}+	\| \Lambda_x^a \bv_{-}  \|_{L^2_t H_x^{2}} \lesssim & \|\Lambda_x^a(  d\bw )\|_{L^2_t L_x^{2}} + \|[\Lambda_x^a, \mathbf{{P}} ] \bv_{-}\|_{L^2_t L_x^{2}} 
			\\
			\lesssim & \| \bw \|_{L^2_t H_x^{1+a}} + + \|[\Lambda_x^a, \mathbf{{P}} ] \bv_{-}\|_{L^2_t L_x^{2}} .
		\end{split}
	\end{equation}
		Using Lemma \ref{jiaohuan3} and \eqref{Ds}, it follows 
	\begin{equation}\label{Dsbbb}
		\begin{split}
			\|[\Lambda_x^a, \mathbf{{P}} ] \bv_{-}\|_{L^2_t L_x^{2}}  \lesssim &  \|\Lambda_x^a h, \Lambda_x^a \bv\|_{L^\infty_t L^\infty_x} \|d^2 \bv_{-} \|_{L^2_t L_x^{2}} 
			\\
			\lesssim & \|h,\bv\|_{L^\infty_t H^s_x}\| \bw \|_{L^2_t H_x^{1}}  .
		\end{split}
	\end{equation}
Substituting \eqref{Dsbbb} to \eqref{Dsbb}, we therefore prove \eqref{Dsb}. Due to \eqref{De0}, we can derive
	\begin{equation}\label{b00}
		\begin{split}
			\mathbf{{P}} ( \mathbf{T} v^\alpha_{-} )=& - \mathbf{T} \big( c^2_s \Theta \epsilon^{\alpha \beta \gamma} \partial_\beta w_\gamma  \big)+ [\mathbf{T}, \mathbf{{P}}]v^\alpha_{-}
			\\
			=& - \mathbf{T} \big( c^2_s \Theta \big) \epsilon^{\alpha \beta \gamma} \partial_\beta w_\gamma  - c^2_s \Theta \mathbf{T} \left\{   W^\alpha - (1-c^{-2}_s) \epsilon^{\alpha \beta \gamma} w_\gamma \partial_\beta h    \right\}
			+  [\mathbf{T}, \mathbf{{P}}]v^\alpha_{-}
			\\
			=& - \mathbf{T} \big( c^2_s \Theta  \big) \epsilon^{\alpha \beta \gamma} \partial_\beta w_\gamma     - c^2_s \Theta \mathbf{T} W^\alpha +  c^2_s \Theta  \mathbf{T} \left\{    (1-c^{-2}_s) \right\}  \epsilon^{\alpha \beta \gamma} w_\gamma \partial_\beta h 
			\\
			&- [\mathbf{T}, \mathbf{{P}}]v^\alpha_{-}     + c^2_s  (1-c^{-2}_s) \Theta \epsilon^{\alpha \beta \gamma} \partial_\beta h \mathbf{T}      w_\gamma     + c^2_s  (1-c^{-2}_s) \Theta \epsilon^{\alpha \beta \gamma} w_\gamma \mathbf{T} ( \partial_\beta h   ). 
		\end{split}
	\end{equation}
	We set $\mathbf{{\Gamma}}=({\Gamma}^0,{\Gamma}^1,{\Gamma}^2)$ and
	\begin{equation*}
		\begin{split}
			{\Gamma}^\alpha= &- \mathbf{T} \big( c^2_s \Theta  \big) \epsilon^{\alpha \beta \gamma} \partial_\beta w_\gamma     - c^2_s \Theta \mathbf{T} W^\alpha +  c^2_s \Theta  \mathbf{T} \left\{    (1-c^{-2}_s) \right\}  \epsilon^{\alpha \beta \gamma} w_\gamma \partial_\beta h 
			\\
			&- [\mathbf{T}, \mathbf{{P}}]v^\alpha_{-}     + c^2_s  (1-c^{-2}_s) \Theta \epsilon^{\alpha \beta \gamma} \partial_\beta h \mathbf{T}      w_\gamma     + c^2_s  (1-c^{-2}_s) \Theta  \epsilon^{\alpha \beta \gamma} w_\gamma \mathbf{T} ( \partial_\beta h   ). 
		\end{split}
	\end{equation*}
	Seeing \eqref{EW1}, then $\mathbf{T} \bW$ formulates by $(d\bv,dh)\cdot d\bw$ and $\bw \cdot (d\bv,dh) \cdot (d\bv,dh)$. On the other hand, $[\mathbf{T}, \mathbf{{P}}]\bv_{-} $ expresses by $(dh,d\bv)\cdot d^2\bv_{-}$ and $(d^2h,d^2\bv)\cdot d\bv_{-}$. These imply that $\mathbf{\Gamma}$ includes the following terms:
	\begin{equation*}
		(d\bv,dh)\cdot d\bw, \quad \bw \cdot (d\bv,dh) \cdot (d\bv,dh), \quad	(dh,d\bv)\cdot d^2\bv_{-},  \quad (d^2h,d^2\bv)\cdot d\bv_{-}, \quad \bw \cdot d^2h.
	\end{equation*} 
	Due to \eqref{b00}, we have
	\begin{equation}\label{b02}
		\begin{split}
			\mathbf{{P}} \left(  \Lambda^{s-2}_x (\mathbf{T} v^\alpha_{-} ) \right) =&  \Lambda^{s-2}_x \Gamma^\alpha + [\mathbf{{P}}, \Lambda^{s-2}_x ] \mathbf{T} v^\alpha_{-} .
		\end{split}
	\end{equation}
	Applying Lemma \ref{Ees} on \eqref{b02}, we obtain 
	\begin{equation}\label{Dscc}
		\begin{split}
		\| \partial_t \Lambda^{s-2}_x (\mathbf{T} \bv_{-} )  \|_{L^2_t H^1_x}+ 	\| \Lambda^{s-2}_x (\mathbf{T} \bv_{-} )  \|_{L^2_t H^2_x} \lesssim   \| \Lambda^{s-2}_x \mathbf{{\Gamma}}\|_{L^2_t L^2_x}+ \| [\mathbf{{P}}, \Lambda^{s-2}_x ] \mathbf{T} \bv_{-} \|_{L^2_t L^2_x} .
		\end{split}
	\end{equation} 
	Note that $s\in (\frac74,\frac{15}{8}]$ and $0<a<s-1$. Using H\"older's inequality along with \eqref{Dsb}, we can deduce
	\begin{equation}\label{Dsc}
		\begin{split}
			&\| \Lambda^{s-2}_x \mathbf{{\Gamma}}\|_{L^2_t L^2_x}
			\\
			\lesssim	& 
			\|(d\bv,dh)\cdot d\bw \|_{L^2_t H^{s-2}_x} + \| \bw \cdot (d\bv,dh) \cdot (d\bv,dh) \|_{L^2_t H^{s-2}_x}
			\\
			& + \| (d\bv,dh)\cdot d^2\bv_{-}\|_{L^2_t H^{s-2}_x}+ \|  (d^2h,d^2\bv)\cdot d\bv_{-} \|_{L^2_t H^{s-2}_x}+ \| \bw \cdot d^2h \|_{L^2_t H^{s-2}_x}
			\\
			\lesssim & \| \bw\|_{L^2_t H_x^{1+a}} ( 1+\| h,\bv \|_{L^\infty_t H_x^{s}} ).
		\end{split}
	\end{equation} 
We next claim that
	\begin{equation}\label{Dsd}
	\begin{split}
		\| [ \Lambda^{s-2}_x, \bv \cdot \nabla^2  ] \partial_t \bv_{-} \|_{L^2_x}  
			\lesssim \|\bv\|_{H^s_x} \| \partial_t \bv_{-} \|_{L^\infty_x} 
			+ \|\Lambda^{s-1-a}_x \bv\|_{L^\infty_x} \| \partial_t \bv_{-} \|_{H^{1+a}_x}.
	\end{split}
\end{equation}
To verify \eqref{Dsd}, due to Little-wood Paley decomposition, we obtain
\begin{equation}\label{k39}
	\begin{split}
		P_j [\Lambda_x^{s-2}, \bv\cdot \nabla^2] \partial_t \bv_{-} =& I_1+I_2+I_3,
	\end{split}
\end{equation}
where $ S_{k-1}= \sum_{j'\leq k-1} P_{j'}$ and
\begin{equation*}
	\begin{split}
	 I_1= & \sum_{|k-j|\leq 2}  P_j \left\{  \Lambda_x^{s-2}\left( P_k\bv \cdot \nabla^2 S_{k-1} \partial_t \bv_{-} \right) - P_k\bv \cdot \nabla^2 \Lambda_x^{s-2} S_{k-1} \partial_t \bv_{-} \right\},
		\\
	I_2= & \sum_{|k-j|\leq 2}  P_j \left\{ \Lambda_x^{s-2}\left( S_{k-1} \bv \cdot \nabla^2  P_k \partial_t \bv_{-} \right) - S_{k-1} \bv \cdot \nabla^2 \Lambda_x^{s-2} P_k \partial_t \bv_{-} \right\},
		\\
	I_3= & \sum_{k\geq j-1}  P_j \left\{ \Lambda_x^{s-2}\left( P_k \bv \cdot \nabla^2  P_k \partial_t \bv_{-} \right) - P_k \bv \cdot \nabla^2 \Lambda_x^{s-2} P_k \partial_t \bv_{-} \right\} .
	\end{split}
\end{equation*}
By applying Bernstein estimate and H\"older's inequality, we derive
\begin{equation}\label{k40}
	\|I_1 \|_{L^2_x} + \|I_3 \|_{L^2_x} \lesssim  2^{js}\|P_j \bv\|_{L^2_x} \| \partial_t \bv_{-} \|_{L^\infty_x} .
\end{equation}
Consider
\begin{equation*}
	\begin{split}
		 & \Lambda_x^{s-2}\left( S_{k-1} \bv \cdot \nabla^2  P_k \partial_t \bv_{-} \right) - S_{k-1} \bv \cdot \nabla^2 \Lambda_x^{s-2} P_k \partial_t \bv_{-} 
		\\
		= &  \Lambda_x^{s-2} \nabla \left( S_{k-1} \bv \cdot \nabla  P_k \partial_t \bv_{-} \right) 
		-\Lambda_x^{s-2} ( \nabla S_{k-1} \bv \cdot \nabla  P_k \partial_t \bv_{-}  )
		- S_{k-1} \bv \cdot \nabla^2 \Lambda_x^{s-2} P_k \partial_t \bv_{-} 
		\\
		=& [ \Lambda_x^{s-2} \nabla,  S_{k-1} \bv ]  \nabla  P_k \partial_t \bv_{-} -\Lambda_x^{s-2} ( \nabla S_{k-1} \bv \cdot \nabla  P_k \partial_t \bv_{-}  ).
	\end{split}
\end{equation*}
We therefore can write
\begin{equation*}
	\begin{split}
		I_2= & \sum_{|k-j|\leq 2}  P_j \left(  [ \Lambda_x^{s-2} \nabla,  S_{k-1} \bv ]  \nabla  P_k \partial_t \bv_{-}  \right)- \sum_{|k-j|\leq 2}  P_j \left\{  \Lambda_x^{s-2} ( \nabla S_{k-1} \bv \cdot \nabla  P_k \partial_t \bv_{-}  ) \right\}.
	\end{split}
\end{equation*}
Applying classical commutator estimate, Bernstein estimate and H\"older's inequality, we can obtain
\begin{equation}\label{k42}
	\begin{split}
			\|I_2 \|_{L^2_x} \lesssim &  \sum_{|k-j|\leq 2} \| \Lambda_x^{s-1}S_{k-1} \bv\|_{L^\infty_x} \cdot \|\Delta_k \nabla \partial_t \bv_{-} \|_{L^2_x} .
		\\
		\lesssim & \| \Lambda_x^{s-1-a} \bv\|_{L^\infty_x} \cdot 2^{ja}\|\Delta_j \nabla \partial_t \bv_{-} \|_{L^2_x} .
	\end{split}
\end{equation}
By using \eqref{k39}, \eqref{k40}, and \eqref{k42}, and taking $l^2$ norm for $j$, we finally obtain \eqref{Dsd}. Similarly, we can also conclude\footnote{If it concerns $\partial^2_{tt}\bv_{-}$, we can use \eqref{De0} to transfer it to $\nabla\partial_{t}\bv_{-}$ and $\nabla^2\bv_{-}$.}
	\begin{equation}\label{Dsd2}
		\begin{split}
			\| [\Lambda^{s-2}_x, \bv \partial^2_{tt} ] \partial_t \bv_{-} \|_{L^2_x}  
				\lesssim \|\bv\|_{H^s_x} \| \partial_t \bv_{-} \|_{L^\infty_x} 
				+ \|\Lambda^{s-1-a}_x \bv\|_{L^\infty_x} \| \partial_t \bv_{-} \|_{H^{1+a}_x}.
		\end{split}
	\end{equation}
	and
\begin{equation}\label{Dsd3}
	\begin{split}
		\| [ \Lambda^{s-2}_x, \bv \cdot \partial_t \nabla ] \partial_t \bv_{-} \|_{L^2_x}  
			\lesssim \|\bv\|_{H^s_x} \| \partial_t \bv_{-} \|_{L^\infty_x} 
			+ \|\Lambda^{s-1-a}_x \bv\|_{L^\infty_x} \| \partial_t \bv_{-} \|_{H^{1+a}_x}.
	\end{split}
\end{equation}
Since \eqref{Dsd}, \eqref{Dsd2}, and \eqref{Dsd3}, we can obtain
	\begin{equation}\label{Dse}
	\begin{split}
		\| [\Lambda^{s-2}_x, \mathbf{{P}} ] \mathbf{T} \bv_{-} \|_{L^2_t L^2_x}  
			\lesssim & \|\bv\|_{H^s_x} ( \| \partial_t \bv_{-} \|_{L^2_t L^\infty_x} + \|\bv \cdot \nabla \bv_{-} \|_{L^2_t L^\infty_x} )
			\\
			& + \|\Lambda^{s-1-a}_x \bv\|_{L^\infty_t L^\infty_x} 
			( \| \partial_t \bv_{-} \|_{L^2_t H^{1+a}_x} + \|\bv \cdot \nabla \bv_{-} \|_{L^2_t H^{1+a}_x} ) .
	\end{split}
\end{equation}
By applying Soblev's imbeddings and \eqref{Dsb}, we have
\begin{equation}\label{Dsf}
	\begin{split}
		 \| \partial_t \bv_{-} \|_{L^2_t L^\infty_x} + \|\bv \cdot \nabla \bv_{-} \|_{L^2_t L^\infty_x}   \lesssim  & \| \bw\|_{L^2_t H_x^{1+a}} ( 1+\| h,\bv \|_{L^\infty_t H_x^{s}} ).
	\end{split}
\end{equation}
Due to Soblev's imbeddings again, we derive 
	\begin{equation}\label{Dsg}
	\begin{split}
		 \|\Lambda^{s-1-a}_x \bv\|_{L^\infty_t L^\infty_x} \lesssim \| \bv \|_{L^\infty_t H_x^{s}} .
	\end{split}
\end{equation}
Using \eqref{Dsb} again, we obtain
\begin{equation}\label{Dsh}
	\begin{split}
		\| \partial_t \bv_{-} \|_{L^2_t H^{1+a}_x} + \|\bv \cdot \nabla \bv_{-} \|_{L^2_t H^{1+a}_x} 
			\lesssim & \| \bw\|_{L^2_t H_x^{1+a}} ( 1+\| h,\bv \|_{L^\infty_t H_x^{s}} ).
	\end{split}
\end{equation}
Gathering the estimates \eqref{Dse}, \eqref{Dsf}, \eqref{Dsg}, and \eqref{Dsh}, it yields
	\begin{equation}\label{Dsi}
	\begin{split}
		\| [\mathbf{{P}}, \Lambda^{s-2}_x ] \mathbf{T} \bv_{-} \|_{L^2_t L^2_x}  
		\lesssim \|\bv\|_{L^\infty_t H^s_x} \| \bw\|_{L^2_t H_x^{1+a}} ( 1+\| h,\bv \|_{L^\infty_t H_x^{s}} ).
	\end{split}
\end{equation}
Combining inequalities \eqref{Dscc}, \eqref{Dsc}, and \eqref{Dsi}, we obtain \eqref{etaq}. This concludes the proof of this lemma.
\end{proof}

\subsection{Energy estimates}
We will first introduce a classical energy estimate for quasi-linear hyperbolic systems, as presented in Li and Qin's book \cite{LQ}.
\begin{Lemma}\label{DW3}\cite{LQ}
Let $a\geq 0$ and $(h,\bv,\bw)$ be a solution of \eqref{wrt} on $[0,t]\times \mathbb{R}^2$.	
Assume
\begin{equation}\label{aap}
	\| h, \bv\|_{L^\infty_{[0,t]} L^\infty_x} \leq 2+C_0.
\end{equation}
Then we have the integral inequality
	\begin{equation}\label{ew00}
		\| (h , \bv) \|^2_{H^{a}_x} \lesssim \| ( h_0 , \bv_0 )\|^2_{H^{a}_x}+ \int^t_0 \|(dh,d\bv)\|_{L^\infty_x} \| (h,\bv) \|^2_{H^{a}_x}  d\tau,
	\end{equation}
	and the energy inequality
	\begin{equation}\label{ew02}
		\| (h, \bv) \|_{H^{a}_x} \lesssim  \| ( h_0, \bv_0 ) \|_{H^{a}_x} \exp \left( \int^t_0   \|(dh,d\bv)\|_{L^\infty_x}  d\tau \right) .
	\end{equation}
\end{Lemma}
Next, we will derive the total energy estimates concerning enthalpy, velocity, and vorticity as follows.
\begin{theorem}\label{DW4}
	Let $\frac74 \leq s' \leq s \leq 2$ and $(h,\bv,\bw)$ be a solution of \eqref{wrt} on $[0,t]\times \mathbb{R}^2$. Assume \eqref{aap} hold. Denote
	\begin{equation*}\label{ew04}
	\begin{split}
	E(t)=\| h \|_{H^{s}_x}+\| \bv \|_{H^{s}_x}+ \|  \bw \|_{H^{s'-\frac14}_x}
	+ \| \nabla \bw \|_{L^8_x}.
	\end{split}
\end{equation*}
When $0< t\leq 1$, We have
\begin{equation*}\label{ew06}
	\begin{split}
		E(t) 
		\lesssim & (1+E^5(0))
		\exp \left( \big[ 1+  E^3(0)\int^t_0 \|(d\bv,dh)\|_{L^\infty_x} d\tau \big] \cdot \big[ \int^t_0  \|(d\bv,dh)\|_{L^\infty_x}  d\tau +1 \big] \right).
	\end{split}
\end{equation*}
\end{theorem}
\begin{proof}
By using \eqref{ew00}, we only need to bound $\|  \bw \|_{H^{s_0-\frac14}_x}$ and $\| \nabla \bw \|_{L^8_x}$. To be simple, we set
\begin{equation*}
	\mathring{\bw}=(w^1,w^2). 
\end{equation*}

\textbf{Step 1: $\|  \bw \|_{H^{s'-\frac14}_x}$}. For $\bw=(w^0,\mathring{\bw})$, we get
\begin{equation*}\label{ew01}
	\begin{split}
			\|\bw\|_{H^{s'-\frac14}_x}=  & \|\bw\|_{L^{2}_x} + \|\bw\|_{\dot{H}^{s'-\frac14}_x} .
	\end{split}
\end{equation*}
By elliptic estimate, one can see
\begin{equation*}\label{ew08}
	\begin{split}
		\|  \bw \|_{\dot{H}^{s'-\frac14}_x}= & \|w^0 \|_{\dot{H}^{s'-\frac14}_x}+ \|\mathring{\bw} \|_{\dot{H}^{s'-\frac14}_x}
		\\
		= & \|w^0 \|_{\dot{H}^{s'-\frac14}_x}+ \| \textrm{curl} \mathring{\bw}  \|_{\dot{H}^{s'-\frac54}_x}+ \| \textrm{div} \mathring{\bw} \|_{\dot{H}^{s'-\frac54}_x}.
	\end{split}
\end{equation*}
Due to \eqref{Vor}, we have
\begin{equation}\label{ew10}
	\partial_\alpha w^\alpha =0.
\end{equation}
Therefore, \eqref{ew10} implies that
\begin{equation*}
	\textrm{div} \mathring{\bw}= \partial_\alpha w^\alpha- \partial_t w^0=-\partial_t w^0.
\end{equation*}
Due to \eqref{EW0}, we have
\begin{equation}\label{ew12}
	\partial_t w^\alpha + (v^0)^{-1} v^i \partial_i w^\alpha= (v^0)^{-1} w^\kappa \partial^\alpha v_\kappa - (v^0)^{-1} w^\alpha \partial^\kappa v_\kappa.
\end{equation}
Then we can deduce that
\begin{equation}\label{ew13}
	\textrm{div} \mathring{\bw}=  (v^0)^{-1} v^i \partial_i w^0+ (v^0)^{-1} w^\kappa \partial_t v_\kappa + (v^0)^{-1} w^0 \partial^\kappa v_\kappa.
\end{equation}
On the other hand, a direct calculation shows
\begin{equation}\label{ew14}
	\begin{split}
		\epsilon_{\alpha i 0} \epsilon^{\alpha \beta \gamma} \partial_\beta w_\gamma=& (\delta^\beta_i \delta^\gamma_0 - \delta^\beta_0 \delta^\gamma_i) \partial_\beta w_\gamma
		\\
		=& \partial_i w_0 - \partial_t w_i.
	\end{split}
\end{equation}
By use of \eqref{ew12} and \eqref{ew14}, we have
\begin{equation}\label{ew16}
	\begin{split}
	\partial_i w_0  =& 	\epsilon_{\alpha i 0} \epsilon^{\alpha \beta \gamma} \partial_\beta w_\gamma
		- (v^0)^{-1} v^j \partial_j w_i + (v^0)^{-1} w^\kappa \partial_i v_\kappa - (v^0)^{-1} w_i \partial^\kappa v_\kappa  
		\\
		=& 	\epsilon_{\alpha i 0}  W^\alpha - (1-c^{-2}_s) ( w_i \partial_t h- w_0 \partial_i h )
		- (v^0)^{-1} v^j \partial_j w_i 
		\\
		& + (v^0)^{-1} w^\kappa \partial_i v_\kappa - (v^0)^{-1} w_i \partial^\kappa v_\kappa.
	\end{split}
\end{equation}
Combing \eqref{ew13} and \eqref{ew16}, we further get
\begin{equation}\label{ew18}
	\begin{split}
		\textrm{div} \mathring{\bw}= & (v^0)^{-1} v^i \epsilon_{\alpha i 0} W^\alpha- (v^0)^{-2} v^i v^j \partial_j w_i +(v^0)^{-2} v^i  w^\kappa \partial_i v_\kappa  
		\\
		& - (v^0)^{-2}  v^i w_i \partial^\kappa v_\kappa +
		(v^0)^{-1} w^\kappa \partial_t v_\kappa + (v^0)^{-1} w^0 \partial^\kappa v_\kappa
		\\
		& - (1-c^{-2}_s) ( w_i \partial_t h- w_0 \partial_i h ).
	\end{split}
\end{equation}
From \eqref{ew18}, we know
\begin{equation}\label{ew20}
	\begin{split}
		\| \mathrm{div} \mathring{\bw} \|_{H^{s'-\frac54}_x} \leq  & C\| \bW \|_{H^{s'-\frac54}_x} +(\frac{|\mathring{\bv}|}{v^0})^2 \|\mathring{\bw} \|_{H^{s'-\frac54}_x}  + C\| \bw \cdot (dh, d\bv)\|_{H^{s'-\frac54}_x} .
	\end{split}
\end{equation}
By the definition \eqref{EW2} for $\bW$, we can see
\begin{equation}\label{ew22}
	\begin{split}
		\| \mathrm{curl} \mathring{\bw} \|_{H^{s'-\frac54}_x} \leq  & \| \bW \|_{H^{s'-\frac54}_x} +  \| \bw \cdot  (dh,d\bv)\|_{H^{s'-\frac54}_x} .
	\end{split}
\end{equation}
Taking advantage of \eqref{ew20} and \eqref{ew22}, we get
\begin{equation}\label{ew24}
	\left\{  1- (\frac{|\mathring{\bv}|}{v^0})^2  \right\} \| \mathring{\bw} \|_{\dot{H}^{s'-\frac14}_x} \lesssim  \| \bW \|_{H^{s'-\frac54}_x}+ \| \bw \cdot  (dh,d\bv)\|_{H^{s'-\frac54}_x}.
\end{equation}
By use of \eqref{rsv} and \eqref{aap}, then \eqref{ew24} becomes
\begin{equation}\label{ew25}
	\| \mathring{\bw} \|_{\dot{H}^{s'-\frac14}_x} \lesssim  \| \bW \|_{H^{s'-\frac54}_x}+ \| \bw \cdot (dh,d\bv)\|_{H^{s'-\frac54}_x}.
\end{equation}
From \eqref{ew16} and \eqref{ew25}, we obtain
\begin{equation}\label{ew26}
	\begin{split}
		\| w^0 \|_{\dot{H}^{s'-\frac14}_x} \lesssim  &  \| \bW \|_{H^{s'-\frac54}_x} + \|\mathring{\bw} \|_{H^{s'-\frac54}_x}  + \| \bw \cdot (dh, d\bv) \|_{H^{s'-\frac54}_x} 
		\\
		\lesssim  &  \| \bW \|_{H^{s'-\frac54}_x} + \| \bw \cdot (dh, d\bv)\|_{H^{s'-\frac54}_x} .
	\end{split}
\end{equation}
Combing \eqref{ew25} and \eqref{ew26}, we therefore get
\begin{equation*}\label{ew28}
	\begin{split}
		\| \bw \|_{{H}^{s'-\frac14}_x} \lesssim  & \| \bw \|_{L^2_x}  +  \| \bW \|_{H^{s'-\frac54}_x} + \| \bw \cdot (dh,d\bv)\|_{H^{s'-\frac54}_x} .
	\end{split}
\end{equation*}
Multiplying $w^\alpha$ on \eqref{ew12}, integrating it on $[0,t]\times \mathbb{R}^2$, and summing $\alpha$ over $0$ to $2$, it yields
\begin{equation}\label{ew30}
	\begin{split}
		\|  \bw \|^2_{L^2_x} \lesssim & \|  \bw_0 \|^2_{L^2_x} + \int^t_0 \left( \|(v^0)^{-1} d\bv\|_{L^\infty_x} + \|  (v^0)^{-2} \bv\cdot \nabla \bv \|_{L^\infty_x} \right) \|  \bw \|^2_{L^2_x} d\tau
	\\
	\lesssim & \|  \bw_0 \|^2_{L^2_x} + \int^t_0 \|d\bv\|_{L^\infty_x} \|  \bw \|^2_{L^2_x}  d\tau.
	\end{split}
\end{equation}
By \eqref{EW1}, we get
\begin{equation}\label{ew32}
	\begin{split}
			 (\partial_t + (v^0)^{-1}v^i\partial_i)   W^\alpha 
			=
			& - (v^0)^{-1} \left\{ \epsilon^{\alpha \beta \gamma}  ( c^{-2}_s-1) v^\kappa \partial_\kappa w_\gamma \partial_\beta h 
			+   \epsilon^{\alpha \beta \gamma}  w_\gamma \partial_\beta \left[    ( c^{-2}_s-1) v^\kappa \right]  \partial_\kappa h \right \}
			\\
			& \ -    \epsilon^{\alpha \beta \gamma} (v^0)^{-1} \partial_\beta w_\gamma  \partial^\kappa v_\kappa 
			+2 \epsilon^{\alpha \beta \gamma}  (v^0)^{-1} c^{-3}_s c'_s v^\kappa w_\gamma \partial_\kappa h \partial_\beta h
				\\
			& \ \ -\epsilon^{\alpha \beta \gamma} (v^0)^{-1} \partial_\beta v^\kappa \partial_\kappa w_\gamma
			+ \epsilon^{\alpha \beta \gamma} (v^0)^{-1}\partial_\beta w^\kappa \partial_\gamma v_\kappa  .
	\end{split}
\end{equation}
Seeing the right hand side of \eqref{ew32}, it includes two types terms:
\begin{equation*}
	(dh,d\bv)  \cdot d\bw, \quad \bw \cdot (dh,d\bv)  \cdot (dh,d\bv).
\end{equation*}
Operating $\Lambda_x^{s'-\frac54}$ on \eqref{ew32}, multiplying $\Lambda_x^{s'-\frac54} W^\alpha$, integrating it om $[0,t] \times \mathbb{R}^2$, and summing $\alpha$ over $0$ to $2$, we can derive that
\begin{equation}\label{ew34}
	\begin{split}
		 \| \bW \|^2_{{H}^{s'-\frac54}_x} 
		\lesssim &  \int^t_0 \| d\bv\|_{L^\infty_x} \| \bW \|^2_{{H}^{s'-\frac54}_x} d\tau
		+ \int^t_0 \| (dh, d\bv) \cdot d \bw\|_{{H}^{s'-\frac54}_x} \| \bW \|_{{H}^{s'-\frac54}_x} d\tau
		\\
		& 
		+\| \bW_0 \|^2_{{H}^{s'-\frac54}} + \int^t_0 \| \bw \cdot (dh,d\bv)  \cdot (dh,d\bv) \|_{{H}^{s'-\frac54}_x} \| \bW \|_{{H}^{s'-\frac54}_x} d\tau .
	\end{split}
\end{equation}
On the right hand side of \eqref{ew34}, we have
\begin{equation}\label{ew36}
	\begin{split}
		\| \bW_0 \|_{{H}^{s'-\frac54}} \lesssim \| \bw_0 \|_{{H}^{s'-\frac14}}+  \|\bw_0 \cdot \nabla h_0 \|_{{H}^{s'-\frac54}}\lesssim \| \bw_0 \|_{{H}^{s'-\frac14}_x}+\| (h_0,\bv_0) \|^2_{{H}^{s}}.
	\end{split}
\end{equation}
and
\begin{equation}\label{ew38}
	\begin{split}
		\| \bW \|_{{H}^{s'-\frac54}} \lesssim \| \bw \|_{{H}^{s'-\frac14}}+  \|\bw \cdot d h \|_{{H}^{s'-\frac54}}\lesssim \| \bw \|_{{H}^{s'-\frac14}_x}+\| (h,\bv) \|^2_{{H}_x^{s}}.
	\end{split}
\end{equation}
Utilizing \eqref{ew34}, \eqref{ew36}, \eqref{ew38}, and H\"older's inequality, we obtain
\begin{equation}\label{ew40}
	\begin{split}
		\| \nabla \bw \|^2_{{H}^{s'-\frac54}_x} 
		\lesssim & \| \bw_0 \|^2_{{H}^{s'-\frac14}_x}+\| (h_0,\bv_0) \|^4_{{H}^{s}}+ \int^t_0 \| (dh, d\bv)\|_{L^\infty_x} \| \bw \|^2_{{H}^{s'-\frac14}_x} d\tau
	\\
	& 	+ \int^t_0 \| (h, \bv)\|_{H^s_x} \|\nabla \bw\|_{L^8_x} \| \bw \|_{{H}^{s'-\frac14}_x} d\tau 
 + \int^t_0 \| (h, \bv)\|^3_{H^s_x} \| \bw \|^2_{{H}^{s'-\frac14}_x} d\tau   .
	\end{split}
\end{equation}
\qquad \textbf{Step 2: $\|  \nabla \bw \|_{L^{8}_x}$}. By Hodge's decomposition, we have
\begin{equation}\label{ew42}
	\begin{split}
		\|  \nabla \bw \|_{L^{8}_x}= \|  \nabla w^0 \|_{L^{8}_x}+ \|  \nabla \mathring{\bw} \|_{L^{8}_x}.
	\end{split}
\end{equation}
By using \eqref{ew16}, we can obtain
\begin{equation}\label{ew44}
	\begin{split}
		\|  \nabla w^0 \|_{L^{8}_x} \lesssim  & \| \bW \|_{L^{8}_x} +\|\nabla \mathring{\bw} \|_{L^{8}_x}  + \| \bw \cdot (dh, d\bv)\|_{L^{8}_x} .
	\end{split}
\end{equation}
Due to \eqref{ew18}, it yields
\begin{equation}\label{ew46}
	\begin{split}
		\|  \mathrm{div} \mathring{\bw} \|_{L^{8}_x} \leq  & C\| \bW \|_{L^{8}_x} +(\frac{|\mathring{\bv}|}{v^0})^2\|\nabla \mathring{\bw} \|_{L^{8}_x}  + C\| \bw \cdot (dh, d\bv)\|_{L^{8}_x} .
	\end{split}
\end{equation}
By using \eqref{EW2}, it follows
\begin{equation}\label{ew48}
	\begin{split}
		\| \mathrm{curl} \mathring{\bw} \|_{L^{8}_x} \leq  & \| \bW \|_{L^{8}_x}  + \| \bw \cdot (dh, d\bv)\|_{L^{8}_x}.
	\end{split}
\end{equation}
Taking advantage of \eqref{ew46} and \eqref{ew48}, we get
\begin{equation}\label{ew50}
	\begin{split}
	\left\{ 1- (\frac{|\mathring{\bv}|}{v^0})^2 \right\}	\|  \nabla \mathring{\bw} \|_{L^{8}_x} \leq  & C\| \bW \|_{H^{s'-\frac54}_x}  + C \| \bw \cdot (dh, d\bv)\|_{L^{8}_x} .
	\end{split}
\end{equation}
By using \eqref{rsv} and \eqref{aap}, so \eqref{ew50} can be calculated by
\begin{equation}\label{ew52}
	\| \nabla \mathring{\bw} \|_{L^{8}_x} \lesssim  \| \bW \|_{L^{8}_x}+ \| \bw \cdot (dh,d\bv)\|_{L^{8}_x}.
\end{equation}
Summing up the outcome \eqref{ew42}, \eqref{ew44}, \eqref{ew52}, we therefore derive
\begin{equation*}\label{ew54}
	\| \nabla {\bw} \|_{L^{8}_x} \lesssim  \| \bW \|_{L^{8}_x}+ \| \bw \cdot (dh,d\bv)\|_{L^{8}_x}.
\end{equation*}
Multiplying $W^\alpha | \bW|^{6}$ on \eqref{ew32} and integrating it on $\mathbb{R}^2$, and summing $\alpha$ over $0$ to $2$, we can derive that
\begin{equation}\label{ew56}
	\begin{split}
	 \frac{d}{dt}   \| \bW \|_{L_x^8}^8  \lesssim &  \| d\bv\|_{L^\infty_x} \| \bW \|^8_{L^8_x} 
		+ \| d\bv \cdot d\bw\|_{L^8_x} \| \bW \|^7_{L^8_x} 
		\\
		& 
		+ \| \bw \cdot (d \bu,dh) \cdot (d \bu,dh) \|_{L^8_x} \| \bW \|^7_{L^8_x} .
	\end{split}
\end{equation}
Due to \eqref{ew56}, it yields
\begin{equation}\label{ew58}
	\begin{split}
		\frac{d}{dt}   \| \bW \|_{L_x^8}^2  \lesssim &  \| d\bv\|_{L^\infty_x} \| \bW \|^2_{L^8_x} 
		+ \| d\bv \cdot d\bw\|_{L^8_x} \| \bW \|_{L^8_x} 
		\\
		& 
		+ \| \bw \cdot (d \bu,dh) \cdot (d \bu,dh) \|_{L^8_x} \| \bW \|_{L^8_x} .
	\end{split}
\end{equation}
Integrating \eqref{ew58} from $0$ to $t$, we can deduce that
\begin{equation}\label{ew60}
	\begin{split}
		\| \bW \|_{L_x^8}^2 - \| \bW_0 \|_{L_x^8}^2  \lesssim &  \int^t_0 \| d\bv\|_{L^\infty_x} \| \bW \|^2_{L^8_x} d\tau
		+ \int^t_0 \| d\bv \cdot d\bw\|_{L^8_x} \| \bW \|_{L^8_x} d\tau
		\\
		& 
		+ \int^t_0  \| \bw \cdot (d \bu,dh) \cdot (d \bu,dh) \|_{L^8_x} \| \bW \|_{L^8_x} d\tau.
	\end{split}
\end{equation}
On the left side of \eqref{ew60}, we have
\begin{equation}\label{ew62}
	\| \bW_0 \|_{L_x^8}^2 \leq \| \nabla \bw_0 \|_{L_x^8}^2+ \| \bw_0 \cdot \nabla h_0 \|_{L_x^8}^2 \leq \| \nabla \bw_0 \|_{L_x^8}^2 + \| (h_0,\bv_0 ) \|^4_{H^s} .
\end{equation}
Combining \eqref{ew60}, \eqref{ew62}, and using \eqref{EW2}, it turns out
\begin{equation}\label{ew64}
	\begin{split}
		 & \| \nabla \bw \|^2_{L^8_x} - \| \bw \cdot (d \bv,dh) \|^2_{L^8_x}- (\| \nabla \bw_0 \|_{L_x^8}^2 + \| (h_0,\bv_0 ) \|^4_{H^s})
		\\
		\lesssim 
		& + \int^t_0 \| d\bv\|_{L^\infty_x} ( \| \nabla \bw \|^2_{L^8_x} + \| \bw \cdot (d \bv,dh) \|^2_{L^8_x}) d\tau
		\\
		&+ \int^t_0 \| d\bv \cdot d \bw\|_{L^8_x} ( \| \nabla \bw \|_{L^8_x} + \| \bw \cdot (d \bv,dh) \|_{L^8_x}) d\tau
		\\
		& 
		+ \int^t_0 \| \bw \cdot (d \bv,dh) \cdot (d \bv,dh) \|_{L^8_x} ( \| \nabla \bw\|_{L^8_x} + \| \bw \cdot (d \bv,dh) \|_{L^8_x}) d\tau .
	\end{split}
\end{equation}
For $\frac74\leq s' \leq s\leq 2$, using H\"older's inequality, we can derive
\begin{equation*}
	\begin{split}
		\| \bw \cdot (d \bv,dh) \|_{L^8_x} \leq & \| \bw \cdot (d \bv,dh) \|_{H^{\frac34}_x} 
		\\
		\leq  & \| \bw \|_{H_x^{s_0-\frac34}} \| ( d\bv,dh) \|_{H^{\frac34}_x} .
	\end{split}
\end{equation*}
Using interpolation formula and Young's inequality, we can estimate $\| \bw \cdot (d \bv,dh) \|_{L^8_x} $ by 
\begin{equation}\label{ew66}
	\begin{split}
		\| \bw \cdot (d \bv,dh) \|_{L^8_x} 
		\leq & \| \bw \|^{\frac23}_{H_x^{s_0-1}} \| \bw \|^{\frac13}_{H_x^{s_0-1/4}} \| ( \bv,h) \|_{H^{s}_x} 
		\\
		\leq & \frac{1}{4} \| \bw \|_{H_x^{s_0-1/4}} + C\| \bw \|^{\frac23}_{H_x^{s_0-1}}\| ( \bv,h) \|^{\frac32}_{H^{s}_x}
		\\ 
		\leq & \frac{1}{4} \| \bw \|_{H_x^{s_0-1/4}} + C\| ( \bv,h) \|^{\frac52}_{H^{s}_x} .
	\end{split}
\end{equation}
Due to \eqref{ew64} and \eqref{ew66}, we can get
\begin{equation}\label{ew68}
	\begin{split}
		 & \| \nabla \bw \|^2_{L^8_x} - \frac{1}{4} \| \bw \|^2_{H_x^{s_0-\frac14}}  
		\\
		\lesssim & \| \nabla \bw_0 \|_{L_x^8}^2 + \| (h_0,\bv_0 ) \|^4_{H^s}+ \| ( \bv,h) \|^{5}_{H^{s}_x}
		\\
		& + \int^t_0 \| d\bv\|_{L^\infty_x} ( \| \nabla \bw \|^2_{L^8_x} + \| \bw \cdot (d \bv,dh) \|^2_{L^8_x}) d\tau
		\\
		&+ \int^t_0 \| d\bv \cdot d \bw\|_{L^8_x} ( \| \nabla \bw \|_{L^8_x} + \| \bw \cdot (d \bv,dh) \|_{L^8_x}) d\tau
		\\
		& 
		+ \int^t_0 \| \bw \cdot (d \bv,dh) \cdot (d \bv,dh) \|_{L^8_x} ( \| \nabla \bw\|_{L^8_x} + \| \bw \cdot (d \bv,dh) \|_{L^8_x}) d\tau .
	\end{split}
\end{equation}
By \eqref{ew68}, using H\"older's inequality, we finally obtain
\begin{equation}\label{ew70}
	\begin{split}
		  & \| \nabla \bw \|^2_{L^8_x} - \frac{1}{4} \| \bw \|^2_{H_x^{s_0-\frac14}} 
		\\
		\lesssim & \| \nabla \bw_0 \|_{L_x^8}^2 + \| (h_0,\bv_0 ) \|^4_{H^s}+ \| ( \bv,h) \|^{5}_{H^{s}_x}
		\\
		& + \int^t_0 \| d\bv\|_{L^\infty_x} ( \| \nabla \bw \|^2_{L^8_x} + \|\bw\|^2_{{H}^{s_0-\frac14}_x} \|(\bv,h) \|^2_{{H}^{s}_x} ) d\tau
		\\
		&+ \int^t_0 \| d\bv \|_{L^\infty_x} \|\nabla \bw\|_{L^8_x} ( \| \nabla \bw \|_{L^8_x} + \|\bw\|_{{H}^{s_0-\frac14}_x} \|(\bv,h) \|_{{H}^{s}_x}) d\tau
		\\
		& 
		+ \int^t_0 \|\bw\|_{{H}^{s_0-\frac14}_x} \|(\bv,h) \|^2_{{H}^{s}_x}  ( \| \nabla \bw \|_{L^8_x} + \|\bw\|_{{H}^{s_0-\frac14}_x} \|(\bv,h) \|_{{H}^{s}_x} ) d\tau .
	\end{split}
\end{equation}
Summing up our outcome \eqref{ew00}, \eqref{ew30}, \eqref{ew40}, and \eqref{ew70}, we have
\begin{equation}\label{ew72}
	\begin{split}
		E(t) \lesssim & E(0)+ \| ( \bv,h) \|^{5}_{H^{s}_x}+ \| ( \bv_0,h_0) \|^{4}_{H^{s}_x}
		\\
		& +\int^t_0 ( 1+ \|(d\bv,dh)\|_{L^\infty_x} ) (1+  \|(\bv,h)\|^3_{{H}_x^{s}} )  E(\tau)d\tau
		\\
		\lesssim & (1+ E^5(0)) \exp\left( 5\int^t_0 \|(dh,d\bv)\|_{L^\infty_x} d\tau \right)
		\\
		& +\int^t_0 ( 1+ \|(d\bv,dh)\|_{L^\infty_x} ) (1+  \|(\bv,h)\|^3_{{H}_x^{s}} )  E(\tau)d\tau.
	\end{split}
\end{equation}
Referring to \eqref{ew72} and applying Gronwall's inequality along with \eqref{ew02} for $t\leq 1$, we obtain
\begin{equation*}
	\begin{split}
		E(t) \lesssim & \left\{ (1+ E^5(0)) \exp( \int^t_0 \|(dh,d\bv)\|_{L^\infty_x} d\tau ) \right\} 
		\\
		& \times \exp \left( \int^t_0 ( 1+ \|(d\bv,dh)\|_{L^\infty_x} ) (1+  \|(\bv,h)\|^3_{{H}_x^{s}} )  d\tau \right)
		\\
		\lesssim & (1+E^5(0))\exp( 5\int^t_0 \|(dh,d\bv)\|_{L^\infty_x} d\tau ) 
		\\
		& \times \exp \left( \big[ 1+  E^3(0)\int^t_0 \|(d\bv,dh)\|_{L^\infty_x} d\tau \big] \cdot \int^t_0 ( 1+ \|(d\bv,dh)\|_{L^\infty_x} )   d\tau \right)
		\\
		\lesssim & (1+E^5(0))
		\exp \left( 5\big[ 1+  E^3(0)\int^t_0 \|(d\bv,dh)\|_{L^\infty_x} d\tau \big] \cdot \big[ \int^t_0  \|(d\bv,dh)\|_{L^\infty_x}  d\tau +1 \big] \right).
	\end{split}
\end{equation*}
Therefore, we complete the proof of Lemma \ref{DW4}.
\end{proof}

Moreover, we also need the following energy estimates of some lower-order terms.
\begin{Lemma}\label{yux}
Assume $s \in (\frac74,\frac{15}{8}]$. Let $(h,\bv,\bw)$ be a solution of \eqref{wrt} on $[0,T]\times \mathbb{R}^2$. Let $\mathcal{D}$ and $\bQ$ be stated in \eqref{wrt0}-\eqref{wrt1}. Then the following estimate holds:
	\begin{equation}\label{YYE}
		\| \mathcal{D}, \bQ \|_{ H_x^{s-1}} \lesssim \| dh, d\bv \|_{L_x^\infty} \| h,\bv \|_{H_x^{s}}.
	\end{equation}
\end{Lemma}
\begin{proof}
	Seeing \eqref{wrt0} and \eqref{wrt1}, $\mathcal{D}$ and $\bQ$ are formulated by $(dh,d\bv)\cdot (dh,d\bv)$. By using Lemma \ref{jiaohuan} and the equation \eqref{REEf}, we can get
	\begin{equation*}
		\| \mathcal{D}, \bQ  \|_{H_x^{s-1}} \lesssim  \| dh,d \bv \|_{L_x^\infty} \| dh, d\bv \|_{H_x^{s-1}}\lesssim \| dh, d\bv \|_{L_x^\infty} \| h,\bv \|_{H_x^{s}}.
	\end{equation*}
Therefore, the estimate \eqref{YYE} holds. We have finished the proof of this lemma.
\end{proof}
\subsection{Uniqueness of solutions}
We now present two results regarding the uniqueness of the solution, which can be derived directly from theorem \ref{DW4}.
\begin{corollary}\label{uniq}
	Assume $\frac74<s_0\leq s \leq \frac{15}{8}$ and \eqref{HEw} hold. Consider the Cauchy problem \eqref{wrt} with the initial data $(h_0, \bv_0, \bw_0,$ $\nabla \bw_0) \in H^{s} \times H^{s} \times H^{s_0-\frac14} \times L^8$. If there exists a solution $(h, \bv, \bw)$ for \eqref{wrt} and
	$(h,\bv) \in C([0,T],H_x^s) $, $\bw \in C([0,T],H_x^{s_0-\frac14}) $, $\nabla \bw\in C([0,T],L^8_x)$ and $(dh, d\bv) \in {L^4_{[0,T]} L^\infty_x}$, then it's unique.
\end{corollary}
\begin{corollary}\label{uniq2}
	Assume $s\in (\frac74,\frac{15}{8}]$ and \eqref{HEw} hold. Consider the Cauchy problem \eqref{wrt} with the initial data $(h_0, \bv_0, \bw_0,$ $\nabla \bw_0) \in H^{s} \times H^{s} \times H^{\frac32} \times L^8$. If there exists a solution $(h, \bv, \bw)$ for \eqref{wrt} and
	$(h,\bv) \in C([0,T^*],H_x^s) $, $\bw \in C([0,T^*],H_x^{\frac32}) $, $\nabla \bw\in C([0,T^*],L^8_x)$ and $(dh, d\bv) \in {L^4_{[0,T^*]} L^\infty_x}$, then it's unique.
\end{corollary}

\section{Proof of Theorem \ref{dingli}}\label{ptr}
Due to Corollary \ref{uniq}, it suffices to prove the existence and Strichartz estimates of solutions. We will reduce the proof to a conclusion for small, smooth, compactly supported data.  
\subsection{Reduction to smooth initial data}\label{sec4.1}
We state a proposition as follows.
\begin{proposition}\label{p3}
		Assume $\frac74<s_0\leq s \leq \frac{15}{8}$ and $\delta\in (0,s-\frac74)$. For each $M_0>0$, there exists $T, M_1, M_2>0$ (depending on $C_0,c_0,s, s_0$ and $M_0$) such that, for each smooth initial data
	$(h_0, {\bv}_0, \bw_0)$, which satisfies
	\begin{equation}\label{ig1}
		\begin{split}
			&\|h_0\|_{H^{s}}+\|\bv_0\|_{H^{s}} + \| \bw_0 \|_{H^{s_0-\frac14}} + \| \nabla \bw_0\|_{L^{8}} \leq M_0,
		\end{split}
	\end{equation}
	there exists a smooth solution $(h,\bv,\bw)$ to \eqref{wrt} on $[0,T] \times \mathbb{R}^2$ satisfying
	\begin{equation}\label{e9}
		\begin{split}
			& \|h, \bv\|_{L^\infty_{[0,T]}H_x^{s}}+\|\bw\|_{L^\infty_{[0,T]}H_x^{s_0-\frac14}} +\|\nabla \bw \|_{L^\infty_{[0,T]}L_x^{8}} \leq M_1,
			\\
			& \|h, \bv \|_{L^\infty_{[0,T]\times \mathbb{R}^2}} \leq 1+C_0.
		\end{split}
	\end{equation}
	Furthermore, the solution satisfies
\begin{enumerate}
		\item  dispersive estimate for $h$, $\bv$ and $\bv_+$
	\begin{equation}\label{p303}
		\|d h, d \bv_{+}, d \bv\|_{L^4_{[0,T]} C^\delta_x} \leq M_2,
	\end{equation}

		\item  Let $f$ satisfy the equation \eqref{linear}. For each $1 \leq r \leq s+1$, the Cauchy problem \eqref{linear} is well-posed in $H_x^r \times H_x^{r-1}$. Furthermore, the following energy estimate
	\begin{equation}\label{p304}
		\|f\|_{L^\infty_{[0,T]} H^r_x} + \|\partial_t f\|_{L^\infty_{[0,T]} H^{r-1}_x}  \leq C_{M_0} ( \| f_0\|_{H^r}+ \| f_1\|_{H^{r-1}}),
	\end{equation}
	and Strichartz estimates
	\begin{equation}\label{305}
		\|\left< \nabla \right>^k f\|_{L^2_{[0,T]} L^\infty_x} \leq  C_{M_0}(\| f_0\|_{H^r}+ \| f_1\|_{H^{r-1}}),\ \ \ k<r-1,
	\end{equation}
	hold. The same estimates hold with $\left< \nabla \right>^k$ replaced by $\left< \nabla \right>^{k-1}d$. Above, $C_{M_0}$ is also a constant depending on $C_0, c_0, s, s_0$ and $M_0$.
\end{enumerate}
\end{proposition}
Indeed, Theorem \ref{dingli} is a consequence of Proposition \ref{p3}, which is presented as follows.
\begin{proof}[Proof of Theorem \ref{dingli} by using Proposition \ref{p3}]
	Let $\{ (h_{0k}, {\bv}_{0k}) \}_{k\in \mathbb{Z}^+}$ be a sequence of smooth data converging to $(h_0, {\bv}_0)$ in $H^{s}$. Similarly, noting \eqref{Vor}, we also define $\bw_{0k}=(w^0_{0k},w^1_{0k},w^2_{0k})$ by
	\begin{equation*}
		w^\alpha_{0k}=\epsilon^{\alpha \beta}_{ \ \ \ \gamma}\partial_\beta v^\gamma_{0 k}.
	\end{equation*}
	Therefore, $\{ (h_{0k}, {\bv}_{0k}, \bw_{0k}, \nabla \bw_{0k}) \}_{k\in \mathbb{Z}^+}$ is a smooth sequence which converges to $(h_0, {\bv}_0, \bw_0, \nabla \bw_0)$ in $H^{s} \times H^{s} \times H^{s_0-\frac14} \times L^8$. Due to \eqref{chuzhi1}, for all $k\in \mathbb{Z}^+$, we have
		\begin{equation}\label{unf}
		\|h_{0k}\|_{H^{s}} + \|{\bv}_{0k} \|_{H^{s}} + \|\bw_{0k}\|_{H^{s_0-\frac14}} + \|\nabla \bw_{0k}\|_{L^{8}} \leq 2M_0.
	\end{equation}
Due to \eqref{unf}, applying Proposition \ref{p3}, for every $k \in \mathbb{Z}^+$, there exists a solution $(h_k, \bv_k, \bw_k)$ on $[0,T]\times \mathbb{R}^2$ to \eqref{wrt} with the initial data
	\begin{equation*}
		(h_k, \bv_k, \bw_k)|_{t=0}=(h_{0k}, \bv_{0k}, \bw_{0k}),
	\end{equation*}
	where $T$ depends on $C_0,c_0,s,s_0$ and $M_0$, and $\bw_k=(w^0_{k},w^1_{k},w^2_{k})$ satisfies
	\begin{equation*}\label{unp}
		w^\alpha_{k}=\epsilon^{\alpha \beta}_{ \ \ \ \gamma}\partial_\beta v^\gamma_{ k}.
	\end{equation*}
We also note that the solution of \eqref{wrt} also satisfies \eqref{QHl}. We let $\bU_k=(p_k, v^0_{1k}, v^1_{1k}, v^2_{2k}), k \in \mathbb{Z}^+$, where $p_k=p(h_k)$. For every $k \in \mathbb{Z}^+$, we obtain
	\begin{equation*}
		\begin{split}
			A^0( \bU_k )\partial_t \bU_k+ \textstyle{ \sum^2_{i=1} } A^i( \bU_k )\partial_{i}\bU_k = 0 .
		\end{split}
	\end{equation*}
Therefore, for $j,l\in \mathbb{Z}^{+}$, we have
	\begin{equation*}
	\begin{split}
		A^0( \bU_k )\partial_t ( \bU_k - \bU_l) + \textstyle{ \sum^2_{i=1} } A^i( \bU_k )\partial_{i}( \bU_k - \bU_l) 
		= \left\{ -A^\alpha( \bU_k )+ A^\alpha( \bU_l ) \right\} \partial_\alpha \bU_l.
	\end{split}
\end{equation*}
Then the energy estimate in $L_x^{2}$ for the difference $\bU_k - \bU_l$ gives
	\begin{equation}\label{ra0}
		\frac{d}{dt}\|\bU_k - \bU_l \|^2_{L_x^{2}} \leq C_{k, l} \left(\| d\bU_k, d\bU_l\|_{L^\infty_x}\|\bU_k-\bU_l\|^2_{L^2_x}+ \|\bU_k-\bU_l\|^2_{L^2_x}\| d \bU_l\|_{L_x^\infty} \right),
	\end{equation}
	where $C_{k, l}$ depends on the $L^\infty_x$ norm of $\bU_k$and $\bU_l$. Using the Strichartz estimates \eqref{p303}, we get
	\begin{equation}\label{ra1}
		\| d\bv_k, dh_k \|_{L^4_{[0,T]C^\delta_x}} \leq M_1, \quad \forall k \in \mathbb{Z}^+.
	\end{equation}
	Due to \eqref{e9}, we can obtain
	\begin{equation}\label{raa}
		\|h_k, \bv_k\|_{L^\infty_{[0,T]}H_x^{s}}+\|\bw_k\|_{L^\infty_{[0,T]}H_x^{s_0-\frac14}} +\|\nabla \bw_k \|_{L^\infty_{[0,T]}L_x^{8}} \leq M_1.
	\end{equation}
Integrating \eqref{ra0} from $0$ to $t$ and using \eqref{ra1} yields
	\begin{equation*}
		\begin{split}
			\|(\bU_k-\bU_l)(t,\cdot)\|_{L_x^{2}} &\lesssim \|(\bU_k-\bU_l)(0,\cdot)\|_{L_x^{2}}
			\lesssim\|(h_{0k}-h_{0l}, \bv_{0k}-\bv_{0l}) \|_{L^{2}}.
		\end{split}
	\end{equation*}
	By using Lemma \ref{jiaohuan0}, it follows 
		\begin{equation*}
		\begin{split}
			\|(h_k-h_l, \bv_k-\bv_l)(t,\cdot)\|_{L_x^{2}} &
			\lesssim\|(h_{0k}-h_{0l}, \bv_{0k}-\bv_{0l}) \|_{L^{2}}.
		\end{split}
	\end{equation*}
	Therefore, $\{(h_k,\bv_{k})\}_{k \in \mathbb{Z}^+}$ is a Cauchy sequence in $C([0,T];L_x^{2})$. Denote $(p,u_1,u_2)$ being the limit. Then $(h,\bv) \in C([0,T];L_x^{2})$ and
	\begin{equation*}\label{uk1}
		\lim_{k\rightarrow \infty}(h_k,\bv_{k})=(h,\bv) \ \mathrm{in} \ C([0,T];L_x^{2}).
	\end{equation*}
Using \eqref{raa}, the function $(h,\bv) \in C([0,T];H_x^{s})$. Moreover, we also get
\begin{equation*}
	\|( h, \bv )\|_{L^\infty_{[0,T]} H_x^{s}}+ \|\bw\|_{L^\infty_{[0,T]} H_x^{s_0-\frac14}}
	+ \|\nabla \bw \|_{ L^\infty_{[0,T]} L_x^8 } \leq M_1,
\end{equation*}
This implies that the estimate \eqref{A02} holds. By using interpolation formula and \eqref{uk1}, we conclude
	\begin{equation}\label{hk1}
	\lim_{k\rightarrow \infty} (h_k, \bv_k)=(h, \bv) \ \ \mathrm{in} \ \ C([0,T];H_x^{s'}), \quad s' < s.
\end{equation}
Applying \eqref{wrt}, $\bw_j$ satisfies the following transport equation:
\begin{equation*}
	\begin{split}
 v^\kappa_j \partial_\kappa w^\alpha_j=w^\kappa_j \partial^\alpha v_{\kappa j}- w^\alpha_j \partial_\kappa v^\kappa_j.
	\end{split}
\end{equation*}
By $L^2_x$ energy estimates for the difference term $\bw_j - \bw_l$, we have
\begin{equation*}\label{uup}
	\begin{split}
		\frac{d}{dt} \|  \bw_j - \bw_l \|^2_{L^2_x} \lesssim & \|d\bv_j\|_{L^\infty_x} \|  \bw_j - \bw_l \|^2_{L^2_x}+ \|\bv_j - \bv_l \|_{L^\infty_x} \|  \nabla \bw_l \|_{L^2_x} \|  \bw_j - \bw_l \|_{L^2_x}
		\\
		& + \|  \bw_l \|_{L^\infty_x} \|d\bv_j - d\bv_l \|_{L^2_x}  \|  \bw_j - \bw_l \|_{L^2_x}
		\\
		\lesssim & \|d\bv_j\|_{L^\infty_x} \|  \bw_j - \bw_l \|^2_{L^2_x}+ \|\bv_k - \bv_l \|_{H^{1+a}_x} \| \bw_l \|_{H^1_x} \|  \bw_j - \bw_l \|_{L^2_x}
		\\
		& + \|  \bw_l \|_{H^{1+a}_x} \|\bv_k - \bv_l \|_{H^1_x}  \|  \bw_j - \bw_l \|_{L^2_x},
	\end{split}
\end{equation*}
for $a>0$. Integrating \eqref{uup} from $0$ to $t$, using the Gronwall's inequality, product estimates and Proposition \ref{p3}, we obtain 
	\begin{equation*}
		\begin{split}
			\|\bw_k-\bw_l\|_{L_x^{2}} \lesssim \| \bv_{0k}-\bv_{0l}\|_{H^{\frac74}}+\|h_{0k}-h_{0l} \|_{H^{\frac74}} +\|\bw_{0k}-\bw_{0l} \|_{L^{2}}.
		\end{split}
	\end{equation*} 
As a result, $\{\bw_k\}_{k\in \mathbb{Z}}$ is a Cauchy sequence in the space $L_x^2$. Let $\mathbf{w}$ denote its limit. Since $\mathbf{w}_k$ is also bounded in $C([0,T]; H_x^{s_0 - \frac{1}{4}})$, it follows that $\mathbf{w} \in C([0,T]; H_x^{s_0 - \frac{1}{4}})$.

	It remains for us to prove \eqref{SSr} and \eqref{SE1}. Due to Proposition \ref{p3} again, the sequence $\{(d\bv_k,dh_k)\}_{k \in \mathbb{Z}^+}$ is uniformly bounded in $L^4([0,T];C_x^\delta)$. Consequently, we have
	\begin{equation}\label{ccr}
		\lim_{k \rightarrow \infty} (d\bv_k, dh_k)=(d\bv, dh), \quad \textrm{in} \ L^4([0,T];L_x^\infty).
	\end{equation}
	Similarly, we also get
	\begin{equation}\label{cc4}
		\lim_{k \rightarrow \infty} d\bv_{+k}= d\bv_{+}, \quad \textrm{in} \ L^4([0,T];L_x^\infty).
	\end{equation}
	Combining \eqref{ra1} and \eqref{ccr}, \eqref{cc4}, we get
	\begin{equation*}\label{cc1}
		\| (d \bv, dh, d\bv_{+})\|_{ L^4_{[0,T]}L^\infty_x } \leq M_1.
	\end{equation*}
	This implies the estimate \eqref{SSr} holds. Take $T \leq \frac{1}{(1+M_1)^{\frac43}}$. We therefore obtain 
	\begin{equation*}\label{ccs}
		\|h,\bv\|_{L^\infty([0,T]\times \mathbb{R}^2)} \leq C_0+T^{\frac34}\| d \bv, dh\|_{ L^4_{[0,T]}L^\infty_x } \leq 1+C_0.
	\end{equation*}
	Using \eqref{p304} and \eqref{305}, we have
		\begin{equation*}
		\|f_k \|_{L^\infty_{[0,T]} H^r_x} + \|\partial_t f_k \|_{L^\infty_{[0,T]} H^{r-1}_x}  \leq C_{M_0} ( \| f_0\|_{H^r}+ \| f_1\|_{H^{r-1}}),
	\end{equation*}
	and Strichartz estimates
	\begin{equation*}
		\|\left< \nabla \right>^a f_k\|_{L^2_{[0,T]} L^\infty_x} \leq  C_{M_0}(\| f_0\|_{H^r}+ \| f_1\|_{H^{r-1}}),\ \ \ a<r-1,
	\end{equation*}
	Taking $k\rightarrow \infty$ and setting $f_k \rightarrow f$ in $H^r_x$, the estimate \eqref{SE1} holds. Therefore, we complete the proof of Theorem \ref{dingli}.
\end{proof}
It remains for us to verify Proposition \ref{p3}.

\subsection{Reduction to small, smooth, and compactly supported data}\label{sec4.2}
Note that the propagation speed of \eqref{wrt} is finite. We therefore set the constant $c>0$ being the largest speed of
propagation corresponding to the acoustic metric. We can further reduce Proposition \ref{p3} to this following
result:
\begin{proposition}\label{p1}
	Assume $\frac74<s_0\leq s \leq \frac{15}{8}$ and $\delta\in (0,s-\frac74)$. Suppose the initial data $(h_0, {\bv}_0, \bw_0)$ be smooth, supported in $B(0,c+2)$ and satisfy
	\begin{equation}\label{300}
		\begin{split}
			&\| h_0 \|_{H^s} + \|{\bv}_0\|_{H^s} +  \|\bw_0\|_{H^{s_0-\frac14}}+ \|\nabla \bw_0\|_{L^{8}}  \leq \epsilon_3.
		\end{split}
	\end{equation}
	Then the Cauchy problem of \eqref{wrt} admits a smooth solution $(h,\bv,\bw)$ on $[-2,2] \times \mathbb{R}^2$. Moreover, it has the following properties:
	\begin{enumerate}
	\item energy estimate
	\begin{equation}\label{402}
		\begin{split}
			&\|h,\bv \|_{L^\infty_{[-2,2]} H_x^{s}} + \| \bw \|_{L^\infty_{[-2,2]} H_x^{s_0-\frac14}}+ \|\nabla \bw \|_{L^\infty_{[-2,2]} L_x^{8}} \leq \epsilon_2,
			\\
			& \|h,\bv\|_{L^\infty_{[-2,2]\times \mathbb{R}^2}} \leq 1+C_0.
		\end{split}
	\end{equation}

	\item  dispersive estimate for $h$ and $\bv$
	\begin{equation}\label{s403}
		\| d \bv, d h,d\bv_{+} \|_{L^4_{[-2,2]} C^\delta_x} \leq \epsilon_2,
	\end{equation}
	where $\bv_{+}$ is defined as in \eqref{De0}-\eqref{De}.

\item dispersive estimate for the linear equation 

Let $f$ satisfy
	\begin{equation}\label{linearAa}
		\begin{cases}
			& \square_g f=0, \qquad (t,x) \in [-2,2]\times \mathbb{R}^2,
			\\
			&(f, \partial_t f)|_{t=t_0}=(f_0, f_1) \in H_x^r(\mathbb{R}^2) \times H_x^{r-1}(\mathbb{R}^2),
		\end{cases}
	\end{equation}
	where $g$ is defined in \eqref{met1}. For each $1 \leq r \leq s+1$, the Cauchy problem \eqref{linearAa} has a unique solution in $C([-2,2]; H_x^r) \cap C^1([-2,2]; H_x^{r-1})$. Furthermore, for $k<r-1$, we also have
	\begin{equation}\label{304}
		\|\left< \nabla\right>^k f\|_{L^2_{[-2,2]} L^\infty_x} \lesssim  \| f_0\|_{H_x^r}+ \| f_1\|_{H_x^{r-1}},
	\end{equation}
	and the same estimates hold with $\left< \nabla \right>^k$ replaced by $\left< \nabla \right>^{k-1}d$.
	\end{enumerate}
\end{proposition}
Assuming Proposition \ref{p1} holds, we can prove Proposition \ref{p3} as follows.
\begin{proof}[Proof of Proposition \ref{p3} by using Proposition \ref{p1}] 
	We divide this proof into three steps.

	$\textbf{Step 1: Scaling}$. We note the initial data in Proposition \ref{p3} satisfying
	\begin{equation}\label{a4}
		\begin{split}
			&\|h_0\|_{H^{s}}+ \| \bv_0\|_{H^{s}} + \|\bw_0\|_{H^{s_0-\frac14}} + \|\nabla \bw_0\|_{L^{8}} \leq M_0,
		\end{split}
	\end{equation}
	which is not small. Take the scaling
	\begin{equation*}
		\begin{split}
			&\underline{h}(t,x)=h(Tt,Tx),\quad \underline{{\bv}}(t,x)=\bv(Tt,Tx).
		\end{split}
	\end{equation*}
	Due to \eqref{a4}, we can compute out
	\begin{equation}\label{sca}
		\begin{split}
			&\|\underline{\bv}_0\|_{\dot{H}^{s}}=T^{s-1} \|{\bv}_0\|_{\dot{H}^{s}} \leq M_0 T^{s-1}.
		\end{split}
	\end{equation}
	Similarly, we obtain
	\begin{equation}\label{scab}
		\begin{split}
			& \|\underline{h}_0\|_{\dot{H}^s} \leq M_0 T^{s-1}.
		\end{split}
	\end{equation}
	Also, we can calculate
	\begin{equation}\label{scac}
		\begin{split}
			 \| \underline{\bw}_{0} \|_{\dot{H}^{s_0-\frac14}} 
			=&  \|\epsilon^{\alpha \beta \gamma} \partial_\beta   \underline{v}_{0\gamma}  \|_{\dot{H}^{s_0-\frac14}}
			\\
			  = & T^{s_0+\frac34} \|\epsilon^{\alpha \beta \gamma } \partial_\beta   {v}_{0\gamma}  \|_{\dot{H}^{s_0-\frac14}}
			\\
			 = &T^{s_0+\frac34} \| \nabla \bw_0 \|_{L^{8}}
			\\
		 \leq	& M_0 T^{s_0+\frac34},
		\end{split}
	\end{equation}
	and
		\begin{equation}\label{scacd}
		\begin{split}
			\| \nabla \underline{\bw}_{0} \|_{L^8} = & \|\epsilon^{\alpha \beta \gamma} \partial_\beta  
			\underline{v}_{0\gamma}  \|_{\dot{W}^{1,8}} \\  = & T^{\frac74} \|\epsilon^{\alpha \beta \gamma } \partial_\beta  {v}_{0\gamma}  \|_{\dot{W}^{1,8}}
			\\
			 =& T^{\frac74} \| \bw_0 \|_{\dot{W}^{1,8}}
			\\
			 \leq & M_0 T^{\frac74}.
		\end{split}
	\end{equation}
	Choose sufficiently small $T$ such that
	\begin{equation}\label{sca0}
		\max\{M_0 T^{s-1} ,  M_0 T^{s_0+\frac34}, M_0 T^{\frac74}  \}\ll \epsilon_3.
	\end{equation}
	Therefore, by \eqref{sca}, \eqref{scab}, \eqref{scac}, and \eqref{sca0}, we get
	\begin{equation}\label{sca1}
		\begin{split}
			& \|\underline{\bv}_0\|_{\dot{H}^{s}}+ \|\underline{h}_0\|_{\dot{H}^{s}} + \|\underline{\bw}_0\|_{\dot{H}^{s_0-\frac14}} + \|\nabla\underline{\bw}_0\|_{L^{8}} \leq \epsilon_3.
		\end{split}
	\end{equation}
	Using \eqref{HEw} and \eqref{muu}, we also have
	\begin{equation}\label{sca2}
		\|\underline{\bv}_0,\underline{h}_0\|_{L^\infty}\leq C_0.
	\end{equation}
	Next, to reduce the initial data with compact support, we use the physical localization technique.
	
	$\textbf{Step 2: Localization}$. Note that the propagation speed of \eqref{wrt} is finite and $c$ be the largest one. Let $\chi(x)$ be a smooth function supported in $B(0,c+2)$, and which equals $1$ in $B(0,c+1)$. Define
	\begin{equation*}
		\underline{\bv}_0=(\underline{v}^0_0,\mathring{\underline{\bv}}_0).
	\end{equation*}
	For given $y \in \mathbb{R}^2$, we define the localized initial data for the velocity and density near $y$:
	\begin{equation}\label{sid0}
		\begin{split}
			\mathring{\widetilde{\bv}}_0(x)=&\chi(x-y) ( \underline{\mathring{\bv}}_0(x) - \mathring{\underline{\bv}}_0(y)) ,
			\\
			\widetilde{ h }_0(x)=&\chi(x-y)( \underline{h}_0(x)- \underline{h}_0(y) ).
		\end{split}
	\end{equation}
	Since the velocity $\widetilde{\bv}$ satisfies \eqref{muu}, so we set
	\begin{equation}\label{sid00}
		{\widetilde{v}^0}_0=\sqrt{\mathrm{e}^{2\widetilde{ h }_0}+ |\mathring{\widetilde{\bv}}_0|^2}, \quad \widetilde{\bv}_0=(	{\widetilde{v}^0}_0, \mathring{\widetilde{\bv}}_0).
	\end{equation}
	By using \eqref{sca2}, \eqref{sid0}, and \eqref{sid00}, we have
	\begin{equation}\label{sca3}
		\|\mathring{\widetilde{\bv}}_0,\widetilde{ h }_0\|_{L^\infty}\leq C_0, \quad 1\leq{\widetilde{v}^0}_0 \leq 1+C_0.
	\end{equation}
	Due to \eqref{sid0} and \eqref{sid00}, we can verify that
	\begin{equation}\label{245}
		\| \widetilde{\bv}_0\|_{{H}^{s}} \lesssim  \|\underline{\bv}_0\|_{\dot{H}^{s}}, \quad  \| \widetilde{h}_0 \|_{H^{s}} \lesssim \|\underline{h}_0 \|_{\dot{H}^{s}}.
	\end{equation}
	By use of \eqref{Vor}, we define
	\begin{equation}\label{wde}
		\widetilde{\bw}_0= \epsilon^{\alpha \beta \gamma } \partial_{\beta}  \widetilde{v}_{0\gamma}.
	\end{equation}
	For $\widetilde{\bv}_{0}$ is the initial data, combining \eqref{wde} and \eqref{sca3} yield
	\begin{equation}\label{sid1}
		\| \widetilde{\mathbf{\bw}}_0 \|_{L^2} \lesssim \|\nabla \widetilde{\bv}_{0}\|_{L^2} .
	\end{equation}
	Due to \eqref{wde} and \eqref{sid0}, we infer
	\begin{equation}\label{sid4}
		\begin{split}
			\| \widetilde{\bw}_0 \|_{\dot{H}^{s_0}}\lesssim 
		\|  \underline{\bw}_{0}\|_{\dot{H}^{s_0-\frac14}}, \quad \| \nabla \widetilde{\bw}_0 \|_{L^8}\lesssim 
		\|  \nabla \underline{\bw}_{0}\|_{L^8}.
		\end{split}
	\end{equation}
	Combining \eqref{245}, \eqref{sca1}, \eqref{sid1}, and \eqref{sid4}, we therefore conclude
	\begin{equation}\label{sid}
		\begin{split}
			&	\|\widetilde{ h }_0 \|_{{H}^{s}}+\|\widetilde{\bv}_0 \|_{{H}^{s}}+ \|\widetilde{ \bw }_0 \|_{{H}^{s_0}} + \|\nabla \widetilde{ \bw }_0 \|_{{L}^{8}}\lesssim \epsilon_3.
		\end{split}
	\end{equation}
	\qquad $\textbf{Step 3: Using Proposition \ref{p1}}$. 
	Under the initial conditions \eqref{sca3} and \eqref{sid}, by Proposition \ref{p1}, there is a smooth solution $(\widetilde{h}, \widetilde{\bv}, \widetilde{\bw})$ on $[-2,2]\times \mathbb{R}^2$ such that
	\begin{equation}\label{p}
		\begin{cases}
			\square_{\widetilde{g}} \widetilde{h}=\widetilde{\mathcal{D}},
			\\
			\square_{\widetilde{g}} \widetilde{v}^\alpha=-\widetilde{c}^2_s \widetilde{\Theta} \epsilon^{\alpha \beta \gamma}\partial_\beta w_\gamma+\widetilde{Q}^\alpha,
			\\
			(\widetilde{v}^\kappa + \underline{v}^\kappa_0(y)) \partial_\kappa \widetilde{w}^\alpha =\widetilde{w}^\kappa \partial^\alpha \widetilde{v}_\kappa - \widetilde{w}^\alpha \partial_\kappa \widetilde{v}^\kappa,
			\\
			(\widetilde{h},\widetilde{\bv},\widetilde{\bw})|_{t=0}=(\widetilde{h}_0,\widetilde{\bv}_0,\widetilde{\bw}_0),
		\end{cases}
	\end{equation}
	where the quantities $\widetilde{c}^2_s$, $\widetilde{{g}}$, and $\widetilde{\Theta}$ are defined by
	\begin{equation*}\label{DDE}
		\begin{split}
			\widetilde{c}^2_s&= c^2_s(\widetilde{h}+\underline{h}_0(y)),
			\\
			\widetilde{g}&=g(\widetilde{h}+\underline{h}_0(y),\widetilde{\bv}+\underline{\bv}_0(y) ),
			\\
			\widetilde{\Theta}&={\Theta}(\widetilde{h}+\underline{h}_0(y), \widetilde{\bv}+\underline{\bv}_0(y)),
		\end{split}
	\end{equation*}
	and $\widetilde{\mathcal{D}}, \widetilde{Q}^\alpha$ have the same formulations with $\mathcal{D}$, ${Q}^\alpha$ by replacing $({h},\bv)$ to $(\widetilde{h}+\underline{h}_0(y), \widetilde{\bv}+\underline{\bv}_0(y))$.

Next, we consider a better component of the velocity $\widetilde{\bv}$. We define
	\begin{equation}\label{deu1}
		\widetilde{\mathbf{P}}=\mathrm{I}-\left[ m^{\alpha \beta}+ 2 \mathrm{e}^{ 2( \widetilde{h}+\underline{h}_0(y) ) }(\widetilde{v}^\alpha+ \underline{v}^\alpha_0(y) ) (\widetilde{v}^\beta+ \underline{v}^\beta_0(y) ) \right] \partial^2_{ \alpha \beta}.
	\end{equation}
Referring to \eqref{deu1} and using Lemma \ref{Ees}, we conclude that $\widetilde{\mathbf{P}}$ is a space-time elliptic operator. Define
two vectors $\widetilde{\bv}_{-}=(\widetilde{v}^0_{-},\widetilde{v}^1_{-},\widetilde{v}^2_{-})$ and $\widetilde{\bv}_{+}=(\widetilde{v}^0_{+},\widetilde{v}^1_{+},\widetilde{v}^2_{+})$ by
\begin{equation*}\label{deu2}
		\begin{split}
			 \widetilde{\mathbf{P}} \widetilde{v}^\alpha_{-} & =    \epsilon^{\alpha \beta \gamma} \partial_\beta \widetilde{ w }_\gamma ,
			\\
			\widetilde{\bv}_{+} & =\widetilde{\bv}-\widetilde{\bv}_{-}.
		\end{split}
	\end{equation*}
Denote an operator
	\begin{equation*}\label{deu3}
		\widetilde{\mathbf{T}}= \partial_t + \frac{\widetilde{v}^i+\underline{v}^i_0(y)}{\widetilde{v}^0+ \underline{v}^0_0(y) }\partial_i.
	\end{equation*}
Due to \eqref{wag0}, we have
	\begin{equation*}
		\begin{split}
			\square_g  \widetilde{v}^{\alpha}_{+}&= \widetilde{\Theta}(1-3 \widetilde{c}^2_s)   \mathrm{e}^{-2(\widetilde{h}+\underline{h}_0(y))} ( \widetilde{v}^\beta+ \underline{\widetilde{v}}^\beta_0(y) ) ( \widetilde{v}^\gamma+ \underline{\widetilde{v}}^\gamma_0(y) ) \partial^2_{\beta \gamma}\widetilde{v}^\alpha_{-}- \widetilde{c}^2_s  \widetilde{\Theta} \widetilde{v}^\alpha_{-} + \widetilde{Q}^\alpha.
		\end{split}
	\end{equation*}
	Using Proposition \ref{p1} again, the solution $(\widetilde{h}, \widetilde{{\bv}}, \widetilde{\bw})$ also satisfies
	\begin{equation}\label{see0}
		\|\widetilde{h}, \widetilde{\bv}\|_{L^\infty_{[-2,2]}H_x^s}+  \| \widetilde{\bw} \|_{L^\infty_{[-2,2]}H_x^{s_0-\frac14}} + \|\nabla \widetilde{\bw} \|_{L^\infty_{[-2,2]}L_x^{8}} \leq \epsilon_2,
	\end{equation}
	and
	\begin{equation*}\label{see1}
		\|d\widetilde{\bv}, d\widetilde{h}\|_{L^4_{[-2,2]} C^\delta_x} \leq \epsilon_2.
	\end{equation*}
	Moreover, the linear wave equation endowed with the Lorentzian metric $\widetilde{{g}}$
	\begin{equation*}\label{312}
		\begin{cases}
			&\square_{ \widetilde{{g}}} \tilde{f}=0,
			\\
			&\tilde{f}(t_0,\cdot)=\widetilde{f}_0, \ \partial_t\tilde{f}(t_0,)=\widetilde{f}_1,
		\end{cases}
	\end{equation*}
	also admits a solution $\widetilde{f} \in C([-2,2],H_x^r)\times C^1([-2,2],H_x^{r-1})$. Furthermore, if $k<r-1$, the following Strichartz estimate holds:
	\begin{equation*}\label{s31}
		\|\left< \nabla \right>^k \widetilde{f}\|_{L^2_{[-2,2]} L^\infty_x} \lesssim  \| \widetilde{f}_0\|_{H^r}+ \| \widetilde{f}_1\|_{H^{r-1}}.
	\end{equation*}
On the other hand, set 
	\begin{equation}\label{vw}
		\begin{split}
			\bar{h}(t,x)  &=\widetilde{ h}+ \underline{h}_0(y),
			\\	
			\bar{{\bv}}(t,x)&=(\bar{v}^0, \mathring{\bar{\bv}}), \quad \bar{v}^0=\sqrt{\mathrm{e}^{2\bar{h}}+|\mathring{\bar{\bv}}|^2},
			\\
			\mathring{\bar{\bv}}(t,x)  &= \mathring{\widetilde{\bv}}+\underline{\mathring{\bv}}_0(y),
			\\
			\bar{w}^\alpha (t,x)  &= \epsilon^{\alpha \beta \gamma} \partial_\beta \bar{{v}}_\gamma.
		\end{split}
	\end{equation}
	Therefore, $(\bar{h},\bar{\bv},\bar{\bw})$ is also a solution of \eqref{p}, and its initial data coincides with $(h_0, \bv_0, \bw_0)$ in $B(y,c+1)$. Consider a restriction
	\begin{equation}\label{RS}
		(\bar{h},\bar{\bv},\bar{\bw})|_{\mathrm{K}^y},
	\end{equation}
	where $y\in \mathbb{R}^2$ and $\mathrm{K}^y=\left\{ (t,x): ct+|x-y| \leq c+1, |t| <1 \right\}$. Then \eqref{RS} solves \eqref{p} on $\mathrm{K}^y$. By finite speed of propagation, a smooth solution $(\bar{h}, \bar{\bv}, \bar{\bw})$ solves \eqref{wrt} on $[-1,1] \times \mathbb{R}^2$. Without diffusion, we still use the notation $\bar{h}, \bar{\bv}$ and $\bar{\bw}$ being the restrictions on $\mathrm{K}^y$.

	From \eqref{vw}, by time-space scaling $(t,x)$ to $(T^{-1}t,T^{-1}x)$, we also obtain
	\begin{equation}\label{po0}
		\begin{split}
			&(h, \bv, \bw)=(\bar{h}, \bar{\bv}, \bar{\bw})(T^{-1}t,T^{-1}x),
			\\
			&( h, \bv, \bw)|_{t=0}=(\bar{h}, \bar{\bv}, \bar{\bw})(0,T^{-1}x)=(h_0, \bu_0, \bw_0).
		\end{split}
	\end{equation}
	Therefore, the function $(h, \bv, \bw)$ defined in \eqref{po0} is the solution of \eqref{wrt} according to the uniqueness of solutions, i.e. Corollary \ref{uniq}. To obtain the energy estimates for $(h, \bv,\bw)$ on $[0,T]\times \mathbb{R}^2$, we use the cartesian grid $3^{-\frac12} \mathbb{Z}^2$ in $\mathbb{R}^2$, and a corresponding smooth partition of unity
	\begin{equation*}
		\textstyle{\sum}_{y \in 3^{-\frac12} \mathbb{Z}^2} \psi(x-y)=1,
	\end{equation*}
	such that the function $\psi$ is supported in the unit ball. Therefore, we have
	\begin{equation}\label{vwa}
		\begin{split}
			& \mathring{\bar{\bv}}=\textstyle{\sum}_{y \in 3^{-\frac12} \mathbb{Z}^2} \psi(x-y)(\mathring{\widetilde{\bv}}+\underline{\mathring{\bv}}_0(y) ),
			\\
			& \bar{h}=\textstyle{\sum}_{y \in 3^{-\frac12} \mathbb{Z}^2} \psi(x-y)(\widetilde{h}+\underline{h}_0(y) ),
			\\
			& \bar{v}^0=\sqrt{\mathrm{e}^{2\bar{h}}+|\mathring{\bar{\bv}}|^2}, \quad \bar{\bv}=(\bar{v}^0, \mathring{\bar{\bv}}) , 
			\\
			&\bar{w}^\alpha =\epsilon^{\alpha \beta \gamma} \partial_\beta \bar{v}_\gamma.
		\end{split}
	\end{equation}
	Due to \eqref{see0}, \eqref{sca2}, and \eqref{vwa}, we shall obtain
	\begin{equation}\label{po03}
		\begin{split}
			& \|\bar{h},\bar{\bv}\|_{L^\infty_{[0,1]}L^\infty_x} \leq \|\widetilde{h},\widetilde{\bv} \|_{H_x^{s}} + \|\underline{h}_0, \underline{\bv}_0\|_{L^\infty_x} \leq 1+C_0.
		\end{split}
	\end{equation}
	By changing the coordinates, it yields
	\begin{equation}\label{po2}
		\begin{split}
			& \|d{\bv}, d{h}\|_{L^4_{[0,T]}C^{\delta}_x}
			\\
			\leq & T^{-(\frac34+{\delta})}\|d\widetilde{\bv}, d\widetilde{h}\|_{L^4_{[0,1]}C^{\delta}_x}
			\\
			\leq
			&C T^{-(\frac34+{\delta})} (\|\widetilde{\bv}_0\|_{H_x^s}+ \|\widetilde{h}_0\|_{H_x^s}+ \| \widetilde{\bw}_0 \|_{H_x^{s_0-\frac14}} + \|\nabla \widetilde{\bw}_0 \|_{L_x^{8}})
			\\
			\leq & C T^{-(\frac34+{\delta})}(T^{s-1}\|{\bu}_0\|_{\dot{H}_x^s}+ T^{s-1} \|h_0\|_{\dot{H}_x^s}
			+ T^{s_0-\frac14}\| {\bw}_0 \|_{\dot{H}_x^{s_0-\frac14}} +  T^{\frac74}\|\nabla \bw_0\|_{L^8_x} )
			\\
			\leq & C (\|{\bv}_0\|_{H_x^s}+ \|{h}_0\|_{H_x^s}+ \| {\bw}_0 \|_{H_x^{s_0-\frac14}}).
		\end{split}
	\end{equation}
Above, we use ${\delta} \in (0,s-\frac74)$ and $\frac74<s_0\leq s \leq 2$. By using \eqref{po03} and changing of coordinates from $(t,x) \rightarrow (T^{-1}t,T^{-1}x)$, we can obtain
	\begin{equation}\label{po05}
		\begin{split}
			& \|{\bv}, {h}\|_{L^\infty_{[0,T]\times \mathbb{R}^3}} = \|{\bar{\bv}}, \bar{h}\|_{L^\infty_{[0,1]}L^{\infty}_x}  \leq 1+C_0,
		\end{split}
	\end{equation}
	and
	\begin{equation}\label{po0a}
		\begin{split}
			& v^0= \sqrt{\mathrm{e}^{2h}+ |\mathring{\bv}|^2} \geq \mathrm{e}^h.
		\end{split}
	\end{equation}
	Applying Theorem \ref{DW4} along with equations \eqref{po2}, \eqref{po05}, and \eqref{po0a}, we conclude that $(h, \bv, \bw)$ satisfies \eqref{ig1}, \eqref{e9} and \eqref{p303}.

	It remains for us to prove \eqref{p304} and \eqref{305}. For $1\leq r \leq s+1$, we consider the following homogeneous linear wave equation
	\begin{equation}\label{po3}
		\begin{cases}
			\square_{{g}} f=0, \qquad [0,T]\times \mathbb{R}^2,
			\\
			(f,f_t)|_{t=0}=(f_0,f_1)\in H_x^r \times H^{r-1}_x.
		\end{cases}
	\end{equation}
	For \eqref{po3} is a scaling invariant system, we can transfer \eqref{po3} to the localized linear equation
	\begin{equation}\label{po4}
		\begin{cases}
			\square_{\widetilde{g}} {f}^y=0, \quad [0,1]\times \mathbb{R}^2,
			\\
			({f}^y,{f}^y_t)|_{t=0}=(f_0^y,f^y_1),
		\end{cases}
	\end{equation}
	where
	\begin{equation*}\label{po5}
		\begin{split}
			& f_0^y=\chi(x-y)(\widetilde{f}_0-\widetilde{f}_0(y)), \quad f^y_1=\chi(x-y)\widetilde{f}_1,
			\\ & \widetilde{f}_0=f_0(Tx),\quad \widetilde{f}_1=f_1(Tx).
		\end{split}
	\end{equation*}
	Let
	\begin{equation}\label{po6}
		\widetilde{f}=\textstyle{\sum}_{y\in 3^{-\frac12}\mathbb{Z}^2}\psi(x-y)({f}^y+\widetilde{f}_0(y)).
	\end{equation}
	Seeing \eqref{po4}, and using Proposition \ref{p1} again, for $k<r-1$, we get
	\begin{equation}\label{po7}
		\begin{split}
			\|\left< \nabla \right>^{k-1} d {f}^y\|_{L^4_{[0,1]} L^\infty_x} \leq  & C(\| \chi(x-y)(\widetilde{f}_0-\widetilde{f}_0(y)) \|_{H_x^r}+ \| \chi(x-y)\widetilde{f}_1 \|_{H_x^{r-1}})
			\\
			\leq & C(\|\widetilde{f}_0\|_{\dot{H}_x^r}+ \| \widetilde{f}_1 \|_{\dot{H}_x^{r-1}}).
		\end{split}
	\end{equation}
	Due to the finite speed of propagation, we can claim $\widetilde{f}={f}^y+\widetilde{f}_0(y)$ in $\mathrm{K}^y$. Applying \eqref{po6} and \eqref{po7}, we obtain
	\begin{equation}\label{po8}
		\begin{split}
			\|\left< \nabla \right>^{k-1} d \widetilde{f}\|_{L^4_{[0,1]} L^\infty_x}
			= & \sup_{y\in 3^{-\frac12}\mathbb{Z}^2}\|\left< \nabla \right>^{k-1} d ({f}^y+\widetilde{f}_0(y))\|_{L^4_{[0,1]} L^\infty_x}
			\\
			= & \sup_{y\in 3^{-\frac12}\mathbb{Z}^2}\|\left< \nabla \right>^{k-1} d {f}^y\|_{L^4_{[0,1]} L^\infty_x}
			\\
			\leq & C(\|\widetilde{f}_0\|_{\dot{H}_x^r}+ \| \widetilde{f}_1 \|_{\dot{H}_x^{r-1}}).
		\end{split}
	\end{equation}
	By a change of coordinates: $(t,x)\rightarrow (T^{-1}t,T^{-1}x)$, we have
	\begin{equation}\label{po9}
		\begin{split}
			\|\left< \nabla \right>^{k-1} d{f}\|_{L^4_{[0,T]} L^\infty_x} = & T^{k-\frac14}\|\left< \nabla \right>^{k-1} d\widetilde{f}\|_{L^4_{[0,1]} L^\infty_x}.
		\end{split}
	\end{equation}
	Taking $T\leq 1$, combing \eqref{po8} with \eqref{po9}, if $k<r-1$, then we have
	\begin{equation}\label{po10}
		\begin{split}
			\|\left< \nabla \right>^{k-1} d{f}\|_{L^4_{[0,T]} L^\infty_x}
			\leq & CT^{r-1-k}(\|{f}_0\|_{\dot{H}_x^r}+ \| {f}_1 \|_{\dot{H}_x^{r-1}})
			\\
			\leq & C(\|{f}_0\|_{{H}_x^r}+ \| {f}_1 \|_{{H}_x^{r-1}}).
		\end{split}
	\end{equation}
	Using the energy estimates for \eqref{po3} and \eqref{po2}, we can carry out
	\begin{equation}\label{po11}
		\begin{split}
			\|d {f}\|_{L^\infty_{[0,T]} H^{r-1}_x} 
			\leq  & C_{M_0}(\| {f}_0\|_{H_x^r}+ \| {f}_1\|_{H_x^{r-1}}) \int^T_0   \|d \bv,dh \|_{L^\infty_x} d\tau
			\\
			\leq & C_{M_0}(\| {f}_0\|_{H_x^r}+ \| {f}_1\|_{H_x^{r-1}}).
		\end{split}
	\end{equation}
	Due to Newton-Leibniz formula $f=\int^{t}_0 \partial_ t fd\tau -f_0$, we therefore get
	\begin{equation}\label{po12}
		\begin{split}
			\|f\|_{L^\infty_{[0,T]} L^{2}_x}
			\leq & T \|d {f}\|_{L^\infty_{[0,T]} L^{2}_x}+ \| {f}_0\|_{L_x^2}
			\\
			\leq & C_{M_0}(\| {f}_0\|_{H_x^r}+ \| {f}_1\|_{H_x^{r-1}}).
		\end{split}
	\end{equation}
	Adding \eqref{po11} and \eqref{po12}, we therefore obtain
	\begin{equation}\label{po13}
		\begin{split}
			\|{f}\|_{L^\infty_{[0,T]} H^{r}_x}+ \|\partial_t {f}\|_{L^\infty_{[0,T]} H^{r-1}_x} \leq  C_{M_0}(\| {f}_0\|_{H_x^r}+ \| {f}_1\|_{H_x^{r-1}}).
		\end{split}
	\end{equation}
	Because of \eqref{po10}, \eqref{po11}, \eqref{po12}, and \eqref{po13}, the estimates \eqref{p304} and \eqref{305} hold. Thus, we complete the proof of Proposition \ref{p3}.
\end{proof}
It remains for us to verify Proposition \ref{p1}.

\section{Proof of Proposition \ref{p1}}\label{sec4.3}
In this section, we will establish Proposition \ref{p1} through a bootstrap argument. For technical reasons, it is convenient for us to replace the original acoustic metric by a truncated one. To do this, we define two vectors
\begin{equation*}\label{0e}
	\mathbf{0}=(0,0,0)^{\mathrm{T}}, \quad \mathbf{1}=(1,0,0)^{\mathrm{T}}.
\end{equation*}
We note that $h=0, \bv=\mathbf{1}, \bw=\mathbf{0}$ is an equilibrium state for the system \eqref{wrt}. Then the initial data in \eqref{300} can be seen as a small perturbation around this equilibrium state. By setting $h=0$ and $\bv=\mathbf{1}$ in $g$ (see \eqref{met1}), we denote this as $g_*$. The inverse matrix of the metric $g_*$ is
\begin{equation*}
	g^{-1}_*=
	\left(
	\begin{array}{cccc}
		-1 & 0 & 0 \\
		0 & c^2_s(0) & 0\\
		0 & 0 & c^2_s(0) 
	\end{array}
	\right ),
\end{equation*}
where $c_s(0)=c_s(h)|_{h=0}$. By a linear change of coordinates which preserves $dt$, we may assume that $g^{\alpha \beta}_*=m^{\alpha \beta}$. Let $\chi$ be a smooth cut-off function supported in the region $B(0,3+2c) \times [-\frac{3}{2}, \frac{3}{2}]$, which equals to $1$ in the region $B(0,2+2c) \times [-1, 1]$. Define $\mathbf{g}= (\mathbf{g}^{\alpha \beta})_{3\times 3}$ and
\begin{equation}\label{AMd3}
	\begin{split}
		\mathbf{g}^{\alpha \beta}=& \chi(t,x)(g^{\alpha \beta}-g^{\alpha \beta }_*)+g^{\alpha \beta }_*,
	\end{split}
\end{equation}
where $g^{\alpha\beta}=\Theta\{ c_s^{2}m^{\alpha\beta}+\mathrm{e}^{-2h}(c_s^2-1)v^\alpha v^\beta \}$ is stated as in \eqref{met1}. Moreover, using \eqref{AMd3}, we have
\begin{equation*}
	\mathbf{g}=g, \quad \textrm{if} \ (t,x)\in [-1,1]\times \mathbb{R}^2.
\end{equation*}
Therefore, $\mathbf{g}$ is an extension of the metric $g$. Consider the following wave-transport system:
\begin{equation}\label{XT}
	\begin{cases}
		\square_{\mathbf{g}} h=\mathcal{D},
		\\
		\square_{\mathbf{g}} v^\alpha=-c^2_s \Theta  \epsilon^{\alpha \beta \gamma}\partial_\beta w_\gamma+Q^\alpha,
		\\
		v^\kappa \partial_\kappa w^\alpha = w^\kappa \partial^\alpha v_\kappa - w^\alpha \partial_\kappa v^\kappa,
	\end{cases}
\end{equation}
where $\mathcal{D}$ and $Q^\alpha$ ($\alpha=0,1,2$) have the same formulations as in equations \eqref{wrt0} and \eqref{wrt1}. Consequently, the system \eqref{XT} is an extension of \eqref{wrt} from $[-1,1]\times \mathbb{R}^2$ to $[-2,2]\times \mathbb{R}^2$.

\begin{definition}\label{fh}(Definition for $\mathcal{H}$)
Assume $\frac74<s_0\leq s \leq \frac{15}{8}$ and $\delta\in (0,s-\frac74)$. We denote by $\mathcal{H}$ the family of smooth solutions $(h, \bv, \bw)$ to the equation \eqref{XT} for $t \in [-2,2]$, with the initial data $(h_0, {\bv}_0, \bw_0)$ supported in $B(0,2+c)$, and for which
\begin{equation}\label{401}
	\|\bv_0\|_{H^s} + \|h_0\|_{H^s}+\| \bw_0\|_{H^{s_0-\frac14}} + \|\nabla \bw_0\|_{L^{8}}   \leq \epsilon_3,
\end{equation}
\begin{equation}\label{402a}
	\| h \|_{L^\infty_{[-2,2]} H_x^{s}}+ \|\bv\|_{L^\infty_{[-2,2]} H_x^{s}}+\| \bw\|_{L^\infty_{[-2,2]} H_x^{s_0-\frac14}}+ \| \nabla \bw\|_{L^\infty_{[-2,2]} L_x^{8}} \leq 2 \epsilon_2,
\end{equation}
\begin{equation}\label{403}
	\| d \bv, d h,d\bv_{+} \|_{L^4_{[-2,2]} C_x^\delta} \leq 2 \epsilon_2.
\end{equation}
\end{definition}
The bootstrap argument can be stated as follows.
\subsection{A bootstrap argument}
\begin{proposition}\label{p4}
Assume $s\in (\frac74,\frac{15}{8}]$ and that \eqref{a0} holds. Then there exists a continuous functional $\Re: \mathcal{H} \rightarrow \mathbb{R}^{+}$, satisfying $\Re(0, \mathbf{1}, \mathbf{0}) = 0$, such that for each $(h, \bv, \bw) \in \mathcal{H}$ satisfying $\Re(h,\bv,\bw) \leq 2 \epsilon_1$ the following hold:

	\begin{enumerate}
	\item The function $h$, $\bv$, and $\bw$ satisfies 
	\begin{equation}\label{404a}
		\Re(h,\bv,\bw) \leq \epsilon_1.
	\end{equation}

	\item The energy estimate holds:
	\begin{equation}\label{404}
		\|h\|_{L^\infty_{[-2,2]} H_x^{s}}+ \|\bv\|_{L^\infty_{[-2,2]} H_x^{s}}+  \|\bw\|_{L^\infty_{[-2,2]} H_x^{s_0-\frac14}} +  \|\nabla \bw\|_{L^\infty_{[-2,2]} L_x^{8}} \leq \epsilon_2,
	\end{equation}
	and the Strichartz estimate
	\begin{equation}\label{405}
		\|d h,d \bv \|_{L^4_{[-2,2]} C^\delta_x} \leq \epsilon_2.
	\end{equation}

	\item For any $1 \leq k \leq s+1$, and for each $t_0\in [-2,2)$, the linear wave equation
	\begin{equation*}\label{linearB}
		\begin{cases}
			& \square_{\mathbf{g}} f=0, \qquad (t,x) \in [-2,2]\times \mathbb{R}^2,
			\\
			&(f, \partial_t f)|_{t=t_0}=(f_0,f_1) \in H^r(\mathbb{R}^2)\times H^{r-1}(\mathbb{R}^2),
		\end{cases}
	\end{equation*}
	admits a unique solution $f \in C([-2,2],H_x^r) \times C^1([-2,2],H_x^{r-1})$. Furthermore, the Strichartz estimate \eqref{304} holds.
	\end{enumerate}
\end{proposition}

Based on Proposition \ref{p4}, we are ready to prove Proposition \ref{p1}.
\begin{proof}[Proof of Proposition \ref{p1} by using Proposition \ref{p4}]
	Consider the initial data in Proposition \ref{p1} satisfying
	\begin{equation*}
		\|h_0\|_{H^s}+\|\bv_0\|_{H^s} +  \| \bw_0\|_{H^{s_0-\frac14}} +  \| \nabla \bw_0\|_{L^{8}}  \leq \epsilon_3.
	\end{equation*}
	Denote by $\mathcal{A}$ the subset of those $\gamma \in [0,1]$ such that the equation \eqref{XT} admits a smooth solution $(h_\gamma, \bv_\gamma, \bw_\gamma)$ having the initial data
	\begin{equation*}
		\begin{split}
			h_\gamma(0)=&\gamma h_0,
			\\
			\mathring{\bv}_\gamma(0)=&\gamma \mathring{\bv}_0,
			\\
			v^0_\gamma(0)=& \sqrt{\mathrm{e}^{2h_\gamma(0) }+|\mathring{\bv}_\gamma(0)|^2},
			\\
			w^\alpha_\gamma(0)=& \epsilon^{\alpha \beta }_{\ \ \ \kappa} \partial_\beta v^\kappa_{0\gamma},
		\end{split}
	\end{equation*}
	and such that $\Re(h_\gamma,\bv_\gamma, \bw_\gamma) \leq \epsilon_1$ and \eqref{404}, \eqref{405} hold. If $\gamma=0$, then
	\begin{equation*}
		(h_\gamma, {\bv}_\gamma, \bw_\gamma)(t,x)=(0,\mathbf{1},\mathbf{0}),
	\end{equation*}
	is a smooth solution of \eqref{XT} with initial data
	\begin{equation*}
		(h_\gamma, {\bv}_\gamma, \bw_\gamma)(0,x)=(0,\mathbf{1},\mathbf{0}).
	\end{equation*}
	Thus, the set $\mathcal{A}$ is not empty. If we can prove that $\mathcal{A}=[0,1]$, thus $1 \in \mathcal{A}$. Also, \eqref{300}, \eqref{402}, \eqref{s403}, \eqref{linearA}, and \eqref{304} follow from Proposition \ref{p4}. As a result, Proposition \ref{p1} holds. It is sufficient to prove that $\mathcal{A}$ is both open and closed in $[0,1]$.

	(1) $\mathcal{A}$ is open. Let $\gamma \in \mathcal{A}$. Then $(h_\gamma, \bv_\gamma, \bw_\gamma)$ is a smooth solution to \eqref{XT}, where
	\begin{equation*}
		w^\alpha_\gamma= \epsilon^{\alpha \beta}_{\ \ \ \kappa} \partial_\beta v^\kappa_\gamma .
	\end{equation*}
	Let $\beta$ be close to $\gamma$.  Since $(h^\gamma,\bv^\gamma,\bw^\gamma)$ is smooth, a perturbation argument
	shows that the equation \eqref{XT} has a smooth solution $(h^\beta, \bv^\beta, \bw^\beta)$, which
	depends continuously on $\beta$. By the continuity of $\Re$, for $\beta$ close to $\gamma$, we infer
	\begin{equation*}
		\Re(h_\beta,\bv_\beta,\bw_\beta) \leq 2\epsilon_1,
	\end{equation*}
	and \eqref{401}, \eqref{402a}, \eqref{403} hold. Due to Proposition \ref{p4}, we get
	\begin{equation*}
		\Re(h_\beta,\bv_\beta,\bw_\beta) \leq \epsilon_1,
	\end{equation*}
	and also \eqref{404}-\eqref{405} hold. Thus, we have showed that $\beta \in \mathcal{A}$.

	(2) $\mathcal{A}$ is closed. Let $\gamma_k \in \mathcal{A}, k \in \mathbb{N}^+$. Let $\gamma$ satisfy $\lim_{k \rightarrow \infty} \gamma_k = \gamma$.
	Then there exists a sequence $\{(h_{\gamma_k}, \bv_{\gamma_k},\bw_{\gamma_k}) \}_{k \in \mathbb{N}^+}$ being the smooth solutions to \eqref{XT} and
	\begin{align*}
		& \|h_{\gamma_k}, \bv_{\gamma_k}\|_{L^\infty_{[-2,2]} H_x^{s}}
		+ \| \bw_{\gamma_k} \|_{L^\infty_{[-2,2]} H_x^{s_0-\frac14}}
		 +\| \nabla \bw_{\gamma_k} \|_{L^\infty_{[-2,2]} L_x^{8}}+  \| d h_{\gamma_k}, d \bv_{\gamma_k}  \|_{L^4_{[-2,2]} C_x^\delta} \leq \epsilon_2.
	\end{align*}
	Then there exists a subsequence $\{(h_{\gamma_k}, \bv_{\gamma_k}, \bw_{\gamma_k}) \}_{k \in \mathbb{N}^+}$  such that 
	\begin{equation*}
		\lim_{k \rightarrow \infty} (h_{\gamma_k}, \bv_{\gamma_k}, \bw_{\gamma_k})=(h_{\gamma}, \bv_{\gamma}, \bw_{\gamma}) .
	\end{equation*}
	Moreover, $(h_{\gamma}, \bv_{\gamma}, \bw_{\gamma})$ satisfies
	\begin{equation*}
		\begin{split}
			&\| h_{\gamma}, \bv_{\gamma}\|_{L^\infty_{[-2,2]} H_x^{s}}+ \| \bw_{\gamma} \|_{L^\infty_{[-2,2]} H_x^{s_0-\frac14}}
			+ \| \nabla \bw_{\gamma} \|_{L^\infty_{[-2,2]} L_x^{8}}+ \| dh_{\gamma} ,  d \bv_{\gamma}\|_{L^4_{[-2,2]} C_x^\delta} \leq \epsilon_2,
		\end{split}
	\end{equation*}
	Thus, $\Re(h_\gamma,\bv_\gamma,\bw_\gamma) \leq \epsilon_1$. Then $\gamma \in \mathcal{A}$. At this stage, we have proved Proposition \ref{p1}.
\end{proof}
Note that there is no precise definition of the functional $ \Re $. We now give the definition.
\subsection{Definition for the functional}
Let us first define foliations for the space-time and introduce new norms.

Assume $(h,\bv,\bw) \in \mathcal{H}$. Give a extension for $\mathbf{g}$ ($\mathbf{g}$ is defined in \eqref{AMd3}), which equals to the Minkowski metric for $t \in [-2, -\frac{3}{2}]$. When no confusion can arise, we still use the notation $\mathbf{g}$.

For each $\theta \in \mathbb{S}^1$, we consider a foliation of the slice $t=-2$ by taking level sets of the function $r_\theta(-2,x)=\theta \cdot x+2$. Then $\theta \cdot dx-t$ is a null covector field over $t=-2$, and it is co-normal to the level sets of $r_\theta(-2)$. Let $\Gamma_{\theta}$ be the graph of a null covector field given by $dr_{\theta}$, We define the hypersurface $\Sigma_{\theta,r}$ for $r \in \mathbb{R}$ as the level sets of $r_{\theta}$. The characteristic hypersurface $\Sigma_{\theta,r}$ is thus the flowout of the set $\theta \cdot x=r-2$ along the null geodesic flow in the direction $\theta$ at $t=-2$.

We introduce an orthonormal sets of coordinates on $\mathbb{R}^2$ by setting $x_{\theta}=\theta \cdot x$. Let $x'_{\theta}$ be given orthonormal coordinates on the hyperplane prependicular to $\theta$, which then define coordinates on $\mathbb{R}^2$ by projection along $\theta$. Then $(t,x'_{\theta})$ induce the coordinates on $\Sigma_{\theta,r}$, and $\Sigma_{\theta,r}$ is given by
\begin{equation*}
	\Sigma_{\theta,r}=\left\{ (t,x): x_{\theta}-\phi_{\theta, r}=0  \right\},
\end{equation*}
for a smooth function $\phi_{\theta, r}(t,x'_{\theta})$. We now introduce two norms for functions defined on $[-2,2] \times \mathbb{R}^2$, for $a \geq 1$,
\begin{equation*}\label{d0}
	\begin{split}
		&\vert\kern-0.25ex\vert\kern-0.25ex\vert f\vert\kern-0.25ex\vert\kern-0.25ex\vert_{a, \infty} = \sup_{-2 \leq t \leq 2} \sup_{0 \leq j \leq 1} \| \partial_t^j f(t,\cdot)\|_{H^{a-j}(\mathbb{R}^2)},
		\\
		& \vert\kern-0.25ex\vert\kern-0.25ex\vert  f\vert\kern-0.25ex\vert\kern-0.25ex\vert_{a,2} = \big( \sup_{0 \leq j \leq 1} \int^{2}_{-2} \| \partial_t^j f(t,\cdot)\|^2_{H^{a-j}(\mathbb{R}^2)} dt \big)^{\frac{1}{2}}.
	\end{split}
\end{equation*}
The same notation applies for functions in $[-2,2] \times \mathbb{R}^2$. We denote
\begin{equation*}
	\vert\kern-0.25ex\vert\kern-0.25ex\vert f\vert\kern-0.25ex\vert\kern-0.25ex\vert_{a,2,\Sigma_{\theta,r}}=\vert\kern-0.25ex\vert\kern-0.25ex\vert f|_{\Sigma_{\theta,r}} \vert\kern-0.25ex\vert\kern-0.25ex\vert_{a,2},
\end{equation*}
where the right hand side is the norm of the restriction of $f$ to ${\Sigma_{\theta,r}}$, taken over the $(t,x'_{\theta})$ variables used to parametrise ${\Sigma_{\theta,r}}$. Similarly, the notation
\begin{equation*}
	\|f\|_{H^{a}(\Sigma_{\theta,r})},
\end{equation*}
denotes the $H^{a}(\mathbb{R})$ norm of $f$ restricted to the time $t$ slice of ${\Sigma_{\theta,r}}$ using the $x'_{\theta}$ coordinates on ${\Sigma^t_{\theta,r}}$.

For $\frac74<s_0<s\leq \frac{15}{8}$, we define the functional $\Re$ by
\begin{equation}\label{500}
	\Re = \sup_{\theta, r} \vert\kern-0.25ex\vert\kern-0.25ex\vert d \phi_{\theta,r}-dt\vert\kern-0.25ex\vert\kern-0.25ex\vert_{s_0-\frac14,2,{\Sigma_{\theta,r}}}.
\end{equation}

\subsection{The proof of Proposition \ref{p4}} We postpone the proof and proceed to state the following two propositions.
\begin{proposition}\label{r2}
	Assume $\frac74<s_0\leq s\leq \frac{15}{8} $ and $ \delta_0 \in (0, s_0-\frac{7}{4})$. 	Let $(h,\bv, \bw) \in \mathcal{H}$ such that $\Re(h, \bv, \bw) \leq 2 \epsilon_1$. Then

	\begin{equation}\label{Re}
		\Re(h, \bv, \bw) \lesssim \epsilon_2.
	\end{equation}
	Furthermore, for each $t$ it holds that
	\begin{equation}\label{502}
		\|d \phi_{\theta,r}(t,\cdot)-dt \|_{C^{1,\delta_0}_{x'}} \lesssim \epsilon_2+  \| d {\mathbf{g}}(t,\cdot) \|_{C^{\delta_0}_x(\mathbb{R}^2)}.
	\end{equation}
\end{proposition}

\begin{proposition}\label{r3}
Assume $\frac74<s_0\leq s\leq \frac{15}{8} $. Suppose that $(h,\bv,\bw) \in \mathcal{{H}}$ and $\Re(h,\bv,\bw)\leq 2 \epsilon_1$.
	For any $1 \leq r \leq s+1$, and for each $t_0 \in [-2,2]$, the linear, non-homogenous equation
	\begin{equation*}
		\begin{cases}
			& \square_{\mathbf{g}} f=F, \qquad (t,x) \in [-2,2]\times \mathbb{R}^2,
			\\
			& f(t,x)|_{t=t_0}=f_0 \in H_x^r(\mathbb{R}^2), 
			\\
			& \partial_t f(t,x)|_{t=t_0}=f_1 \in H_x^{r-1}(\mathbb{R}^2),
		\end{cases}
	\end{equation*}
	admits a solution $f \in C([-2,2],H_x^r) \times C^1([-2,2],H_x^{r-1})$ and the following estimates holds:
	\begin{equation}\label{lw0}
		\| f\|_{L_t^\infty H_x^r}+ \|\partial_t f\|_{L_t^\infty H_x^{r-1}} \lesssim \|f_0\|_{H_x^r}+ \|f_1\|_{H_x^{r-1}}+\|F\|_{L^1_tH_x^{r-1}}.
	\end{equation}
	Additionally, the following estimates hold, provided $a<r-\frac34$,
	\begin{equation}\label{lw1}
		\| \left<\nabla \right>^a f\|_{L^4_{t}L^\infty_x} \lesssim \|f_0\|_{H_x^r}+ \|f_1\|_{H_x^{r-1}}+\|F\|_{L^1_tH_x^{r-1}}.
	\end{equation}
	The similar estimate \eqref{lw1} holds if we replace $\left<\nabla \right>^a$ by $\left<\nabla \right>^{a-1}d$.
\end{proposition}
Based on Propositions \ref{r2} and \ref{r3}, we are now ready to present the proof of Proposition \ref{p4}. 
\begin{proof}[ Proof of Proposition \ref{p4} ]
Due to Proposition \ref{r2} and Proposition \ref{r3}, it remains for us to verify \eqref{404} and \eqref{405}. Applying Theorem \ref{DW4}, \eqref{401}, and \eqref{403}, we have
\begin{equation}\label{w00}
		\|h\|_{L^\infty_{[-2,2]} H_x^{s}}+ \|\bv\|_{L^\infty_{[-2,2]} H_x^{s}}+  \|\bw\|_{L^\infty_{[-2,2]} H_x^{s_0-\frac14}} +  \|\nabla \bw\|_{L^\infty_{[-2,2]} L_x^{8}} \lesssim \epsilon_3.
\end{equation}
Note that $\epsilon_3 \ll \epsilon_2$. Therefore, the energy estimate \eqref{404} holds. By applying Proposition \ref{r3} to \eqref{XT}, for $\delta\in (0,s-\frac74)$, the following Strichartz estimate for $d\bv_{+}$ holds:
\begin{equation}\label{rtt}
	\| d\bv_{+} \|_{L^4_t C^\delta_x} \lesssim \|\partial_t \mathbf{T} \bv_{-} \|_{L^1_tH_x^{s-1}}+\|\nabla \mathbf{T} \bv_{-} \|_{L^1_tH_x^{s-1}}+  \| \bv_{-}\|_{L^1_tH_x^{s-1}}+\|\bQ\|_{L^1_tH_x^{s-1}}.
\end{equation}
The estimates \eqref{rtt}, \eqref{w00} combining with Lemma \ref{yuy} (taking $a=\frac12$), Lemma \ref{yux}, we therefore get
\begin{equation}\label{w02}
	\begin{split}
		\| d\bv_{+} \|_{L^4_t C^\delta_x} \lesssim & \| \bw\|_{L^2_t H_x^{\frac32}} (1+\| h,\bv \|^2_{L^\infty_t H_x^{s}}) + \|dh,d\bv\|_{L^1_t L_x^\infty} \| h,\bv \|_{L^\infty_t H^s_x}
		\\
		\lesssim & \epsilon_3 (1+ \epsilon^2_3) + \epsilon_3 \epsilon_2
		\\
		\lesssim & \epsilon_3.
	\end{split}
\end{equation}
By Sobolev imbedding $C^\delta_x(\mathbb{R}^2) \hookrightarrow H_x^{1+2\delta}(\mathbb{R}^2)$ and $L^4_t([-2,2]) \hookrightarrow H^{\frac14}_t([-2,2]) $, we then get
\begin{equation}\label{w04}
	\begin{split}
		\| d \bv_{-}\|_{L^4_t C^\delta_x} & \lesssim \| d \bv_{-} \|_{H^\frac14_t H_x^{1+2\delta}}
		\\
		& \lesssim \| \Lambda^{\frac14+2\delta}_{x}d \bv_{-} \|_{H^\frac14_t H_x^{\frac34}} 
		\\
		& \lesssim \| \Lambda^{\frac14+2\delta}_{x} \bv_{-} \|_{H^2_{t,x}} .
	\end{split}
\end{equation} 
By using Lemma \ref{jiaohuan3} and \ref{Ees}, we derive that
\begin{equation}\label{w05}
	\begin{split}
		\| \Lambda^{\frac14+2\delta}_{x} \bv_{-} \|_{H^2_{t,x}} \lesssim & \|\Lambda_x^{\frac14+2\delta} d\bw \|_{L^2_t L_x^{2}} + \|[\Lambda_x^{\frac14+2\delta}, \mathbf{{P}} ] v^\alpha_{-}\|_{L^2_t L_x^{2}} 
		\\
		\lesssim & \| d\bw \|_{L^2_t H_x^{{\frac14+2\delta}}} + \|\Lambda_x^{\frac14+2\delta} \bv\|_{L^\infty_t L^\infty_x} \|d^2 \bv \|_{L^2_t L_x^{2}} 
		\\
		\lesssim & \| \bw \|_{L^2_t H_x^{\frac54+2\delta}} ( 1+ \|h,\bv\|_{L^\infty_t H^s_x}).
	\end{split}
\end{equation}
Substituting \eqref{w05} to \eqref{w04}, and using \eqref{w00}, it yields
\begin{equation}\label{w06}
	\begin{split}
		\| d \bv_{-}\|_{L^4_t C^\delta_x} & \lesssim \| \bw \|_{L^2_t H_x^{\frac54+2\delta}} ( 1+ \|h,\bv\|_{L^\infty_t H^s_x})
		\\
		& \lesssim \epsilon_3.
	\end{split}
\end{equation} 
Adding \eqref{w02} and \eqref{w06}, the following estimate holds:
\begin{equation}\label{w07}
	\begin{split}
		\| d\bv \|_{L^4_t C^\delta_x} \lesssim \| d \bv_{+}\|_{L^4_t C^\delta_x}+ \| d \bv_{-}\|_{L^4_t C^\delta_x} \lesssim \epsilon_3.
	\end{split}
\end{equation}
Since \eqref{w06} holds and $\epsilon_3 \ll \epsilon_2$, we can derive the estimate \eqref{405}. Therefore, the proof of this proposition is complete.
\end{proof}
It remains for us to prove Propositions \ref{r2} and \ref{r3}. Since these proofs are lengthy, we will prove Proposition \ref{r2} in Section \ref{secp4}, and the proof of Proposition \ref{r3} will be presented in Section \ref{sec7}.
\section{Proof of Proposition \ref{r2}}\label{secp4}
Observing the conclusions in Proposition \ref{r2}, it depends on the regularity of the null hypersurfaces. We first establish the following result:
\begin{proposition}\label{r1}
	Assume $\frac74<s_0\leq s\leq \frac{15}{8} $.	Let $(h,\bv,\bw) \in \mathcal{H}$ so that $\Re(h,\bv,\bw) \leq 2 \epsilon_1$. Then
	\begin{equation*}\label{501}
		\vert\kern-0.25ex\vert\kern-0.25ex\vert  {\mathbf{g}}^{\alpha \beta}-\mathbf{m}^{\alpha \beta} \vert\kern-0.25ex\vert\kern-0.25ex\vert_{s_0-\frac14,2,{\Sigma_{\theta,r}}} + \vert\kern-0.25ex\vert\kern-0.25ex\vert \lambda ({\mathbf{g}}^{\alpha \beta}-P_{<\lambda} {\mathbf{g}}^{\alpha \beta}), d P_{ <\lambda } {\mathbf{g}}^{\alpha \beta}, \lambda^{-1} \nabla P_{<\lambda} d {\mathbf{g}}^{\alpha \beta}  \vert\kern-0.25ex\vert\kern-0.25ex\vert_{s_0-\frac54,2,{\Sigma_{\theta,r}}} \lesssim \epsilon_2.
	\end{equation*}
\end{proposition}
To prove Proposition \ref{r1}, including Proposition \ref{r2}, it suffices to restrict our attention to the case where $\theta=(0,1)$ and $r=0$. We fix this choice, and suppress $\theta$ and $r$ in our notation. We use $(x_2, x')$ instead of $(x_{\theta}, x'_{\theta})$. Then $\Sigma$ is defined by
\begin{equation*}
	\Sigma=\left\{ x_2- \phi(t,x')=0 \right\}.
\end{equation*}
Let $ \Delta_{x'}$ be a Laplacian operator on $\Sigma$, and
\begin{equation}\label{Sig}
	\Lambda_{x'}= (- \Delta_{x'})^{\frac12}.
\end{equation}
The hypothesis $\Re\leq 2 \epsilon_1$ implies that
\begin{equation}\label{503}
	\vert\kern-0.25ex\vert\kern-0.25ex\vert d \phi_{\theta,r}(t,\cdot)-dt \vert\kern-0.25ex\vert\kern-0.25ex\vert_{s_0-\frac14,2, \Sigma} \leq 2 \epsilon_1.
\end{equation}
By using \eqref{503} and Sobolev imbedding, we have
\begin{equation}\label{504}
	\|d \phi(t,x')-dt \|_{L^4_t C^{1,\delta_0}_{x'}} + \| \partial_t d \phi(t,x')\|_{L^4_t C^{\delta_0}_{x'}} \lesssim \epsilon_1,
\end{equation}
where we use $s_0>\frac74$ and $\delta_0 \in (0,s_0-\frac74)$. Noting that $(h,\bv,\bw) \in \mathcal{H}$, we can obtain
\begin{equation}\label{5021}
	\vert\kern-0.25ex\vert\kern-0.25ex\vert \bv \vert\kern-0.25ex\vert\kern-0.25ex\vert_{s,\infty}+\vert\kern-0.25ex\vert\kern-0.25ex\vert h\vert\kern-0.25ex\vert\kern-0.25ex\vert_{s,\infty}+\vert\kern-0.25ex\vert\kern-0.25ex\vert \bw \vert\kern-0.25ex\vert\kern-0.25ex\vert_{s_0-\frac14,\infty} +\|d\bv,dh, d\bv_{+}\|_{L^4_tC^\delta_x} \lesssim \epsilon_2.
\end{equation}
We are now ready to prove Proposition \ref{r1}.
\subsection{Characteristic energy estimates and the proof of Proposition \ref{r1}}
To start, let us introduce two lemmas about characteristic energy estimates, which are referred to in \cite{ST}.
\begin{Lemma}\label{te0}(\cite{ST}, Lemma 5.5)
	Assume $s\in (\frac74,2]$. Let $\tilde{f}(t,x)=f(t,x',x_2+\phi(t,x'))$. Then we have
	\begin{equation*}
		\vert\kern-0.25ex\vert\kern-0.25ex\vert \tilde{f}\vert\kern-0.25ex\vert\kern-0.25ex\vert_{s-\frac14,\infty}\lesssim \vert\kern-0.25ex\vert\kern-0.25ex\vert f\vert\kern-0.25ex\vert\kern-0.25ex\vert_{s-\frac14,\infty}, \quad \|d\tilde{f}\|_{L^4_tL_x^\infty}\lesssim \|d f\|_{{L^4_tL_x^\infty}}, \quad  \|\tilde{f}\|_{H^{s-\frac14}_{x}}\lesssim \|f\|_{H^{s-\frac14}_{x}}.
	\end{equation*}
\end{Lemma}
\begin{Lemma}\label{te2}(\cite{ST}, Lemma 5.4)
	For $r\geq 1$, we have
	\begin{equation*}
		\begin{split}
			\sup_{t\in[-2,2]} \| f\|_{H^{r-\frac{1}{2}}(\mathbb{R}^2)} & \lesssim \vert\kern-0.25ex\vert\kern-0.25ex\vert f \vert\kern-0.25ex\vert\kern-0.25ex\vert_{r,2},
			\\
			\sup_{t\in[-2,2]} \| f\|_{H^{r-\frac{1}{2}}(\Sigma^t)} & \lesssim \vert\kern-0.25ex\vert\kern-0.25ex\vert f \vert\kern-0.25ex\vert\kern-0.25ex\vert_{r,2,\Sigma}.
		\end{split}
	\end{equation*}
	If $r> \frac{3}{2}$, then
	\begin{equation*}
		\vert\kern-0.25ex\vert\kern-0.25ex\vert f_1f_2\vert\kern-0.25ex\vert\kern-0.25ex\vert_{r,2}\lesssim  \vert\kern-0.25ex\vert\kern-0.25ex\vert f_2 \vert\kern-0.25ex\vert\kern-0.25ex\vert_{r,2} \vert\kern-0.25ex\vert\kern-0.25ex\vert f_1\vert\kern-0.25ex\vert\kern-0.25ex\vert_{r,2}.
	\end{equation*}
	Similarly, if $r>1$, then
	\begin{equation*}
		\vert\kern-0.25ex\vert\kern-0.25ex\vert f_1f_2\vert\kern-0.25ex\vert\kern-0.25ex\vert_{r,2,\Sigma}\lesssim  \vert\kern-0.25ex\vert\kern-0.25ex\vert f_2 \vert\kern-0.25ex\vert\kern-0.25ex\vert_{r,2,\Sigma} \vert\kern-0.25ex\vert\kern-0.25ex\vert f_1 \vert\kern-0.25ex\vert\kern-0.25ex\vert_{r,2,\Sigma}.
	\end{equation*}
\end{Lemma}
Next, we prove two characteristic energy estimates for the hyperbolic system \eqref{QHl}.
\begin{Lemma}\label{te1}
Assume $\frac74<s_0\leq s\leq \frac{15}{8} $. Suppose that $\bU=(p(h), \mathrm{e}^{-h}v^1, \mathrm{e}^{-h}v^2 )^{\mathrm{T}}$ satisfies the hyperbolic symmetric system
	\begin{equation}\label{505}
		A^0(\bU) \partial_t\bU+ \sum^{2}_{i=1}A^i(\bU) \partial_i \bU= 0.
	\end{equation}
	Then
	\begin{equation}\label{te10}
		\begin{split}
			\vert\kern-0.25ex\vert\kern-0.25ex\vert  \bU\vert\kern-0.25ex\vert\kern-0.25ex\vert_{s_0-\frac14,2,\Sigma} & \lesssim \|d \bU \|_{L^4_t L^{\infty}_x}+ \| \bU\|_{L^{\infty}_tH_x^{s_0-\frac14}}.
		\end{split}
	\end{equation}
\end{Lemma}

\begin{proof}
	By a change of coordinates $x_2 \rightarrow x_2-\phi(t,x')$ and setting $\tilde{\bU}(t,x)=U(t,x',x_2+\phi(t,x'))$, the system \eqref{505} is transformed to
	\begin{equation*}
		A^0(\tilde{\bU}) \partial_t \tilde{\bU}+ \sum_{i=1}^2 A^i(\tilde{\bU}) \partial_{i} \tilde{\bU}= - \partial_t \phi  \partial_2 \tilde{\bU} - \sum_{\alpha=0}^2 A^\alpha(\tilde{\bU}) \partial_{\alpha}\phi \partial_\alpha \tilde{\bU}.
	\end{equation*}
	For $\phi$ is independent of $x_2$, we further get
	\begin{equation}\label{U}
		A^0(\tilde{\bU}) \partial_t \tilde{\bU}+ \sum_{i=1}^2 A^i(\tilde{\bU}) \partial_{i} \tilde{\bU}= - \partial_t \phi  \partial_2 \tilde{\bU} - \sum_{\alpha=0}^2 A^\alpha(\tilde{\bU}) \partial_{\alpha}\phi \partial_\alpha \tilde{\bU}.
	\end{equation}
	To prove \eqref{te10}, we first establish the $0$-order estimate. A direct calculation on$[-2,2]\times \mathbb{R}^2$ shows that
	\begin{equation*}
		\begin{split}
			\vert\kern-0.25ex\vert\kern-0.25ex\vert \tilde{\bU}\vert\kern-0.25ex\vert\kern-0.25ex\vert^2_{0,2,\Sigma} & \lesssim \| d\tilde{\bU} \|_{L^1_t L_x^\infty}\|\tilde{\bU}\|_{L_x^2} + \|\nabla d\phi\|_{L^1_t L_{x'}^\infty}\|\tilde{\bU}\|_{L_x^2}
			\\
			& \lesssim \| d\tilde{\bU} \|_{L^1_t L_x^\infty}\|\tilde{\bU}\|_{L_x^2} + \|\nabla d\phi\|_{L^1_t L_x^\infty}\|\tilde{\bU}\|_{L_x^2}.
		\end{split}
	\end{equation*}
	By using Lemma \ref{te0}, \eqref{504} and \eqref{5021}, we can prove that
	\begin{equation}\label{U0}
		\vert\kern-0.25ex\vert\kern-0.25ex\vert \bU\vert\kern-0.25ex\vert\kern-0.25ex\vert_{0,2,\Sigma} \lesssim \|d \bU \|_{L^2_t L^{\infty}_x}+ \| \bU\|_{L^{\infty}_tL_x^2}.
	\end{equation}
	Next, we will prove the $(s_0-\frac14)$-order estimate. Operating the derivative of $\partial^{\beta}_{x'}$($1 \leq |\beta| \leq s_0-\frac14$) on \eqref{U} and integrating it on $[-2,2]\times \mathbb{R}^2$, we get
	\begin{equation}\label{U1}
		\begin{split}
			\| \partial^{\beta}_{x'} \tilde{\bU}\|^2_{L^2_{\Sigma}} & \lesssim \| d \tilde{\bU} \|_{L^1_t L^{\infty}_x} \| \partial^{\beta}_{x} \tilde{\bU}\|_{L^{\infty}_tL_x^2} +|I_1|+|I_2|,
		\end{split}
	\end{equation}
	where
	\begin{equation*}
		\begin{split}
			&I_1= -\int_{-2}^2\int_{\mathbb{R}^2} \partial^{\beta}_{x'}  \big( \partial_t \phi  \partial_3 \tilde{\bU} \big) \cdot \Lambda^{\beta}_{x'} \tilde{\bU}  dxd\tau,
			\\
			& I_2= -\sum^2_{\alpha=0}\int_{-2}^2\int_{\mathbb{R}^2} \partial^{\beta}_{x'} \big( A^\alpha(\tilde{\bU}) \partial_{\alpha}\phi \partial_\alpha \tilde{\bU} \big) \cdot \partial^{\beta}_{x'} \tilde{\bU}   dxd\tau.
		\end{split}
	\end{equation*}
	Write $I_1$ as
	\begin{equation*}
		\begin{split}
			I_1 =& -\int_{-2}^2\int_{\mathbb{R}^2} \big( \partial^{\beta}_{x'}(\partial_t \phi  \partial_3 \tilde{\bU})-\partial_t \phi \partial_3 \partial^{\beta}_{x'} \tilde{\bU} \big)   \partial^{\beta}_{x'} \tilde{\bU}dxd\tau
			\\
			& + \int_{-2}^2\int_{\mathbb{R}^2} \partial_t \phi \cdot \partial_3 \partial^{\beta}_{x'} \tilde{\bU} \cdot \partial^{\beta}_{x'} \tilde{\bU}  dx d\tau,
			\\
			=& -\int_{-2}^2\int_{\mathbb{R}^2} [ \partial^{\beta}_{x'}, \partial_t \phi  \partial_3] \tilde{\bU} \cdot \partial^{\beta}_{x'} \tilde{\bU}dxd\tau
		\end{split}
	\end{equation*}
	Similarly, $I_2$ can be written by
	\begin{equation*}
		\begin{split}
			I_2 = & -\sum^2_{\alpha=0}\int_{-2}^2\int_{\mathbb{R}^2} \big( \partial^{\beta}_{x'} \big( A^\alpha(\tilde{\bU}) \partial_{\alpha}\phi \partial_\alpha \tilde{\bU}) -  A^\alpha(\tilde{\bU}) \partial_{\alpha}\phi \partial_\alpha \partial^{\beta}_{x'} \tilde{\bU} \big) \cdot \partial^{\beta}_{x'} \tilde{\bU}  dxd\tau
			\\
			& +\sum^2_{\alpha=0} \int_{-2}^2\int_{\mathbb{R}^2} \big( A^\alpha(\tilde{\bU}) \partial_{\alpha}\phi \big) \cdot \partial_\alpha(\partial^{\beta}_{x'} \tilde{\bU}) \cdot \partial^{\beta}_{x'} \tilde{\bU}dxd\tau.
		\end{split}
	\end{equation*}
	By using Lemma \ref{jiaohuan}, we infer that
	\begin{equation}\label{U2}
		\begin{split}
			|I_1|
			& \lesssim  \big( \|\partial^\beta \tilde \bU\|_{L^{\infty}_tL_x^2} \| d \partial_t \phi \|_{L^1_tL_{x}^\infty}+ \sup_{\theta, r} \|\partial^{\beta}_{x'} \partial_t \phi\|_{L^2(\Sigma_{\theta,r})} \| d \tilde{\bU}\|_{L^1_tL_x^\infty} \big)\cdot\|\partial^\beta \tilde \bU\|_{L^\infty_tL_x^2}
		\end{split}
	\end{equation}
	and
	\begin{equation}\label{U3}
		\begin{split}
			|I_2| \lesssim  & \big(\| \partial^{\beta}_{x'} \tilde{\bU} \|_{L^2_tL_x^2} \|d \nabla \phi\|_{L^1_t L^\infty_x} + \|d\tilde{\bU}\|_{L^1_tL_x^\infty} \sup_{\theta,r}\|\partial^{\beta}_{x'}d \phi\|_{L^2(\Sigma_{\theta,r})}  \big) \cdot \|\partial^{\beta}_{x'} \tilde{\bU} \|_{L^\infty_tL_x^2}
			\\
			& + \big( \| d \tilde{\bU} \|_{L^2_t L_x^\infty} \| \nabla \phi\| _{L^2_tL_x^\infty}+ \|\tilde{\bU}\|_{L^2_t L_x^\infty} \|\nabla^2\phi\|_{L^2_tL_{x}^\infty} \big)  \cdot \|\partial^\beta \tilde{\bU} \|^2_{L^\infty_t L_x^2}.
		\end{split}
	\end{equation}
	Take sum of $1\leq \beta \leq s_0-\frac14$ on \eqref{U1}. By using Lemma \ref{te0}, \eqref{U2}, \eqref{U3}, \eqref{5021}, and \eqref{504}, we obtain
	\begin{equation}\label{U4}
		\begin{split}
			\vert\kern-0.25ex\vert\kern-0.25ex\vert \partial_{x'} \bU\vert\kern-0.25ex\vert\kern-0.25ex\vert_{s_0-\frac54,2,\Sigma} & \lesssim \|d \bU \|_{L^4_t L^{\infty}_x}+\|d \bU\|_{L^{\infty}_tH_x^{s_0-1}}.
		\end{split}
	\end{equation}
	Operating $\nabla$ on \eqref{505}, we have
	\begin{equation*}
		\begin{split}
			A^0(\bU) \partial_t(\nabla \bU) + \sum^2_{i=1} A^i(\bU) \partial_i(\nabla \bU)=-\sum^2_{\alpha=0}\nabla (A^\alpha(\bU)) \partial_\alpha \bU,
		\end{split}
	\end{equation*}
	In a similar way, we can obtain
	\begin{equation}\label{U40}
		\begin{split}
			\vert\kern-0.25ex\vert\kern-0.25ex\vert \nabla \bU \vert\kern-0.25ex\vert\kern-0.25ex\vert_{s_0-\frac14,2,\Sigma} & \lesssim \|d \bU \|_{L^4_t L^{\infty}_x}+\|d \bU\|_{L^{\infty}_tH_x^{s_0-1}}.
		\end{split}
	\end{equation}
	Thanks to $\partial_t \bU=- \sum^2_{i=1}(A^0)^{-1}A^i(\bU)\partial_i \bU$ and Lemma \ref{te2}, we can carry out
	\begin{equation}\label{U5}
		\begin{split}
			\vert\kern-0.25ex\vert\kern-0.25ex\vert \partial_t \bU \vert\kern-0.25ex\vert\kern-0.25ex\vert_{s_0-\frac14,2,\Sigma}
			& \lesssim \vert\kern-0.25ex\vert\kern-0.25ex\vert \bU \vert\kern-0.25ex\vert\kern-0.25ex\vert_{s_0-\frac14,2,\Sigma} \vert\kern-0.25ex\vert\kern-0.25ex\vert\partial_t \bU\vert\kern-0.25ex\vert\kern-0.25ex\vert_{s_0-\frac14,2,\Sigma}
			\lesssim \|d \bU \|_{L^4_t L^{\infty}_x}.
		\end{split}
	\end{equation}
	The estimate \eqref{U5} together with \eqref{U0}, \eqref{U4}, \eqref{U40} gives \eqref{te10}. This conclude the proof of this lemma.
\end{proof}
\begin{Lemma}\label{fre}
	Assume $\frac74<s_0\leq s\leq \frac{15}{8} $. Let $\bU$ satisfy the assumption in Lemma \ref{te1}. Then
	\begin{equation}\label{508}
		\vert\kern-0.25ex\vert\kern-0.25ex\vert  \lambda(\bU-P_{<\lambda} \bU), d P_{<\lambda} \bU, \lambda^{-1} d \nabla P_{<\lambda} \bU \vert\kern-0.25ex\vert\kern-0.25ex\vert_{s_0-\frac{5}{4},2,\Sigma} \lesssim \epsilon_2. 
	\end{equation}
\end{Lemma}
\begin{proof}
	Let $\Delta_0$ be a standard multiplier of order $0$ on $\mathbb{R}^2$, such that $\Delta_0$ is additionally bounded on $L^\infty(\mathbb{R}^2)$. Clearly, we have
	\begin{equation*}
		A^0(\bU) \partial_t(\Delta_0U)+ \sum^2_{i=1}A^i(\bU) \partial_i(\Delta_0 \bU)= -[\Delta_0, A^\alpha(\bU)]\partial_{\alpha}\bU.
	\end{equation*}
	Applying Lemma \ref{te1}, we derive
	\begin{equation}\label{60}
		\vert\kern-0.25ex\vert\kern-0.25ex\vert \Delta_0U\vert\kern-0.25ex\vert\kern-0.25ex\vert^2_{s_0-\frac{1}{4},2,\Sigma} \lesssim \|d \bU \|_{L^4_t L^{\infty}_x}\| \Delta_0 \bU\|^2_{L^{\infty}_tH_x^{s_0}}+\| [\Delta_0, A^\alpha(\bU)]\partial_{\alpha}\bU \|_{L^{2}_tH_x^{s_0-\frac{1}{4}}}\| \Delta_0 \bU\|_{L^{\infty}_tH_x^{s_0-\frac{1}{4}}}.
	\end{equation}
	Due to commutator estimates, we obtain
	\begin{equation*}
		\| [\Delta_0, A^\alpha(\bU)]\partial_{\alpha}\bU \|_{L^{2}_tH_x^{s_0-\frac{1}{4}}} \lesssim \|d \bU \|_{L^4_t L^{\infty}_x}\| \Delta_0 \bU\|_{L^{\infty}_tH_x^{s_0-\frac{1}{4}}}.
	\end{equation*}
	According to the inequality \eqref{60}, it turns out
	\begin{equation}\label{60e}
		\vert\kern-0.25ex\vert\kern-0.25ex\vert \Delta_0 \bU \vert\kern-0.25ex\vert\kern-0.25ex\vert^2_{s_0-\frac{1}{4},2,\Sigma} \lesssim \|d \bU \|_{L^4_t L^{\infty}_x}\| \Delta_0 \bU\|^2_{L^{\infty}_tH_x^{s_0-\frac{1}{4}}}.
	\end{equation}
	To get a bound for $\lambda(\bU-P_{<\lambda} \bU)$, we write
	\begin{equation*}
		\lambda(\bU-P_{<\lambda} \bU)= \lambda \sum_{\mu\geq \lambda}P_\mu \bU,
	\end{equation*}
	where $P_\mu \bU$ satisfies the above conditions for $\Delta_0 \bU$. Applying \eqref{60e} and replacing $s_0-\frac{1}{4}$ to $s_0-\frac{5}{4}$, we get
	\begin{equation*}
		\begin{split}
			\vert\kern-0.25ex\vert\kern-0.25ex\vert \lambda(\bU-P_{<\lambda} \bU) \vert\kern-0.25ex\vert\kern-0.25ex\vert^2_{s_0-\frac{5}{4},2,\Sigma}
			=	& \sum_{\mu\geq \lambda} \vert\kern-0.25ex\vert\kern-0.25ex\vert \mu P_\mu \bU \cdot \frac{\lambda}{\mu}  \vert\kern-0.25ex\vert\kern-0.25ex\vert^2_{s_0-\frac{5}{4},2,\Sigma}
			\\
			\lesssim & \| \bU \|^2_{L^{\infty}_tH_x^{s_0-\frac{1}{4}}} \lesssim \epsilon^2_2.
		\end{split}
	\end{equation*}
	Taking square of the above inequality, which yields
	\begin{equation*}
		\vert\kern-0.25ex\vert\kern-0.25ex\vert \lambda(\bU-P_{<\lambda} \bU) \vert\kern-0.25ex\vert\kern-0.25ex\vert_{s_0-\frac{5}{4},2,\Sigma}
		\lesssim  \epsilon_2.
	\end{equation*}
	Finally, applying \eqref{60} to $\Delta_0=P_{<\lambda}$ and $\Delta_0=\lambda^{-1}\nabla P_{<\lambda}$ shows that
	\begin{equation*}
		\vert\kern-0.25ex\vert\kern-0.25ex\vert d P_{<\lambda} \bU\vert\kern-0.25ex\vert\kern-0.25ex\vert_{s_0-\frac{5}{4},2,\Sigma} +\vert\kern-0.25ex\vert\kern-0.25ex\vert \lambda^{-1} d \nabla P_{<\lambda} \bU \vert\kern-0.25ex\vert\kern-0.25ex\vert_{s_0-\frac{5}{4},2,\Sigma} \lesssim \epsilon_2.
	\end{equation*}
	Therefore, the estimate \eqref{508} holds. We have finished the proof of this lemma.
\end{proof}
Taking advantage of Lemma \ref{te1}, inequalities \eqref{402a} and \eqref{403}, we can directly obtain:
\begin{corollary}\label{vte}
	Assume $\frac74<s_0\leq s\leq \frac{15}{8} $. 	Suppose $(h, \bv, \bw) \in \mathcal{H}$. Then the following estimate holds:
	\begin{equation*}\label{e017}
		\vert\kern-0.25ex\vert\kern-0.25ex\vert \bv \vert\kern-0.25ex\vert\kern-0.25ex\vert_{s_0-\frac14,2,\Sigma}+ \vert\kern-0.25ex\vert\kern-0.25ex\vert h \vert\kern-0.25ex\vert\kern-0.25ex\vert_{s_0-\frac14,2,\Sigma}  \lesssim \epsilon_2.
	\end{equation*}
\end{corollary}
Next, let us prove the characteristic energy estimates for vorticity.
\begin{Lemma}\label{te20}
 Assume $\frac74<s_0\leq s\leq \frac{15}{8} $. 	Suppose that $(h, \bv, \bw) \in \mathcal{H}$. Then we have 
	\begin{equation}\label{te201}
		\begin{split}
			\vert\kern-0.25ex\vert\kern-0.25ex\vert  \bw  \vert\kern-0.25ex\vert\kern-0.25ex\vert_{s_0-\frac14,2,\Sigma}+
			\vert\kern-0.25ex\vert\kern-0.25ex\vert  d \bw  \vert\kern-0.25ex\vert\kern-0.25ex\vert_{s_0-\frac54,2,\Sigma} \lesssim  \epsilon_2.
		\end{split}
	\end{equation}
\end{Lemma}
\begin{proof}
	By a change of coordinates $x_2 \rightarrow x_2-\phi(t,x')$ and setting $\tilde{\bw}(t,x)= \bw(t,x',x_2+\phi(t,x'))$, the third equation in \eqref{wrt} is transformed to
	\begin{equation*}
		\tilde{v}^\kappa	\partial_\kappa \tilde{w}^{\alpha} = - \tilde{v}^\kappa \partial_\kappa \phi \tilde{w}^{\alpha}.
	\end{equation*}
	For $\phi$ is independent of $x_2$, we then get
	\begin{equation}\label{VW1}
		\partial_t \tilde{w}^{\alpha}+ (\tilde{v}^0)^{-1} \tilde{v}^i \partial_i \tilde{w}^{\alpha}= - \partial_t \phi \tilde{w}^{\alpha} - (\tilde{v}^0)^{-1} \tilde{v}^1 \partial_1 \phi \tilde{w}^{\alpha}.
	\end{equation}
	Multiplying with $\tilde{w}_{\alpha}$ on \eqref{VW1}, and integrating it on $[0,t] \times \mathbb{R}^2$, we can obtain 
	\begin{equation*}
		\begin{split}
			\| \tilde{\bw} \|^2_{0,2,\Sigma}\leq & C \|d \bv\|_{L^4_t L^\infty_x}\| \tilde{\bw} \|^2_{L^\infty_t L_x^{2}}+ C\| d \phi -dt\|_{s_0-\frac14,2,\Sigma} \| \tilde{\bw} \|^2_{L^\infty_t L_x^{2}} .
		\end{split}
	\end{equation*}
Applying the inequalities \eqref{401} and \eqref{403}, it yields
	\begin{equation}\label{VW2}
	\| \tilde{\bw} \|_{0,2,\Sigma} \lesssim \epsilon_2.
\end{equation}
For $\bw=(w^0,\mathring{\bw})$, we get
\begin{equation*}
	\begin{split}
		\|\nabla \bw\|_{s_0-\frac54,2,\Sigma}=  & \|\nabla w^0\|_{s_0-\frac54,2,\Sigma} + \|\nabla {\mathring{\bw}}\|_{s_0-\frac54,2,\Sigma} .
	\end{split}
\end{equation*}
By Hodge's decomposition, we know
\begin{equation}\label{VW5}
	\nabla {\mathring{\bw}}= (-\Delta)^{-1}\nabla^2 \left( \textrm{curl} {\mathring{\bw}} +  \textrm{div} {\mathring{\bw}} \right). 
\end{equation}
Due to \eqref{ew16}, \eqref{ew18}, and \eqref{VW5}, we can see
\begin{equation}\label{VW6}
	\begin{split}
		\|\nabla \bw\|_{s_0-\frac54,2,\Sigma}\lesssim  & \|\bW\|_{s_0-\frac54,2,\Sigma} + \|(-\Delta)^{-1}\nabla^2 \bW\|_{s_0-\frac54,2,\Sigma}+ \|\bw\cdot (dh,d\bv)\|_{s_0-\frac54,2,\Sigma} .
	\end{split}
\end{equation}
Using \eqref{EW1}, by a change of coordinates, we have
	\begin{equation}\label{VW7}
	\begin{split}
		\tilde{v}^\kappa \partial_\kappa    \tilde{W}^\alpha   
		=& \tilde{F}^\alpha - \tilde{v}^\kappa \partial_\kappa \phi \tilde{W}^\alpha,
	\end{split}
\end{equation}
where 
\begin{equation*}
	\begin{split}
		F^\alpha = 	& -\epsilon^{\alpha \beta \gamma} \partial_\beta v^\kappa \partial_\kappa w_\gamma- \epsilon^{\alpha \beta \gamma} \partial_\beta w_\gamma  \partial^\kappa v_\kappa 
		+ \epsilon^{\alpha \beta \gamma}\ \partial_\beta w^\kappa \partial_\gamma v_\kappa
		\\
		&	+ \epsilon^{\alpha \beta \gamma} w_\gamma \partial_\beta \left\{ ( c^{-2}_s-1) v^\kappa \right\} \partial_\kappa h 
		- \epsilon^{\alpha \beta \gamma} ( c^{-2}_s-1) v^\kappa \partial_\kappa w_\gamma \partial_\beta h 
		\\
		&	+2 \epsilon^{\alpha \beta \gamma}  c^{-3}_s c'_s v^\kappa w_\gamma \partial_\kappa h \partial_\beta h.
	\end{split}
\end{equation*}
	Operating derivatives $\Lambda_{x'}^{\alpha}$ ($0\leq \alpha \leq s_0-\frac54$) on \eqref{VW7}, multiplying with $\Lambda_{x'}^{\alpha} \tilde{\bW}$, and integrating on $[0,t] \times \mathbb{R}^2$, we can obtain 
	\begin{equation*}
		\begin{split}
			\| \bW \|^2_{s_0-\frac14,2,\Sigma}\leq & C \|d \bv\|_{L^4_t L^\infty_x}\| {\bw} \|^2_{L^\infty_t H_x^{s_0-\frac14}}+ C\| d \phi -dt\|_{s-\frac14,2,\Sigma} \| {\bw} \|^2_{L^\infty_t H_x^{s_0-\frac14}}
			\\
			&+  C\| d \phi -dt\|_{s_0-\frac14,2,\Sigma}\| \tilde{\bw} \|_{L^\infty_t H_x^{s_0-\frac14}} 
			{\| \nabla {\bw} \|_{L^\infty_t L^8_x} } (1+\|\bv\|_{L^\infty_t H^{s_0}_x}).
		\end{split}
	\end{equation*}
	Taking advantage of inequalities \eqref{401} and \eqref{403}, we find that
	\begin{equation}\label{VW9}
		\| \bW \|_{s_0-\frac54,2,\Sigma} \lesssim \epsilon_2.
	\end{equation}
	Multiplying $(v^0)^{-1}$ on \eqref{EW1} and operating $(-\Delta)^{-1}\nabla^2$, we have
	\begin{equation}\label{VW1q}
		\begin{split}
		\partial_t (-\Delta)^{-1}\nabla^2 W^\alpha + (v^0)^{-1} v^i \partial_i (-\Delta)^{-1}\nabla^2 W^\alpha= & (-\Delta)^{-1}\nabla^2 \left\{ (v^0)^{-1}   F^\alpha  \right\}
		\\
		& + [(v^0)^{-1} v^i \partial_i, (-\Delta)^{-1}\nabla^2]W^\alpha.
		\end{split}
	\end{equation}
	The commutator estimate tells us
	\begin{equation}\label{VW8}
	\|	[(v^0)^{-1} v^i \partial_i, (-\Delta)^{-1}\nabla^2]\bW \|_{H^{s_0-\frac54}_x} \lesssim \|d\bv\|_{L^\infty_x} \|\bW \|_{H^{s_0-\frac54}_x} .
	\end{equation}
	Similarly, by change of coordinates on \eqref{VW1q}, we can establish
		\begin{equation}\label{VW11}
		\| (-\Delta)^{-1}\nabla^2 \bW \|_{s_0-\frac54,2,\Sigma} \lesssim \epsilon_2.
	\end{equation}
	By trace theorem, we have
	\begin{equation}\label{VW13}
		\begin{split}
			\|\bw\cdot (dh,d\bv)\|_{s_0-\frac54,2,\Sigma} \lesssim & \|\bw\cdot (dh,d\bv)\|_{L^2_t H_x^{s_0-\frac34}}
			\\
			\lesssim &  \|\bw\|_{L^\infty_t H_x^{s_0-\frac14}}\| h,\bv\|_{L^2_t H_x^{s}}
			\\
			\lesssim & \epsilon^2_2.
		\end{split}
	\end{equation}
	Combining \eqref{VW6}, \eqref{VW9}, \eqref{VW11}, and \eqref{VW13}, we get
	\begin{equation}\label{VW15}
		\begin{split}
			\|\nabla \bw\|_{s_0-\frac54,2,\Sigma}\lesssim  & \epsilon_2.
		\end{split}
	\end{equation}
By using \eqref{VW1} and \eqref{VW15}, it tells us
	\begin{equation}\label{VW17}
		\begin{split}
			\|\partial_t \bw\|_{s_0-\frac54,2,\Sigma}\lesssim  & \epsilon_2.
		\end{split}
	\end{equation}
	Due to $\partial_{x'}\bw =  \nabla \bw \partial_{x'} \phi$, we obtain
	\begin{equation}\label{VW19}
	\| \tilde{\bw} \|_{s_0-\frac14,2,\Sigma}\lesssim	\|\nabla \tilde{\bw} \|_{s_0-\frac54,2,\Sigma}  .
	\end{equation}
	Combining \eqref{VW2}, \eqref{VW15}, \eqref{VW17} with \eqref{VW19} yields the estimate \eqref{te201}. This concludes the proof of this lemma.
\end{proof}
Finally, we are ready to prove Proposition \ref{r1}.
\begin{proof}[Proof of Proposition \ref{r1}]
	From Lemma \ref{fre}, it only remains for us to verify
	\begin{equation*}
		\begin{split}
			\vert\kern-0.25ex\vert\kern-0.25ex\vert {\mathbf{g}}^{\alpha \beta}-\mathbf{m}^{\alpha \beta}\vert\kern-0.25ex\vert\kern-0.25ex\vert_{s_0-\frac{1}{4},2,\Sigma_{\theta,r}} \lesssim \epsilon_2.
		\end{split}
	\end{equation*}
	Due to Lemma \ref{te1}, \eqref{504}, and \eqref{5021}, we get
	\begin{equation*}
		\sup_{\theta,r}\vert\kern-0.25ex\vert\kern-0.25ex\vert \bv \vert\kern-0.25ex\vert\kern-0.25ex\vert_{s_0-\frac{1}{4},2,\Sigma_{\theta,r}}+ \sup_{\theta,r}\vert\kern-0.25ex\vert\kern-0.25ex\vert h \vert\kern-0.25ex\vert\kern-0.25ex\vert_{s_0-\frac{1}{4},2,\Sigma_{\theta,r}} \lesssim \epsilon_2.
	\end{equation*}
	By using the definition of $\mathbf{g}$ (see \eqref{AMd3}), and Lemma \ref{te2}, the following estimate
	\begin{equation*}
		\begin{split}
			\vert\kern-0.25ex\vert\kern-0.25ex\vert  {\mathbf{g}}^{\alpha \beta}-\mathbf{m}^{\alpha \beta}\vert\kern-0.25ex\vert\kern-0.25ex\vert _{s_0-\frac{1}{4},2,\Sigma_{\theta,r}} & \lesssim \vert\kern-0.25ex\vert\kern-0.25ex\vert \bv\vert\kern-0.25ex\vert\kern-0.25ex\vert_{s_0-\frac{1}{4},2,\Sigma_{\theta,r}}+\vert\kern-0.25ex\vert\kern-0.25ex\vert \bv \cdot \bv\vert\kern-0.25ex\vert\kern-0.25ex\vert_{s_0-\frac{1}{4},2,\Sigma_{\theta,r}}+\vert\kern-0.25ex\vert\kern-0.25ex\vert c_s^2-c_s^2(0)\vert\kern-0.25ex\vert\kern-0.25ex\vert_{s_0-\frac{1}{4},2,\Sigma_{\theta,r}}
			\\
			& \lesssim \epsilon_2,
		\end{split}
	\end{equation*}
	holds for $s_0>\frac74$. Therefore, the estimate \eqref{501} holds. This concludes the proof of this lemma.
\end{proof}
It also remains for us to obtain Proposition \ref{r2}. Next, we introduce a new frame on the null hypersurface $\Sigma$.
\subsection{Null frame}
We introduce a null frame along $\Sigma$ as follows. Let
\begin{equation*}
	V=(dr)^*,
\end{equation*}
where $r$ is the defining function of the foliation $\Sigma$, and where $*$ denotes the identification of covectors and vectors induced by $\mathbf{g}$. Then $V$ is the null geodesic flow field tangent to $\Sigma$. Let
\begin{equation*}\label{600}
	\sigma=dt(V), \qquad l=\sigma^{-1} V.
\end{equation*}
Thus $l$ is the $\mathbf{g}$-normal field to $\Sigma$ normalized so that $dt(l)=1$, hence
\begin{equation*}\label{601}
	l=\left< dt,dx_2-d\phi\right>^{-1}_{\mathbf{g}} \left( dx_2-d \phi \right)^*.
\end{equation*}
Therefore, the coefficients $l^j$ are smooth functions of $\bv, \rho$ and $d \phi$. Conversely, we have
\begin{equation}\label{602}
	dx_2-d \phi =\left< l,\partial_{2}\right>^{-1}_{\mathbf{g}} l^*.
\end{equation}
Seeing from \eqref{602}, $d \phi$ is also a smooth function of $\bv, \rho$ and the coefficients of $l$.

Next we introduce the vector fields $e_1$ tangent to the fixed-time slice $\Sigma^t$ of $\Sigma$. We do this by applying Grahm-Schmidt orthogonalization in the metric $\mathbf{g}$ to the $\Sigma^t$-tangent vector fields $\partial_{1}+ \partial_{1} \phi \partial_{2}$.

Finally, we denote
\begin{equation*}
	\underline{l}=l+2\partial_t.
\end{equation*}
It follows that $\{l, \underline{l}, e_1 \}$ form a null frame in the sense that
\begin{align*}
	& \left<l, \underline{l} \right>_{\mathbf{g}} =2, \qquad \qquad \ \ \left< e_1, e_1\right>_{\mathbf{g}}=1,
	\\
	& \left<l, l \right>_{\mathbf{g}} =\left<\underline{l}, \underline{l} \right>_{\mathbf{g}}=0, \quad \left<l, e_1 \right>_{\mathbf{g}}=\left<\underline{l}, e_1 \right>_{\mathbf{g}}=0 .
\end{align*}
The coefficients of each fields is a smooth function of $h$, $\bv$, and $d \phi$, and by assumption we also have the pointwise bound
\begin{equation*}
	| e_1 - \partial_{1} |  + | l- (\partial_t+\partial_{2}) | + | \underline{l} - (-\partial_t+\partial_{2})|  \lesssim \epsilon_1.
\end{equation*}
Based on the above setting, we introduce a result about the decomposition of curvature tensor.
\begin{Lemma}\label{LLQ}(\cite{ST}, Lemma 5.8)
	Assume $\frac74<s_0\leq s\leq \frac{15}{8} $. Suppose $f$ satisfying $$\mathbf{g}^{\alpha \beta} \partial^2_{\alpha \beta}f=F.$$
	Let $(t,x',\phi(t,x'))$ denote the projective parametrisation of $\Sigma$, and for $0 \leq \alpha, \beta \leq 1$, let $/\kern-0.55em \partial_\alpha$ denote differentiation along $\Sigma$ in the induced coordinates. Then, for $0 \leq \alpha, \beta \leq 1$, one can write
	\begin{equation*}
		/\kern-0.55em \partial_\alpha /\kern-0.55em \partial_\beta (f|_{\Sigma}) = l(f_2)+ f_1,
	\end{equation*}
	where
	\begin{equation*}
		\| f_2 \|_{L^2_t H^{s_0-\frac54}_{x'}(\Sigma)}+\| f_1 \|_{L^1_t H^{s_0-\frac54}_{x'}(\Sigma)} \lesssim \|df\|_{L^\infty_t H_x^{s_0-\frac54}}+ \|df\|_{L^4_t L_x^\infty}+  \| F\|_{L^1_t H^{s_0-\frac54}_{x'}(\Sigma)}.
	\end{equation*}
\end{Lemma}
\begin{corollary}\label{Rfenjie}
Assume $\frac74<s_0\leq s\leq \frac{15}{8} $, $ \delta\in (0, s-\frac{7}{4})$, and $\delta_0\in (0,s_0-\frac74)$. Let $R$ be the Riemann curvature tensor for the metric ${\mathbf{g}}$. Let $e_0=l$. Then for any $0 \leq a, b, c,d \leq 2$, we can write
	\begin{equation}\label{603}
		\left< R(e_a, e_b)e_c, e_d \right>_{\mathbf{g}}|_{\Sigma}=l(f_2)+f_1,
	\end{equation}
	where $|f_1|\lesssim |d \bw|+ |d \mathbf{g} |^2$ and $|f_2| \lesssim |d {\mathbf{g}}|$. Moreover, the characteristic energy estimates
	\begin{equation}\label{604}
		\|f_2\|_{L^2_t H^{s_0-\frac54}_{x'}(\Sigma)}+\|f_1\|_{L^1_t H^{s_0-\frac54}_{x'}(\Sigma)} \lesssim \epsilon_2,
	\end{equation}
	holds. Additionally, for any $t \in [-2,2]$, it follows
	\begin{equation}\label{605}
		\|f_2(t,\cdot)\|_{C^{\delta_0}_{x'}(\Sigma^t)} \lesssim \|d \mathbf{g}\|_{C^{\delta_0}_x(\mathbb{R}^2)}.
	\end{equation}
\end{corollary}
\begin{proof}
	Due to the definition of curvature tensor, we have
	\begin{equation*}
		\left< R(e_a, e_b)e_c, e_d \right>_{\mathbf{g}}= R_{\alpha \beta \mu \nu}e^\alpha_a e^\beta_b e_c^\mu e_d^\nu,
	\end{equation*}
	where
	\begin{equation*}
		R_{\alpha \beta \mu \nu}= \frac12 \left[ \partial^2_{\alpha \mu} \mathbf{g}_{\beta \nu}+\partial^2_{\beta \nu} \mathbf{g}_{\alpha \mu}-\partial^2_{\beta \mu} \mathbf{g}_{\alpha \nu}-\partial^2_{\alpha \nu} \mathbf{g}_{\beta \mu} \right]+ F(\mathbf{g}^{\alpha \beta}, d \mathbf{g}_{\alpha \beta}),
	\end{equation*}
	where $F$ is a sum of products of coefficients of $\mathbf{g}^{\alpha \beta} $ with quadratic forms in $d \mathbf{g}_{\alpha \beta}$. Due to Lemma \ref{LLQ}, the estimate \eqref{603} holds. Next, we will prove \eqref{604} and \eqref{605}.

	Applying Proposition \ref{r1}, the term $F$ satisfies the bound required of $f_1$. It suffices for us to consider
	\begin{equation*}
		\frac12 e^\alpha_a e^\beta_b e_c^\mu e_d^\nu  \left[ \partial^2_{\alpha \mu} \mathbf{g}_{\beta \nu}+\partial^2_{\beta \nu} \mathbf{g}_{\alpha \mu}-\partial^2_{\beta \mu} \mathbf{g}_{\alpha \nu}-\partial^2_{\alpha \nu} \mathbf{g}_{\beta \mu} \right].
	\end{equation*}
	We therefore look at the term $ e^\alpha_a e_c^\mu \partial^2_{\alpha \mu} \mathbf{g}_{\beta \nu} $, which is typical. By \eqref{503} and Proposition \ref{r1}, we get
	\begin{equation}\label{LLL}
		\vert\kern-0.25ex\vert\kern-0.25ex\vert  l^\alpha - \delta^{\alpha 0} \vert\kern-0.25ex\vert\kern-0.25ex\vert _{s_0-\frac14,2,\Sigma} +\vert\kern-0.25ex\vert\kern-0.25ex\vert  \underline{l}^\alpha + \delta^{\alpha 0}-2\delta^{\alpha 2} \vert\kern-0.25ex\vert\kern-0.25ex\vert _{s_0-\frac14,2,\Sigma} + \vert\kern-0.25ex\vert\kern-0.25ex\vert  e^\alpha_1- \delta^{\alpha 1} \vert\kern-0.25ex\vert\kern-0.25ex\vert _{s_0-\frac14,2,\Sigma} \lesssim \epsilon_1.
	\end{equation}
	By using \eqref{LLL} and Proposition \ref{r1}, the term $ e_a (e_c^\mu) \partial_{ \mu} \mathbf{g}_{\beta \nu}$ satisfies the bound required of $f_1$, we therefore consider $e_a(e_c(\mathbf{g}_{\beta \nu}))$. Since the coefficients of $e_c$ in the basis $/\kern-0.55em \partial_\alpha$ have tangential derivatives bounded in $L^2_tH^{s_0-\frac54}_{x'}(\Sigma)$, we are reduced by Lemma \ref{LLQ} to verifying that
	\begin{equation*}
		\| \mathbf{g}^{\alpha \beta } \partial^2_{\alpha \beta} \mathbf{g}_{\mu \nu} \|_{L^1_t H^{s_0-\frac54}_{x'}(\Sigma)} \lesssim \epsilon_2.
	\end{equation*}
	Note $\mathbf{g}^{\alpha \beta } \partial^2_{\alpha \beta} \mathbf{g}_{\mu \nu}= \square_{\mathbf{g}} \mathbf{g}_{\mu \nu}$. By use of Corollary \ref{vte} and Lemma \ref{te20}, we have
	\begin{equation*}
		\begin{split}
			\| \square_{\mathbf{g}} \mathbf{g}_{\mu \nu}\|_{L^1_t H^{s_0-\frac54}_{x'}(\Sigma)}
			\lesssim & \ \| \square_{\mathbf{g}} \bv\|_{L^2_t H^{s_0-\frac54}_{x'}(\Sigma)}+\| \square_{\mathbf{g}} h \|_{L^2_t H^{s_0-\frac54}_{x'}(\Sigma)}
			\\
			\lesssim & \ \| d \bw \|_{L^2_t H^{s_0-\frac54}_{x'}(\Sigma)}+ \| (d\bv,dh)\cdot (d\bv,dh) \|_{L^1_t H^{s_0-\frac54}_{x'}(\Sigma)}
			\\
			\lesssim & \ \| d \bw \|_{L^2_t H^{s_0-\frac54}_{x'}(\Sigma)}+ [ 2-(-2)  ]^{\frac14}\| d\bv, dh \|_{L^4_t L^\infty_x }\| d\bv, dh\|_{L^2_t H^{s_0-\frac54}_{x'}(\Sigma)}
			\\
			\lesssim & \ \epsilon_2.
		\end{split}
	\end{equation*}
	Above, $d \bw$ and $(d{\mathbf{g}})^2$ are included in $f_1$. Therefore, we have finished the proof of this corollary.
\end{proof}
Based on the above null frame, we can discuss the estimates for connection coefficients as follows.
\subsection{Estimates for connection coefficients}
Define
\begin{equation*}
	\chi= \left<D_{e_1}l,e_1 \right>_{\mathbf{g}}, \qquad l(\ln \sigma)=\frac{1}{2}\left<D_{l}\underline{l},l \right>_{\mathbf{g}}. \quad 
\end{equation*}
For $\sigma$, we set the initial data $\sigma=1$ at the time $-2$. Thanks to Proposition \ref{r1}, we have
\begin{equation}\label{606}
	\|\chi \|_{L^2_t H^{s_0-\frac54}_{x'}(\Sigma)} + \| l(\ln \sigma)\|_{L^2_t H^{s_0-\frac54}_{x'}(\Sigma)} \lesssim \epsilon_1.
\end{equation}
In a similar way, if we expand $l=l^\alpha /\kern-0.55em \partial_\alpha$ in the tangent frame $\partial_t, \partial_{x'}$ on $\Sigma$, then
\begin{equation}\label{607}
	l^0=1, \quad \|l^1\|_{s_0-\frac14,2,\Sigma} \lesssim \epsilon_1.
\end{equation}
Next, we will establish some estimates for connections along the hypersurface.
\begin{Lemma}\label{chi}
	Assume $\frac74<s_0\leq s\leq \frac{15}{8} $, $ \delta\in (0, s-\frac{7}{4})$, and $ \delta_0 \in (0, s_0-\frac{7}{4})$. Let $\chi$ be defined as before. Then
	\begin{equation}\label{608}
		\|\chi\|_{L^2_t H^{s_0-\frac54}_{x'}(\Sigma)} \lesssim \epsilon_2.
	\end{equation}
	Furthermore, for any $t \in [-2,2]$,
	\begin{equation}\label{609}
		\| \chi \|_{C^{\delta_0}_{x'}(\Sigma^t)} \lesssim \epsilon_2+ \|d \mathbf{g} \|_{C^{\delta_0}_{x}(\mathbb{R}^2)}.
	\end{equation}
\end{Lemma}
\begin{proof}
	Along the null hypersurface, $\chi$ satisfies a transport equation (see Klainerman and Rodnianski \cite{KR2}). This tells us:
	\begin{equation*}
		l(\chi)=\left< R(l,e_1)l, e_1 \right>_{\mathbf{g}}-\chi^2-l(\ln \sigma)\chi.
	\end{equation*}
	Due to Corollary \ref{Rfenjie}, we can rewrite the above equation by
	\begin{equation}\label{610}
		l(\chi-f_2)=f_1-\chi^2-l(\ln \sigma)\chi,
	\end{equation}
	where
	\begin{equation}\label{611}
		\|f_2\|_{L^2_t H^{s_0-\frac54}_{x'}(\Sigma)}+\|f_1\|_{L^1_t H^{s_0-\frac54}_{x'}(\Sigma)} \lesssim \epsilon_2,
	\end{equation}
	and for any $t \in [0,T]$,
	\begin{equation}\label{612}
		\|f_2(t,\cdot)\|_{C^{\delta_0}_{x'}(\Sigma^t)} \lesssim \|d \mathbf{g}\|_{C^{\delta_0}_x(\mathbb{R}^2)}.
	\end{equation}
	Let $\Lambda_{x'}$ be defined in \eqref{Sig}. To be simple, we set
	\begin{equation*}\label{ef}
		F=f_1-\chi^2-l(\ln \sigma)\chi.
	\end{equation*}
	Applying integration by parts on \eqref{610}, we obtain
	\begin{equation}\label{613}
		\begin{split}
			\|\Lambda_{x'}^{s_0-\frac54}(\chi-f_2)(t,\cdot) \|_{L^2_{x'}(\Sigma^t)}
			\lesssim \ &\| [\Lambda_{x'}^{s_0-\frac54},l](\chi-f_2) \|_{L^1_tL^2_{x'}(\Sigma^t)}+ \| \Lambda_{x'}^{s_0-\frac54}F \|_{L^1_tL^2_{x'}(\Sigma^t)}.
		\end{split}
	\end{equation}
	By using $s_0>\frac74$, $H^{s_0-\frac54}_{x'}(\Sigma^t)$ is an algebra. This ensures that
	\begin{equation}\label{614}
		\begin{split}
			\| \Lambda_{x'}^{s_0-\frac54} F \|_{L^1_tL^2_{x'}(\Sigma^t)} &\lesssim \|f_1\|_{L^1_tH^{s_0-\frac54}_{x'}(\Sigma^t)}+ \|\chi\|^2_{L^2_tH^{s_0-\frac54}_{x'}(\Sigma^t)}
			\\
			& \quad + \|\chi\|_{L^2_tH^{s_0-\frac54}_{x'}(\Sigma^t)}\cdot\|l(\ln \sigma)\|_{L^2_tH^{s_0-\frac54}_{x'}(\Sigma^t)}
			\\
			& \quad + \|\mu\|_{L^2_tH^{s_0-\frac54}_{x'}(\Sigma^t)}\cdot\|\chi\|_{L^2_tH^{s_0-\frac54}_{x'}(\Sigma^t)}.
		\end{split}
	\end{equation}
	For $a \geq 0$, a direct calculation tells us
	\begin{equation*}
		\begin{split}
			[\Lambda_{x'}^{a},l]f=& \Lambda_{x'}^{a} (l^{\alpha} /\kern-0.55em \partial_{\alpha}  f ) - l^{\alpha} /\kern-0.55em \partial_{\alpha} \Lambda_{x'}^{a} f
			\\
			=& \Lambda_{x'}^{a} /\kern-0.55em \partial_{\alpha}  ( l^{\alpha} f ) - \Lambda_{x'}^{a} ( /\kern-0.55em \partial_{\alpha} l^\alpha f)- l^\alpha \Lambda_{x'}^{a} /\kern-0.55em \partial_{\alpha} f
			\\
			=&  - \Lambda_{x'}^{a} ( /\kern-0.55em \partial_{\alpha} l^\alpha f) + [\Lambda_{x'}^{a} /\kern-0.55em \partial_{\alpha} , l^{\alpha} ]f.
		\end{split}
	\end{equation*}
	As a result, we can bound $\| [\Lambda_{x'}^{s_0-\frac54},l](\chi-f_2) \|_{L^2_{x'}(\Sigma^t)}$ by
	\begin{equation}\label{k20}
		\begin{split}
			\| [\Lambda_{x'}^{s_0-\frac54},l](\chi-f_2) \|_{L^2_{x'}(\Sigma^t)} 
			\leq & \| /\kern-0.55em \partial_{\alpha} l^{\alpha} (\chi-f_2)(t,\cdot) \|_{H^{s_0-\frac54}_{x'}(\Sigma^t)}
			\\
			& + \|[\Lambda_{x'}^{s_0-\frac54} /\kern-0.55em \partial_{\alpha}, l^{\alpha}](\chi-f_2)(t,\cdot) \|_{L^{2}_{x'}(\Sigma^t)}.
		\end{split}	
	\end{equation}
	Note $l^0=1$ and $s_0>\frac74$. Applying Kato-Ponce's commutator estimate and Sobolev embeddings, the right hand side of \eqref{k20} can be bounded by
	\begin{equation}\label{615}
		\|l^1(t,\cdot)\|_{H^{s_0-\frac54}_{x'}(\Sigma^t)} \| \Lambda_{x'}^{s_0-\frac54}(\chi-f_2)(t,\cdot) \|_{L^{2}_{x'}(\Sigma^t)} .
	\end{equation}
	Because of \eqref{606}, \eqref{607}, \eqref{611}, \eqref{613}, \eqref{614} and \eqref{615}, we thus prove that
	\begin{equation}\label{k22}
		\sup_t \|(\chi-f_2)(t,\cdot)\|_{H^{s_0-\frac54}_{x'}(\Sigma^t)}  \lesssim \epsilon_2.
	\end{equation}
	The estimate \eqref{k22} combining with \eqref{611} yield \eqref{608}. Due to \eqref{610}, we find
	\begin{equation}\label{616}
		\begin{split}
			\| \chi-f_2\|_{C^{\delta_0}_{x'}} & \lesssim \| f_1 \|_{L^1_tC^{\delta_0}_{x'}}+ \|\chi^2\|_{L^1_tC^{\delta_0}_{x'} }+\|l(\ln \sigma)\chi\|_{L^1_tC^{\delta_0}_{x'}}.
		\end{split}
	\end{equation}
	Using Sobolev's imbedding, we have
	\begin{equation}\label{619}
		H^{s_0-\frac54}(\mathbb{R})\hookrightarrow C^{\delta_0}(\mathbb{R}), \qquad \delta_0 \in (0,s_0-\frac74).
	\end{equation}
	Taking advantage of \eqref{612}, \eqref{616}, and \eqref{619}, we get \eqref{609}. At this stage, we have finished the proof of this lemma.
\end{proof}
Finally, we present a proof for Proposition \ref{r2} as follows.
\subsection{The proof of Proposition \ref{r2}}
First, we will verify \eqref{Re}. By using \eqref{501} and \eqref{602} for $\vert\kern-0.25ex\vert\kern-0.25ex\vert \mathbf{g}-\mathbf{m} \vert\kern-0.25ex\vert\kern-0.25ex\vert_{s_0-\frac54,2,\Sigma}$, then the inequality \eqref{Re} follows from the following bound:
\begin{equation*}
	\vert\kern-0.25ex\vert\kern-0.25ex\vert l-(\partial_t-\partial_{2})\vert\kern-0.25ex\vert\kern-0.25ex\vert_{s_0-\frac14,2,\Sigma} \lesssim \epsilon_2.
\end{equation*}
Above, it is understood that one takes the norm of the coefficients of $l-(\partial_t-\partial_{2})$ in the standard frame on $\mathbb{R}^{1+2}$. The geodesic equation, together with the bound for Christoffel symbols $\|\Gamma^\alpha_{\beta \gamma}\|_{L^4_t L^\infty_x} \lesssim \|d {\mathbf{g}} \|_{L^4_t L^\infty_x}\lesssim \epsilon_2$, imply that
\begin{equation*}
	\|l-(\partial_t-\partial_{2})\|_{L^\infty_{t,x}} \lesssim \epsilon_2.
\end{equation*}
Therefore, it suffices for us to bound the tangential derivatives of the coefficients for $l-(\partial_t-\partial_{2})$ in the norm $L^2_t H^{s_0-\frac54}_{x'}(\Sigma)$. By using Proposition \ref{r1}, we can estimate the Christoffel symbols
\begin{equation*}
	\|\Gamma^\alpha_{\beta \gamma} \|_{L^2_t H^{s_0-\frac54}_{x'}(\Sigma^t)} \lesssim \epsilon_2.
\end{equation*}
Note that $H^{s_0-\frac54}_{x'}(\Sigma^t)$ is a algebra when $s_0>\frac74$. We thus derive
\begin{equation*}
	\|\Gamma^\alpha_{\beta \gamma} e_1^\beta l^\gamma\|_{L^2_t H^{s_0-\frac54}_{x'}(\Sigma^t)} \lesssim \epsilon_2.
\end{equation*}
In what follows, we can establish the following bound:
\begin{equation}\label{fb}
	\| \left< D_{e_1}l, e_1 \right>\|_{L^2_t H^{s_0-\frac54}_{x'}(\Sigma^t)}+ \| \left< D_{e_1}l, \underline{l} \right>\|_{L^2_t H^{s_0-\frac54}_{x'}(\Sigma^t)}+\|\left< D_{l}l, \underline{l} \right>\|_{L^2_t H^{s_0-\frac54}_{x'}(\Sigma^t)} \lesssim \epsilon_2.
\end{equation}
In fact, the first term in \eqref{fb} is $\chi$, which can be bounded by using Lemma \ref{chi}. For the second term in the left side of \eqref{fb}, consider
\begin{equation*}
	\left< D_{e_1}l, \underline{l} \right>=\left< D_{e_1}l, 2\partial_t \right>=-2\left< D_{e_1}\partial_t,l \right>.
\end{equation*}
Then it can be bounded by use of Proposition \ref{r1}. Similarly, using Proposition \ref{r1}, we can obtain the desired estimate for the last term in \eqref{fb}. Therefore, the rest of proof is to obtain
\begin{equation*}
	\| d \phi(t,x')-dt \|_{C^{1,\delta_0}_{x'}(\mathbb{R})}  \lesssim \epsilon_2+ \| d\mathbf{g}(t,\cdot)\|_{C^{\delta_0}_x(\mathbb{R}^2)}.
\end{equation*}
In order to do this, it suffices for us to show that
\begin{equation*}
	\|l(t,\cdot)-(\partial_t-\partial_{x_2})\|_{C^{1,\delta_0}_{x'}(\mathbb{R})} \lesssim \epsilon_2+ \| d \mathbf{g} (t,\cdot)\|_{C^{\delta_0}_x(\mathbb{R}^2)}.
\end{equation*}
It's clear that the coefficients of $e_1$ are small in $C^{\delta_0}_{x'}(\Sigma^t)$ perturbations of their constant coefficient analogs. Thus, it's left for us to prove
\begin{equation*}
	\|\left< D_{e_1}l, e_1 \right>(t,\cdot)\|_{C^{\delta_0}_{x'}(\Sigma^t)}
	+\|\left< D_{e_1}l, \underline{l} \right>(t,\cdot)\|_{C^{\delta_0}_{x'}(\Sigma^t)}  \lesssim \epsilon_2+ \| d\mathbf{g}(t,\cdot)\|_{C^{\delta_0}_x(\mathbb{R}^2)}.
\end{equation*}
Above, the first term is bounded by Lemma \ref{chi}, and the second by using
\begin{equation*}
	\|\left< D_{e_1}\partial_t, l \right>(t,\cdot)\|_{C^{\delta_0}_{x'}(\Sigma^t)} \lesssim  \| d\mathbf{g}(t,\cdot)\|_{C^{\delta_0}_x(\mathbb{R}^2)}.
\end{equation*}
Therefore, the estimate \eqref{502} holds. At this stage, we complete the proof of Proposition \ref{r2}.

\section{The proof of Proposition \ref{r3}}\label{sec7}
In this part, our goal is to give a proof of Proposition \ref{r3} and establish Strichartz estimates of solutions.
\subsection{The proof of Proposition \ref{r3}}
We first introduce Strichartz estimates for linear inhomogeneous wave equations.
\begin{proposition}\label{r5}
Assume $s\in (\frac74, \frac{15}{8}]$. Suppose that $(h,\bv,\bw) \in \mathcal{{H}}$ and $\Re(h, \bv, \bw) \leq 2 \epsilon_1$. For each $1 \leq r \leq s+1$, then the linear, homogenous equation
	\begin{equation*}
		\begin{cases}
			& \square_{\mathbf{g}} F=0,
			\\
			&F(t,x)|_{t=t_0}=F_0, \quad \partial_t F(t,x)|_{t=t_0}=F_1,
		\end{cases}
	\end{equation*}
	is well-posed for the initial data $(F_0, F_1)$ in $H_x^r \times H_x^{r-1}$. Moreover, the solution satisfies, respectively, the energy estimates
	\begin{equation*}
		\| F\|_{L^\infty_tH_x^r}+ \| \partial_t F\|_{L^\infty_tH_x^{r-1}} \leq C \big( \|F_0\|_{H_x^r}+ \|F_1\|_{H_x^{r-1}} \big),
	\end{equation*}
	and the Strichartz estimates
	\begin{equation}\label{SL}
		\| \left<\nabla \right>^a F\|_{L^4_{t}L^\infty_x} \leq C \big( \|F_0\|_{H_x^r}+ \|F_1\|_{H_x^{r-1}} \big), \quad a<r-\frac34.
	\end{equation}
	The similar estimate \eqref{SL} also holds if we replace $\left<\nabla \right>^a$ by $\left<\nabla \right>^{a-1}d$.
\end{proposition}
\begin{remark}
	The proof of Proposition \ref{r5} will be given in Section \ref{Ap}. 
\end{remark}
Now we are ready to prove Proposition \ref{r3} by use of Proposition \ref{r5}.
\begin{proof}[Proof of Proposition \ref{r3} by using Proposition \ref{r5}]
	By using Duhamel's principle and Proposition \ref{r3}, the conclusions of Proposition \ref{r5} hold .
\end{proof}

We can also obtain the following Strichartz estimates of solutions.
\subsection{Strichartz estimates for solutions}
\begin{proposition}\label{r6}
	Assume $s\in (\frac74, \frac{15}{8}]$. Suppose $(h, \bv, \bw) \in \mathcal{H}$ and $\Re(h, \bv, \bw)\leq 2 \epsilon_1$. Let $\bv_{+}$ be defined in \eqref{De}. Then for $a<s-\frac34$, we have
	\begin{equation}\label{fgh}
		\|\left< \nabla \right>^a h, \left< \nabla \right>^a \bv_{+}\|_{L^4_t L^\infty_x}
		\lesssim  \| h_0\|_{H^s}+\| \bv_0\|_{H^s}+\| \bw_0\|_{H^{1+}} .
	\end{equation}
	The similar estimate \eqref{fgh} holds if we replace $\left<\nabla \right>^a$ by $\left<\nabla \right>^{a-1}d$.
\end{proposition}
\begin{proof}
	Note \eqref{wag0}. The functions $h$ and $\bv_{+}$ satisfy the system
	\begin{equation}\label{fcr}
		\begin{cases}
			& \square_{\mathbf{g}} h= \mathcal{D},
			\\
		&	\square_{\mathbf{g}} \bv_{+}
			= \Theta\mathrm{e}^{-2h}(v^0)^2  (1-3 c^2_s)    \mathbf{T} \mathbf{T} \bv_{-}- c^2_s  \Theta \bv_{-} + \bQ,
		\end{cases}
	\end{equation}
	By using the Strichartz estimate in Proposition \ref{r5} (taking $r=s$ and $a<r-\frac34$) 
	, we obtain
	\begin{equation*}
		\begin{split}
			& \|\left< \nabla \right>^a h, \left< \nabla \right>^a \bv_{+}\|_{L^4_t L^\infty_x} 
			\\
			\lesssim & \| h_0, \bv_0\|_{H^s}+  \|\mathbf{T}\mathbf{T}\bv_{-}\|_{L^1_{[-2,2]}H_x^{s-1}}+\|\mathcal{D},\bQ \|_{L^1_{[-2,2]}H_x^{s-1}} +\|\bv_{-} \|_{L^1_{[-2,2]}H_x^{s-1}} .
		\end{split}
	\end{equation*}
Due to \eqref{etaq} and Lemma \ref{yux}, we further get
	\begin{equation}\label{fgH}
		\begin{split}
			\|\left< \nabla \right>^a h, \left< \nabla \right>^a \bv_{+}\|_{L^4_t L^\infty_x} 
			\lesssim & \| h_0 \|_{H^s}+ \| \bv_0\|_{H^s} +  \|dh,d\bv \|_{L^4_{[-2,2]}L_x^\infty}  \|h,\bv \|_{L^\infty_{[-2,2]}H_x^s} 
			\\
			& + \|\bw\|_{L^\infty_{[-2,2]}H_x^{1+}}   (1+  \|h,\bv \|^2_{L^\infty_{[-2,2]}H_x^s} ) .
		\end{split}
	\end{equation}
	Due to \eqref{402a} and \eqref{403}, the estimate \eqref{fgH} becomes
	\begin{equation*}
		\begin{split}
			\|\left< \nabla \right>^a h, \left< \nabla \right>^a \bv_{+}\|_{L^4_t L^\infty_x} & \lesssim \| h_0\|_{H^s}+\| \bv_0\|_{H^s} +\| \bw_{0}\|_{H^{1+}}, \qquad a<s-\frac34.
		\end{split}
	\end{equation*}
\end{proof}
The Proposition \ref{r6} states a Strichartz estimate of solutions with a very low regularity of the velocity, density, and vorticity. This also motivates the following proposition.
\begin{proposition}\label{r4}
	Assume $\frac74<s_0\leq s\leq \frac{15}{8} $ and $ \delta\in (0, s-\frac{7}{4})$. Suppose $(h, \bv, \bw) \in \mathcal{H}$ and $\Re(h,\bv, \bw)\leq 2 \epsilon_1$. Let $\bv_{+}$ be defined in \eqref{De}. Then we have
	\begin{equation}\label{SR0}
		\begin{split}
			\| d h, d \bv\|_{L^2_t C^\delta_x} & \lesssim \| h \|_{L^\infty_t H_x^s}+\|\bv\|_{L^\infty_t H_x^s}+ {\| \bw\|_{L^\infty_t H_x^{\frac54+2\delta}}}.
		\end{split}
	\end{equation}
	Furthermore, the Strichartz estimates
	\begin{equation}\label{strr}
		\|d \bv, d h, d \bv_{+}\|_{L^4_t C^\delta_x} \leq \epsilon_2,
	\end{equation}
	and the energy estimates
	\begin{equation}\label{eef}
		\| \bv\|_{L^\infty_tH_x^s} +\|h\|_{L^\infty_tH_x^s}+\| \bw \|_{L^\infty_tH_x^{s_0-\frac14}}+ \| \nabla \bw \|_{L^\infty_tL^8_x} \leq \epsilon_2,
	\end{equation}
	hold.
\end{proposition}
\begin{proof}
	Due to \eqref{fcr}, and using the Strichartz estimate in Proposition \ref{r5} (taking $r=s$ and $a=1$), we have
	\begin{equation}\label{rrr}
		\begin{split}
			\| d h, d \bv_+\|_{L^4_t C^\delta_x}  \lesssim \| h \|_{L^\infty_t H_x^s}+\|\bv\|_{L^\infty_t H_x^s}+ {\| \bw\|_{L^\infty_t H_x^{1+\delta}}}.
		\end{split}
	\end{equation}
	Recall \eqref{De0} and \eqref{De1}. By Sobolev imbedding $C^\delta_x(\mathbb{R}^2) \hookrightarrow H_x^{1+2\delta}(\mathbb{R}^2)$ and $L^4_t([-2,2]) \hookrightarrow H^{\frac14}_t([-2,2]) $, we then get
	\begin{equation}\label{pp3}
		\begin{split}
			\| d \bv_{-}\|_{L^4_t C^\delta_x} & \lesssim \| d \bv_{-} \|_{H^\frac14_t H_x^{1+2\delta}}
			\\
			& \lesssim \| \Lambda^{\frac14+2\delta}_{x}d \bv_{-} \|_{H^\frac14_t H_x^{\frac34}} 
			\\
			& \lesssim \| \Lambda^{\frac14+2\delta}_{x} \bv_{-} \|_{H^2_{t,x}} .
		\end{split}
	\end{equation} 
	By using Lemma \ref{jiaohuan3} and \ref{Ees}, we derive that
		\begin{equation}\label{pp3a}
		\begin{split}
			\| \Lambda^{\frac14+2\delta}_{x} \bv_{-} \|_{H^2_{t,x}} \lesssim & \|\Lambda_x^{\frac14+2\delta} d\bw \|_{L^2_t L_x^{2}} + \|[\Lambda_x^{\frac14+2\delta}, \mathbf{{P}} ] v^\alpha_{-}\|_{L^2_t L_x^{2}} 
			\\
			\lesssim & \|  d\bw \|_{L^2_t H_x^{{\frac14+2\delta}}} + \|\Lambda_x^{\frac14+2\delta} \bv\|_{L^\infty_t L^\infty_x} \|d^2 \bv \|_{L^2_t L_x^{2}} 
			\\
			\lesssim & \| \bw \|_{L^2_t H_x^{\frac54+2\delta}} ( 1+ \|h,\bv\|_{L^\infty_t H^s_x}).
		\end{split}
	\end{equation}
	Substituting \eqref{pp3a} to \eqref{pp3}, it follows
		\begin{equation}\label{pp3b}
		\begin{split}
			\| d \bv_{-}\|_{L^4_t C^\delta_x} & \lesssim \| \bw \|_{L^2_t H_x^{\frac54+2\delta}} ( 1+ \|h,\bv\|_{L^\infty_t H^s_x}).
		\end{split}
	\end{equation} 
	By using \eqref{rrr} and \eqref{pp3b}, we therefore obtain \eqref{SR0}. Considering $\frac74<s_0\leq s\leq \frac{15}{8} $ and $ \delta\in (0, s-\frac{7}{4})$, it yields
	\begin{equation}\label{key}
		\frac54+2\delta < \frac32< s_0-\frac14.
	\end{equation} 
	Due to \eqref{SR0} and \eqref{key}, applying Theorem \ref{DW4}, we conclude that
	\begin{equation}\label{pp2}
		\begin{split}
			\| d h,  d \bv_{+}, d \bv\|_{L^4_t C^\delta_x} 
			&  \lesssim  \| h \|_{L^\infty_t H_x^s}+\|\bv\|_{L^\infty_t H_x^s}+ {\| \bw\|_{L^\infty_t H_x^{s_0-\frac14}}}
			\\
			& \lesssim \| h_0\|_{H^s}+\|\bv_0\|_{H^s} + \| \bw_0\|_{H^{s_0-\frac14}} + \|\nabla \bw_0\|_{L_x^8}.
		\end{split}
	\end{equation}
	By using \eqref{401} and \eqref{pp2}, we have
	\begin{equation}\label{900}
		\begin{split}
			\| d h, d \bv_{+}, d \bv\|_{L^2_t C^\delta_x} \lesssim \epsilon_3.
		\end{split}
	\end{equation}
	Note that $\epsilon_3 \ll \epsilon_2$. Taking advantage of \eqref{900}, we therefore obtain \eqref{strr}. By using Theorem \ref{DW4} and \eqref{strr}, the estimate \eqref{eef} holds. Therefore, we have finished the proof of this proposition. 
\end{proof}
\begin{remark}
	The Strichartz estimate \eqref{SR0} requires a very low regularity of the velocity, density, and vorticity. We hope that the regularity exponent in \eqref{SR0} is optimal.
\end{remark}
It only remains for us to establish Proposition \ref{r5}. This will be discussed in the next section.

\section{Proof of Proposition \ref{r5}}\label{Ap}
In this section, we present a proof of Proposition \ref{r5} using wave-packet approximation, following the approach of Smith and Tataru in \cite{ST}. 

\subsection{The proof of Proposition \ref{r5}}
In order to prove Proposition \ref{r5}, we first reduce it to phase space. Given a frequency scale $\lambda \geq 1$ ($\lambda=2^j, j\in \mathbb{N}^{+}$), we consider the smooth coefficients
\begin{equation*}
	\mathbf{g}_{\lambda}= P_{<\lambda} \mathbf{g}=\sum_{\lambda'<\lambda} P_{\lambda'},
\end{equation*}
where $\mathbf{g}$ is defined in \eqref{AMd3}. We now introduce a proposition in the phase space, which will be very useful.
\begin{proposition}\label{A1}
	Suppose $(h,\bv,\bw) \in \mathcal{{H}}$ and $\Re(h,\bv,\bw)\leq 2 \epsilon_1$. Let $f$ satisfy
	\begin{equation}\label{linearA}
		\begin{cases}
			\square_{\mathbf{g}} f=0, \quad (t,x)\in [-2,2]\times \mathbb{R}^2,\\
			(f, \partial_t f)|_{t=t_0}=(f_0, f_1).
		\end{cases}
	\end{equation}
	Then for each $(f_0,f_1) \in H^1 \times L^2$ there exists a function $f_{\lambda} \in C^\infty([-2,2]\times \mathbb{R}^2)$, with
	\begin{equation*}
		\mathrm{supp} \widehat{f_\lambda(t,\cdot)} \subseteq \{ \xi \in \mathbb{R}^2: \frac{\lambda}{8} \leq |\xi| \leq 8\lambda \},
	\end{equation*}
	such that
	\begin{equation}\label{Yee}
		\begin{cases}
			& \| \square_{\mathbf{g}_\lambda} f_{\lambda} \|_{L^4_{[-2,2]} L^2_x} \lesssim \epsilon_0 (\| f_0\|_{H^1}+\| f_1 \|_{L^2} ),
			\\
			& f_\lambda(-2)=P_\lambda f_0, \quad \partial_t f_{\lambda} (-2)=P_{\lambda} f_1.
		\end{cases}
	\end{equation}
	Additionally, if $r>\frac34$, then the following Strichartz estimates holds:
	\begin{equation}\label{Ase}
		\| f_{\lambda} \|_{L^4_{[-2,2]} L^\infty_x} \lesssim \epsilon_0^{-\frac{1}{4}} \lambda^{r-1} ( \| P_\lambda f_0 \|_{H^1} + \| P_\lambda f_1 \|_{L^2} ) .
	\end{equation}
\end{proposition}
Next, we give a proof for Proposition \ref{r5} based on Proposition \ref{A1}.
\begin{proof}[Proof of Proposition \ref{r5} by use of Proposition \ref{A1}]
	We divide the proof into several cases.
	
	(i) $r=1$. Using basic energy estimates for \eqref{linear}, we have
	\begin{equation*}
		\begin{split}
			\| \partial_t f \|_{L^2_x} + \|\nabla f \|_{L^2_x}  \lesssim & \ (\| f_0\|_{H^1}+ \| f_1\|_{L^2}) \exp(\int^t_0 \| d \mathbf{g} \|_{L^\infty_x} d\tau)
			\\
			\lesssim & \ \| f_0\|_{H^1}+ \| f_1\|_{L^2}.
		\end{split}
	\end{equation*}
	Then the Cauchy problem \eqref{linearA} holds a unique solution $f \in C([-2,2],H_x^1)$ and $\partial_t f \in C([-2,2],L_x^2)$. It remains to show that the solution $f$ also satisfies the Strichartz estimate \eqref{SL}.

	Without loss of generality, we take $t_0=0$. For any given initial data $(f_0,f_1) \in H^1 \times L^2$, and $t_0 \in [-2,2]$, we take a Littlewood-Paley decomposition
	\begin{equation*}
		f_0=\sum_{\lambda}P_{\lambda}f_0, \qquad  f_1=\sum_{\lambda}P_{\lambda}f_1,
	\end{equation*}
	and for each $\lambda$ we take the corresponding $f_{\lambda}$ as in Proposition \ref{A1}.  If we set
	\begin{equation*}
		f=\sum_{\lambda}f_{\lambda},
	\end{equation*}
	then $f$ matches the initial data $(f_0,f_1)$ at the time $t=t_0$, and also satisfies the Strichartz estimates \eqref{Ase}. In fact, $f$ is also an approximate solution for $\square_{\mathbf{g}}$ in the sense that
	\begin{equation*}
		\| \square_{\mathbf{g}} f \|_{L^2_t L^2_x} \lesssim \epsilon_0(\| f_0 \|_{H^1}+\| f_1 \|_{L^2} ).
	\end{equation*}
	We can derive the above bound by using the decomposition
	\begin{equation*}
		\square_{\mathbf{g}} f=\textstyle{\sum_{\lambda}} \square_{\mathbf{g}_\lambda} f_\lambda+ \textstyle{\sum_{\lambda}} \square_{\mathbf{g}-\mathbf{g}_\lambda} f_\lambda.
	\end{equation*}
	The first term can be controlled by Proposition \ref{A1}. As for the second term, using $\mathbf{g}^{00}=-1$, we can rewrite
	\begin{equation*}
		\begin{split}
			\textstyle{\sum_{\lambda}} \square_{\mathbf{g}-\mathbf{g}_\lambda} f_\lambda= \textstyle{\sum_{\lambda}} ({\mathbf{g}-\mathbf{g}_\lambda}) \nabla  df_\lambda.
		\end{split}
	\end{equation*}
	By using H\"older's inequality, it follows that
	\begin{equation}\label{yv1}
		\begin{split}
			\| \textstyle{\sum_{\lambda}} \square_{\mathbf{g}-\mathbf{g}_\lambda} f_\lambda \|_{L^2_x} \lesssim \mathrm{sup}_{\lambda} \left(\lambda \| {\mathbf{g}-\mathbf{g}_\lambda} \|_{L^\infty_x} \right) \left( \sum_{\lambda}  \| df_\lambda \|^2_{L^2_x} \right)^{\frac12}.
		\end{split}
	\end{equation}
	Due to Bernstein's inequality, we have
	\begin{equation}\label{yv2}
		\begin{split}
			\mathrm{sup}_{\lambda} \left(\lambda \| {\mathbf{g}-\mathbf{g}_\lambda} \|_{L^\infty_x} \right) \lesssim & \ \mathrm{sup}_{\lambda} \big(\lambda \sum_{\mu > \lambda}\| \mathbf{g}_\mu \|_{L^\infty_x} \big)
			\\
			\lesssim & \ \mathrm{sup}_{\lambda} \big(\lambda \sum_{\mu > \lambda}\mu^{-(1+\delta)}\|d \mathbf{g}_\mu \|_{C^\delta_x} \big)
			\\
			\lesssim & \ \|d \mathbf{g} \|_{C^\delta_x}(\textstyle{\sum_{\mu}} \mu^{-\delta}) \lesssim \ \|d \mathbf{g} \|_{C^\delta_x}.
		\end{split}
	\end{equation}
	Combining \eqref{yv1} and \eqref{yv2}, we get
	\begin{equation*}
		\begin{split}
			\textstyle{\sum_{\lambda}} \square_{\mathbf{g}-\mathbf{g}_\lambda} f_\lambda \lesssim \epsilon_0(\| f_0 \|_{H^1}+\| f_1 \|_{L^2} ).
		\end{split}
	\end{equation*}
	For given $F\in L^1_t L^2_x$, we set
	\begin{equation*}
		\mathbf{M} F=\int^t_0 f^{\tau}(t,x)d\tau,
	\end{equation*}
	where $f^{\tau}(t,x)$ is the approximate solution formed above with the Cauchy data
	\begin{equation*}
		f^{\tau}(\tau,x)=0, \quad \partial_t f^{\tau}(\tau,x)=F(\tau,\cdot).
	\end{equation*}
	By calculating
	\begin{equation*}
		\square_{\mathbf{g}} \mathbf{M}F=\int^t_0 \square_{\mathbf{g}} f^{\tau}(t,x) d\tau+F,
	\end{equation*}
	it follows that
	\begin{equation*}
		\| \square_{\mathbf{g}} \mathbf{M}F-F\|_{L^2_t L^2_x} \lesssim \| \square_{\mathbf{g}}f^{\tau}\|_{L^1_t L^2_x} \lesssim \epsilon_0 \| F \|_{L^2_t L^2_x}.
	\end{equation*}
	Using the contraction principle, we can write the solution $f$ in the form
	\begin{equation*}
		f= \tilde{f}+ \mathbf{M}F,
	\end{equation*}
	where $\tilde{f}$ is the approximation solution formed above for initial data $(f_0,f_1)$ specified at time $t=0$, and
	\begin{equation*}
		\| F\|_{L^2_t L^2_x} \lesssim \epsilon_0 ( \| f_0 \|_{H^1}+\| f_1\|_{L^2}).
	\end{equation*}
	The Strichartz estimates now follow since they holds for each $f^\tau$, $\tau \in [0,t]$. By Duhamel's principle, we can also obtain the Strichartz estimates for the linear, nonhomogeneous wave equation
	\begin{equation*}
		\begin{cases}
			\square_{\mathbf{g}} f= F',
			\\
			f|_{t=0}=f_0, \quad \partial_tf|_{t=0}=f_1.
		\end{cases}
	\end{equation*}
	That means
	\begin{equation*}
		\| \left< \nabla \right>^a f \|_{L^4_t L^\infty_x} \lesssim  \|f_0\|_{H^1}+ \| f_1 \|_{L^2}+ \|F'\|_{L^1_t L^2_x} , \quad   a<\frac14.
	\end{equation*}
	(ii) $1 < r \leq s+1$. Based on the above result, we plan to transform the initial data in $H^1 \times L^2$. Operating $\left< \nabla \right>^{r-1}$ on \eqref{linearA}, we have 
	\begin{equation*}
		\square_{\mathbf{g}} \left< \nabla \right>^{r-1}f=-[\square_{\mathbf{g}}, \left< \nabla \right>^{r-1}]f.
	\end{equation*}
	Let $\left< \nabla \right>^{r-1}f=\bar{f}$. Then $\bar{f}$ is a solution to
	\begin{equation}\label{qd}
		\begin{cases}
			\square_{\mathbf{g}}\bar{f}=-[\square_{\mathbf{g}}, \left< \nabla \right>^{r-1}]\left< \nabla \right>^{1-r}\bar{f},
			\\
			(\bar{f}(t_0), \partial_t \bar{f}(t_0)) \in H^1 \times L^2.
		\end{cases}
	\end{equation}
	To handle this case, we need to bound the right term as
	\begin{equation}\label{qqd}
		\| [\square_{\mathbf{g}}, \left< \nabla \right>^{r-1}]\left< \nabla \right>^{1-r}\bar{f}\|_{L^4_t L^2_x} \lesssim \epsilon_0( \| d\bar{f}\|_{L^\infty_t L^2_x}+  \| \left< \nabla \right>^m d\bar{f}\|_{L^4_t L^\infty_x}),
	\end{equation}
	provided $m>1-s$. To prove it, we apply analytic interpolation to the family
	\begin{equation*}
		\bar{f} \rightarrow   [\square_{\mathbf{g}}, \left< \nabla \right>^{r-1}]\left< \nabla \right>^{1-r}\bar{f}.
	\end{equation*}
	For $\mathrm{Re}z=0$, noting $\mathbf{g}^{00}=-1$, we use the commutator estimate (c.f. (3.6.35) of \cite{KP}) to get
	\begin{equation*}
		\| [\mathbf{g}^{\alpha \beta}, \left< \nabla \right>^z] \partial^2_{\alpha \beta} \bar{f} \|_{L^2_x}=\| [\mathbf{g}^{\alpha i}, \left< \nabla \right>^z] \partial_i ( \partial_{\alpha} \bar{f} )\|_{L^2_x} \lesssim \| d \mathbf{g} \|_{L^\infty_x} \| d \bar{f} \|_{L^2_x}.
	\end{equation*}
	For $\mathrm{Re}z=s$, we use the Kato-Ponce commutator estimate
	\begin{equation}\label{qdy}
		\begin{split}
			\| [\mathbf{g}^{\alpha \beta}, \left< \nabla \right>^z] \left< \nabla \right>^{-z} \partial^2_{\alpha \beta} \bar{f} \|_{L^2_x} 
			= 	& \| [\mathbf{g}^{\alpha i}, \left< \nabla \right>^z] \left< \nabla \right>^{-z} \partial_{i} ( \partial_{\alpha } \bar{f} ) \|_{L^2_x} 
			\\
			\lesssim & \| d \mathbf{g} \|_{L^\infty_x} \| d \bar{f} \|_{L^2_x}+  \| \mathbf{g}^{\alpha \beta}- \mathbf{m}^{\alpha \beta}\|_{H^{s}_x}   \|\left< \nabla \right>^{-z} \nabla d\bar{f} \|_{L^\infty_x}.
		\end{split}
	\end{equation}
	Taking advantage of
	\begin{equation*}
		\| d \mathbf{g} \|_{L^4_t L^\infty_x} +  \| \mathbf{g}^{\alpha \beta}- \mathbf{m}^{\alpha \beta}\|_{L^\infty_t H^{s}_x} \lesssim \epsilon_0,
	\end{equation*}
	we can bound \eqref{qdy} by
	\begin{equation*}
		\| [\mathbf{g}^{\alpha \beta}, \left< \nabla \right>^z] \left< \nabla \right>^{-z} \nabla^2_{\alpha \beta} \bar{f} \|_{L^2_x} \lesssim \epsilon_0( \| d \bar{f} \|_{L^2_x}+  \|\left< \nabla \right>^{-z} \nabla d\bar{f}  \|_{L^\infty_x}).
	\end{equation*}
	Let us go back \eqref{qd}. Using the discussion in case $r=1$, for $ \theta<0$, we obtain
	\begin{equation}\label{eh}
		\begin{split}
			& \| \left< \nabla \right>^{(\theta-1)} d \bar{f} \|_{L^4_t L^\infty_x} 
			\\
			\lesssim  & \  \epsilon_0( \|\left< \nabla \right>^{r-1}f_0\|_{H^1}+ \| \left< \nabla \right>^{r-1} f_1 \|_{L^2}+ \| d \bar{f} \|_{L^2_x}+  \|\left< \nabla \right>^{-r} \nabla d \bar{f} \|_{L^4_t L^\infty_x} ) ,
			\\
			\lesssim & \ \epsilon_0( \|f_0\|_{H^r}+ \| f_1 \|_{H^{r-1}} + \| d \bar{f} \|_{L^2_x}+  \|\left< \nabla \right>^{-r} \nabla d \bar{f} \|_{L^4_t L^\infty_x} ).
		\end{split}
	\end{equation}
	Taking $\theta = -r+1$ in \eqref{eh}, we can see
	\begin{equation}\label{eeh}
		\|\left< \nabla \right>^{-r} \nabla d \bar{f} \|_{L^4_t L^\infty_x} \lesssim \epsilon_0( \|f_0\|_{H^r}+ \| f_1 \|_{H^{r-1}}).
	\end{equation}
	By \eqref{qd}, \eqref{qqd}, and \eqref{eeh}, we have
	\begin{equation*}
		\begin{split}
			\| \left< \nabla \right>^a f, \left< \nabla \right>^{a-1} d f \|_{L^2_t L^\infty_x} \lesssim  
			& \  \epsilon_0( \|f_0\|_{H^r}+ \| f_1 \|_{H^{r-1}} ) , \quad   a<r-1.
		\end{split}
	\end{equation*}
\end{proof}
\begin{remark}\label{rel}
	From Proposition \ref{A1}, it implies that there is a good approximate solution $f_\lambda$ for the problem
	\begin{equation}\label{linearD}
		\begin{cases}
			\square_{\mathbf{g}_\lambda} f=0, \quad (t,x)\in [-2,2]\times \mathbb{R}^2,\\
			(f, \partial_t f)|_{t=t_0}=(P_\lambda f_0, P_\lambda f_1).
		\end{cases}
	\end{equation}
	In the case of $\epsilon_0 \lambda\leq 1$, we can take $f_\lambda=P_\lambda f$, where $f$ is the exact solution of \eqref{linearD}.
	Applying energy estimates for \eqref{linearD}, we can see
	\begin{equation}\label{eet}
		\|df\|_{L^\infty_{[-2,2]} L^2_x} \lesssim \| P_\lambda f_0\|_{H^1}+\|P_\lambda f_1\|_{L^2}.
	\end{equation}
	Moreover, for $\mathbf{g}^{00}_\lambda=-1$, so we have
	\begin{equation*}
		\begin{split}
			\| \square_{\mathbf{g}_\lambda} f_\lambda \|_{L^2_{[-2,2]} L^2_x} 
			\lesssim & 
			\| [ \mathbf{g}^{\alpha i}_\lambda, \partial_\alpha P_\lambda ] \partial_i f \|_{ L^2_{[-2,2]} L^2_x } 
			+ \| P_\lambda (\partial_\alpha \mathbf{g}^{\alpha i}_\lambda ) \partial_i f \|_{ L^2_{[-2,2]} L^2_x }
			\\
			\lesssim & \| d \mathbf{g}_\lambda \|_{L^2_{[-2,2]} L^\infty_x} \| d {f} \|_{L^\infty_{[-2,2]} L^2_x}
			\\
			\lesssim & \epsilon_0 ( \| P_\lambda f_0\|_{H^1}+\|P_\lambda f_1\|_{L^2} ).
		\end{split}
	\end{equation*}
	The Strichartz estimate \eqref{Ase} follows from Sobolev imbeddings and \eqref{eet}. Hence, the rest of the paper is to establishing Proposition \ref{A1} in the case that
	\begin{equation*}
		\epsilon_0 \lambda \gg 1.
	\end{equation*}
\end{remark}
\subsection{Proof of Proposition \ref{A1}}
By using Remark \ref{rel}, it suffices for us to consider the problem if the frequency $\lambda$ is large enough, namely
\begin{equation*}
	\lambda \geq \epsilon_0^{-1}.
\end{equation*}
In this case, we can construct an approximate solution to \eqref{linearA} by using wave packets. We also note that the proposition \ref{A1} has two parts: \eqref{Yee} and \eqref{Ase}. The conclusion of the first part, i.e. \eqref{Yee} will be established in the following Proposition \ref{szy} and \ref{szi}. The inequality \eqref{Ase} will be proved in Proposition \ref{szt}.

Before our proof, we introduce a spatially localized mollifier $T_\lambda$ by
\begin{equation*}
	T_\lambda f = \chi_\lambda * f, \quad \chi_\lambda=\lambda^2 \chi(\lambda^{-1} y),
\end{equation*}
where $\chi \in C^\infty_0(\mathbb{R}^2)$ is supported in the ball $|x| \leq \frac{1}{32}$, and has integral $1$. By choosing $\chi$ appropriately, any function $\phi$ with frequency support contained in $\{\xi\in \mathbb{R}^2:|\xi|\leq 4\lambda\}$ can be factored $\phi=T_\lambda \widetilde{\phi}$, where $\| \widetilde{\phi} \|_{L^2_x} \approx \| \phi \|_{L^2_x}$.


\subsubsection{\textbf{A normalized wave packet}}\label{cwp}
Let's begin with the definition of a normalized wave packet.
\begin{definition}[\cite{ST},Definition 8.1]
	Let the hypersurface $\Sigma_{\omega,r}$ and the geodesic $\gamma$ be defined in Section \ref{secp4}. A normalized wave packet centered around $\gamma$ is a function $f$ of the form
	\begin{equation*}
		f=\epsilon_0^{\frac14} \lambda^{-\frac54} T_\lambda(u \theta),
	\end{equation*}
	where
	\begin{equation*}
		u(t,x)=\delta(x_{\omega}-\phi_{\omega,r}(t,x'_\omega)), \quad \theta=\theta_0( (\epsilon_0 \lambda)^{\frac12}(x'_\omega-\gamma'_\omega(t))  ).
	\end{equation*}
	Here, $\theta_0$ is a smooth function supported in the set $|x'| \leq 1$, with uniform bounds on its derivatives $|\partial^\alpha_{x'} \theta_0(x')| \leq c_\alpha$.
\end{definition}
\begin{remark}
	The advantage of this definition is that the derivatives of wave packets involve only the tangential behavior of the restrictions of various functions to the characteristic surfaces $\Sigma_{\omega,r}$, as opposed to their regularity within the support of the wave packet. Therefore, we don't need to discuss the behavior of the null foliations $\cup_{r\in \mathbb{R}}\Sigma_{\omega,r}$ in
	transversal directions.
\end{remark}
We give two notations here. The notation $L(\varphi,\psi)$ means a translation invariant bilinear operator of the form
\begin{equation*}
	L(\varphi,\psi)(x)=\int K(y,z)\varphi(x+y)\psi(x+z)dydz,
\end{equation*}
where $K(y,z)$ is a finite measure. If $X$ is a Sobolev spaces, we then denote $X_\kappa$ the same space but with the norm obtained by dimensionless rescaling by $\kappa$,
\begin{equation*}
	\| \varphi \|_{X_\kappa}=\| \varphi(\kappa \cdot)\|_{X}.
\end{equation*}
Since $\frac74<s_0\leq s \leq \frac{15}{8}$, we get $s_0-\frac54>\frac12$. For $\kappa<1$, we have $\| \varphi \|_{H^{s_0-\frac54}_\kappa(\mathbb{R})} \lesssim \| \varphi \|_{H^{s_0-\frac54}(\mathbb{R})}$.

Based on the above definition, let us calculate what we will get when we operate $\square_{\mathbf{g}_\lambda}$ on wave packets.
\begin{proposition}\label{np}
	Let $f$ be a normalized packet. Then there exists another normalized wave packet $\tilde{f}$ and functions $\phi_m(t,x'_\omega), m=0,1,2$, so that
	\begin{equation}\label{np1}
		\square_{\mathbf{g}_\lambda} P_\lambda f= L(d \mathbf{g}, d \tilde{P}_\lambda \tilde{f})+ \epsilon_0^{\frac14}\lambda^{-\frac54}P_{\lambda}T_{\lambda}\sum_{m=0,1,2}
		\psi_m\delta^{(m)}(x'_\omega-\phi_{\omega,r}),
	\end{equation}
	where the functions $\psi_m=\psi_m(t,x'_\omega)$ satisfy the scaled Sobolev estimates
	\begin{equation}\label{np2}
		\| \psi_m\|_{L^2_t H^{s_0-\frac54}_{a,x'_\omega}} \lesssim \epsilon_0 \lambda^{1-m}, \quad m=0,1,2, \quad a=(\epsilon_0 \lambda)^{-\frac12}.
	\end{equation}
\end{proposition}
\begin{proof}[Proof of Proposition \ref{np}]
	For brevity, we consider the case $\omega=(0,1)$. Then $x_\omega=x_2$, and $x'_\omega=x'$. We write
	\begin{equation}\label{js0}
		\square_{\mathbf{g}_\lambda} P_\lambda f
		= \lambda^{-1} ( [\square_{\mathbf{g}_\lambda}, P_\lambda T_\lambda]+P_\lambda T_\lambda \square_{\mathbf{g}_\lambda} )(u\theta).
	\end{equation}
	For the first term in \eqref{js0}, noting $\mathbf{g}_\lambda$ supported at frequency $\leq 4 \lambda$, then we can write
	\begin{equation*}
		[\square_{\mathbf{g}_\lambda}, P_\lambda T_\lambda]=[\square_{\mathbf{g}_\lambda}, P_\lambda T_\lambda]\tilde{P}_\lambda \tilde{T}_\lambda
	\end{equation*}
	for some multipliers $\tilde{P}_\lambda$ and $\tilde{T}_\lambda$ which have the same properties as $P_\lambda$ and $T_\lambda$. Therefore, by using the kernal bounds for $P_\lambda T_\lambda$, we conclude that
	\begin{equation*}
		[\square_{\mathbf{g}_\lambda}, P_\lambda T_\lambda]f=L(d\mathbf{g}, df).
	\end{equation*}
	For the second term in \eqref{js0}, we use the Leibniz rule
	\begin{equation}\label{js1}
		\square_{\mathbf{g}_\lambda}(u\theta)=\theta \square_{\mathbf{g}_\lambda} u+(\mathbf{g}^{\alpha \beta}_\lambda+\mathbf{g}^{\beta \alpha}_\lambda)\partial_\alpha u \partial_\beta \theta + u \square_{\mathbf{g}_\lambda} \theta.
	\end{equation}
	Let $\nu$ denote the conormal vector field along $\Sigma$. We thus know $\nu=dx_2-d\phi(t,x')$. In the following, we take the greek indices $0 \leq \alpha, \beta \leq 1$.

	For the first term in \eqref{js1}, we can compute out
	\begin{equation*}
		\begin{split}
			\mathbf{g}_{\lambda}^{\alpha \beta} \partial^2_{\alpha \beta} u =& \mathbf{g}_{\lambda}^{\alpha \beta}(t,x',\phi)\nu_\alpha \nu_\beta \delta^{(2)}_{x_2-\phi}
			-2 (\partial_2 \mathbf{g}_{\lambda}^{\alpha \beta})(t,x',\phi) \nu_\alpha \nu_\beta \delta^{(1)}_{x_2-\phi}
			\\
			&+(\partial^2_2 \mathbf{g}_{\lambda}^{\alpha \beta})(t,x',\phi) \nu_\alpha \nu_\beta \delta^{(0)}_{x_2-\phi}- \mathbf{g}_{\lambda}^{\alpha \beta}(t,x',\phi) \partial^2_{\alpha \beta}\phi \delta^{(1)}_{x_2-\phi}
			\\
			&+ \partial_2\mathbf{g}_{\lambda}^{\alpha \beta}(t,x',\phi) \partial^2_{\alpha \beta}\phi \delta^{(0)}_{x_2-\phi}.
		\end{split}
	\end{equation*}
	Above, $\delta^{(m)}_{x_2-\phi}=(\partial^m \delta)(x_2-\phi) $. Due to Leibniz rule, we can take
	\begin{equation*}
		\begin{split}
			\psi_0&=\theta \big\{ (\partial^2_2 \mathbf{g}_{\lambda}^{\alpha \beta})(t,x',\phi)\nu_\alpha \nu_\beta+ (\partial_2 \mathbf{g}_{\lambda}^{\alpha \beta})(t,x',\phi)\partial^2_{\alpha \beta}\phi  \big\},
			\\
			\psi_1&=\theta \big\{ 2(\partial_2 \mathbf{g}_{\lambda}^{\alpha \beta})(t,x',\phi)\nu_\alpha \nu_\beta-  \mathbf{g}_{\lambda}^{\alpha \beta}(t,x',\phi)\partial^2_{\alpha \beta}\phi \big\},
			\\
			\psi_2&=\theta( \mathbf{g}_{\lambda}^{\alpha \beta}-\mathbf{g}^{\alpha \beta})\nu_\alpha \nu_\beta,
		\end{split}
	\end{equation*}
	Taking advantage of \eqref{Re}, Proposition \ref{r1}, and Corollary \ref{vte}, we can conclude that this settings of $\psi_0, \psi_1,$ and $\psi_2$ satisfy the estimates \eqref{np2}.

	For the second term in \eqref{js0}, we have
	\begin{equation*}
		\begin{split}
			(\mathbf{g}^{\alpha \beta}_\lambda+\mathbf{g}^{\beta \alpha}_\lambda)\partial_\alpha u \partial_\beta \theta=& \frac12\nu_\alpha (\mathbf{g}^{\alpha \beta}_\lambda+\mathbf{g}^{\beta \alpha}_\lambda)(t,x',\phi)\partial_\beta \theta \delta^{(1)}_{x_2-\phi}
			\\
			&- \frac12\nu_\alpha \partial_2 (\mathbf{g}^{\alpha \beta}_\lambda+\mathbf{g}^{\beta \alpha}_\lambda)(t,x',\phi)\partial_\beta \theta \delta^{(0)}_{x_2-\phi}.
		\end{split}
	\end{equation*}
	Then we can take
	\begin{equation*}
		\psi_0= \frac12\nu_\alpha \partial_2 (\mathbf{g}^{\alpha \beta}_\lambda+\mathbf{g}^{\beta \alpha}_\lambda)(t,x',\phi)\partial_\beta \theta, \quad \psi_1= \frac12\nu_\alpha (\mathbf{g}^{\alpha \beta}_\lambda+\mathbf{g}^{\beta \alpha}_\lambda)(t,x',\phi)\partial_\beta \theta.
	\end{equation*}
	Thanks to \eqref{Re}, Proposition \ref{r1}, and Corollary \ref{vte}, we conclude that this settings of $\psi_0$ and $\psi_1$ satisfy the estimate \eqref{np2}.

	For the third term in \eqref{js0}, we take
	\begin{equation*}
		\psi_0=\mathbf{g}^{\alpha \beta}_\lambda(t,x',\phi)\partial^2_{\alpha \beta}\theta.
	\end{equation*}
	By using \eqref{Re}, Proposition \ref{r1}, and Corollary \ref{vte} again, this settings of $\psi_0$ satisfies the estimates \eqref{np2}.
\end{proof}
As a immediate consequence, we can obtain the following corollary.
\begin{corollary}\label{np0}
	Let $f$ be a normalized wave packet. Then the following estimates hold:
	\begin{equation*}\label{np3}
		\|d P_\lambda f \|_{L^\infty_t L^2_x} \lesssim 1, \quad \| \square_{\mathbf{g}} P_\lambda f \|_{L^2_t L^2_x} \lesssim \epsilon_0.
	\end{equation*}
\end{corollary}
From Proposition \ref{np} and Corollary \ref{np0}, a single normalized wave packet is not enough for us to construct approximate solutions to a linear wave equation. Therefore, we discuss the superposition of wave packets as follows.
\subsubsection{\textbf{Superpositions of wave packets}}
The index $\omega$ stands for the initial orientation of the wave packet at $t=-2$, which varies over a maximal collection of approximately $\epsilon_0^{-\frac12}\lambda^{\frac12}$ unit vectors separated by at least $\epsilon_0^{\frac12}\lambda^{-\frac12}$. For each $\omega$, we have the orthonormal coordinate system $(x_\omega, x'_\omega)$ of $\mathbb{R}^2$, where $x_\omega=x \cdot \omega$, and $x'_\omega$ are projective along $\omega$.


Next, we decompose $\mathbb{R}^2$ into a parallel tiling of rectangles of size $(8\lambda)^{-1}$ in the $x_{\omega}$ direction, and $(4\epsilon_0 \lambda)^{-\frac12}$ in the other directions $x'_{\omega}$. The index $j \in \mathbb{N}$ corresponds to a counting of these rectangles in this decomposition. We let $R_{\omega,j}$ denote the collection of the doubles of these rectangles, and let $\Sigma_{\omega,j}$ denote the element of the $\Sigma_{\omega}$ foliation upon which $R_{\omega,j}$ is centered. 
The distinct elements of the foliation, denoted by $\Sigma_{\omega,j}$ are thus separated by at least $(8\lambda)^{-1}$ at $t = −2$, and thus by $(9\lambda)^{-1}$ at all values of $t$, as shown in \eqref{pr0} below. Let $\gamma_{\omega,j}$ denote the null geodesic contained in $\Sigma_{\omega,j}$ which passes through the center of $R_{\omega,j}$ at time $t = −2$.

We denote the slab $T_{\omega,j}$ as
\begin{equation}\label{twj}
	T_{\omega,j}=\Sigma_{\omega,j} \cap \{ |x'_{\omega}-\gamma_{\omega,j}| \leq (\epsilon_0 \lambda)^{-\frac12}\},
\end{equation}
and let
\begin{equation*}
	\Sigma_\omega=\cup_{r\in\mathbb{R}}\Sigma_{\omega,r}.
\end{equation*}
For each $\omega$ the slabs $T_{\omega,j}$ satisfy a finite-overlap condition; indeed, slabs associated to different elements of $\Sigma_\omega$ are disjoint, and those associated to the
same $\Sigma_\omega$ have finite overlap in the $x'_{\omega}$ variable, since the flow restricted to any
$\Sigma_{\omega,r}$ is $C^1$ close to translation. We next introduce some geometry properties for these slabs.

By \eqref{600} and \eqref{606}, then the estimate
\begin{equation*}
	|dr_{\theta}-(\theta \cdot d x-dt)| \lesssim \epsilon_1,
\end{equation*}
holds pointwise uniformly on $[-2,2]\times \mathbb{R}^2$. This also implies that
\begin{equation}\label{pr0}
	| \phi_{\theta,r}(t,x'_{\theta})-\phi_{\theta,r'}(t,x'_{\theta})-(r-r')| \lesssim \epsilon_1|r-r'|.
\end{equation}
On the other hand, \eqref{502} tells us
\begin{equation}\label{pr1}
	\| d^2_{x'_{\omega}}\phi_{\omega,r}(t,x'_{\omega})-d^2_{x'_{\omega}}\phi_{\omega,r'}(t,x'_{\omega})\|_{L^\infty_{x'_{\omega}}} \lesssim \epsilon_2+ \bar{\rho}(t),
\end{equation}
where set $$\bar{\rho}(t)=\| d \mathbf{g} \|_{C^{\delta_0}_x}.$$
By using \eqref{pr0} and \eqref{pr1}, we get
\begin{equation*}\label{pr2}
	\| d_{x'_{\omega}}\phi_{\omega,r}(t,x'_{\omega})-d_{x'_{\omega}}\phi_{\omega,r'}(t,x'_{\omega})\|_{L^\infty_{x'_{\omega}}} \lesssim (\epsilon_2+ \bar{\rho}(t))^{\frac12}|r-r'|^{\frac12}.
\end{equation*}
For $dx_{\omega}-d\phi_{\omega,r}$ is null and also $|d \mathbf{g} | \leq \bar{\rho}(t)$, this also implies H\"older-$\frac12$ bounds on $d\phi_{\omega,r}$. Therefore, we suppose that $(t,x)\in \Sigma_{\omega,r}$ and $(t,y)\in \Sigma_{\omega,r'}$, that $|x'_{\omega}-y'_{\omega}| \leq 2(\epsilon_0 \lambda)^{-\frac12}$, and that $|r-r'|\leq 2 \lambda^{-1}$. Using \eqref{601}, we can obtain
\begin{equation*}
	|l_{\omega}(t,x)-l_{\omega}(t,y)| \lesssim \epsilon_0^{\frac12}\lambda^{-\frac12}+\epsilon_0^{-\frac12}\bar{\rho}(t)\lambda^{-\frac12}.
\end{equation*}
Due to $\dot{\gamma}_{\omega}=l_{\omega}$ and $\|\bar{\rho}\|_{L^4_t} \lesssim \epsilon_0$, any geodesic in $\Sigma_{\omega}$ which intersects a slab $T_{\omega,j}$ should be contained in the similar slab of half the scale.

We now introduce a lemma for superpositions of wave packets from a certain fixed time.
\begin{Lemma}\label{SWP}[\cite{ST},Lemma 8.5 and Lemma 8.6]
	Let $\frac74<s_0\leq s \leq \frac{15}{8}$, $\delta_0 \in (0, s_0-\frac74)$ and $0<\mu < \delta_0$. Let a scalar function $\bar{v}(t,x)$ be formulated by
	\begin{equation*}\label{swp}
		\bar{v}(t,x)=\epsilon_0^{\frac14}\lambda^{-\frac14} P_{\lambda} \sum_{\omega,j}T_\lambda(\psi^{\omega,j}\delta_{x_{\omega}-\phi_{\omega,j}(t,x'_{\omega})}).
	\end{equation*}
	Set $a=(\epsilon_0 \lambda)^{-\frac12}$. Then we have
	\begin{equation}\label{swp0}
		\| \bar{v}(t) \|^2_{L^2_x}\lesssim \sum_{\omega,j}\| \psi^{\omega,j}\|^2_{H^{\frac12+\mu}_a}, \qquad \text{if} \ \ \ \ \bar{\rho}(t) \leq \epsilon_0,
	\end{equation}
	and
	\begin{equation}\label{swp1}
		\| \bar{v}(t) \|^2_{L^2_x}\lesssim \epsilon_0^{-1} \bar{\rho}(t)  \sum_{\omega,j}\| \psi^{\omega,j}\|^2_{H^{\frac12+\mu}_a},  \quad \text{if} \ \ \bar{\rho}(t) \geq \epsilon_0 .
	\end{equation}
	
	\begin{remark}
		Thanks to \eqref{swp0} and \eqref{swp1}, we can carry out
		\begin{equation}\label{swp3}
			\| \bar{v}(t) \|^2_{L^2_x}\lesssim \left(1+\epsilon_0^{-1} \bar{\rho}(t)\right)  \sum_{\omega,j}\| \psi^{\omega,j}\|^2_{H^{\frac12+\mu}_a}.
		\end{equation}
	\end{remark}
\end{Lemma}
\begin{proposition}[\cite{ST},Proposition 8.4]\label{szy}
	Let $f=\sum_{\omega,j}a_{\omega,j}f^{\omega,j}$, where $f^{\omega,j}$ are normalized wave packets supported in $T_{\omega,j}$. Then we have
	\begin{equation}\label{ese}
		\| d P_\lambda f\|_{L^\infty_t L^2_x} \lesssim (\sum_{\omega,j} a^2_{\omega,j})^{\frac12},
	\end{equation}
	and
	\begin{equation}\label{ese1}
		\| \square_{\mathbf{g}_{\lambda}} P_\lambda f \|_{L^1_t L^2_x} \lesssim \epsilon_0 (\sum_{\omega,j} a^2_{\omega,j})^{\frac12}.
	\end{equation}
\end{proposition}
\begin{proof}
	We first prove a weaker estimate comparing with \eqref{ese}
	\begin{equation}\label{ese2}
		\| d P_\lambda f\|_{L^2_t L^2_x} \lesssim (\sum_{\omega,j} a^2_{\omega,j})^{\frac12}.
	\end{equation}
	By using \eqref{swp3} and replacing $P_\lambda$ by $\lambda^{-1}\nabla P_\lambda$, and $\psi^{\omega,j}=a_{\omega,j}\zeta^{\omega,j}$, we have
	\begin{equation*}
		\| \nabla P_\lambda f(t)\|^2_{L^2_x} \lesssim (1+\epsilon_0^{-1}\bar{\rho}(t)) \sum_{\omega,j} a^2_{\omega,j}.
	\end{equation*}
	Due to the fact $\|\bar{\rho}\|_{L^4_t} \lesssim \epsilon_0$, we can see that
	\begin{equation}\label{ese3}
		\| \nabla P_\lambda f\|^2_{L^2_t L^2_x} \lesssim \sum_{\omega,j} a^2_{\omega,j} .
	\end{equation}
	We also need to get the similar estimate for the time derivatives. We can calculate
	\begin{equation*}
		\partial_t h = \dot{\gamma}(t) (\epsilon_0 \lambda)^{\frac12} \tilde{h}, \quad \partial_t \delta(x_\omega-\phi_{\omega,j})=\partial_t \phi_{\omega,j} \delta^{(1)}(x_\omega-\phi_{\omega,j}).
	\end{equation*}
	For $\dot{\gamma} \in L^\infty_t$ and $\partial_t \phi_{\omega,j} \in L^\infty_t$, then we have
	\begin{equation}\label{ese4}
		\| \partial_t P_\lambda f\|^2_{L^2_t L^2_x} \lesssim \sum_{\omega,j} a^2_{\omega,j} .
	\end{equation}
	Together with \eqref{ese3} and \eqref{ese4}, we have proved \eqref{ese2}.

	To prove \eqref{ese1}, we use the formula \eqref{np1}. Considering the right hand of \eqref{np1}, by using \eqref{ese2}, we can bound the first term by
	\begin{equation}\label{k26}
		\| L(d\mathbf{g}, d \tilde{P}_{\lambda} \tilde{f}) \|_{L^1_tL^2_x} \lesssim \|d\mathbf{g}\|_{L^4_tL^\infty_x} \|d \tilde{P}_{\lambda} \tilde{f} \|_{L^2_tL^2_x} \lesssim \epsilon_0 (\sum_{\omega,j} a^2_{\omega,j})^{\frac12}.
	\end{equation}
	It only remains for us to estimate the second right term on \eqref{np1}. If we set
	\begin{equation}\label{k28}
		\vartheta=\epsilon_0^{\frac12} \lambda^{-\frac54} P_\lambda T_\lambda \left(\sum_{\omega,j}a_{\omega,j}\cdot \sum_{m=0,1,2}\psi^{\omega,j}_m \delta^{(m)}_{x_{\omega}-\phi_{\omega,j}} \right),
	\end{equation}
	we have
	\begin{equation*}
		\| \vartheta \|^2_{L^2_x}=(1+\bar{\rho}(t)\epsilon_0^{-1}) \sum_{\omega,j}a^2_{\omega,j} \sum_{m=0,1,2} \lambda^{m-1}\|\psi^{\omega,j}_m (t) \|^2_{H^{1+\mu}_a}.
	\end{equation*}
	By \eqref{np2}, we therefore get
	\begin{equation}\label{k30}
		\begin{split}
			& \| \vartheta \|^2_{L^1_t L^2_x}
			\\
			\lesssim & \left(\int^2_{-2}[1+\bar{\rho}(t)\epsilon_0^{-1}]dt \right) \left( \int^{2}_{-2} \sum_{\omega,j}a^2_{\omega,j}  \sum_{m=0,1,2} \lambda^{2(m-1)}\|\psi^{\omega,j}_m (t) \|^2_{H^{1+\mu}_a} dt \right)
			\\
			\lesssim & \epsilon_0 (\sum_{\omega,j} a^2_{\omega,j})^{\frac12}.
		\end{split}
	\end{equation}
	Due to \eqref{np1}, using \eqref{k26}, \eqref{k28}, and \eqref{k30}, we have proved \eqref{ese1}. Using \eqref{ese2} and \eqref{ese1}, and classical energy estimates for linear wave equation, we obtain \eqref{ese}.
\end{proof}
\subsubsection{\textbf{Matching the initial data}}
Although we have constructed the approximate solutions using superpositions of normalized wave packets, we also need to complete this construction, i.e. matching the initial data for the solutions. 
\begin{proposition}[\cite{ST}, Proposition 8.7]\label{szi}
	Given any initial data $(f_0,f_1) \in H^1 \times L^2$, there exists a function of the form
	\begin{equation*}
		f=\sum_{\omega,j}a_{\omega,j}f^{\omega,j},
	\end{equation*}
	where the function $f^{\omega,j}$ are normalized wave packets, such that
	\begin{equation*}
		P_\lambda f(-2)=P_\lambda f_0, \quad \partial_t P_\lambda f(-2)=P_\lambda f_1.
	\end{equation*}
	Furthermore,
	\begin{equation*}
		\sum_{\omega,j}a_{\omega,j}^2 \lesssim \| f_0 \|^2_{H^1}+ \| f_1 \|^2_{L^2}.
	\end{equation*}
\end{proposition}
\subsubsection{\textbf{Overlap estimates}} Since the foliations $\Sigma_{\omega,r}$ varing with $\omega$ and $r$, so a fixed $\Sigma_{\omega,r}$ may intersect with other $\Sigma_{\omega',r'}$. As a result, we should be clear about the number of $\lambda$-slabs which contain two given points in the space-time $[-2,2]\times \mathbb{R}^2$.
\begin{corollary}[\cite{ST}, Proposition 9.2]\label{corl}
	For all points $P_1=(t_1,x_1)$ and $P_2=(t_2,x_2)$ in space-time $\mathbb{R}^{+} \times \mathbb{R}^2$ , and $\epsilon_0 \lambda \geq 1$, the number $N_{\lambda}(P_1,P_2)$ of slabs of scale $\lambda$ that contain both $P_1$ and $P_2$ satisfies the bound
	\begin{equation*}
		\begin{split}
			N_{\lambda}(P_1,P_2)\lesssim & \epsilon_0^{-\frac12} \lambda^{-\frac12} |t_1-t_2|^{-\frac12}.
		\end{split}
	\end{equation*}
\end{corollary}
After the construction of approximate solutions, we still need to prove the key estimate \eqref{Ase}. 
\subsubsection{\textbf{The proof of \eqref{Ase}}}
To start the proof, let us define $\mathcal{T}=\cup_{\omega,j}T_{\omega,j}$, where $T_{\omega,j}$ is set in \eqref{twj}. We also denote $\chi_{{J}}$ be a smooth cut-off function, and $\chi_J=1$ on a set $J$.
\begin{proposition}\label{szt}
	Let $t \in [-2,2]$ and
	\begin{equation*}
		f=\sum_{J \in \mathcal{T}}a_{J}\chi_{J}f_{J},
	\end{equation*}
	where $\sum_{J \in \mathcal{T}}a_{J}^2 \leq 1$ and $f_J$ are normalized wave packets in $J$. Then
	\begin{equation}\label{Aswr}
		\|f \|_{L^4_t L^\infty_x} \lesssim \epsilon_0^{-\frac14}  (\ln \lambda)^{\frac14}.
	\end{equation}
\end{proposition}
\begin{proof}
	This proof follows Smith-Tataru's paper \cite{ST}(Proposition 10.1 on page 353).
	
	Let us first make a partition of the time-interval $[-2,2]$. By decomposition, there exists a partition $\left\{ I_j \right\}$ of the  interval $[-2,2]$ into disjoint subintervals $I_j$ such that with the size of each $I_j$, $|I_j| \approx \lambda^{-1}$, and the number of subintervals $j \approx\lambda$. 
	We claim that, by the Mean Value Theorem, there exists a number $t_j$ such that
	\begin{equation}\label{Aswrq}
		\|f \|^4_{L^4_t L^\infty_x} \leq \sum_j \|f \|^4_{L^4_{I_j} L^\infty_x} \leq \sum_j \|f(t_j,\cdot) \|^4_{ L^\infty_x} ,
	\end{equation}
	where $t_j$ is located in $I_j$, and $|t_{j+1}-t_j| \approx \lambda^{-1}$. Let us explain the \eqref{Aswrq} as follows. To be simple, we let  $I_{j_0}=[0,\lambda^{-1}]$ and $I_{j_0+1}=[\lambda^{-1},2\lambda^{-1}]$.By mean value theorem, on $I_{j_0}$, we have
	\begin{equation*}
		\|f \|^4_{L^4_{I_{j_0}} L^\infty_x} = \|f(t_{j_0},\cdot) \|^4_{ L^\infty_x}\lambda^{-1}.
	\end{equation*}
	We also have
	\begin{equation*}
		\|f \|^4_{L^4_{I_{j_0+1}} L^\infty_x} = \|f(t_{j_0+1},\cdot) \|^4_{ L^\infty_x}\lambda^{-1}.
	\end{equation*}
	If $t_{j_0} \leq \frac12 \lambda^{-1}$ or $t_{j_0+1} \geq \frac32 \lambda^{-1}$, then $|t_{j_0+1}-t_{j_0}| \geq \frac12 \lambda^{-1}$. Otherwise, $t_{j_0} \in [\frac12 \lambda^{-1}, \lambda^{-1}] $ and $t_{j_0+1} \in [\lambda^{-1}, \frac32 \lambda^{-1}]$. In this case,
	we combine $I_{j_0}, I_{j_0+1}$ together and set $I^*_{j_0}= I_{j_0}\cup I_{j_0+1}$. On the new interval  $I^*_{j_0}$, we can see that
	\begin{equation*}
		\|f \|^4_{L^4_{I^*_{j_0}} L^\infty_x} = 2\lambda^{-1}\|f(t^*_{j_0},\cdot) \|^4_{ L^\infty_x},
	\end{equation*}
	and
	\begin{equation*}
		2\|f(t^*_{j_0},\cdot) \|^4_{ L^\infty_x} = \|f(t_{j_0},\cdot) \|^4_{ L^\infty_x}+\|f(t_{j_0+1},\cdot) \|^4_{ L^\infty_x}. 
	\end{equation*}
	For $f$ is a contituous function, we then get
	\begin{equation*}
		\|f(t^*_{j_0},\cdot) \|_{ L^\infty_x} = \left\{  \frac12( \|f(t_{j_0},\cdot) \|^4_{ L^\infty_x} +\|f(t_{j_0+1},\cdot) \|^4_{ L^\infty_x} ) \right\}^{\frac14}, \qquad t^*_{j_0} \in [\frac12\lambda^{-1}, \frac32\lambda^{-1}].
	\end{equation*}
	When no confusion arise, we still set $t_{j_0}=t^*_{j_0} \in [\frac12\lambda^{-1}, \frac32\lambda^{-1}]$. On the next time-interval $I_{j_0+2}=[2\lambda^{-1}, 3\lambda^{-1}]$, we have
	\begin{equation*}
		\|f \|^4_{L^4_{I_{j_0+2}} L^\infty_x} = \|f(t_{j_0+1},\cdot) \|^4_{ L^\infty_x}\lambda^{-1}, \quad t_{j_0+1} \in [2\lambda^{-1}, 3\lambda^{-1}].
	\end{equation*}
	We thus obtain $|t_{j_0+1}-t_{j_0}| \geq \frac12 \lambda^{-1}$. In this way, we can decompose $[-2,2]$.

	Therefore, to prove \eqref{Aswr}, and combining with \eqref{Aswrq}, we only need to show that
	\begin{equation}\label{wee0}
		\sum_j |f(t_j,x_j)|^4 \lesssim \epsilon_0^{-1} \ln \lambda,
	\end{equation}
	where $x_j$ is arbitrarily chosen. We then set the points $P_j=(t_j,x_j)$.

	Since each points lies in at most $\approx \epsilon_0^{-\frac12}\lambda^{\frac12}$ slabs, so we may assume that $|a_J| \geq \epsilon_0^{\frac12}\lambda^{-\frac12}$. Then we decompose the sum $f=\sum_{J \in \mathcal{T}}a_{J}\chi_{J}f_{J}$ dyadically with respect to the size of $a_J$. 
	We next decompose the sum over $j$ via a dyadic decomposition in the numbers of slabs containing $(t_j,x_j)$. We may assume that we are summing over $M$ points $(t_j,x_j)$, each of which is contained in approximately $L$ slabs. Then $|f(t_j,x_j)|\lesssim N^{-\frac12}L$ and
	\begin{equation}\label{wee1}
		\sum_j |f(t_j,x_j)|^4 \lesssim L^4 M N^{-2}.
	\end{equation}
	Combining \eqref{wee0} with \eqref{wee1}, so we only prove
	\begin{equation}\label{wee2}
		L^4 M N^{-1} \lesssim \epsilon_0^{-1} \ln \lambda .
	\end{equation}
	This is a counting problem, which we will prove it by calculating in two different ways. It's based on the number $K$ of pairs $(i,j)$ for which $P_i$ and $P_j$ are contained in a common slab, counted with multiplicity. For $J \in \mathcal{T}$, we denote by $n_J$ the number of points $P_j$ contained in $J$. Then
	\begin{equation*}
		K \approx \sum_{n_J \geq 2} n_J^2 \gtrsim N^{-1}  (\sum_{n_J \geq 2} n_J)^2.
	\end{equation*}
	Note that $ \sum_{J \in \mathcal{T}} n_J \approx ML$. We consider it into two cases. If
	\begin{equation*}
		\sum_{n_J \geq 2} n_J \leq \sum_{n_J =1 } n_J,
	\end{equation*}
	then $N\approx ML$. In this case, combining with the fact that $L \lesssim \epsilon_0^{-\frac12}\lambda^{\frac12}$, then \eqref{wee2} holds. Otherwise, we have
	\begin{equation}\label{wee3}
		K \gtrsim N^{-1}M^2L^2.
	\end{equation}
	In this case, by using Corollary \ref{corl}, we obtain
	\begin{equation}\label{wee4}
		K \lesssim \epsilon_0^{-\frac12} \lambda^{-\frac12} \sum_{1\leq i,j \leq M, i\neq j} |t_i-t_j|^{-\frac12}. 
	\end{equation}
	The sum is maximized in the case that $t_j$ are close as possible, i.e. if the $t_j$ are consecutive multiples of $\lambda^{-1}$. Therefore, for $M \approx \lambda$, we can update \eqref{wee4} as
	\begin{equation}\label{wee5}
		\begin{split}
			K \lesssim & \epsilon_0^{-\frac12} \lambda^{-\frac12} \lambda^{\frac12}  \sum_{1\leq i,j \leq M, i\neq j} |i-j|^{-\frac12} 
			\\
			\lesssim & \epsilon_0^{-\frac12}  (\sum_{1\leq i,j \leq M, i\neq j} |i-j|^{-1})^{\frac12}  (\sum_{1\leq i,j \leq M, i\neq j} 1 )^{\frac12}
			\\ 
			\lesssim & \epsilon_0^{-\frac12} ( M \ln M )^{\frac12} \cdot ( M^2 )^{\frac12}
			\\
			=& M^{\frac32} \epsilon_0^{-\frac12} ( \ln \lambda )^{\frac12}. 
		\end{split}
	\end{equation}
	Combining \eqref{wee5} and \eqref{wee3}, we get \eqref{wee2}. Therefore, we have finished the proof of Proposition \ref{szt}.
\end{proof}

\section{Proof of Theorem \ref{dingli2}}\label{Sec8}
In this section, we will prove the existence of solutions for Theorem \ref{dingli2} (uniqueness was already established in Corollary \ref{uniq2}). First, we reduce Theorem \ref{dingli2} to the case of smooth initial data with bounded frequency support (see Proposition \ref{DDL2} below). Next, we present a small data result with smoother vorticity than that in Proposition \ref{DDL2} (see Proposition \ref{DDL3} below). Finally, based on Proposition \ref{DDL3}, we develop a semi-classical method\footnote{This approach is inspired by Bahouri-Chemin \cite{BC1,BC2}, Tataru \cite{T1}, and Ai-Ifrim-Tataru \cite{AIT}.} to prove the Strichartz estimate involving the original vorticity in Proposition \ref{DDL2}, thereby completing the proof of Proposition \ref{DDL2}.

\subsection{Proof of Theorem \ref{dingli2}}\label{keypra}
Denote $P_{j}$ being the Littlewood-Paley operator with the frequency $2^j (j \in \mathbb{Z})$. Denote a sequence of initial data $(h_{0j},\bv_{0j})$ satisfy
\begin{equation}\label{Dss0}
	h_{0j}=P_{\leq j}h_0, \quad \bv_{0j}=P_{\leq j}\bv_0,
\end{equation}
where $(h_0, \bv_0)$ is stated as \eqref{czB}, and $P_{\leq j}=\textstyle{\sum}_{k\leq j}P_k$. Following \eqref{Vor}, we define $\bw_{0j}=(w^0_{0j},w^1_{0j},w^2_{0j})$ as
\begin{equation}\label{Qss0}
	w^\alpha_{0j}=\epsilon^{\alpha \beta \gamma}\partial_\beta v_{0j\gamma}.
\end{equation}
Due to \eqref{czB} and \eqref{Qss0}, we can obtain
\begin{equation}\label{Ess0}
	\begin{split}
		\|\bw_{0j}\|_{H^{\frac32}}
		\leq &  M_0  .
	\end{split}
\end{equation}
Similarly, we have
\begin{equation}\label{Essa}
	\begin{split}
		\|\nabla \bw_{0j}\|_{L^{8}}
		\leq &   M_0  .
	\end{split}
\end{equation}
Adding \eqref{Dss0}, \eqref{Ess0}, and \eqref{Essa}, we can see
\begin{equation}\label{pu0}
	\|h_{0j}\|_{H^s}+ \|\bv_{0j}\|_{H^s}+ \| \bw_{0j} \|_{H^{\frac32}}  + \|\nabla \bw_{0j}\|_{L^{8}}\leq CM_0 .
\end{equation}
Using \eqref{HEw}, \eqref{Dss0} and \eqref{Qss0}, we get
\begin{equation}\label{pu00}
	| \bv_{0j}, h_{0j} | \leq C_0, \quad c_s|_{t=0}\geq c_0>0.
\end{equation}

Before we give a proof of Theorem \ref{dingli2}, let us now introduce Proposition \ref{DDL2}.

\begin{proposition}\label{DDL2}
	Let $s$ be stated as in Theorem \ref{dingli2}. Let \eqref{HEw}-\eqref{czB} hold. Let $(h_{0j}, \bv_{0j}, \bw_{0j})$ be stated in \eqref{Dss0} and \eqref{Qss0}. For each $j\geq 1$, consider Cauchy problem \eqref{wrt} with the initial data $(h_{0j},\bv_{0j},  \bw_{0j})$. Then for all $j \geq 1$, there exists two positive constants $T^{*}>0$ and ${M}_{2}>0$ ($T^*$ and ${M}_{2}$ only depends on ${M}_{0},s, C_0$ and $c_0$) such that \eqref{wrt} has a unique solution $(h_{j},\bv_{j})\in C([0,T^*];H_x^s)$, $\bw_{j}\in C([0,T^*];H_x^{\frac32}) $. To be precise,

	\begin{enumerate}
		\item the solution $h_j, \bv_j$ and $\bw_j$ satisfy the energy estimates
		\begin{equation}\label{Duu0}
			\|h_j\|_{L^\infty_{[0,T^*]}H_x^{s}}+\|\bv_j\|_{L^\infty_{[0,T^*]}H_x^{s}}+ \|\bw_j\|_{L^\infty_{[0,T^*]}H_x^{\frac32}}
			+ \|\nabla \bw_j\|_{L^\infty_{[0,T^*]}L_x^{8}} \leq {M}_{2},
		\end{equation}
		and
		\begin{equation}\label{Duu00}
			\|h_j, \bv_j\|_{L^\infty_{[0,T^*]\times \mathbb{R}^2}} \leq 2+C_0.
		\end{equation}
		\item the solution $h_j$ and $\bv_j$ satisfy the Strichartz estimate
		\begin{equation}\label{Duu2}
			\|dh_j, d\bv_j\|_{L^4_{[0,T^*]}L_x^\infty} \leq {M}_{2}.
		\end{equation}
		\item  for $s-\frac{3}{4} \leq r \leq \frac{11}{4}$, consider the following linear wave equation
		\begin{equation}\label{Duu21}
			\begin{cases}
				\square_{{g}_j} f_j=0, \qquad [0,T^*]\times \mathbb{R}^2,
				\\
				(f_j,\partial_t f_j)|_{t=0}=(f_{0j},f_{1j}),
			\end{cases}
		\end{equation}
		where $(f_{0j},f_{1j})=(P_{\leq j}f_0,P_{\leq j}f_1)$ and $(f_0,f_1)\in H_x^r \times H^{r-1}_x$. Then there is a unique solution $f_j$ on $[0,T^*]\times \mathbb{R}^2$. Moreover, for $a\leq r-(s-1)$, the following estimates
		\begin{equation}\label{Duu22}
			\begin{split}
				&\|\left< \nabla \right>^{a} {f}_j,\left< \nabla \right>^{a-1} d{f}_j\|_{L^2_{[0,T^*]} L^\infty_x}
				\leq  {M}_3(\|{f}_0\|_{{H}_x^r}+ \|{f}_1 \|_{{H}_x^{r-1}}),
				\\
				&\|{f}_j\|_{L^\infty_{[0,T^*]} H^{r}_x}+ \|\partial_t {f}_j\|_{L^\infty_{[0,T^*]} H^{r-1}_x} \leq  {M}_3 (\| {f}_0\|_{H_x^r}+ \| {f}_1\|_{H_x^{r-1}}),
			\end{split}
		\end{equation}
		hold, where ${M}_3$ is a universal constant depends on $C_0, c_0, M_0, s$.
	\end{enumerate}
\end{proposition}

Based on Proposition \ref{DDL2}, we are ready to prove Theorem \ref{dingli2}.
\medskip\begin{proof}[Proof of Theorem \ref{dingli2} by using Proposition \ref{DDL2}] 
	
	If we set $\bU_j=(p(h_j), \mathrm{e}^{-h_j}\bv_j)$, by Lemma \ref{QH} and \ref{Wrs}, for $j \in \mathbb{N}^{+}$ we have
	\begin{equation*}
			A^0 (\bU_j) \partial_t \bU_j+ A^i (\bU_j) \partial_i \bU_j=0,
	\end{equation*}
	and
		\begin{equation*}
			\partial_t w^\alpha_j + (v^0_j)^{-1} v^i_j \partial_i w^\alpha_j = (v^0_j)^{-1} w^\kappa_j \partial^\alpha v_{j \kappa} - (v^0_j)^{-1} w^\alpha_j \partial_\kappa v^\kappa_j.
	\end{equation*}
	Thus, for any $j, l \in \mathbb{N}^{+}$, we obtain
	\begin{equation}\label{ur1}
		\begin{split}
			A^0 (\bU_j) \partial_t (\bU_j-\bU_l)+ A^i (\bU_j) \partial_i (\bU_j-\bU_l)=&-\left\{   A^\alpha (\bU_j)-A^\alpha (\bU_l) \right\} \partial_\alpha \bU_l,
		\end{split}
	\end{equation}
	and
	\begin{equation}\label{ur2}
		\begin{split}
			& \partial_t ( w^\alpha_j -w^\alpha_l ) + (v^0_j)^{-1} v^i_j \partial_i ( w^\alpha_j - w^\alpha_l )  
			\\
			= & - \left\{  (v^0_j)^{-1} v^i_j - (v^0_l)^{-1} v^i_l \right\} \partial_i w^\alpha_l +
			 \left\{  (v^0_j)^{-1} w^\kappa_j - (v^0_l)^{-1} w^\kappa_l \right\} \partial^\alpha v_{j \kappa} 
			 \\
			 & + (v^0_l)^{-1} w^\kappa_l  \partial^\alpha ( v_{j \kappa} - v_{l \kappa} )
			-\left\{  (v^0_j)^{-1} w^\alpha_j - (v^0_l)^{-1} w^\alpha_l \right\}   \partial_\kappa v^\kappa_j
			- (v^0_l)^{-1} w^\alpha_l     \partial_\kappa ( v^\kappa_j -  v^\kappa_l ).
		\end{split}
	\end{equation}
	Due to \eqref{ur1} and \eqref{ur2}, using \eqref{Duu0}-\eqref{Duu2}, we can show that
	\begin{equation*}
		\| \bU_{j}-\bU_{l}\|_{L^\infty_{[0, T^*]}H^1_x}+\| \bw_{j}-\bw_{l}\|_{L^\infty_{[0, T^*]}H^{\frac12}_x} \leq C_{_2} (\| \bv_{0j}-\bv_{0l}\|_{H^{s}}+ \|h_{0j}-h_{0l}\|_{H^{s}}+ \| \bw_{0j}-\bw_{0l}  \|_{H^{\frac32}}).
	\end{equation*}
	Here $C_{2}$ is a positive constant depending on ${M}_2$. By Lemma \ref{jiaohuan0}, so we get
	\begin{equation*}
		\| \bv_{j}-\bv_{l},h_{j}-h_{l}\|_{L^\infty_{[0, T^*]}H^1_x}+\| \bw_{j}-\bw_{l}\|_{L^\infty_{[0, T^*]}H^{\frac12}_x} \leq C_{2} (\| \bv_{0j}-\bv_{0l}\|_{H^{s}}+ \|h_{0j}-h_{0l}\|_{H^{s}}+ \| \bw_{0j}-\bw_{0l}  \|_{H^{\frac12}}).
	\end{equation*}
	Thus, $\{(h_{j}, \bv_{j}, \bw_{j})\}_{j\in \mathbb{N}^+}$ is a Cauchy sequence in $H^1_x\times H^1_x \times H^{\frac12}_x$. Let the limit be $(h, \bv, \bw)$. Then
	\begin{equation}\label{rwe}
		\begin{split}
			(h_{j}, \bv_{j}, \bw_{j})\rightarrow & (h, \bv, \bw) \quad \quad \ \ \quad \text{in} \ \ H_x^{1} \times H_x^{1} \times H_x^{\frac12}.
		\end{split}
	\end{equation}
	By using \eqref{Duu0}, a subsequence of $\{(h_{j}, \bv_{j}, \bw_{j})\}_{j\in \mathbb{N}^+}$ is weakly convergent. Therefore, if $m\rightarrow \infty$, then
	\begin{equation}\label{rwe0}
		\begin{split}
			(h_{j_m}, \bv_{j_m}, \bw_{j_m}, \nabla \bw_{j_m})\rightharpoonup & (h,  \bv, \bw, \nabla \bw) \quad \qquad \ \ \text{in} \ \ H_x^{s} \times H_x^{s} \times H_x^{\frac32} \times L^8_x.
		\end{split}
	\end{equation}
	Due to \eqref{rwe} and \eqref{rwe0}, then $(h, \bv,\bw)$ is a strong solution of \eqref{wrt} and $(h, \bv,\bw,\nabla \bw)\in H_x^{s} \times H_x^{s} \times H_x^{\frac32}\times L^8_x$. Also, points (1), (2), and (3) of Theorem \ref{dingli2} hold. Therefore, we have finished the proof of Theorem \ref{dingli2}.
	
\end{proof}
At this stage, all that remains is to prove Proposition \ref{DDL2}. We will postpone this proof until Section \ref{keypro}. Instead, we turn to introduce Proposition \ref{DDL3}, which, compared to Proposition \ref{DDL2}, provides a small-data existence result with a smoother vorticity.
\subsection{Proposition \ref{DDL3} for small data}
By the finite propagation speed of system \eqref{Wrs}, we denote $c>0$ as the largest speed.  
\begin{proposition}\label{DDL3}
	Assume $s\in (\frac74,\frac{15}{8}], \delta\in (0, s-\frac74)$, $\delta_1=\frac{s-\frac74}{10}$ and \eqref{a0} hold. For each small, smooth initial data $(h_0, \bv_0, \bw_0)$ supported in $B(0,c+2)$ which satisfies
	\begin{equation}\label{DP30}
		\begin{split}
			&\|h_0 \|_{H^{s}} + \|\bv_0 \|_{H^{s}}+   \| \bw_0\|_{H^{\frac32+\delta_1}}+ \|\nabla \bw_0\|_{L^{8}}  \leq \epsilon_3,
		\end{split}
	\end{equation}
	there exists a smooth solution $(h, \bv, \bw)$ to \eqref{wrt} on $[-2,2] \times \mathbb{R}^2$ satisfying
	\begin{equation}\label{DP31}
	\|h\|_{L^\infty_{[-2,2]}H_x^{s}} +	\|\bv\|_{L^\infty_{[-2,2]}H_x^{s}}+ \|\bw\|_{L^\infty_{[-2,2]}H_x^{\frac32+\delta_1}}+ \|\nabla \bw\|_{L^\infty_{[-2,2]}L_x^{8}} \leq \epsilon_2.
	\end{equation}
	Furthermore, the solution satisfies the following properties

	\begin{enumerate}
		\item dispersive estimate for $h$, $\bv$, and $\bv_+$
		\begin{equation}\label{DP32}
			\|d h, d \bv, d \bv_+\|_{L^4_{[-2,2]} C^{\delta}_x} \leq \epsilon_2.
		\end{equation}

		\item For any $t_0 \in [-2,2]$, let $f$ satisfy equation
		\begin{equation}\label{ppp}
			\begin{cases}
				&\square_{g} f=0,
				\\
				&(f,\partial_t f)|_{t=t_0}=(f_0,f_1).
			\end{cases}
		\end{equation} For each $1 \leq r \leq s+1$, the Cauchy problem \eqref{ppp} is well-posed in $H_x^r \times H_x^{r-1}$, and the following estimate holds:
		\begin{equation}\label{DP33}
			\|\left< \nabla \right>^a f\|_{L^4_{[-2,2]} L^\infty_x} \lesssim  \| f_0\|_{H_x^r}+ \| f_1\|_{H_x^{r-1}},\quad \ a<r-\frac34,
		\end{equation}
		and the same estimates hold if we replace $\left< \nabla \right>^a$ by $\left< \nabla \right>^{a-1}d$.
	\end{enumerate}
\end{proposition}

\begin{proof}
	For $\frac74+\delta_1<s$, by replacing the regularity exponents in Proposition \ref{p1} to $s=s, s_0=\frac74+\delta_1$, the conclusion of Proposition \ref{DDL3} follows immediately.
\end{proof}
\begin{remark}
	This Strichartz estimates \eqref{DP33} is stronger than \eqref{Duu2}, since the regularity of vorticity in Proposition \ref{DDL3} is higher than that in Proposition \ref{DDL2}.
\end{remark}

\subsection{Proof of Proposition \ref{DDL2}}\label{keypro}
To prove Proposition \ref{DDL2}, we divide the proof into three parts. First, as presented in subsection \ref{esess}, we obtain a solution to Proposition \ref{DDL2} over a short time interval. Second, we derive robust Strichartz estimates for the low- and mid-frequency components of the solutions within these short intervals, albeit with a loss of derivatives, as demonstrated in subsection \ref{esest}. Finally, in subsection \ref{finalk}, we extend these solutions to a regular time interval. Meanwhile, the solutions to the linear wave equations can also be extended to this regular time interval, as shown in subsection \ref{finalq}.
\subsubsection{Energy estimates and Strichartz estimates on a short time-interval.} \label{esess}
Take the scaling
\begin{equation*}
	\begin{split}
		& \widetilde{h}_{0j}=\rho_{0j}(Tt,Tx), \quad  \widetilde{\bv}_{0j}={\bv}_{0j}(Tt,Tx).
	\end{split}
\end{equation*}
Referring to \eqref{Vor}, we define $\widetilde{\bw}_{0j}=(\widetilde{w}^0_{0j}, \widetilde{w}^1_{0j}, \widetilde{w}^2_{0j})$ as
\begin{equation*}
	\widetilde{w}^\alpha_{0j}=\epsilon^{\alpha \beta \gamma} \partial_\beta \widetilde{v}_{0j\gamma}.
\end{equation*}
Thus, we can derive
\begin{equation}\label{Dss2}
	\begin{split}
		\| \widetilde{h}_{0j} \|_{\dot{H}^{s}} + \| \widetilde{\bv}_{0j} \|_{\dot{H}^{s}}    \leq  & T^{s-1}\| ({h}_{0j} \|_{\dot{H}^{s}} + \| {\bv}_{0j} \|_{\dot{H}^{s}} ),
		\\
		\| \widetilde{\bw}_{0j}\|_{\dot{H}^{\frac32+\delta_1}} \leq  & T^{\frac12+\delta_{1}} \|{\bw}_{0j} \|_{\dot{H}^{\frac32+\delta_1}}
		\\
		\leq & T^{\frac12+\delta_{1}} 2^{\delta_{1} j }\|{\bw}_{0j} \|_{\dot{H}^{\frac32}}.
	\end{split}
\end{equation}
Similarly, we obtain
\begin{equation*}\label{Dssa}
	\begin{split}
		\| \nabla \widetilde{\bw}_{0j}\|_{L^{8}} 
		\leq & T^{\frac74} \|\nabla {\bw}_{0j} \|_{L^{8}}.
	\end{split}
\end{equation*}
Define
\begin{equation}\label{DTJ}
	T^*_j=2^{-\delta_1j}(CM_0)^{-3},
\end{equation}
where $CM_0$ is from \eqref{pu0}. Taking $T$ in \eqref{Dss2} as $T^*_j$, for $s>\frac74$,  \eqref{Dss2} yields
\begin{equation*}\label{pp7}
	\begin{split}
		& \| \widetilde{h}_{0j}\|_{\dot{H}^{s}} + \| \widetilde{\bv}_{0j} \|_{\dot{H}^{s}} +  \| \widetilde{\bw}_{0j}\|_{\dot{H}^{\frac32+\delta_1}}+ \|\nabla \widetilde{\bw}_{0j}\|_{L^{8}}  \leq 2^{-\frac{\delta^2_1}2 j}( 1+CM_0 )^{-\frac12} .
	\end{split}
\end{equation*}
Note $\| \widetilde{\bv}_{0j}\|_{L^\infty}\leq  \|{\bv}_{0j}\|_{L^\infty}$ and $\| \widetilde{h}_{0j}\|_{L^\infty}  \leq  \|{h}_{0j}\|_{L^\infty} $. Due to \eqref{pu00}, we have
\begin{equation*}\label{pp60}
	\begin{split}
	\| \widetilde{h}_{0j}\|_{L^\infty}  +	\| \widetilde{\bv}_{0j}\|_{L^\infty}  \leq  \|{h}_{0j}\|_{L^\infty} +  \|{\bv}_{0j}\|_{L^\infty}  \leq  C_0.
	\end{split}
\end{equation*}
For a small parameter $\epsilon_3$ stated in Proposition \ref{DDL3}, we choose $N_0=N_0(s,M_0,c_0,C_0)$ such that
\begin{equation}\label{pp8}
	\begin{split}
		& 2^{-\frac{\delta^2_1 }2N_0 } ( 1+CM_0 )^{-\frac12} \ll \epsilon_3,
		\\
		&  2^{-\delta_{1} N_0 } (1+C_*)^3\{ 1+ {C^3_*} (CM_0)^{-3}(1+ CM_0)^{-3} \} \leq 1 ,
		\\
		&  C2^{-\frac{3}{4}\delta_1N_0}(CM_0)^{-\frac94} ( 1+CM_0 )^3  [\frac{1}{3}(1-2^{-\delta_{1}})]^{-2} \leq 2 .
	\end{split}	
\end{equation}
Above, $C_*$ is denoted by
\begin{equation}\label{Cstar}
	\begin{split}
		C_*=CM_0\mathrm{e}^2\exp\{  CM_0 \textrm{e}^2 \}  .
	\end{split}	
\end{equation}
Therefore, for $j \geq N_0$, we have
\begin{equation}\label{Dss3}
	\begin{split}
		\| \widetilde{h}_{0j}\|_{\dot{H}^{s}}+\| \widetilde{\bv}_{0j} \|_{\dot{H}^{s}}+\| \widetilde{\bw}_{0j}\|_{\dot{H}^{\frac32+\delta_1}}+\| \nabla \widetilde{\bw}_{0j}\|_{L^8}  \leq  \epsilon_3.
	\end{split}
\end{equation}
Since the estimate \eqref{Dss3} is formulated in terms of homogeneous norms, we need to use physical localization techniques to extend these bounds to the inhomogeneous case.

Define a smooth function $\chi\in C_0^\infty(\mathbb{R}^2)$ supported in $B(0,c+2)$, and $\chi$ equals $1$ in $B(0,c+1)$. Given $y \in \mathbb{R}^2$, we define the localized initial data for the velocity and density near $y$:
\begin{equation}\label{yyy0}
	\begin{split}
		\bar{\bv}_{0j}(x)=&\chi(x-y)\left( \widetilde{\bv}_{0j}(x)- \widetilde{\bv}_{0j}(y)\right),
		\\
		\bar{h}_{0j}(x)=&\chi(x-y)\left( \widetilde{h}_{0j}(x)-\widetilde{h}_{0j}(y)\right).
	\end{split}
\end{equation}
Referring to \eqref{Vor}, we set
\begin{equation}\label{yyy1}
	\bar{w}_{0j}=\epsilon^{\alpha \beta \gamma} \partial_\beta \bar{v}_{0j\gamma}.
\end{equation}
Taking advantage of \eqref{Dss3}, \eqref{yyy0} and \eqref{yyy1}, we obtain
\begin{equation}\label{sS4}
	\begin{split}
		& \| \bar{h}_{0j}\|_{{H}^{s}}+ \| \bar{\bv}_{0j} \|_{{H}^{s}}+\| \bar{\bw}_{0j} \|_{{H}^{\frac32+\delta_1}}+\| \nabla \bar{\bw}_{0j} \|_{{L}^{8}}  
		\\
		\leq  & C(\| \widetilde{h}_{0j}\|_{\dot{H}^{s}}+ \| \widetilde{\bv}_{0j} \|_{\dot{H}^{s}}+\| \widetilde{\bw}_{0j} \|_{\dot{H}^{\frac32+\delta_1}}+\| \nabla \widetilde{\bw}_{0j} \|_{L^{8}})
		\\
		\leq & \epsilon_3.
	\end{split}
\end{equation}
Using Proposition \ref{DDL3}, for every $j$, there is a smooth solution $(\bar{h}_j, \bar{\bv}_j, \bar{\bw}_j)$ on $[-2,2]\times \mathbb{R}^2$ satisfying
\begin{equation}\label{yz0}
	\begin{cases}
		\square_{\widetilde{g}_j} \bar{h}_j=\widetilde{\mathcal{D}}, 
		\\
		\square_{\widetilde{g}_j} \bar{v}_j^\alpha=-\widetilde{c}^2_s \widetilde{\Theta} \epsilon^{\alpha \beta \gamma} \partial_\beta \bar{w}_\gamma +\widetilde{{Q}}^\alpha,
		\\
		( {\bar v}^\kappa_j+\tilde{v}^\kappa_{0j}(y) ) \partial_\kappa \bar{\bw}^\alpha_j= \bar{w}_j^\kappa \partial^\alpha \bar{v}_{j\kappa} - \bar{w}_j^\alpha \partial_\kappa \bar{v}_j^{\kappa} ,
		\\
		(\bar{h}_j, \bar{\bv}_j, \bar{\bw}_j)|_{t=0}=(\bar{h}_{0j}, \bar{\bv}_{0j}, \bar{\bw}_{0j}).
	\end{cases}
\end{equation}
Above, the quantities $\widetilde{c}^2_s$ and $\widetilde{{g}_j}$ are defined by
\begin{equation}\label{QRY}
	\begin{split}
		\widetilde{c}^2_s &= c^2_s (\bar{h}_j+\tilde{h}_{0j}(y)),
		\\
		\widetilde{g}_j &=g({\bar h}_j+\tilde{h}_{0j}(y),  {\bar\bv}_j+\tilde{\bv}_{0j}(y) ),
	\end{split}
\end{equation}
and $\widetilde{{Q}}^{\alpha}, \alpha=0,1,2$ and $\widetilde{\mathcal{D}}$ have the same formulations with ${Q}^{\alpha}, \mathcal{D}$ by replacing $(h, {\bv})$ to $(\bar{\rho}_j+\tilde{\rho}_0(y), \bar{\bv}_j+\tilde{\bv}_0(y))$. Using Proposition \ref{DDL3} again, on $[-2,2]\times \mathbb{R}^2$, we infer
\begin{equation*}\label{Dsee0}
\|\bar{h}_j\|_{L^\infty_{[-2,2]}H_x^{s}}+	\|\bar{\bv}_j\|_{L^\infty_{[-2,2]}H_x^{s}}+  \| \bar{\bw}_j \|_{L^\infty_{[-2,2]}H_x^{\frac32+\delta_1}}+ \|\nabla \bar{\bw}_j \|_{L^\infty_{[-2,2]}L_x^{8}} \leq \epsilon_2,
\end{equation*}
and
\begin{equation*}\label{Dsee1}
	\|d\bar{h}_j , d\bar{\bv}_j \|_{L^4_{[-2,2]}C^{\delta}_x} \leq \epsilon_2.
\end{equation*}
Furthermore, using \eqref{DP33}, for $1\leq r \leq s+1$, the linear equation\footnote{Here $\widetilde{g}_j$ is given by \eqref{QRY}.}
\begin{equation}\label{Dsee2}
	\begin{cases}
		&\square_{ \widetilde{g}_j } F=0,\qquad \qquad \ \ \qquad [-2,2]\times \mathbb{R}^2,
		\\
		&({F}, \partial_t {F})|_{t=t_0}=({F}_0, {F}_1), \ \ \ t_0 \in [-2,2],
	\end{cases}
\end{equation}
admits a solution ${F} \in C([-2,2],H_x^{r})\times C^1([-2,2],H_x^{r-1})$. Moreover, for $k<r-1$, the following estimate holds\footnote{For all $j$, the initial norm of the initial data \eqref{sS4} is uniformly controlled by the same small parameter $\epsilon_3$, and the regularity of the initial data \eqref{sS4} only depends on $s$, so the constant in \eqref{3s1} is uniform for all $\widetilde{{g}}$ depending on $j$.}:
\begin{equation}\label{3s1}
	\begin{split}
		\|\left< \nabla \right>^{k}F, \left< \nabla \right>^{k-1}d{F}\|_{L^4_{[-2,2]} L^\infty_x} \leq  & C(\| {F}_0\|_{H_x^r}+ \| {F}_1\|_{H_x^{r-1}} ).
	\end{split}
\end{equation}
Note \eqref{yz0}. Then the function $(\bar{h}_j+\widetilde{h}_{0j}(y),  \bar{\bv}_j+\widetilde{\bv}_{0j}(y),  \bar{\bw}_j)$ is also a solution of the following system
\begin{equation*}\label{yz1}
	\begin{cases}
		\square_{\widetilde{g}_j} ( \bar{h}_j+\widetilde{h}_{0j}(y) )=\widetilde{\mathcal{D}}, 
		\\
		\square_{\widetilde{g}_j} (\bar{v}^\alpha _j+\widetilde{v}^\alpha_{0j}(y) )=-\widetilde{c}^2_s \widetilde{\Theta} \epsilon^{\alpha \beta \gamma} \partial_\beta \bar{w}_\gamma +\widetilde{{Q}}^\alpha,
		\\
		( {\bar v}^\kappa_j+\tilde{v}^\kappa_{0j}(y) ) \partial_\kappa \bar{\bw}^\alpha_j= \bar{w}_j^\kappa \partial^\alpha \bar{v}_{j\kappa} - \bar{w}_j^\alpha \partial_\kappa \bar{v}_j^{\kappa} ,
		\\
		(\bar{\bv}_j+\widetilde{\bv}_{0j}(y), \bar{\rho}_j+\widetilde{\rho}_{0j}(y),  \bar{w}_j )|_{t=0}=(\widetilde{\bv}_{0j}, \widetilde{\rho}_{0j}, \widetilde{w}_{0j}).
	\end{cases}
\end{equation*}
For $y\in \mathbb{R}^2$, we consider the restrictions 
\begin{equation}\label{Dsee3}
	\left( \bar{\bv}_j+\widetilde{\bv}_{0j} (y) \right)|_{\mathrm{K}^y},
	\quad (\bar{h}_j+\widetilde{h}_{0j} (y) )|_{\mathrm{K}^y},
	\\
	\quad \bar{w}_j|_{\mathrm{K}^y},
\end{equation}
where $$\mathrm{K}^y=\left\{ (t,x): ct+|x-y| \leq c+1, |t| <1 \right\}.$$ Then the restrictions \eqref{Dsee3} solve \eqref{wrt} on $\mathrm{K}^y$. Due to finite speed of propagation, a smooth solution $(\bar{\bv}_j+\widetilde{\bv}_{0j}(y), \bar{\rho}_j+ \widetilde{\rho}_{0j}(y), \bar{\bw}_j)$ solves \eqref{wrt} on $\mathrm{K}^y$.

Let a function $\psi$ be supported in the unit ball such that
\begin{equation*}
	\textstyle{\sum}_{y \in 3^{-\frac12} \mathbb{Z}^2 } \psi(x-y)=1.
\end{equation*}
Therefore, the function
\begin{equation}\label{Dsee4}
	\begin{split}
		\widetilde{\bv}_j(t,x)  &=\textstyle{\sum}_{y \in 3^{-\frac12} \mathbb{Z}^2 }\psi(x-y) (\bar{\bv}_j+\widetilde{\bv}_{0j}(y)),
		\\
		\widetilde{h}_j(t,x)  &=\textstyle{\sum}_{y \in 3^{-\frac12} \mathbb{Z}^2}\psi(x-y) ( \bar{\rho}_j+ \widetilde{\rho}_{0j}(y)),
		\\
		\widetilde{w}^\alpha_j(t,x)  &= \epsilon^{\alpha \beta \gamma}\partial_\beta \widetilde{v}_{j\gamma},
	\end{split}
\end{equation}
is a smooth solution of \eqref{wrt} on $[-1,1]\times \mathbb{R}^2$ with the initial data $(\widetilde{h}_j,\widetilde{\bv}_j,  \widetilde{\bw}_j )|_{t=0}=(\widetilde{h}_{0j}, \widetilde{\bv}_{0j}, \widetilde{\bw}_{0j})$,

On the other hand, the initial data $(\widetilde{h}_{0j}, \widetilde{\bv}_{0j}, \widetilde{\bw}_{0j})$ is a scaling of $(h_{0j},\bv_{0j},  \bw_{0j})$ with the space-time scale $T^*_j$, and the system \eqref{wrt} is scaling-invariant. Therefore, the function
\begin{equation*}
	({h}_{j}, {\bv}_{j}, {\bw}_{j})=(\widetilde{\bv}_{j},\widetilde{\rho}_{j}, \widetilde{\bw}_{j}) ((T^*_j)^{-1}t,(T^*_j)^{-1}x),
\end{equation*}
is a solution of \eqref{wrt} on $[0,T^*_j]\times \mathbb{R}^2$ ($T^*_j$ is defined as in \eqref{DTJ}) with the initial data
\begin{equation*}
	({h}_{j}, {\bv}_{j}, {\bw}_{j})|_{t=0}=({h}_{0j},{\bv}_{0j}, {\bw}_{0j}).
\end{equation*}
Referring \eqref{Dsee4} and using \eqref{Dsee2}-\eqref{3s1}, we can see
\begin{equation}\label{Dsee5}
	\begin{split}
		\|d\widetilde{h}_j, d\widetilde{\bv}_j \|_{L^4_{[0,1]}C^{\delta}_x}
		\leq & \sup_{y \in 3^{-\frac12} \mathbb{Z}^2}  \|d\bar{h}_j, d\bar{\bv}_j \|_{L^4_{[0,1]}{C^{\delta}_x}}
		\\
		\leq & C(\|\bar{h}_{0j}\|_{H^s}+ \|\bar{\bv}_{0j}\|_{H^s}+  \| \bar{\bw}_{0j} \|_{H^{\frac32}}+  \| \nabla \bar{\bw}_{0j} \|_{L^{8}}).
	\end{split}
\end{equation}
By changing of coordinates $(t,x)\rightarrow ((T_j^*)^{-1}t,(T_j^*)^{-1}x)$ for each $j\geq 1$, we get
\begin{equation}\label{Dsee6}
	\begin{split}
		\|d{h}_j, d{\bv}_j\|_{L^4_{[0,T^*_j]}C^{\delta}_x}
		\leq & (T^*_j)^{-(\frac34+{\delta})}\|d\widetilde{h}_j, d\widetilde{\bv}_j\|_{L^4_{[0,1]}C^{\delta}_x}.
	\end{split}
\end{equation}
Using \eqref{DTJ}, \eqref{Dsee5}, \eqref{Dsee6}, and ${\delta} \in (0,s-\frac74)$, it follows
\begin{equation}\label{Dsee7}
	\begin{split}
		\|d{h}_j, d{\bv}_j \|_{L^4_{[0,T^*_j]}C^{\delta}_x}
		\leq	&C (T^*_j)^{-(\frac34+{\delta})}(\|\bar{h}_{0j}\|_{H_x^s}+ \|\bar{\bv}_{0j}\|_{H_x^s}+  \| \bar{\bw}_{0j} \|_{H_x^{\frac32}})
		\\
		\leq & C (T^*_j)^{-(\frac34+{\delta})} \left\{  (T^*_j)^{s-1}\|{\bv}_{0j}, {h}_{0j}\|_{\dot{H}_x^s}
		+ (T^*_j)^{\frac32}\| {\bw}_{0j} \|_{\dot{H}_x^{\frac32}} + (T^*_j)^{\frac74}\|\nabla {\bw}_{0j} \|_{L_x^{8}} \right\}
		\\
		\leq & C (\|{h}_{0j}\|_{H^s}+ \|{\bv}_{0j}\|_{H^s}+  \| {\bw}_{0j} \|_{H^{\frac32}}+ \| \nabla {\bw}_{0j} \|_{L^{8}}).
	\end{split}
\end{equation}
Comining \eqref{pu0} and \eqref{Dsee7}, we obtain
\begin{equation}\label{yz4}
	\begin{split}
		\|d{h}_j, d{\bv}_j\|_{L^4_{[0,T^*_j]}C^{\delta}_x}
		\leq & C M_0.
	\end{split}
\end{equation}
Set
\begin{equation*}\label{Etr}
	E(T^*_{j})=\|h_{j}\|_{L^\infty_{[0,T^*_{j}]} H^{s}_x}+ \|\bv_{j}\|_{L^\infty_{[0,T^*_{j}]} H^{s}_x}+
	\|\bw_{j}\|_{L^\infty_{[0,T^*_{j}]} H^{\frac32}_x}+
	\|\nabla \bw_{j}\|_{L^\infty_{[0,T^*_{j}]} L^{8}_x}.
\end{equation*}
Due to \eqref{DTJ}, \eqref{pp8}, and \eqref{yz4}, we get
\begin{equation*}\label{Aab}
	\begin{split}
		\|d{h}_j, d{\bv}_j \|_{L^4_{[0,T^*_j]}C^{\delta}_x} \leq (T_j^*)^{\frac34}	\|d{h}_j, d{\bv}_j \|_{L^4_{[0,T^*_j]}C^{\delta}_x}
		\leq 2.
	\end{split}
\end{equation*}
By using Theorem \ref{DW4}, we have
\begin{align*}
	\label{AMM3}
	&  E(T^*_j) \leq C (1+M_0)^5 \exp\{ C (1+M_0)^3 \} .
\end{align*}
Similarly, for $1\leq r \leq s+1$, there exists a unique solution for the Cauchy problem
\begin{equation}\label{ru03}
	\begin{cases}
		\square_{{g}_j} F=0, \quad (0,T^*_j]\times \mathbb{R}^2,
		\\
		(F,\partial_t F)|_{t=0}=(F_0,F_1)\in H_x^{r} \times H^{r-1}_x.
	\end{cases}
\end{equation}
Moreover, for $a<r-1$, we have
\begin{equation}\label{ru04}
	\begin{split}
		\|\left<\nabla \right>^{a-1} dF\|_{L^4_{[0,T^*_j]} L^\infty_x}
		\leq & C(\|{F}_0\|_{{H}_x^{r}}+ \| {F}_1 \|_{{H}_x^{r-1}}),
	\end{split}
\end{equation}
and
\begin{equation*}\label{ru05}
	\begin{split}
		\|{F}\|_{L^\infty_{[0,T^*_j]} H^{r}_x}+ \|\partial_t {F}\|_{L^\infty_{[0,T_j]} H^{r-1}_x} \leq  C(\| {F}_0\|_{H_x^{r}}+ \| {F}_1\|_{H_x^{r-1}}).
	\end{split}
\end{equation*}
Gathering the results above, there is a uniform bound for $h_{j}, \bv_{j}$ and $\bw_j$. However, the length of the interval $[0, T_j^*]$ is not uniform and depends on $j$. To construct a solution as the limit of these smooth solutions, we must extend them over a uniformly regular time interval.

\subsubsection{A loss of Strichartz estimates on a short time-interval.} \label{esest}
We will discuss it into two cases\footnote{Here, we mainly inspired by Tataru \cite{T1}, Bahouri-Chemin \cite{BC2}, and Ai-Ifrim-Tataru \cite{AIT}. Of course, our work based on the property of Strichartz estimates in Proposition \ref{r6} and careful analysis on vorticity. Moreover, we eventually conclude it by induction method.}, the high frequency and low frequency for $d{h}_j$ and $d{\bv}_j$.

\textbf{Case 1: high frequency ($k \geq j$).} Due to \eqref{yz4}, we get
\begin{equation}\label{yz6}
	\begin{split}
		& \| d \bv_j, d h_j \|_{L^4_{[0,T^*_j]}C^a_x} \leq CM_0, \quad  a \in [0,s-\frac74).
	\end{split}
\end{equation}
For $k \geq j$, using Bernstein inequality and \eqref{yz6}, we have
\begin{equation}\label{yz9}
	\begin{split}
		\|P_{k} d h_j, P_{k} d \bv_j  \|_{L^4_{[0,T^*_j]}L^\infty_x}
		\leq & 2^{-ka}\| P_k d h_j, P_k d \bv_j  \|_{L^4_{[0,T^*_j]}C^a_x}
		\\
		\leq & C2^{-ja} \|d h_j  ,  d \bv_j \|_{L^4_{[0,T^*_j]}C^a_x}
		\\
		\leq &  C2^{-ja} M_0.
	\end{split}
\end{equation}
Taking $a=9\delta_1$ in \eqref{yz9}, we obtain
\begin{equation}\label{yz10}
	\begin{split}
		\| P_{k} d h_j , P_{k} d \bv_j\|_{L^4_{[0,T^*_j]}L^\infty_x}
		\leq & C   2^{-\delta_{1}k} \cdot (1+M_0)^3 2^{-7\delta_{1}j}, \quad k \geq j,
	\end{split}
\end{equation}
\quad \textbf{Case 2: low frequency($k<j$).} In this case, it's much different from the high frequency. Fortunately, there is some good estimates for difference terms $P_k(d{h}_{m+1}-d{h}_m)$ and $P_k(d{\bv}_{m+1}-d{\bv}_{m})$. Following \eqref{De}, we set
\begin{equation*}\label{etad0}
	\begin{split}
		& \bv_m=\bv_{+m}+ \bv_{-m}, 
		\\
		&\bv_{\pm m}=(v^0_{\pm m},v^1_{\pm m},v^2_{\pm m}),
		\\
		& (\mathbf{Id}-\mathbf{{P}}) v^\alpha_{-m}= \epsilon^{\alpha \beta \gamma} \partial_\beta w_{m\gamma}.
	\end{split}	
\end{equation*}
We will obtain some good estimates of $P_k(d{h}_{m+1}-d{h}_m)$ and $P_k(d{\bv}_{m+1}-d{\bv}_m)$ by using \eqref{ru03}-\eqref{ru04}. Indeed, this good estimate is from a loss of derivatives of Strichartz estimate. We claim that
\begin{align}\label{yu0}
	& \|h_{m+1}-h_{m}, \bv_{m+1}-\bv_{m}\|_{L^\infty_{ [0,T^*_{m+1}] } L^2_x} \leq C 2^{-(s- \delta_1)m} M_0,
	\\\label{yu1}
	& \|\bw_{m+1}-\bw_{m}\|_{L^\infty_{ [0,T^*_{m+1}]  } L^2_x} \leq C2^{- (\frac32- \delta_1)m}M_0.
\end{align}
To verify \eqref{yu0} and \eqref{yu1}, we start from the initial data. Applying Bernstein's inequality, we get
\begin{equation}\label{yz12}
	\begin{split}
		\|{\bv}_{0(m+1)}-{\bv}_{0m}\|_{L^2}
		\lesssim & 2^{-sm} \|{\bv}_{0(m+1)}-{\bv}_{0m}\|_{\dot{H}^s}.
	\end{split}	
\end{equation}
Similarly, we obtain
\begin{equation}\label{yz13}
	\begin{split}
		& \|{h}_{0(m+1)}-{h}_{0m}\|_{L^2} \lesssim   2^{-sm}\|{h}_{0(m+1)}-{h}_{0m}\|_{\dot{H}^s},
		\\
		& \|{\bw}_{0(m+1)}-{\bw}_{0m}\|_{L^2} \lesssim   2^{-{\frac32}m}\|{\bw}_{0(m+1)}-{\bw}_{0m}\|_{\dot{H}^{\frac32}}.
	\end{split}	
\end{equation}
Based on \eqref{yz12} and \eqref{yz13}, we can use the structure of the system \eqref{Wrs} to prove \eqref{yu0}-\eqref{yu1}. Let ${\bU}_{m}=(p(h_{m}), \mathrm{e}^{-h_m}\bv_{m})^\mathrm{T} $. Then ${\bU}_{m+1}-{\bU}_{m}$ satisfies
\begin{equation}\label{yz14}
	\begin{cases}
		& A^0({\bU}_{m+1}) \partial_t ( {\bU}_{m+1}- {\bU}_{m}) + A^i({\bU}_{m+1}) \partial_i ( {\bU}_{m+1}- {\bU}_{m})=\Pi_m,
		\\
		& ( {\bU}_{m+1}-{\bU}_{m} )|_{t=0}= {\bU}_{0(m+1)}-{\bU}_{0m},
	\end{cases}
\end{equation}
where
\begin{equation}\label{Fhz}
	\Pi_m=-[A^0({\bU}_{m+1})-A^0({\bU}_{m}) ]\partial_t  {\bU}_{m}- [A^i({\bU}_{m+1})-A^i({\bU}_{m}) ]\partial_i  {\bU}_{m}.
\end{equation}
From \eqref{Fhz}, we can see
\begin{equation*}\label{Fhz0}
	|\Pi_m|\lesssim |{\bU}_{m+1} - {\bU}_{m}| \cdot |  d{\bU}_{m} |.
\end{equation*}
Multiplying ${\bU}_{m+1}- {\bU}_{m}$ on \eqref{yz14} and integrating it on $\mathbb{R}^2$, we have
\begin{equation}\label{yz15}
	\frac{d}{dt}	\|{\bU}_{m+1}-{\bU}_{m} \|^2_{L^2_x} \lesssim   \| (d {\bU}_{m+1}, d{\bU}_{m}) \|_{L^\infty_x}\| {\bU}_{m+1}-{\bU}_{m} \|^2_{L^2_x}.
\end{equation}
Integrating \eqref{yz15} on $[0,T^*_{m+1}]$, using Gronwall's inequality and Strichartz estimate \eqref{yz4}, it yields
\begin{equation*}
	\begin{split}
		\| {\bU}_{m+1}-{\bU}_{m} \|_{L^\infty_{[0,T^*_{m+1} ]}L^2_x} \lesssim &  \| {\bU}_{0(m+1)}-{\bU}_{0m} \|_{L^2_x} .
	\end{split}
\end{equation*}
Due to \eqref{yz12} and \eqref{yz13}, we then have
\begin{equation*}
	\begin{split}
		\| {\bU}_{m+1}-{\bU}_{m} \|_{L^\infty_{[0,T^*_{m+1} ]}L^2_x} \leq C2^{-(s- \delta_1)m}M_0.
	\end{split}
\end{equation*}
By using Lemma \ref{jiaohuan0}, we show that
\begin{equation*}
	\begin{split}
		\| {\bv}_{m+1}-{\bv}_{m}, {\rho}_{m+1}-{\rho}_{m} \|_{L^\infty_{[0,T^*_{m+1}]}L^2_x} \leq C2^{-(s- \delta_1)m}M_0.
	\end{split}
\end{equation*}
Thus, the estimate \eqref{yu0} holds. To prove \eqref{yu1}, we consider the transport equation of ${\bw}_{m+1} - {\bw}_m$:

\begin{equation}\label{AMM7}
\begin{split}
& \partial_t ({w}^\alpha_{m+1}- {w}^\alpha_{m}) + ( {v}^0_{m+1} )^{-1} {v}^i_{m+1} \partial_i ({w}^\alpha_{m+1}- {w}^\alpha_{m})
	\\
	= &   - \left\{  (v^0_{m+1})^{-1} v^i_{m+1} - (v^0_m)^{-1} v^i_m \right\} \partial_i w^\alpha_m +
	\left\{  (v^0_{m+1})^{-1} w^\kappa_{{m+1}} - (v^0_l)^{-1} w^\kappa_m \right\} \partial^\alpha v_{({m+1}) \kappa} 
	\\
	& + (v^0_m)^{-1} w^\kappa_l  \partial^\alpha ( v_{(m+1) \kappa} - v_{m \kappa} )
	-\left\{  (v^0_{m+1})^{-1} w^\alpha_j - (v^0_l)^{-1} w^\alpha_m \right\}   \partial_\kappa v^\kappa_{m+1}
	\\
	& - (v^0_l)^{-1} w^\alpha_m    \partial_\kappa ( v^\kappa_{m+1} -  v^\kappa_m )
\end{split}	
\end{equation}
Multiplying with ${w}_{(m+1)\alpha}- {w}_{m\alpha}$ on \eqref{AMM7} and integrating it on $[0,t]\times \mathbb{R}^2$, we derive
\begin{equation}\label{AMM8}
	\begin{split}
		\|{\bw}_{m+1}- {\bw}_{m} \|^2_{L^2_x} \leq  & \| {\bw}_{0(m+1)}- {\bw}_{0m} \|^2_{L^2_x}+ C\int^t_{0}  \| \nabla {\bv}_{m+1}\|_{L^\infty_x}\| {\bw}_{m+1}- {\bw}_{m} \|^2_{L^2_x}d\tau
		\\
		& + C\int^t_{0} \| {\bv}_{m+1}-{\bv}_{m} \|_{L^{\frac83}_x}\| {\bw}_{m+1}-{\bw}_{m} \|_{L^2_x}\|\nabla {\bw}_{m} \|_{L^8_x}d\tau
		\\
		\leq  & \| {\bw}_{0(m+1)}- {\bw}_{0m} \|^2_{L^2_x}+ C\int^t_{0}  \| \nabla {\bv}_{m+1}\|_{L^\infty_x}\| {\bw}_{m+1}- {\bw}_{m} \|^2_{L^2_x}d\tau
		\\
		& + C\int^t_{0} \| {\bv}_{m+1}-{\bv}_{m} \|_{H^{\frac14}_x}\| {\bw}_{m+1}-{\bw}_{m} \|_{L^2_x}\|\nabla {\bw}_{m} \|_{L^8_x}d\tau.
	\end{split}
\end{equation}
We also note that
\begin{equation}\label{kkk}
	\| {\bv}_{m+1}-{\bv}_{m} \|_{H^{\frac14}_x} \lesssim \| {\bv}_{m+1}-{\bv}_{m} \|^{1-\frac{1}{4s}}_{L^{2}_x} \| {\bv}_{m+1}-{\bv}_{m} \|^\frac{1}{4s}_{H^{s}_x}  .
\end{equation}
Taking advantage of \eqref{yz13}, \eqref{yu0}, \eqref{yz6}, \eqref{AMM8}, and \eqref{kkk}, and using Gronwall's inequality gives
\begin{equation*}
	\begin{split}
		\|{\bw}_{m+1}- {\bw}_{m} \|_{L^\infty_{[0,T^*_{m+1}]} L^2_x}  \lesssim   2^{-(\frac32-\delta_1)m} M_0.
	\end{split}
\end{equation*} 
This implies that \eqref{yu1} holds. On the other hand, for $\delta_1<s-\frac74$, using Strichartz estimates \eqref{ru04}(similar to \eqref{SR0}), it yields
\begin{equation}\label{see80}
	\begin{split}
		& \|\nabla({h}_{m+1}-{h}_{m}) \|_{L^4_{[0,T^*_{m+1}]} C^{\delta_{1}}_x}+\|\nabla ({\bv}_{+(m+1)}-{\bv}_{+m}) \|_{L^4_{[0,T^*_{m+1}]} C^{\delta_{1}}_x}
		\\
		\leq  &  C\| {h}_{m+1}-{h}_{m}, {\bv}_{m+1}-{\bv}_{m} \|_{L^\infty_{[0,T^*_{m+1}]} H^{\frac74+2\delta_{1}}_x}+\| {\bw}_{m+1}-{\bw}_{m} \|_{L^\infty_{[0,T^*_{m+1}]} H^{1+2\delta_{1}}_x}.
	\end{split}
\end{equation}
Noting $s=2+10\delta_{1}$ and using \eqref{yu0}-\eqref{yu1}, we can bound \eqref{see80} by
\begin{equation}\label{see99}
	\begin{split}
		& \|\nabla({h}_{m+1}-{h}_{m}) \|_{L^4_{[0,T^*_{m+1}]} C^{\delta_{1}}_x}+\|\nabla ({\bv}_{+(m+1)}-{\bv}_{+m}) \|_{L^4_{[0,T^*_{m+1}]} C^{\delta_{1}}_x}
		\leq   C 2^{-7\delta_1m} M_0.
	\end{split}
\end{equation}
Applying \eqref{yu1} and Sobolev imbeddings, we can establish
\begin{equation}\label{siq}
	\begin{split}
		\|\nabla ({\bv}_{-(m+1)}-{\bv}_{-m}) \|_{L^4_{[0,T^*_{m+1}]} C^{\delta_1}_x}
		\leq & \| \nabla ({\bv}_{-(m+1)}-{\bv}_{-m}) \|_{ L^\infty_{[0,T^*_{m+1}]} C^{\delta_1}_x}
		\\
		\leq  & C \|{\bw}_{m+1}- {\bw}_{m} \|_{L^\infty_{[0,T^*_{m+1}]} H^{1+2\delta_{1}}_x}
		\\
		\leq & C  2^{-7\delta_{1} m} (1+M_0)^2.
	\end{split}
\end{equation}
Adding \eqref{see99} and \eqref{siq}, it yields
\begin{equation}\label{siw}
	\begin{split}
		& \|\nabla( {h}_{m+1}-{h}_{m} ) \|_{L^4_{[0,T^*_{m+1}]} C^{\delta_{1}}_x}+\| \nabla ({\bv}_{m+1}-{\bv}_{m}) \|_{L^4_{[0,T^*_{m+1}]} C^{\delta_{1}}_x}\leq  C 2^{-7\delta_{1} m} (1+M_0)^2.
	\end{split}
\end{equation}
By using \eqref{Wrs} and \eqref{siw}, we obtain
\begin{equation}\label{sie}
	\begin{split}
		& \|\partial_t({h}_{m+1}-{h}_{m})\|_{L^4_{[0,T^*_{m+1}]} C^{\delta_{1}}_x}+\| \partial_t({\bv}_{m+1}-{\bv}_{m}) \|_{L^2_{[0,T^*_{m+1}]} C^{\delta_{1}}_x}
		\\
		\leq  & \|\nabla({h}_{m+1}-{h}_{m}) , \nabla ({\bv}_{m+1}-{\bv}_{m}) \|_{L^4_{[0,T^*_{m+1}]} C^{\delta_{1}}_x} \cdot (1+ \|\bv_m, h_m\|_{L^\infty_{[0,T^*_{m+1}]} H^s_x} )
		\\
		\leq & C (1+M_0)^3 2^{-6\delta_{1} m}.
	\end{split}
\end{equation}
Due to \eqref{siw} and \eqref{sie}, for $k<m$, we get
\begin{equation}\label{Sia}
	\|P_k d({h}_{m+1}-{h}_{m}), P_k d ({\bv}_{m+1}-{\bv}_{m}) \|_{L^4_{[0,T^*_{m+1}]} L^\infty_x}
	\leq   C2^{-\delta_1k}\cdot   2^{-6\delta_1m}(1+M_0)^3,
\end{equation}
and
\begin{equation*}\label{kz4}
	\|P_k d({h}_{m+1}-{h}_{m}), P_k d ({\bv}_{m+1}-{\bv}_{m}) \|_{L^1_{[0,T^*_{m+1}]} L^\infty_x}
	\leq   C2^{-\delta_1k} \cdot  2^{-6\delta_1m} (1+M_0)^3.
\end{equation*}
\subsubsection{Uniform energy and Strichartz estimates on a regular time-interval $[0,T_{N_0}^*]$.}\label{finalk}
Recall \eqref{DTJ}, then $T_{N_0}^*=2^{-\delta_{1}N_0}( CM_0)^{-3}$. Recall \eqref{yz10} and \eqref{Sia}. Therefore, when $k\geq j$ and $j\geq N_0$, using \eqref{yz10} and H\"older's inequality, we have
\begin{equation}\label{kf01}
	\begin{split}
		\| P_{k} d \bv_j, P_{k} d h_j \|_{L^1_{[0,T^*_j]}L^\infty_x} \leq & (T^*_j)^{\frac34}\| P_{k} d \bv_j, P_{k} dh_j \|_{L^4_{[0,T^*_j]}L^\infty_x}
		\\
		\leq & (T^*_{N_0})^{\frac34} \cdot C  (1+M_0)^3 2^{-\delta_{1}k} 2^{-7\delta_1j}.
	\end{split}
\end{equation}
When $k< j$ and $m\geq N_0$, using \eqref{Sia} and H\"older's inequality, we obtain
\begin{equation}\label{kf02}
	\|P_k d({h}_{m+1}-{h}_{m}), P_k d({\bv}_{m+1}-{\bv}_{m}) \|_{L^1_{[0,T^*_{m+1}]} L^\infty_x}
	\leq    (T^*_{N_0})^{\frac34} \cdot C  (1+M_0)^3 2^{-\delta_{1}k} 2^{-6\delta_1m}.
\end{equation}
On the other side, when $k\geq j$ and $j< N_0$, using \eqref{yz10}, we get
\begin{equation}\label{kf03}
	\begin{split}
		\| P_{k} d \bv_j, P_{k} d h_j \|_{L^1_{[0,T^*_{N_0}]}L^\infty_x} \leq & (T^*_{N_0})^{\frac34} \| P_{k} d \bv_j, P_{k} d h_j \|_{L^4_{[0,T^*_{N_0}]}L^\infty_x}
		\\
		\leq & (T^*_{N_0})^{\frac34} \| P_{k} d \bv_j, P_{k} d h_j \|_{L^4_{[0,T^*_j]}L^\infty_x}
		\\
		\leq & (T^*_{N_0})^{\frac34} \cdot C  (1+M_0)^3 2^{-\delta_{1}k} 2^{-7\delta_1j}.
	\end{split}
\end{equation}
When $k< j$ and $m< N_0$, using \eqref{Sia} and H\"older's inequality, it yields
\begin{equation}\label{kf04}
	\begin{split}
		& \|P_k d({h}_{j+1}-{h}_{j}), P_k d({\bv}_{j+1}-{\bv}_{j}) \|_{L^1_{[0,T^*_{N_0}]} L^\infty_x}
		\\
		\leq & (T^*_{N_0})^{\frac34} \|P_k d({h}_{j+1}-{h}_{j}), P_k d({\bv}_{j+1}-{\bv}_{j}) \|_{L^4_{[0,T^*_{N_0}]} L^\infty_x}
		\\
		\leq & (T^*_{N_0})^{\frac34} \|P_k d({h}_{j+1}-{h}_{j}), P_k d({\bv}_{j+1}-{\bv}_{j}) \|_{L^4_{[0,T^*_{j+1}]} L^\infty_x}
		\\
		\leq    & (T^*_{N_0})^{\frac34} \cdot C  (1+M_0)^3 2^{-\delta_{1}k} 2^{-6\delta_1j}.
	\end{split}
\end{equation}
Due to a different time-interval for the sequence  $(h_j, \bv_j,w_j)$, we will discuss the solutions $(h_j, \bv_j,w_j)$ if  $j \leq N_0$ or  $j \geq N_0+1$ as follows.

\textbf{Case 1: $j \leq N_0$.} In this case, $(h_j, \bv_j, w_j)$ exists on $[0,T_j^*]$, and $[0,T^*_{N_0}]\subseteq [0,T^*_{j}]$. As a result, we don't need to extend solutions $(h_j, \bv_j, w_j)$ if  $j \leq N_0$. Using \eqref{yz4} and H\"older's inequality, we get
\begin{equation*}
	\| d\bv_j, dh_j\|_{L^1_{[0,T^*_{N_0}]}L^\infty_x} \leq (T^*_{N_0})^{\frac34} \| d\bv_j, dh_j\|_{L^4_{[0,T^*_{N_0}]}L^\infty_x} \leq C(T^*_{N_0})^{\frac34} ( 1+ M_0 ).
\end{equation*}
By \eqref{pp8}, this yields
\begin{equation*}\label{ky0}
	\| d\bv_j, dh_j\|_{L^1_{[0,T^*_{N_0}]}L^\infty_x} \leq 2.
\end{equation*}
By using Newton-Leibniz's formula and \eqref{pu00}, it follows that
\begin{equation*}\label{ky1}
	\|\bv_j, h_j\|_{L^\infty_{[0,T^*_{N_0}] \times \mathbb{R}^3}}\leq |\bv_{0j}, h_{0j}|+ \| d\bv_j, dh_j\|_{L^1_{[0,T^*_{N_0}]}L^\infty_x} \leq 2+C_0.
\end{equation*}
Using the energy theorem \ref{DW4}, we get
\begin{equation*}\label{ky2}
	 E(T^*_{N_0})  \leq C (1+M_0)^5 \exp\{ C (1+M_0)^3 \} .
\end{equation*}
\quad \textbf{Case 2: $j \geq N_0+1$.} In this case, we expect to extend the time interval $I_1=[0,T^*_j]$ to $[0,T^*_{N_0}]$. Our idea is to use \eqref{yz10}, \eqref{Sia}, and Theorem \ref{DW4} to calculate the energy at time $T_j^*$. Starting at $T_j^*$, we can also get a new time-interval.

We set
\begin{equation*}\label{ti1}
	I_1=[0,T^*_j]=[t_0,t_1], \quad \quad |I_1|=(CM_0)^{-3}2^{-\delta_{1} j}.
\end{equation*}
By using frequency decomposition, we get
\begin{equation}\label{kz03}
	\begin{split}
		d \bv_j= & \textstyle{\sum}^{\infty}_{k=j} d \bv_j+ \textstyle{\sum}^{j-1}_{k=1}P_k d \bv_j
		\\
		=& P_{\geq j}d \bv_j+ \textstyle{\sum}^{j-1}_{k=1} \textstyle{\sum}_{m=k}^{j-1} P_k  (d\bv_{m+1}-d\bv_m)+ \textstyle{\sum}^{j-1}_{k=1}P_k d \bv_k .
	\end{split}
\end{equation}
Similarly, we also have
\begin{equation}\label{kz04}
	\begin{split}
		d h_j= & P_{\geq j}d h_j+ \textstyle{\sum}^{j-1}_{k=1} \textstyle{\sum}_{m=k}^{j-1} P_k  (d h_{m+1}-d h_m)+ \textstyle{\sum}^{j-1}_{k=1}P_k d h_k.
	\end{split}
\end{equation}
When $k\geq j$, using \eqref{kf01} and \eqref{kf03}, we have\footnote{In the case $j\geq N_0$ we use \eqref{kf01}. In the case $j < N_0$, we take $T^*_j = T^*_{N_0}$ and use \eqref{kf03} to give a bound on $[0,T^*_{N_0}]$.}
\begin{equation}\label{kz01}
	\begin{split}
		\| P_{k} d h_j, P_{k} d \bv_j,  \|_{L^1_{[0, T^*_j]}L^\infty_x} \leq   &   (T^*_{N_0})^{\frac34} \cdot C (1+M_0)^3 2^{-\delta_{1}k} 2^{-7\delta_{1}j}.
	\end{split}
\end{equation}
When $k< j$, using \eqref{kf02} and \eqref{kf04}, it follows\footnote{In the case $m\geq N_0$ we use \eqref{kf02}. In the case $m < N_0$, we take $T^*_{m+1} = T^*_{N_0}$ and use \eqref{kf04} to give a bound on $[0,T^*_{N_0}]$.}
\begin{equation}\label{kz02}
	\|P_k d({h}_{m+1}-{h}_{m}), P_k d({\bv}_{m+1}-{\bv}_{m}) \|_{L^1_{[0,T^*_{m+1}]} L^\infty_x}
	\leq    (T^*_{N_0})^{\frac34} \cdot C  (1+M_0)^3 2^{-\delta_{1}k} 2^{-6\delta_1m}.
\end{equation}
We will use  and \eqref{kz01}-\eqref{kz02} to give a precise analysis of \eqref{kz03}-\eqref{kz04} and get some new time-intervals, and then we try to extend $\rho_j$ from $[0,T^*_j]$ to $[0,T^*_{N_0}]$. Our strategy is as follows. In step 1, we extend it from $[0,T^*_j]$ to $[0,T^*_{j-1}]$. In step 2, we use induction methods to conclude these estimates and also extend it to $[0,T_{N_0}^*]$.

\textbf{Step 1: Extending $[0,T^*_j]$ to $[0,T^*_{j-1}] \ (j \geq N_0+1)$.} To start, referring \eqref{DTJ}, we need to calculate $E(T^*_j)$ for obtaining a length of a new time-interval. Then we shall calculate $\|d h_j, d \bv_j \|_{L^1_{[0,T^*_j]}L^\infty_x}$. Using \eqref{kz03} and \eqref{kz04}, we derive that
\begin{equation*}
	\begin{split}
		\|d \bv_j, d h_j\|_{L^1_{[0,T^*_j]} L^\infty_x }
		\leq & 	\|P_{\geq j}d\bv_j, P_{\geq j}d h_j\|_{L^1_{[0,T^*_j]} L^\infty_x }  + \textstyle{\sum}^{j-1}_{k=1} \|P_k d\bv_k, P_k dh_k\|_{L^1_{[0,T^*_j]}L^\infty_x}\\
		& + \textstyle{\sum}^{j-1}_{k=1} \textstyle{\sum}_{m=k}^{j-1}  \|P_k  (d\bv_{m+1}-d\bv_m), P_k  (dh_{m+1}-dh_m)\|_{L^1_{[0,T^*_j]}L^\infty_x}
		\\
		\leq & 	\|P_{\geq j}d\bv_j, P_{\geq j}d h_j\|_{L^1_{[0,T^*_j]}L^\infty_x} + \textstyle{\sum}^{j-1}_{k=1} \|P_k d\bv_k, P_k dh_k\|_{L^1_{[0,T^*_k]}L^\infty_x}\\
		& + \textstyle{\sum}^{j-1}_{k=1} \textstyle{\sum}_{m=k}^{j-1}\| P_k  (d\bv_{m+1}-d\bv_m), P_k  (dh_{m+1}-dh_m)\|_{L^1_{[0,T^*_{m+1}]}L^\infty_x}
		\\
		\leq & C(1+M_0)^3 (T^*_{N_0})^{\frac34}  \textstyle{\sum}_{k=j}^{\infty} 2^{-\delta_{1} k} 2^{-7\delta_{1} j}
		\\
		& + C(1+M_0)^3 (T^*_{N_0})^{\frac34} \textstyle{\sum}^{j-1}_{k=1} 2^{-\delta_{1} k} 2^{-7\delta_{1} k}
		\\
		& +  C(1+M_0)^3 (T^*_{N_0})^{\frac34} \textstyle{\sum}^{j-1}_{k=1} \textstyle{\sum}_{m=k}^{j-1} 2^{-\delta_{1} k} 2^{-6\delta_{1} m}
		\\
		\leq & C(1+M_0)^3 (T^*_{N_0})^{\frac34} [\frac{1}{3}(1-2^{-\delta_{1}})]^{-2}.
	\end{split}
\end{equation*}
we get
\begin{equation}
	\begin{split}\label{kz05}
		\|d \bv_j, d h_j\|_{L^1_{[0,T^*_j]}L^\infty_x}
		\leq  C(1+M_0)^3(T^*_{N_0})^{\frac34} [\frac{1}{3}(1-2^{-\delta_{1}})]^{-2} \leq 2.
	\end{split}
\end{equation}
By using \eqref{kz05} and \eqref{pu00}, we get
\begin{equation*}
	\begin{split}\label{kp05}
		\| \bv_j,  h_j\|_{L^\infty_{[0,T^*_j]}L^\infty_x} \leq \| \bv_{0j},  h_{0j}\|+	\|d \bv_j, d h_j\|_{L^1_{[0,T^*_j]}L^\infty_x}
		\leq  C_0+2.
	\end{split}
\end{equation*}
By \eqref{kz05}, \eqref{pp8} and Theorem \ref{DW4}, we have
\begin{equation}\label{kz06}
	\begin{split}
		E(T^*_{j}) \leq C (1+M_0)^5 \exp\{ C (1+M_0)^3 \}  =C_*. 
	\end{split}
\end{equation}
Above, $C_*$ is stated in \eqref{Cstar}. Starting at the time $T^*_j$, seeing \eqref{DTJ} and \eqref{kz06}, we can obtain an extending time-interval with a length of $(C_*)^{-3}2^{-\delta_{1} j}$. But, if $T^*_j + (C_*)^{-3}2^{-\delta_{1} j} \geq T^*_{j-1}$, we have finished this step. Else, we need to extend it again.

\textbf{Case 1:} $T^*_j + (C_*)^{-3}2^{-\delta_{1} j} \geq T^*_{j-1} $. In this case, we can get a new interval
\begin{equation*}\label{deI2}
	I_2=[T^*_j, T^*_{j-1}], \quad |I_2|= (2^{\delta_{1}}-1) (CM_0)^{-3}2^{-\delta_{1} j}.
\end{equation*}
Moreover, referring \eqref{kz01} and \eqref{kz02}, we derive
\begin{equation}\label{kz07}
	\| P_{k} d h_j, P_{k} d \bv_j \|_{L^1_{[T^*_j, T^*_{j-1}]}L^\infty_x}
	\leq  (T^*_{N_0})^{\frac34}  \cdot C  (1+C_*)^3 2^{-\delta_{1}k} 2^{-7\delta_{1}j}, \quad k \geq j,
\end{equation}
and $k<j$,
\begin{equation}\label{kz08}
	\|P_k (d{h}_{j}-d{h}_{j-1}), P_k (d{\bv}_{j}-d{\bv}_{j-1}) \|_{L^1_{[T^*_j, T^*_{j-1}]} L^\infty_x}
	\leq   (T^*_{N_0})^{\frac34}  \cdot C  (1+C_*)^3 2^{-\delta_{1}k} 2^{-6\delta_1(j-1)}.
\end{equation}
Using \eqref{pp8} and $j \geq N_0+1$, \eqref{kz07} and \eqref{kz08} yields
\begin{equation}\label{kz09}
	\| P_{k} d h_j, P_{k} d \bv_j \|_{L^1_{[T^*_j, T^*_{j-1}]}L^\infty_x}
	\leq  (T^*_{N_0})^{\frac34}  \cdot C  (1+M_0)^3 2^{-\delta_{1}k} 2^{-6\delta_{1}j}, \quad k \geq j,
\end{equation}
and $k<j$,
\begin{equation}\label{kz10}
	\|P_k (d{h}_{j}-d{h}_{j-1}), P_k (d{\bv}_{j}-d{\bv}_{j-1}) \|_{L^1_{[T^*_j, T^*_{j-1}]} L^\infty_x}
	\leq    (T^*_{N_0})^{\frac34}  \cdot C   (1+M_0)^3 2^{-\delta_{1}k} 2^{-5\delta_1(j-1)}.
\end{equation}
Therefore, we obtain the following estimate
\begin{equation}\label{kz11}
	\begin{split}
		& \|dh_j, d \bv_j\|_{L^1_{[0,T^*_{j-1}]} L^\infty_x}
		\\
		\leq & 	\|P_{\geq j}d h_j, P_{\geq j}d\bv_j \|_{L^1_{[0,T^*_{j-1}]} L^\infty_x}  + \textstyle{\sum}^{j-1}_{k=1} \|P_k dh_k, P_k d\bv_k\|_{L^1_{[0,T^*_{j-1}]} L^\infty_x}
		\\
		& + \textstyle{\sum}^{j-1}_{k=1} \|P_k  (d\bv_{j}-d\bv_{j-1}), P_k  (dh_{j}-dh_{j-1})\|_{L^1_{[0,T^*_{j-1}]} L^\infty_x}
		\\
		& + \textstyle{\sum}^{j-2}_{k=1} \textstyle{\sum}_{m=k}^{j-2}  \|P_k  (d\bv_{m+1}-d\bv_m), P_k  (dh_{m+1}-dh_m)\|_{L^1_{[0,T^*_{j-1}]} L^\infty_x}.
	\end{split}
\end{equation}
Due to \eqref{kz01} and \eqref{kz09}, it yields
\begin{equation}\label{kz12}
	\begin{split}
		\|P_{\geq j}dh_j, P_{\geq j}d\bv_j\|_{L^1_{[0,T^*_{j-1}]} L^\infty_x}
		\leq & C   (1+M_0)^3 (T^*_{N_0})^{\frac34}  \textstyle{\sum}^{\infty}_{k=j}2^{-\delta_{1}k} (2^{-7\delta_1 j}+ 2^{-6\delta_{1} j})
		\\
		\leq & C   (1+M_0)^3  (T^*_{N_0})^{\frac34}  \textstyle{\sum}^{\infty}_{k=j} 2^{-\delta_{1}k} 2^{-6\delta_{1} j}\times 2.
	\end{split}
\end{equation}
Due to \eqref{kz02} and \eqref{kz10}, it follows
\begin{equation}\label{kz13}
	\begin{split}
		& \textstyle{\sum}^{j-1}_{k=1} \|P_k  (d\bv_{j}-d\bv_{j-1}), P_k  (dh_{j}-dh_{j-1})\|_{L^1_{[0,T^*_{j-1}]} L^\infty_x}
		\\
		\leq & C   (1+M_0)^3 (T^*_{N_0})^{\frac34}  \textstyle{\sum}^{j-1}_{k=1} 2^{-\delta_{1}k} (2^{-6\delta_1 (j-1)}+ 2^{-5\delta_{1} (j-1)})
		\\
		\leq & C   (1+M_0)^3 (T^*_{N_0})^{\frac34}  \textstyle{\sum}^{j-1}_{k=1} 2^{-\delta_{1}k} 2^{-5\delta_{1} (j-1)} \times 2.
	\end{split}
\end{equation}
Inserting \eqref{kz12}-\eqref{kz13} into \eqref{kz11}, we derive that
\begin{equation}\label{kz14}
	\begin{split}
		\|d \bv_j, d h_j\|_{L^1_{[0,T^*_{j-1}]} L^\infty_x}
		\leq & 	C (1+M_0)^3 (T^*_{N_0})^{\frac34}   \textstyle{\sum}_{k=j}^{\infty} 2^{-\delta_{1} k} 2^{-6\delta_{1} j}\times 2
		\\
		& + C    (1+M_0)^3 (T^*_{N_0})^{\frac34}  \textstyle{\sum}^{j-1}_{k=1} 2^{-\delta_{1}k} 2^{-5\delta_{1} (j-1)} \times 2
		\\
		& + C (1+M_0)^3 (T^*_{N_0})^{\frac34}  \textstyle{\sum}^{j-1}_{k=1} 2^{-\delta_{1} k} 2^{-7\delta_{1} k}
		\\
		& +  C (1+M_0)^3 (T^*_{N_0})^{\frac34}  \textstyle{\sum}^{j-2}_{k=1} \textstyle{\sum}_{m=k}^{j-2} 2^{-\delta_{1} k} 2^{-6\delta_{1} m}
		\\
		\leq & 	C (1+M_0)^3 (T^*_{N_0})^{\frac34}  [\frac13(1-2^{-\delta_{1}})]^{-2}.
	\end{split}
\end{equation}
By using \eqref{kz14}, \eqref{pp8} and Theorem \ref{DW4}, we can establish
\begin{equation}\label{kz15}
	\begin{split}
		&|\bv_j, h_j| \leq 2+C_0,
		\\
		& E(T^*_{j-1}) \leq C_*.
	\end{split}
\end{equation}
\quad \textbf{Case 2:} $T^*_j + (C_*)^{-3}2^{-\delta_{1} j} < T^*_{j-1} $. In this case, we record
\begin{equation*}
	I_2=[T^*_j, t_2], \quad |I_2| = (C_*)^{-3}2^{-\delta_{1} j}.
\end{equation*}
Referring \eqref{kz01} and \eqref{kz02}, we can derive
\begin{equation}\label{kz16}
	\| P_{k} d \bv_j, P_{k} d h_j \|_{L^1_{I_2}L^\infty_x}
	\leq  (T^*_{N_0})^{\frac34} \cdot C  (1+C_*)^3 2^{-\delta_{1}k} 2^{-7\delta_{1}j}, \quad k \geq j,
\end{equation}
and
\begin{equation}\label{kz17}
	\|P_k (dh_{j}-dh_{j-1}), P_k  (d{\bv}_{j}-d{\bv}_{j-1}) \|_{L^1_{I_2} L^\infty_x}
	\leq    (T^*_{N_0})^{\frac34} \cdot C  (1+C_*)^3 2^{-\delta_{1}k} 2^{-6\delta_1(j-1)}, \quad k<j.
\end{equation}
Similarly, on $I_1 \cup I_2$, we have
\begin{equation*}
	\begin{split}
		\|d\bv_j, dh_j\|_{L^1_{ I_1 \cup I_2 } L^\infty_x }
		\leq & 	\|P_{\geq j}d\bv_j, P_{\geq j}d h_j\|_{L^1_{I_1 \cup I_2} L^\infty_x }   + \textstyle{\sum}^{j-1}_{k=1} \|P_k d\bv_k, P_k dh_k\|_{L^1_{[0,T^*_k]} L^\infty_x }
		\\
		& + \textstyle{\sum}^{j-1}_{k=1} \|P_k  (d\bv_{j}-d\bv_{j-1}), P_k  (dh_{j}-dh_{j-1})\|_{L^1_{I_1 \cup I_2 } L^\infty_x }
		\\
		& + \textstyle{\sum}^{j-2}_{k=1} \textstyle{\sum}_{m=k}^{j-2}  \|P_k  (d\bv_{m+1}-d\bv_m), P_k  (dh_{m+1}-dh_m)\|_{L^1_{I_1 \cup I_2} L^\infty_x } .
	\end{split}
\end{equation*}
Noting that $ I_1 \cup I_2 \subseteq T^*_k$ when $k \leq j-1$, we therefore obtain
\begin{equation}\label{kz11A}
	\begin{split}
		& \|d \bv_j, d h_j\|_{L^1_{ I_1 \cup I_2 } L^\infty_x }
		\\
		\leq & 	\|P_{\geq j}d\bv_j, P_{\geq j}d h_j\|_{L^1_{I_1 \cup I_2} L^\infty_x }   + \textstyle{\sum}^{j-1}_{k=1} \|P_k d\bv_k, P_k dh_k\|_{L^1_{I_1 \cup I_2 } L^\infty_x }
		\\
		& + \textstyle{\sum}^{j-1}_{k=1} \|P_k  (d \bv_{j}-d \bv_{j-1}), P_k  (dh_{j}-dh_{j-1})\|_{L^1_{I_1 \cup I_2 } L^\infty_x }
		\\
		& + \textstyle{\sum}^{j-2}_{k=1} \textstyle{\sum}_{m=k}^{j-2}  \|P_k  (d\bv_{m+1}-d\bv_m), P_k  (dh_{m+1}-dh_m)\|_{L^1_{[0,T^*_{m+1}]} L^\infty_x }
	\end{split}
\end{equation}
Inserting \eqref{kz01}, \eqref{kz02} and \eqref{kz16} and \eqref{kz17} to \eqref{kz11A}, it follows that
\begin{equation}\label{kz11a}
	\begin{split}		
		\|d \bv_j, d h_j\|_{L^1_{ I_1 \cup I_2 } L^\infty_x }
		\leq & 	C (1+M_0)^3 (T^*_{N_0})^{\frac34} \textstyle{\sum}_{k=j}^{\infty} 2^{-\delta_{1} k} 2^{-6\delta_{1} j}\times 2
		\\
		& + C    (1+M_0)^3 (T^*_{N_0})^{\frac34} \textstyle{\sum}^{j-1}_{k=1} 2^{-\delta_{1}k} 2^{-5\delta_{1} (j-1)} \times 2
		\\
		& + C (1+M_0)^3 (T^*_{N_0})^{\frac34} \textstyle{\sum}^{j-1}_{k=1} 2^{-\delta_{1} k} 2^{-7\delta_{1} k}
		\\
		& +  C (1+M_0)^3 (T^*_{N_0})^{\frac34} \textstyle{\sum}^{j-2}_{k=1} \textstyle{\sum}_{m=k}^{j-2} 2^{-\delta_{1} k} 2^{-6\delta_{1} m}
		\\
		\leq & 	C (1+M_0)^3 (T^*_{N_0})^{\frac34} [\frac13(1-2^{-\delta_{1}})]^{-2}.
	\end{split}
\end{equation}
Applying \eqref{kz11a} and Theorem \ref{DW4}, we get
\begin{equation*}\label{kz15f}
	\begin{split}
		& |\bv_j,h_j| \leq 2+C_0,
		\\
		& E(t_2) \leq C_*.
	\end{split}
\end{equation*}
Therefore, we can repeat the process with a length with $(C_*)^{-1}2^{-\delta_{1} j}$ till extending it to $T^*_{j-1}$. Moreover, on every new time-interval with $(C_*)^{-1}2^{-\delta_{1} j}$, the estimates \eqref{kz16} and \eqref{kz17} hold. Set
\begin{equation}\label{times1}
	X_1= \frac{T^*_{j-1}-T^*_j}{C^{-3}_*2^{-\delta_{1} j}}= (2^{\delta_{1}}-1)C^3_* (CM_0)^{-3}.
\end{equation}
Then we need a maximum of $X_1$-times to reach the time $T^*_{j-1}$ both in case 2 (it's also adapt to case 1 for calculating the times). As a result, we can calculate
\begin{equation*}
	\begin{split}
		\|d \bv_j, d h_j\|_{L^1_{[0,T^*_{j-1}]} L^\infty_x}
		\leq & 	\|P_{\geq j}d\bv_j, P_{\geq j}d h_j\|_{L^1_{[0,T^*_{j-1}]} L^\infty_x}
		+ \textstyle{\sum}^{j-1}_{k=1} \|P_k d\bv_k, P_k dh_k\|_{L^1_{[0,T^*_{k}]} L^\infty_x}
		\\
		+ & \textstyle{\sum}^{j-1}_{k=1} \|P_k  (d\bv_{j}-d\bv_{j-1}), P_k  (dh_{j}-dh_{j-1})\|_{L^1_{[0,T^*_{j-1}]} L^\infty_x}.
		\\
		& + \textstyle{\sum}^{j-2}_{k=1} \textstyle{\sum}_{m=k}^{j-2}  \|P_k  (d\bv_{m+1}-d\bv_m), P_k  (dh_{m+1}-dh_m)\|_{L^1_{[0,T^*_{m}]} L^\infty_x}
	\end{split}
\end{equation*}	
Due to \eqref{kz01}, \eqref{kz02}, \eqref{kz16}, \eqref{kz17}, and \eqref{pp8}, we obtain
\begin{equation}\label{kzqt}
	\begin{split}
		& \|d \bv_j, d h_j\|_{L^1_{[0,T^*_{j-1}]} L^\infty_x}
		\\
		\leq & C (1+CM_0)^3 (T^*_{N_0})^{\frac34} \textstyle{\sum}_{k=j}^{\infty} 2^{-\delta_{1} k} 2^{-6\delta_{1} j}\times (2^{\delta_{1}}-1)C^3_* (CM_0)^{-3}
		\\
		& + C    (1+CM_0)^3 (T^*_{N_0})^{\frac34} \textstyle{\sum}^{j-1}_{k=1} 2^{-\delta_{1}k} 2^{-5\delta_{1} (j-1)} \times (2^{\delta_{1}}-1)C^3_* (CM_0)^{-3}
		\\
		& + C (1+CM_0)^3(T^*_{N_0})^{\frac34} \textstyle{\sum}^{j-1}_{k=1} 2^{-\delta_{1} k} 2^{-7\delta_{1} k}
		\\
		& +  C (1+CM_0)^3(T^*_{N_0})^{\frac34} \textstyle{\sum}^{j-2}_{k=1} \textstyle{\sum}_{m=k}^{j-2} 2^{-\delta_{1} k} 2^{-6\delta_{1} m}
		\\
		\leq & 	C (1+M_0)^3 (T^*_{N_0})^{\frac34}[\frac13(1-2^{-\delta_{1}})]^{-2}.
	\end{split}
\end{equation}
Therefore, by using \eqref{kzqt} and Theorem \eqref{DW4}, we get
\begin{equation}\label{kz270}
	\begin{split}
		& \| \bv_j, h_j \|_{L^\infty_{ [0,T^*_{j-1}]\times \mathbb{R}^3}} \leq 2+C_0,
		\\
		& E(T^*_{j-1}) \leq C_*.
	\end{split}
\end{equation}
At this stage, both in case 1 or case 2, seeing from \eqref{kz270}, \eqref{kzqt}, \eqref{kz14}, \eqref{kz15}, through a maximum of $X_1=(2^{\delta_{1}}-1)C^3_* (CM_0)^{-3}$ times with each length $C_*^{-3}2^{-\delta_{1}j}$ or $(2^{\delta_{1}}-1)(CM_0)^{-3}2^{-\delta_{1} j}$, we can extend the solutions $(h_j,\bv_j,\bw_j)$ from $[0,T^*_j]$ to $[0,T^*_{j-1}]$, and
\begin{equation}\label{kz27}
	\begin{split}
		& \| \bv_j, h_j \|_{L^\infty_{ [0,T^*_{j-1}]\times \mathbb{R}^2}} \leq 2+C_0, \quad E(T^*_{j-1}) \leq C_*,
		\\
		& \|d \bv_j, d h_j\|_{L^1_{[0,T^*_{j-1}]} L^\infty_x}
		\leq  	C(1+M_0)^3 (T^*_{N_0})^{\frac34}[\frac13(1-2^{-\delta_{1}})]^{-2}.
	\end{split}		
\end{equation}
From \eqref{kz27}, we have extended the solutions $(h_m,\bv_m,\bw_m)$ ($m\in[N_0, j-1]$) from $[0,T^*_m]$ to $[0,T^*_{m-1}]$. Moreover, referring \eqref{kz01} and \eqref{kz02}, \eqref{kz27}, and \eqref{times1}, we get
\begin{equation}\label{kz29}
	\begin{split}
		& \| P_{k} d \bv_m, P_{k} d h_m \|_{L^1_{[0,T^*_{m-1}]}L^\infty_x}
		\\
		\leq  & \| P_{k} d\bv_m, P_{k} d h_m \|_{L^1_{[0,T^*_{m}]}L^\infty_x}
		\\
		& +\| P_{k} d \bv_m, P_{k} d h_m \|_{L^1_{[T^*_{m},T^*_{m-1}]}L^\infty_x}
		\\
		\leq  & C(1+CM_0)^3(T^*_{N_0})^{\frac34} 2^{-\delta_{1}k} 2^{-7\delta_{1}m} + C(1+C_*)^3 2^{-\delta_{1}k} 2^{-7\delta_{1}m} \times (2^{\delta_{1}}-1)C^3_* (CM_0)^{-3}
		\\
		\leq  & C(1+M_0)^3(T^*_{N_0})^{\frac34} 2^{-\delta_{1}k} 2^{-7\delta_{1}m} + C(1+M_0)^3 2^{-\delta_{1}k} 2^{-6\delta_{1}m} \times (2^{\delta_{1}}-1).
	\end{split}	
\end{equation}
For $k \geq m\geq N_0+1$, due to \eqref{pp8} and \eqref{kz29}, it yields
\begin{equation}\label{kz31}
	\begin{split}
		\| P_{k} d \bv_m, P_{k} d h_m \|_{L^1_{[0,T^*_{m-1}]}L^\infty_x}
		\leq & C(1+M_0)^3(T^*_{N_0})^{\frac34} 2^{-\delta_{1}k} 2^{-5\delta_{1}m} \times 2^{\delta_{1}},
	\end{split}	
\end{equation}
Similarly, if $m\geq N_0+1$, using \eqref{pp8}, for $k<j$, the following estimate holds:
\begin{equation}\label{kz34}
	\begin{split}
		& \|P_k (d{h}_{m}-d{h}_{m-1}), P_k (d{\bv}_{m}-d{\bv}_{m-1}) \|_{L^1_{[0,T^*_{m-1}]} L^\infty_x}
		\\
		\leq    &  \|P_k (d{h}_{m}-d{h}_{m-1}), P_k  (d{\bv}_{m}-d{\bv}_{m-1}) \|_{L^1_{[0,T^*_{m}]} L^\infty_x}
		\\
		& +\|P_k (d {h}_{m}-d {h}_{m-1}), P_k (d {\bv}_{m}-d {\bv}_{m-1}) \|_{L^1_{[T^*_m,T^*_{m-1}]} L^\infty_x}
		\\
		\leq & C  (1+CM_0)^3 (T^*_{N_0})^{\frac34} 2^{-\delta_{1}k} 2^{-6\delta_1(m-1)}
		\\
		& + C  (1+C_*)^3 2^{-\delta_{1}k} 2^{-6\delta_1(m-1)} \times (2^{\delta_{1}}-1)C^3_* (CM_0)^{-3}
		\\
		\leq & C(1+M_0)^3(T^*_{N_0})^{\frac34} 2^{-\delta_{1}k} 2^{-5\delta_{1}(m-1)} \times 2^{\delta_{1}}.
	\end{split}	
\end{equation}
\quad \textbf{Step 2: Extending time interval $[0,T^*_j]$ to $[0,T^*_{N_0}]$}. Based the above analysis in Step 1, we can give a induction by achieving the goal. We assume the solutions $(h_j, \bv_j, w_j)$ can be extended from $[0,T^*_j]$ to $[0,T^*_{j-l}]$ through a maximam $X_l$ times and
\begin{equation}\label{kz41}
	\begin{split}
		X_l=& \frac{T_{j-l}^*- T_{j}^*}{C^{-3}_* 2^{-\delta_{1} j}}
		\\
		=& \frac{(CM_0)^{-1}( 2^{-\delta_{1}(j-l)} - 2^{-\delta_{1}j} )}{C^{-3}_* 2^{-\delta_{1} j}}
		\\
		=& \frac{C^3_*}{(CM_0)^3} (2^{\delta_{1}l}-1).
	\end{split}	
\end{equation}
Moreover, the following bounds
\begin{equation}\label{kz42}
	\begin{split}
		\| P_{k} d \bv_j, P_{k} d h_j \|_{L^1_{[0,T^*_{j-l}]}L^\infty_x}
		\leq & C(1+M_0)^3 (T^*_{N_0})^{\frac34} 2^{-\delta_{1}k} 2^{-5\delta_{1}j} \times 2^{\delta_{1}l}, \qquad k \geq j,
	\end{split}	
\end{equation}
and
\begin{equation}\label{kz43}
	\begin{split}
		& \|P_k (d{h}_{m+1}-d{h}_{m}), P_k (d{\bv}_{m+1}-d{\bv}_{m}) \|_{L^1_{[0,T^*_{m-l}]} L^\infty_x}
		\\
		\leq  & C(1+M_0)^3 (T^*_{N_0})^{\frac34} 2^{-\delta_{1}k} 2^{-5\delta_{1}m } \times 2^{\delta_{1}l}, \quad k<j,
	\end{split}	
\end{equation}
and
\begin{equation}\label{kz44}
	\|d \bv_j, d h_j\|_{L^1_{[0,T^*_{j-l}]} L^\infty_x}
	\leq  	C(1+M_0)^3 (T^*_{N_0})^{\frac34}[\frac13(1-2^{-\delta_{1}})]^{-2}.
\end{equation}
and
\begin{equation}\label{kz45}
	|\bv_j, h_j| \leq 2+C_0, \quad E(T^*_{j-l}) \leq C_*.
\end{equation}
In the following, we will check the estimates \eqref{kz41}, \eqref{kz42}, \eqref{kz43}, \eqref{kz44}, and \eqref{kz45} hold when $l=1$, and it also holds when we replace $l$ by $l+1$.

Using \eqref{kz27}, \eqref{kz31}, \eqref{times1}, and \eqref{kz34}, then \eqref{kz42}-\eqref{kz45} hold by taking $l=1$. Let us now check it for $l+1$. In this case, it implies that  $T^*_{j-l} \leq T^*_{N_0}$. Therefore, $j-l\geq N_0+1$ should hold. Starting at the time $T^*_{j-l}$, seeing \eqref{DTJ} and \eqref{kz45}, we shall get an extending time-interval of $(h_j, \bv_j,w_j)$ with a length of $(C_*)^{-3}2^{-\delta_{1} j}$. Hence, we can go to the case 2 in step 1, and the length every new time-interval is $C_*^{-3} 2^{-\delta_{1} j}$. Therefore, the times is
\begin{equation}\label{kz46}
	X= \frac{T^*_{j-(l+1)}-T^*_{j-l}}{C^{-3}_*2^{-\delta_{1} j}}= 2^{\delta_{1} l}(2^{\delta_{1}}-1)C^3_* (CM_0)^{-3}.
\end{equation}
Thus, we can deduce that
\begin{equation}\label{kz40}
	X_{l+1}=X_l+X= (2^{\delta_{1}(l+1)}-1)C^3_* (CM_0)^{-3}.
\end{equation}
Moreover, for $k \geq j$, we have
\begin{equation*}\label{kz48}
	\begin{split}
		\| P_{k} d \bv_j, P_{k} dh_j \|_{L^1_{[0, T^*_{j-(l+1)}]}L^\infty_x}\leq & \| P_{k} d \bv_j, P_{k} dh_j \|_{L^1_{[0,T^*_{j-l}]}L^\infty_x}
		\\
		& +	\| P_{k} d \bv_j, P_{k} dh_j \|_{L^1_{[T^*_{j-l}, T^*_{j-(l+1)}]}L^\infty_x}.
	\end{split}
\end{equation*}
Using \eqref{kz01} and \eqref{kz45}
\begin{equation}\label{kz49}
	\begin{split}
		& \| P_{k} d \bv_j, P_{k} dh_j \|_{L^1_{[T^*_{j-l}, T^*_{j-(l+1)}]}L^\infty_x}
		\\
		\leq   & C  (1+C_*)^3(T^*_{N_0})^{\frac34} 2^{-\delta_{1}k} 2^{-7\delta_{1}j}\times 2^{\delta_{1} l}(2^{\delta_{1}}-1)C^3_* (CM_0)^{-3}, \quad k\geq j.
	\end{split}
\end{equation}
Due to \eqref{pp8}, it yields
\begin{equation*}
	\begin{split}
		(1+C_*)^3 C^3_* (CM_0)^{-3} (1+CM_0)^{-3}2^{-\delta_{1} N_0} \leq 1.
	\end{split}
\end{equation*}
Hence, from \eqref{kz49} we have
\begin{equation}\label{kz50}
	\begin{split}
		& \| P_{k} d \bv_j, P_{k} dh_j \|_{L^1_{[T^*_{j-l}, T^*_{j-(l+1)}]}L^\infty_x}
		\\
		\leq   & C  (1+M_0)^3(T^*_{N_0})^{\frac34} 2^{-\delta_{1}k} 2^{-5\delta_{1}j}\times 2^{\delta_{1} l}(2^{\delta_{1}}-1), \quad k\geq j.
	\end{split}
\end{equation}
Using \eqref{kz42} and \eqref{kz50}, so we get
\begin{equation}\label{kz51}
	\begin{split}
		\| P_{k} d \bv_j, P_{k} d h_j \|_{L^1_{[0,T^*_{j-(l+1)}]}L^\infty_x}
		\leq & C(1+M_0)^3 (T^*_{N_0})^{\frac34} 2^{-\delta_{1}k} 2^{-5\delta_{1}j} \times 2^{\delta_{1}(l+1)}, \qquad k \geq j.
	\end{split}	
\end{equation}
If $k<j$, we can derive
\begin{equation}\label{kz52}
	\begin{split}
		& \|P_k (d{h}_{m+1}-d{h}_{m}), P_k (d{\bv}_{m+1}-d{\bv}_{m}) \|_{L^1_{[0,T^*_{m-(l+1)}]} L^\infty_x}
		\\
		\leq  & \|P_k (d{h}_{m+1}-d{h}_{m}), P_k (d{\bv}_{m+1}-d{\bv}_{m}) \|_{L^1_{[0,T^*_{m-l}]} L^\infty_x}
		\\
		& + \|P_k (d{h}_{m+1}-d{h}_{m}), P_k (d{\bv}_{m+1}-d{\bv}_{m}) \|_{L^1_{[T^*_{m-l},T^*_{m-(l+1)}]} L^\infty_x}.
	\end{split}	
\end{equation}
When we extend the solutions $(h_j, \bv_j, w_j)$ from $[0,T^*_{j-l}]$ to $[0,T^*_{j-(l+1)}]$, then the solutions $(h_m, \bv_m, w_m)$ is also extended from $[0,T^*_{m-l}]$ to $[0,T^*_{m-(l+1)}]$. Seeing \eqref{kz02} and \eqref{kz45}, we can obtain
\begin{equation}\label{kz53}
	\begin{split}
		& \|P_k (d{h}_{m+1}-d{h}_{m}), P_k (d{\bv}_{m+1}-d{\bv}_{m}) \|_{L^1_{[T^*_{m-l},T^*_{m-(l+1)}]} L^\infty_x}
		\\
		\leq  & C(1+C_*)^3 (T^*_{N_0})^{\frac34} 2^{-\delta_{1}k} 2^{-6\delta_{1}m } \times 2^{\delta_{1} l}(2^{\delta_{1}}-1)C^3_* (CM_0)^{-3}.
	\end{split}	
\end{equation}
Hence, for $m-l \geq N_0+1$, using \eqref{pp8}, it yields
\begin{equation*}
	(1+C_*)^3  C^3_* (CM_0)^{-3} (1+CM_0)^{-3} 2^{-\delta_{1}N_0 } \leq 1.
\end{equation*}
Based on the above results, \eqref{kz53} becomes
\begin{equation}\label{kz54}
	\begin{split}
		& \|P_k (d{h}_{m+1}-d{h}_{m}), P_k (d{\bv}_{m+1}-d{\bv}_{m}) \|_{L^1_{[T^*_{m-l},T^*_{m-(l+1)}]} L^\infty_x}
		\\
		\leq  & C(1+M_0)^3(T^*_{N_0})^{\frac34} 2^{-\delta_{1}k} 2^{-5\delta_{1}m } \times 2^{\delta_{1} l}(2^{\delta_{1}}-1).
	\end{split}	
\end{equation}
Inserting \eqref{kz43} and \eqref{kz54} to \eqref{kz52}, we have
\begin{equation}\label{kz55}
	\begin{split}
		& \|P_k (d{h}_{m+1}-d{h}_{m}), P_k (d{\bv}_{m+1}-d{\bv}_{m}) \|_{L^1_{[0,T^*_{m-(l+1)}]} L^\infty_x}
		\\
		\leq  & C(1+M_0)^3(T^*_{N_0})^{\frac34} 2^{-\delta_{1}k} 2^{-5\delta_{1}m } \times 2^{\delta_{1}(l+1)}, \quad k<j,
	\end{split}	
\end{equation}
Next, we can bound 
\begin{equation}\label{kz56}
	\begin{split}
		& \|d \bv_j, d h_j\|_{L^1_{[0,T^*_{j-(l+1)}]} L^\infty_x}
		\\
		\leq & \|P_{\geq j}d\bv_j, P_{\geq j}d h_j\|_{L^1_{[0,T^*_{j-(l+1)}]} L^\infty_x}
		\\
		& + \textstyle{\sum}^{j-1}_{k=j-l}\textstyle{\sum}^{j-1}_{m=k} \|P_k  (d\bv_{m+1}-d\bv_{m}), P_k  (dh_{m+1}-dh_{m})\|_{L^1_{[0,T^*_{j-(l+1)}]} L^\infty_x}
		\\
		& + \textstyle{\sum}^{j-1}_{k=1} \|P_k d\bv_k, P_k dh_k\|_{L^1_{[0,T^*_{j-(l+1)}]} L^\infty_x}
		\\
		=& 	\|P_{\geq j}d\bv_j, P_{\geq j}d h_j\|_{L^1_{[0,T^*_{j-(l+1)}]} L^\infty_x}
		\\
		& + \textstyle{\sum}^{j-1}_{k=j-l} \|P_k d\bv_k, P_k dh_k\|_{L^1_{[0,T^*_{j-(l+1)}]} L^\infty_x}
		+ \textstyle{\sum}^{j-(l+1)}_{k=1} \|P_k d\bv_k, P_k dh_k\|_{L^1_{[0,T^*_{j-(l+1)}]} L^\infty_x}
		\\
		& + \textstyle{\sum}^{j-1}_{k=1} \textstyle{\sum}^{j-1}_{m=j-(l+1)} \|P_k  (d\bv_{m+1}-d\bv_{m}), P_k  (dh_{m+1}-dh_{m})\|_{L^1_{[0,T^*_{j-(l+1)}]} L^\infty_x}
		\\
		& + \textstyle{\sum}^{j-(l+2)}_{k=1} \textstyle{\sum}_{m=k}^{j-(l+2)}  \|P_k  (d\bv_{m+1}-d\bv_m), P_k  (dh_{m+1}-dh_m)\|_{L^1_{[0,T^*_{j-(l+1)}]} L^\infty_x}
		\\
		=& R_1+  R_2+ R_3+ R_4+ R_5,
	\end{split}
\end{equation}
where
\begin{equation}\label{Theta}
	\begin{split}
		R_1= &  \|P_{\geq j}d\bv_j, P_{\geq j}d h_j\|_{L^1_{[0,T^*_{j-(l+1)}]} L^\infty_x} ,
		\\
		R_2= & \textstyle{\sum}^{j-(l+2)}_{k=1}\textstyle{\sum}^{j-(l+2)}_{m=k} \|P_k  (d\bv_{m+1}-d\bv_{m}), P_k  (dh_{m+1}-dh_{m})\|_{L^1_{[0,T^*_{j-(l+1)}]} L^\infty_x},
		\\
		R_3= & \textstyle{\sum}^{j-1}_{k=1}\textstyle{\sum}^{j-1}_{m=j-(l+1)} \|P_k  (d\bv_{m+1}-d\bv_{m}), P_k  (dh_{m+1}-dh_{m})\|_{L^1_{[0,T^*_{j-(l+1)}]} L^\infty_x},
		\\
		R_4 =& \textstyle{\sum}^{j-(l+1)}_{k=1}	\|P_{k}d\bv_k, P_{k}dh_k\|_{L^1_{[0,T^*_{j-(l+1)}]} L^\infty_x},
		\\
		R_5 =& \textstyle{\sum}^{j-1}_{k=j-l}	\|P_{k}d\bv_k, P_{k}dh_k\|_{L^1_{[0,T^*_{j-(l+1)}]} L^\infty_x}.
	\end{split}
\end{equation}
On time-interval $[0,T^*_{j-(l+1)}]$, we note that there is no growth for $R_2$ and $R_4$ in this extending process. For example, considering $R_2$, the existing time-interval of $P_k  (d\bv_{m+1}-d\bv_{m})$ is actually $[0,T^*_{m+1}]$, and $[0,T^*_{j-(l+1)}] \subseteq [0,T^*_{m+1}]$ if $m \geq j-(l+2)$. Therefore, we can use the bounds \eqref{kz01} and \eqref{kz02} to handle $R_2$ and $R_4$. While, considering $R_1, R_3$, and $R_5$, we need to calculate the growth in Strichartz estimates. Based on this idea, let us give a precise analysis on \eqref{Theta}.

According to \eqref{kz51}, we can estimate $R_1$ by
\begin{equation}\label{kz57}
	\begin{split}
		R_1 \leq & C(1+CM_0)^3(T^*_{N_0})^{\frac34} \textstyle{\sum}^{\infty}_{k=j} 2^{-\delta_{1}k} 2^{-5\delta_{1}j} \times 2^{\delta_{1}(l+1)}.
	\end{split}
\end{equation}
Due to \eqref{kz01}, we have
\begin{equation}\label{kz58}
	\begin{split}
		R_2\leq &	\textstyle{\sum}^{j-(l+2)}_{k=1}\textstyle{\sum}^{j-(l+2)}_{m=k} \|P_k  (d\bv_{m+1}-d\bv_{m}), P_k  (dh_{m+1}-dh_{m})\|_{L^1_{[0,T^*_{m+1}]} L^\infty_x},
		\\
		\leq &  C(1+CM_0)^3 (T^*_{N_0})^{\frac34} \textstyle{\sum}^{j-(l+2)}_{k=1}\textstyle{\sum}^{j-(l+2)}_{m=k} 2^{-\delta_{1}k} 2^{-6\delta_{1}m}.
	\end{split}	
\end{equation}
For $1 \leq k \leq j-1$ and $m \leq j-1$, using \eqref{kz55}, it follows
\begin{equation}\label{kz59}
	\begin{split}
		R_3\leq & \textstyle{\sum}^{j-1}_{k=1}\textstyle{\sum}^{j-1}_{m=j-(l+1)} \|P_k  (d\bv_{m+1}-d\bv_{m}), P_k  (dh_{m+1}-dh_{m})\|_{L^1_{[0,T^*_{m-(l+1)}]} L^\infty_x}
		\\
		\leq &  C  (1+CM_0)^3 (T^*_{N_0})^{\frac34} \textstyle{\sum}^{j-1}_{k=1}\textstyle{\sum}^{j-1}_{m=j-(l+1)}  2^{-\delta_{1}k} 2^{-5\delta_{1}m} \times 2^{\delta_{1}(l+1)}.
	\end{split}
\end{equation}
Due to \eqref{kz01}, it yields
\begin{equation}\label{kz60}
	\begin{split}
		R_4 \leq & \textstyle{\sum}^{j-(l+1)}_{k=1}	\|P_{k}d\bv_k, P_{k}dh_k\|_{L^1_{[0,T^*_{k}]} L^\infty_x}
		\\
		\leq & C  (1+CM_0)^3 (T^*_{N_0})^{\frac34} \textstyle{\sum}^{j-(l+1)}_{k=1} 2^{-\delta_{1}k} 2^{-7\delta_{1}k}.
	\end{split}
\end{equation}
If $j-l \leq k\leq j-1$, then $k+l+1-j \leq l$. By using \eqref{kz42}, we can estimate
\begin{equation}\label{kz61}
	\begin{split}
		R_5 =& \textstyle{\sum}^{j-1}_{k=j-l}	\|P_{k}d\bv_k, P_{k}dh_k\|_{L^1_{[0,T^*_{j-(l+1)}]} L^\infty_x}
		\\
		= & \textstyle{\sum}^{j-1}_{k=j-l}	\|P_{k}d\bv_k, P_{k}dh_k\|_{L^1_{[0,T^*_{k-(k+l+1-j)}]} L^\infty_x}
		\\
		\leq	& C  (1+CM_0)^3 (T^*_{N_0})^{\frac34} \textstyle{\sum}^{j-1}_{k=j-l} 2^{-\delta_{1}k} 2^{-5\delta_{1}j}2^{\delta_{1}(k+l+1-j)}.
	\end{split}
\end{equation}
Inserting \eqref{kz57}, \eqref{kz58}, \eqref{kz59}, \eqref{kz60}, \eqref{kz61} to \eqref{kz56}, it follows
\begin{equation}\label{kz62}
	\begin{split}
		& \|d \bv_j, d h_j\|_{L^1_{[0,T^*_{j-(l+1)}]} L^\infty_x}
		\\
		\leq & C(1+CM_0)^3 (T^*_{N_0})^{\frac34} \textstyle{\sum}^{\infty}_{k=j} 2^{-\delta_{1}k} 2^{-5\delta_{1}j} \times 2^{\delta_{1}(l+1)}
		\\
		& + C(1+CM_0)^3(T^*_{N_0})^{\frac34} \textstyle{\sum}^{j-(l+2)}_{k=1}\textstyle{\sum}^{j-(l+2)}_{m=k} 2^{-\delta_{1}k} 2^{-6\delta_{1}m}
		\\
		&+C  (1+CM_0)^3(T^*_{N_0})^{\frac34} \textstyle{\sum}^{j-1}_{k=1}\textstyle{\sum}^{j-1}_{m=j-(l+1)}  2^{-\delta_{1}k} 2^{-5\delta_{1}m} \times 2^{\delta_{1}(l+1)}
		\\
		& + C  (1+CM_0)^3(T^*_{N_0})^{\frac34} \textstyle{\sum}^{j-(l+1)}_{k=1} 2^{-\delta_{1}k} 2^{-7\delta_{1}k}
		\\
		& + C  (1+CM_0)^3(T^*_{N_0})^{\frac34} \textstyle{\sum}^{j-1}_{k=j-l} 2^{-\delta_{1}k} 2^{-5\delta_{1}k}2^{\delta_{1}(k+l+1-j)}
	\end{split}
\end{equation}
In the case of $j-l \geq N_0+1$ and $j\geq N_0+1$, the estimate \eqref{kz62} yields
\begin{equation}\label{kz63}
	\begin{split}
		& \|d \bv_j, d h_j\|_{L^1_{[0,T^*_{j-(l+1)}]} L^\infty_x}
		\\
		\leq & C(1+CM_0)^3(T^*_{N_0})^{\frac34}(1-2^{-\delta_{1}})^{-2} \big\{
		2^{-\delta_{1}j} 2^{\delta_{1}(l+1)}+ 2^{-6\delta_{1}}
		\\
		& \quad +2^{-5\delta_{1}[j-(l+1)]}  2^{\delta_{1}(l+1)}+ 2^{-6\delta_{1}}+ 2^{-5\delta_{1}(j-l)}2^{\delta_{1}(l+1-j)}
		\big\}
		\\
		\leq & C(1+CM_0)^3(T^*_{N_0})^{\frac34}(1-2^{-\delta_{1}})^{-2} \left\{
		2^{-6\delta_{1}N_0} + 2^{-6\delta_{1}} +	2^{-5\delta_{1}N_0} + 2^{-6\delta_{1}}+ 	2^{-6\delta_{1}N_0} 	\right\}
		\\
		\leq & C(1+M_0)^3(T^*_{N_0})^{\frac34}[\frac13(1-2^{-\delta_{1}})]^{-2}.
	\end{split}
\end{equation}
By using \eqref{pu00}, \eqref{pp8}, \eqref{kz63}, and Theorem \ref{DW4}, we have proved
\begin{equation}\label{kz64}
	\begin{split}
		|\bv_j, h_j|\leq 2+C_0, \quad  E(T^*_{j-(l+1)}) \leq C_*.
	\end{split}
\end{equation}
Gathering \eqref{kz40}, \eqref{kz51}, \eqref{kz55}, \eqref{kz63} and \eqref{kz64}, the estimates \eqref{kz41}-\eqref{kz45} hold for $l+1$. Thus, our induction hold \eqref{kz41}-\eqref{kz45} for $l=1$ to $l=j-N_0$. Therefore, we can extend the solutions $(h_j,\bv_j,\bw_j)$ from $[0,T^*_j]$ to $[0,T^*_{N_0}]$ when $j \geq N_0$. We denote
\begin{equation}\label{Tstar}
	T^*=\min\{ 1, T^*_{N_0} \}=\min\{ 1,  (CM_0)^{-3}2^{-\delta_{1} N_0} \}.
\end{equation}
Taking $l=j-N_0$ in \eqref{kz44}-\eqref{kz45}, we therefore get
\begin{equation}\label{kz65}
	\begin{split}
		& E(T^*) \leq  C_*, \quad \|\bv_j, h_j\|_{L^\infty_{[0,T^*]} L^\infty_x} \leq 2+C_0,
		\\
		& \|d \bv_j, d h_j\|_{L^1_{[0,T^*]} L^\infty_x}
		\leq  C(1+M_0)^3(T^*_{N_0})^{\frac34}[\frac13(1-2^{-\delta_{1}})]^{-2} \leq 2,
	\end{split}
\end{equation}
where $N_0$ and $C_*$ (depending on $C_0, c_0, s, M_0$) are denoted in \eqref{pp8} and \eqref{Cstar}. Similarly, we conclude
\begin{equation}\label{kz66}
	\begin{split}
		\|d \bv_j, d h_j\|_{L^4_{[0,T^*]} L^\infty_x}
		\leq & C(1+M_0)^3 [\frac13(1-2^{-\delta_{1}})]^{-2}.
	\end{split}
\end{equation}
Due to \eqref{Tstar}, \eqref{kz65}, \eqref{kz66}, we have proved \eqref{Duu0}, \eqref{Duu00}, and \eqref{Duu2}.

It still remains for us to prove \eqref{Duu21} and \eqref{Duu22}. We present the proof in the following subsection.
\subsubsection{Strichartz estimates of linear wave equation on time-interval $[0,T^*_{N_0}]$.}\label{finalq}
We still expect the behaviour of a linear wave equation endowed with $g_j=g(h_j,\bv_j)$. So we claim a theorem as follows
\begin{proposition}\label{rut}
	For $s-\frac34 \leq r \leq \frac{11}{4}$, there is a solution $f_j$ on $[0,T^*_{N_0}]\times \mathbb{R}^2$ satisfying the following linear wave equation
	\begin{equation}\label{ru01}
		\begin{cases}
			\square_{{g}_j} f_j=0,
			\\
			(f_j,\partial_t f_j)|_{t=0}=(f_{0j},f_{1j}),
		\end{cases}
	\end{equation}
	where $(f_{0j},f_{1j})=(P_{\leq j}f_0,P_{\leq j}f_1)$ and $(f_0,f_1)\in H_x^r \times H^{r-1}_x$. Moreover, for $a\leq r-(s-1) $, we have
	\begin{equation}\label{ru02}
		\begin{split}
			&\|\left< \nabla \right>^{a-1} d{f}_j\|_{L^4_{[0,T^*_{N_0}]} L^\infty_x}
			\leq  C_{M_0}(\|{f}_0\|_{{H}_x^r}+ \|{f}_1 \|_{{H}_x^{r-1}}),
			\\
			&\|{f}_j\|_{L^\infty_{[0,T^*_{N_0}]} H^{r}_x}+ \|\partial_t {f}_j\|_{L^\infty_{[0,T^*_{N_0}]} H^{r-1}_x} \leq  C_{M_0}(\| {f}_0\|_{H_x^r}+ \| {f}_1\|_{H_x^{r-1}}).
		\end{split}
	\end{equation}
\end{proposition}
\begin{proof}[Proof of Proposition \ref{rut}.] Firstly, there are some Strichartz estimates for short time intervals. By applying the extension method described in subsection \ref{esest} and summing the short-time estimates corresponding to these intervals, we can derive a specific type of Strichartz estimate that incurs a loss of derivatives. Due to	
	\begin{equation}\label{ru060}
		a+9\delta_{1} \leq r-(s-1)+9\delta_{1} < r-(s-1)+(s-\frac74) <r-\frac34,
	\end{equation}	
	using \eqref{ru01}, \eqref{ru03} and \eqref{ru04}, we therefore have
	\begin{equation}\label{ru06}
		\begin{split}
			\| \left< \nabla \right>^{a-1+9\delta_{1}} df_j\|_{L^4_{[0,T^*_j]} L^\infty_x} \leq & C( \| f_{0j} \|_{H^r_x} + \| f_{1j} \|_{H^{r-1}_x}  )
			\\
			\leq & C ( \| f_{0} \|_{H^r_x} + \| f_{1} \|_{H^{r-1}_x}  ).
		\end{split}
	\end{equation}
	Next, we will discuss the cases of high frequency and low frequency. For $k\geq j$, which corresponds to the high frequency, we apply  Bernstein's inequality to obtain
	\begin{equation}\label{ru07}
		\begin{split}
			\| \left< \nabla \right>^{a-1} P_k df_j\|_{L^4_{[0,T^*_j]} L^\infty_x}
			=& C2^{-9\delta_{1} k}\| \left< \nabla \right>^{a-1+9\delta_{1}} P_k df_j\|_{L^4_{[0,T^*_j]} L^\infty_x}
			\\
			\leq & C2^{-\delta_{1} k} 2^{-8\delta_{1} j} \| \left< \nabla \right>^{a-1+9 \delta_{1}} df_j\|_{L^4_{[0,T^*_j]} L^\infty_x}.
		\end{split}
	\end{equation}
	Combining \eqref{ru06} and \eqref{ru07}, for $a \leq r-(s-1)$, we obtain
	\begin{equation}\label{ru070}
		\begin{split}
			\| \left< \nabla \right>^{a-1} P_k df_j\|_{L^4_{[0,T^*_j]} L^\infty_x}
			\leq & C2^{-\delta_{1} k} 2^{-8\delta_{1} j} ( \| f_{0} \|_{H^r_x} + \| f_{1} \|_{H^{r-1}_x}  ), \quad k\geq j.
		\end{split}
	\end{equation}
	For $k\leq j$, which corresponds to the low frequency, we need to discuss the difference terms. For any integer $m\geq 1$, we have
	\begin{equation*}
		\begin{cases}
			\square_{{g}_m} f_m=0, \quad [0,T^*_{m}]\times \mathbb{R}^2,
			\\
			(f_m,\partial_t f_m)|_{t=0}=(f_{0m},f_{1m}),
		\end{cases}
	\end{equation*}
	and
	\begin{equation*}
		\begin{cases}
			\square_{{g}_{m+1}} f_{m+1}=0, \quad [0,T^*_{m+1}]\times \mathbb{R}^2,
			\\
			(f_{m+1},\partial_t f_{m+1})|_{t=0}=(f_{0(m+1)},f_{1(m+1)}).
		\end{cases}
	\end{equation*}
	Then the difference term $f_{m+1}-f_m$ is satisfied by
	\begin{equation*}\label{ru08}
		\begin{cases}
			\square_{{g}_{m+1}} (f_{m+1}-f_m)=({g}^{\alpha i}_{m+1}-{g}^{\alpha i}_{m} )\partial^2_{\alpha i} f_m, \quad [0,T^*_{m+1}]\times \mathbb{R}^2,
			\\
			(f_{m+1}-f_m,\partial_t (f_{m+1}-f_m))|_{t=0}=(f_{0(m+1)}-f_{0m},f_{1(m+1)}-f_{1m}).
		\end{cases}
	\end{equation*}
	Due to \eqref{ru04}, \eqref{ru060}, and Duhamel's principle, which yields
	\begin{equation}\label{ru09}
		\begin{split}
			& \| \left< \nabla \right>^{a-1} P_k (d f_{m+1}-df_m)\|_{L^4_{[0,T^*_{m+1}]} L^\infty_x}
			\\
			\leq & C\| f_{0(m+1)}-f_{0m}\|_{H_x^{r-9\delta_{1}}} + C\|f_{1(m+1)}-f_{1m}\|_{H_x^{r-1-9\delta_{1}}}
			\\
			& \ + C\|({g}_{m+1}-{g}_{m}) \cdot\nabla d f_m\|_{L^1_{[0,T^*_{m+1}]} H^{r}_x}
			\\
			\leq & C2^{-9\delta_{1} m}( \| f_{0} \|_{H^r_x} + \| f_{1} \|_{H^{r-1}_x}  ) + C \| {g}_{m+1}-{g}_{m}\|_{L^1_{[0,T^*_{m+1}]} L^\infty_x} \| \nabla d f_m\|_{L^\infty_{[0,T^*_{m+1}]} H^{r-1}_x} .
		\end{split}
	\end{equation}
	By applying energy estimates and the Bernstein inequality, it follows that
	\begin{equation}\label{ru10}
		\begin{split}
			\| \nabla d f_m\|_{L^\infty_{[0,T^*_{m+1}]} H^{r-1}_x} \leq & C\| d f_m\|_{L^\infty_{[0,T^*_{m+1}]} H^{r}_x}
			\\
			\leq & C(\|f_{0m}\|_{H^{r+1}_x}+\|f_{1m}\|_{H^{r}_x})
			\\
			\leq & C2^m (\|f_{0m}\|_{H^r_x}+\|f_{1m}\|_{H^{r-1}_x}) .
		\end{split}
	\end{equation}
	Using H\"older's inequality, the Strichartz estimates \eqref{ru04} and Duhamel's principle, we obtain
	\begin{equation}\label{ru11}
		\begin{split}
			& 2^m \| {g}_{m+1}-{g}_{m}\|_{L^1_{[0,T^*_{m+1}]} L^\infty_x}
			\\
			\leq & 2^m (T^*_{m+1})^{\frac34}\| {\bv}_{m+1}-{\bv}_{m},  h_{m+1}-h_{m}\|_{L^4_{[0,T^*_{m+1}]} L^\infty_x}
			\\
			\leq & C2^m ( \| {\bv}_{m+1}-{\bv}_{m},  h_{m+1}-h_{m}\|_{L^\infty_{[0,T^*_{m+1}]} H^{\frac34+\delta_{1}}_x} + \| {w}_{m+1}-{w}_{m}\|_{L^\infty_{[0,T^*_{m+1}]} H^{\delta_{1}}_x})
			\\
			\leq & C2^{-8\delta_{1}m } (1+CM_0)^3.
		\end{split}
	\end{equation}
	Above, we also use \eqref{yu0}, \eqref{yu1}, and \eqref{DTJ}. According to \eqref{ru10} and \eqref{ru11}, for $k<j$ and $k\leq m$, \eqref{ru09} tells us that
	\begin{equation}\label{ru12}
		\begin{split}
			\| \left< \nabla \right>^{a-1} P_k (df_{m+1}-df_m)\|_{L^4_{[0,T^*_{m+1}]} L^\infty_x}
			\leq  C2^{-\delta_{1} k}2^{-8\delta_{1} m} \| (f_{0},f_1) \|_{H^r_x\times H^{r-1}_x}  (1+CM_0)^3.
		\end{split}
	\end{equation}
	Although we have established some estimates for difference terms, we also need to discuss the energy estimates. This is necessary because the key point of the extension method (in subsection \ref{esest}) crucially relies on the uniform bound of the energy. To achieve this goal, we take the operator $\left< \nabla \right>^{r-1}$ on \eqref{ru01}, which yields the following result:
	\begin{equation}\label{ru15}
		\begin{cases}
			\square_{{g}_j} \left< \nabla \right>^{r-1}f_j=[ \square_{{g}_j} ,\left< \nabla \right>^{r-1}] f_j,
			\\
			(\left< \nabla \right>^{r-1}f_j,\partial_t \left< \nabla \right>^{r-1}f_j)|_{t=0}=(\left< \nabla \right>^{r-1}f_{0j},\left< \nabla \right>^{r-1}f_{1j}).
		\end{cases}
	\end{equation}
	To estimate $[ \square_{{g}_j} ,\left< \nabla \right>^{r-1}] f_j$, we will divide it into two cases: $s-\frac34<r\leq s$ and $s<r\leq \frac{11}{4}$.
	
	\textbf{Case 1: $s-\frac34<r\leq s$.} In this situation, we note that
	\begin{equation}\label{ru14}
		\begin{split}
			[ \square_{{g}_j} ,\left< \nabla \right>^{r-1}] f_j
			=& [{g}^{\alpha i}_j-\mathbf{m}^{\alpha i}, \left< \nabla \right>^{r-1} \partial_i ] \partial_{\alpha}f_j + \left< \nabla \right>^{r-1}( \partial_i g_j \partial_{\alpha}f_j )
			\\
			= &[{g}_j-\mathbf{m}, \left< \nabla \right>^{r-1} \nabla] df_j + \left< \nabla \right>^{r-1}( \nabla g_j df_j).
		\end{split}
	\end{equation}
	By \eqref{ru14} and Kato-Ponce estimates, we have\footnote{If $r=s$, then $L^{\frac{3}{s-r}}_x=L^\infty_x$.}
	\begin{equation}\label{ru16}
		\| [ \square_{{g}_j} ,\left< \nabla \right>^{r-1}] f_j \|_{L^2_x} \leq C ( \|dg_j\|_{L^\infty_x} \|d f_j \|_{H^{r-1}_x} + \|\left< \nabla \right>^{r} (g_j-\mathbf{m}) \|_{L^{\frac{2}{1+r-s}}_x} \| df_j \|_{L^{\frac{2}{s-r}}_x} )
	\end{equation}
	By Sobolev's inequality, it follows that
	\begin{equation}\label{ru17}
		\|\left< \nabla \right>^{r} (g_j-\mathbf{m}) \|_{L^{\frac{2}{1+r-s}}_x} \leq C\|g_j-\mathbf{m}\|_{H^s_x}.
	\end{equation}
	Taking advantage of the Gagliardo-Nirenberg and Young's inequalities, which yields
	\begin{equation}\label{ru18}
		\begin{split}
			\| df_j \|_{L^{\frac{3}{s-r}}_x} \leq & C\|\left< \nabla \right>^{r-1} df_j \|^{\frac{3/4}{3/4+(2-s)}}_{L^{2}_x} \|\left< \nabla \right>^{r-(s-\frac34)} df_j \|^{\frac{2-s}{3/4+(2-s)}}_{L^{\infty}_x}
			\\
			\leq & C( \|df_j \|_{H^{r-1}_x}+ \|\left< \nabla \right>^{r-(s-\frac34)} df_j \|_{L^{\infty}_x})
		\end{split}
	\end{equation}
	\quad	\textbf{Case 2: $s<r\leq \frac{11}{4}$.} For $s<r\leq \frac{11}{4}$, applying Kato-Ponce estimates, we have
	\begin{equation}\label{ru150}
		\begin{split}
			\| [ \square_{{g}_j} ,\left< \nabla \right>^{r-1}] f_j \|_{L^2_x}
			= & \| {g}^{\alpha i}_j-\mathbf{m}^{\alpha i},\left< \nabla \right>^{r-1}] \partial^2_{\alpha i}f_j \|_{L^2_x}
			\\
			\leq & C ( \|dg_j\|_{L^\infty_x} \|d f_j \|_{H^{r-1}_x} + \|\left< \nabla \right>^{r-1} (g_j-\mathbf{m}) \|_{L^{\frac{2}{1+s-r}}_x} \|\nabla df_j \|_{L^{\frac{2}{r-s}}_x} ) .
		\end{split}
	\end{equation}
	By Sobolev's inequality, it follows that
	\begin{equation}\label{ru151}
		\|\left<\nabla \right>^{r-1} (g_j-\mathbf{m}) \|_{L^{\frac{2}{1+s-r}}_x} \leq C\|g_j-\mathbf{m}\|_{H^s_x}.
	\end{equation}
	Note that $\frac74<s\leq 2$. By applying the Gagliardo-Nirenberg and Young's inequalities, we derive:
	\begin{equation}\label{ru152}
		\begin{split}
			\| \nabla df_j \|_{L^{\frac{3}{1+s-r}}_x} \leq & C\|\left< \nabla \right>^{r-1} df_j \|^{\frac{3/4}{3/4+(2-s)}}_{L^{2}_x} \|\left< \nabla \right>^{r-(s-\frac34)} df_j \|^{\frac{2-s}{3/4+(2-s)}}_{L^{\infty}_x}
			\\
			\leq & C( \|df_j \|_{H^{r-1}_x}+ \|\left< \nabla \right>^{r-(s-\frac34)} df_j \|_{L^{\infty}_x}).
		\end{split}
	\end{equation}
	Therefore, in both cases, for \( s - \frac{3}{4} < r \leq \frac{11}{4} \), using \eqref{ru16}-\eqref{ru18} and  \eqref{ru150}-\eqref{ru152}, we can derive
	\begin{equation*}\label{ru19}
		\begin{split}
			\| [ \square_{{g}_j} ,\left< \nabla \right>^{r-1}] f_j \|_{L^2_x}\|d f_j \|_{H^{r-1}_x} \leq & C  \|dg_j\|_{L^\infty_x} \|d f_j \|^2_{H^{r-1}_x}
			+ C  \| g_j-\mathbf{m} \|_{H^s_x}\|d f_j \|^2_{H^{r-1}_x}
			\\
			& + C  \|g_j-\mathbf{m}\|_{H^s_x} \|\left< \nabla \right>^{r-\frac{s}{2}-1} df_j \|_{L^{\infty}_x} \|d f_j \|_{H^{r-1}_x}
			\\
			\leq & C  ( \|dg_j\|_{L^\infty_x}+ \|g_j-\mathbf{m}\|_{H^s_x}+ \|\left< \nabla \right>^{r-(s-\frac34)} df_j \|_{L^{\infty}_x} ) \|d f_j \|^2_{H^{r-1}_x}
			\\
			&+ C  \|g_j-\mathbf{m}\|^2_{H^s_x}  \|\left< \nabla \right>^{r-(s-\frac34)} df_j \|_{L^{\infty}_x} .
		\end{split}	
	\end{equation*}
	The above estimate, when combined with \eqref{ru15}, yields
	\begin{equation*}\label{ru20}
		\begin{split}
			\frac{d}{dt}\|d f_j \|^2_{H^{r-1}_x}
			\leq & C  (\|dg_j\|_{L^\infty_x}+ \|g_j-\mathbf{m}\|_{H^s_x}+ \|\left< \nabla \right>^{r-(s-\frac34)} df_j \|_{L^{\infty}_x} ) \|d f_j \|^2_{H^{r-1}_x}
			\\
			&+ C  \|g_j-\mathbf{m}\|^2_{H^s_x}  \|\left< \nabla \right>^{r-(s-\frac34)} df_j \|_{L^{\infty}_x} .
		\end{split}	
	\end{equation*}
	Using Gronwall's inequality, we obtain
	\begin{equation}\label{ru21}
		\begin{split}
			\|d f_j(t) \|^2_{H^{r-1}_x}
			\leq & C  (\|d f_j(0) \|^2_{H^{r-1}_x} + C \int^t_0 \|g_j-\mathbf{m}\|^2_{H^s_x}  \|\left< \nabla \right>^{r-(s-\frac34)} df_j \|_{L^{\infty}_x}d\tau)
			\\
			& \qquad \cdot\exp\left\{\int^t_0  \|dg_j\|_{L^\infty_x}+ \|g_j-\mathbf{m}\|_{H^s_x}+ \|\left< \nabla \right>^{r-(s-\frac34)} df_j \|_{L^{\infty}_x} d\tau \right\} .
		\end{split}	
	\end{equation}
	On the other hand, the Newton-Leibniz formula tells us
	\begin{equation}\label{ru22}
		\begin{split}
			\|f_j(t) \|_{L^{2}_x}
			\leq & \| f_j(0) \|_{L^2_x} + \int^t_0 \| \partial_t f_j \|_{L^2_x} d\tau.
		\end{split}	
	\end{equation}
	Summarizing the results from \eqref{ru21}, \eqref{ru22} and \eqref{kz65}, for \( t \in [0, T^*_{N_0}] \), we obtain
	\begin{equation}\label{ru23}
		\begin{split}
			& \| f_j(t) \|^2_{H^{r}_x} + \| d f_j(t) \|^2_{H^{r-1}_x}
			\\
			\leq & C  \textrm{e}^{C_*}(\|f_0 \|^2_{H^{r}_x}+\|f_1 \|^2_{H^{r-1}_x} + C_*)
			\\
			& \quad \cdot\exp \left\{\int^t_0   \|\left< \nabla \right>^{r-\frac{s}{2}-1} df_j \|_{L^{\infty}_x} d\tau \cdot 
			\exp \left( \int^t_0   \|\left< \nabla \right>^{r-(s-\frac34)} df_j \|_{L^{\infty}_x} d\tau \right)  \right\} .
		\end{split}	
	\end{equation}
	For \( a \leq r - (s - 1) \), based on \eqref{ru070}, \eqref{ru12}, and \eqref{ru23}, and following the extension method outlined in subsection \ref{esest}, we can obtain the following bounds:
	\begin{equation}\label{ru234}
		\begin{split}
			\| \left< \nabla \right>^{a-1} P_k df_j\|_{L^4_{[0,T^*_{N_0}]} L^\infty_x}
			\leq & C2^{-\delta_{1} k} 2^{-8\delta_{1} j} ( \| f_{0} \|_{H^r_x} + \| f_{1} \|_{H^{r-1}_x}  ) \times \left\{ 2^{\delta_{1} (j-N_0)} \right\}^{\frac34}
			\\
			\leq & C2^{-\frac34\delta_{1}N_0} 2^{-\delta_{1} k} 2^{-4\delta_{1} j} ( \| f_{0} \|_{H^r_x} + \| f_{1} \|_{H^{r-1}_x}  ), \quad k \geq j,
		\end{split}
	\end{equation}
	and
	\begin{equation}\label{ru25}
		\begin{split}
			& \| \left< \nabla \right>^{a-1} P_k (df_{m+1}-df_m)\|_{L^4_{[0,T^*_{N_0}]} L^\infty_x}
			\\
			\leq & C2^{-\delta_{1} k}2^{-8\delta_{1} m}( \| f_{0} \|_{H^r_x} + \| f_{1} \|_{H^{r-1}_x}  )(1+CM_0)^3 \times \left\{ 2^{\delta_{1} (m+1-N_0)} \right\}^{\frac34}
			\\
			\leq & C2^{-\frac34\delta_{1}N_0} 2^{-\delta_{1} k}2^{-4\delta_{1} m}( \| f_{0} \|_{H^r_x} + \| f_{1} \|_{H^{r-1}_x}  ) (1+CM_0)^3 , \quad k<j.
		\end{split}
	\end{equation}
	By decomposing the frequency, we obtain
	\begin{equation*}
		df_j= P_{\geq j} df_j+ \textstyle{\sum}^{j-1}_{k=1} \textstyle{\sum}^{j-1}_{m=k} P_k (df_{m+1}-df_m)+ \textstyle{\sum}^{j-1}_{k=1} P_k df_k.
	\end{equation*}
	By using \eqref{ru234} and \eqref{ru25}, we get
	\begin{equation}\label{ru230}
		\begin{split}
			\| \left< \nabla \right>^{a-1} df_{j} \|_{L^4_{[0,T^*_{N_0}]} L^\infty_x}
			\leq & C( \| f_{0} \|_{H^r_x} + \| f_{1} \|_{H^{r-1}_x}  )(1+CM_0)^3  [\frac13(1-2^{-\delta_{1}})]^{-2}.
		\end{split}
	\end{equation}
	Applying \eqref{ru23} and \eqref{ru230}, we obtain
	\begin{equation*}\label{ru26}
		\begin{split}
			\|f_j\|_{L^\infty_{[0,T^*_{N_0}]} H^{r}_x}+ \|d f_j\|_{L^\infty_{[0,T^*_{N_0}]} H^{r-1}_x }
			\leq & A_* \exp ( B_* \mathrm{e}^{B_*} ),
		\end{split}	
	\end{equation*}
	where $C_*$ is as stated in \eqref{Cstar} and 
	\begin{equation*}
		\begin{split}
			A_{*}=& C  \mathrm{e}^{C_*}(\|f_0 \|_{H^{r}_x}+\|f_1 \|_{H^{r-1}_x} + C_*),
			\\
			B_*=& C(\| f_{0} \|_{H^r_x} + \| f_{1} \|_{H^{r-1}_x})(1+CM_0)^3 [\frac13(1-2^{-\delta_{1}})]^{-2}.
		\end{split}
	\end{equation*}
	At this stage, we have proved \eqref{ru02}.
\end{proof}

\section{Proof of Theorem \ref{thm3} and Theorem \ref{thm4}}\label{sec10}
In this section, we assume the fluid is stiff, with the pressure given by $p(\varrho)=\varrho$. Since Remark \ref{rem}, we can obtain
\begin{equation}\label{nu00}
	c_s\equiv 1, \quad \Theta\equiv 1, \quad g=\Bm.
\end{equation}
Furthermore, due to \eqref{EW2}, we can see
\begin{equation}\label{nu01}
W^\alpha=\epsilon^{\alpha \beta \gamma} \partial_\beta w_\gamma.
\end{equation}
Therefore, we can derive a better wave-transport formulation than in the polytropic case.
\subsection{New formulations and estimates} Let us prove the following lemma:
\begin{Lemma}\label{Wrq}
	Let $(h,\bv)$ be a solution of \eqref{REEf} and $p(\varrho)=\varrho$. Let $\bw$ be defined in \eqref{Vor}. Then the system \eqref{REEf} can be written as
	\begin{equation}\label{wrq}
		\begin{cases}
			& \square h = \mathrm{e}^{-2h} w^\beta w_\beta - 2 \partial_\alpha h \partial^\alpha h- \mathrm{e}^{-2h}\partial_\alpha v^\beta \partial^\alpha v_\beta,
			\\
			& \square v^\alpha = -  W^\alpha ,
			\\
			& v^\kappa \partial_\kappa w^\alpha= w^\kappa \partial^\alpha v_\kappa,
			\\
			& v^\kappa \partial_\kappa W^\alpha= -\epsilon^{\alpha \beta \gamma} \partial_\beta v^\kappa \partial_\kappa w_\gamma 
			+ \epsilon^{\alpha \beta \gamma} \partial_\beta w_\kappa \partial_\gamma v^\kappa.
		\end{cases}
	\end{equation}
\end{Lemma}
\begin{proof}
For $c_s \equiv1$, using \eqref{REEf}, we obtain
\begin{equation}\label{nu03}
	\partial_\kappa v^\kappa=0.
\end{equation}
Inserting \eqref{nu03} into the third equation in \eqref{wrt}, it yields
\begin{equation}\label{nu05}
	v^\kappa \partial_\kappa w^\alpha= w^\kappa \partial^\alpha v_\kappa.
\end{equation}
	Due to \eqref{wrt0} and \eqref{nu00}, we can derive that
	\begin{equation}\label{nu02}
		\begin{split}
			\mathcal{D}=&  -\mathrm{e}^{-2h} \partial_\kappa v^\beta \partial_\beta v^\kappa-2 \partial_\kappa h \partial^\kappa h
			\\
			=& -\mathrm{e}^{-2h} \partial_\kappa v^\beta ( \partial_\beta v^\kappa - \partial^\kappa v_\beta) -\partial_\kappa v^\beta \partial^\kappa v_\beta -2  \partial_\kappa h \partial^\kappa h
			\\
			=& -\mathrm{e}^{-2h} \epsilon_\beta^{\ \ \kappa \alpha} \partial_\kappa v^\beta  w_\alpha -\mathrm{e}^{-2h} \partial_\kappa v^\beta \partial^\kappa v_\beta -2  \partial_\kappa h \partial^\kappa h
			\\
			=& \mathrm{e}^{-2h} w^\beta w_\beta- \mathrm{e}^{-2h} \partial_\alpha v^\beta \partial^\alpha v_\beta- 2 \partial_\alpha h\partial^\alpha h.
		\end{split}
	\end{equation}
	Similarly, using \eqref{wrt1} and \eqref{nu00}, we have
	\begin{equation}\label{nu04}
		Q^\alpha=0, \quad \alpha=0,1,2.
	\end{equation}
	Applying \eqref{nu00}, \eqref{nu03} and \eqref{EW1}, it yields
	\begin{equation}\label{nu06}
		v^\kappa \partial_\kappa W^\alpha= -\epsilon^{\alpha \beta \gamma} \partial_\beta v^\kappa \partial_\kappa w_\gamma 
		 + \epsilon^{\alpha \beta \gamma} \partial_\beta w_\kappa \partial_\gamma v^\kappa.
	\end{equation}
	Substituting \eqref{nu02} and \eqref{nu04} to \eqref{wrt}, and combining with \eqref{nu05} and \eqref{nu06}, we therefore obtain \eqref{wrq}.
\end{proof}

\begin{Lemma}\label{Wagc}
Assume $0<\epsilon'<\epsilon \leq \frac18$. Let $(h,\bv)$ be a solution of \eqref{REEf} and $p(\varrho)=\varrho$. Let $\bw$ be defined in \eqref{Vor}. Then we have the following energy inequality
	\begin{equation}\label{wagc}
			\begin{split}
				\|	\bw  \|^2_{{H}^{1+\epsilon'}_x} \lesssim  & \|\bw_0  \|^2_{{H}^{1+\epsilon'}} + \int^t_0 \left( \|d\bv\|_{L^\infty_x} + \|  d\bv \|_{\dot{B}^{\epsilon'}_{\infty,2}} \right) \|\bw \|^2_{{H}^{1+\epsilon'}_x}  d\tau  + \| \bw \cdot d\bv \|^2_{H^{\epsilon'}_x}.
			\end{split}
	\end{equation}
\end{Lemma}
\begin{proof}
	First, we have
\begin{equation}\label{nu08}
	\| \bw \|_{H^{1+\epsilon'}_x} = \| \bw \|_{L^{2}_x}+\| \bw \|_{\dot{H}^{1+\epsilon'}_x}.
\end{equation}
Noting $\bw=(w^0,\mathring{\bw})$ and using elliptic estimates, we obtain
\begin{equation*}\label{nu10}
	\begin{split}
	\| \bw \|_{\dot{H}^{1+\epsilon'}_x} = & \| w^0 \|_{\dot{H}^{1+\epsilon'}_x}+\| \mathring{\bw} \|_{\dot{H}^{1+\epsilon'}_x}
	\\
	=& \| w^0 \|_{\dot{H}^{1+\epsilon'}_x}+\| \textrm{curl} \mathring{\bw} \|_{\dot{H}^{\epsilon'}_x} +\| \textrm{div} \mathring{\bw} \|_{\dot{H}^{\epsilon'}_x}.
	\end{split}
\end{equation*}
Due to \eqref{ew18}, \eqref{nu00}, and \eqref{nu03}, it yields
\begin{equation*}
	\begin{split}
		\textrm{div} \mathring{\bw}= & (v^0)^{-1} v^i \epsilon_{\alpha i 0} W^\alpha- (v^0)^{-2} v^i v^j \partial_j w_i +(v^0)^{-2} v^i  w^\kappa \partial_i v_\kappa  
	 +
		(v^0)^{-1} w^\kappa \partial_t v_\kappa .
	\end{split}
\end{equation*}
The above equation tells us 
\begin{equation}\label{nu12}
	\begin{split}
		\| \mathrm{div} \mathring{\bw} \|_{H^{\epsilon'}_x} \leq  & C\| \bW \|_{H^{\epsilon'}_x} +(\frac{|\mathring{\bv}|}{v^0})^2 \|\mathring{\bw} \|_{H^{\epsilon'}_x}  + C\| \bw \cdot d\bv\|_{H^{\epsilon'}_x} .
	\end{split}
\end{equation}
By using \eqref{nu01}, we can see
\begin{equation}\label{nu14}
	\begin{split}
		\| \mathrm{curl} \mathring{\bw} \|_{H^{\epsilon'}_x} \leq  & \| \bW \|_{H^{\epsilon'}_x}  .
	\end{split}
\end{equation}
Taking advantage of \eqref{nu12} and \eqref{nu14}, we get
\begin{equation}\label{nu16}
	\left\{  1- (\frac{|\mathring{\bv}|}{v^0})^2  \right\} \| \mathring{\bw} \|_{\dot{H}^{\epsilon'}_x} \lesssim  \| \bW \|_{H^{\epsilon'}_x}+ \| \bw \cdot  d\bv\|_{H^{\epsilon'}_x}.
\end{equation}
By use of \eqref{rsv}, then \eqref{nu16} becomes
\begin{equation}\label{nu18}
	\| \mathring{\bw} \|_{\dot{H}^{\epsilon'}_x} \lesssim  \| \bW \|_{H^{\epsilon'}_x}+ \| \bw \cdot d\bv\|_{H^{\epsilon'}_x}.
\end{equation}
By using \eqref{ew16}, \eqref{nu00} and \eqref{nu03}, we obtain
\begin{equation}\label{nu20}
	\begin{split}
		\partial_i w_0  =& 	\epsilon_{\alpha i 0} \epsilon^{\alpha \beta \gamma} \partial_\beta w_\gamma
		- (v^0)^{-1} v^j \partial_j w_i + (v^0)^{-1} w^\kappa \partial_i v_\kappa - (v^0)^{-1} w_i \partial^\kappa v_\kappa  
		\\
		=& 	\epsilon_{\alpha i 0}  W^\alpha 
		- (v^0)^{-1} v^j \partial_j w_i 
		 + (v^0)^{-1} w^\kappa \partial_i v_\kappa .
	\end{split}
\end{equation}
Due to \eqref{nu20} and \eqref{nu18}, we can derive
\begin{equation}\label{nu22}
	\begin{split}
		\| w^0 \|_{\dot{H}^{1+\epsilon'}_x} \lesssim  &  \| \bW \|_{H^{\epsilon'}_x} + \|\mathring{\bw} \|_{H^{\epsilon'}_x}  + \| \bw \cdot d\bv \|_{H^{\epsilon'}_x} 
		\\
		\lesssim  &  \| \bW \|_{H^{\epsilon'}_x} + \| \bw \cdot d\bv\|_{H^{\epsilon'}_x} .
	\end{split}
\end{equation}
Gathering \eqref{nu08}, \eqref{nu18}, and \eqref{nu22}, we get
\begin{equation}\label{nu24}
	\begin{split}
		\| \bw \|_{{H}^{1+\epsilon'}_x} \lesssim  &  \| \bw \|_{L^{2}_x} + \| \bW \|_{H^{\epsilon'}_x} + \| \bw \cdot d\bv\|_{H^{\epsilon'}_x} .
	\end{split}
\end{equation}
Multiplying $w^\alpha$ on \eqref{nu05}, integrating it on $[0,t]\times \mathbb{R}^2$, and summing $\alpha$ over $0$ to $2$, it yields
\begin{equation}\label{nu26}
	\begin{split}
		\|  \bw \|^2_{L^2_x} \lesssim & \|  \bw_0 \|^2_{L^2_x} + \int^t_0 \left( \|(v^0)^{-1} d\bv\|_{L^\infty_x} + \|  (v^0)^{-2} \bv\cdot \nabla \bv \|_{L^\infty_x} \right) \|  \bw \|^2_{L^2_x} d\tau
		\\
		\lesssim & \|  \bw_0 \|^2_{L^2_x} + \int^t_0 \|d\bv\|_{L^\infty_x} \|  \bw \|^2_{L^2_x}  d\tau.
	\end{split}
\end{equation}
Operating $\Lambda_x^{\epsilon'}$ on the equation \eqref{nu06}, we have
\begin{equation}\label{nu28}
\begin{split}
	\partial_t \Lambda_x^{\epsilon'} W^\alpha + (v^0)^{-1}v^i \partial_i \Lambda_x^{\epsilon'} W^\alpha= & 
	-[\Lambda_x^{\epsilon'} ,(v^0)^{-1}v^i \partial_i]W^\alpha -\epsilon^{\alpha \beta \gamma} \Lambda_x^{\epsilon'}  \left\{  (v^0)^{-1} \partial_\beta v^\kappa \partial_\kappa w_\gamma \right\}
	\\
	&   + \epsilon^{\alpha \beta \gamma} \Lambda_x^{\epsilon'}  \left\{ (v^0)^{-1} \partial_\beta w_\kappa \partial_\gamma v^\kappa \right\}.
\end{split}	
\end{equation}
Multiplying $\Lambda_x^{\epsilon'} W^\alpha$ on \eqref{nu28}, integrating it on $[0,t]\times \mathbb{R}^2$, and summing $\alpha$ over $0$ to $2$, we can derive
\begin{equation}\label{nu30}
	\begin{split}
	\|	\Lambda_x^{\epsilon'} \bW  \|^2_{L^2_x} \lesssim  & \|\Lambda_x^{\epsilon'}	\bW_0  \|^2_{L^2_x} + \int^t_0 \|d\bv\|_{L^\infty_x} \|\Lambda_x^{\epsilon'}	\bW  \|^2_{L^2_x}  d\tau 
	\\
	& + \int^t_0 \| [\Lambda_x^{\epsilon'} ,(v^0)^{-1}v^i \partial_i] \bW \|_{L^2_x} \|	\Lambda_x^{\epsilon'} \bW  \|_{L^2_x} d\tau
	\\
	& + \int^t_0 \| \Lambda_x^{\epsilon'}( d\bv \cdot d\bw ) \|_{L^2_x} \|	\Lambda_x^{\epsilon'} \bW  \|_{L^2_x} d\tau.
	\end{split}
\end{equation}
By using Lemma \ref{jiaohuan3} and \ref{pro}, we can bound the estimate \eqref{nu30} by
\begin{equation*}\label{nu32}
	\begin{split}
		\|	\Lambda_x^{\epsilon'} \bW  \|^2_{L^2_x} \lesssim  & \|\Lambda_x^{\epsilon'}	\bW_0  \|^2_{L^2_x} + \int^t_0 \|d\bv\|_{L^\infty_x} \|\Lambda_x^{\epsilon'}	\bW  \|^2_{L^2_x}  d\tau 
		\\
		& + \int^t_0 \|  d\bv \|_{\dot{B}^{\epsilon'}_{\infty,2}} \|d	\bw  \|_{L^2_x} \|	\Lambda_x^{\epsilon'} \bW  \|_{L^2_x} d\tau
		\\
		& + \int^t_0 \|  d\bv \|_{L^\infty_x} \|	d \bw  \|_{L^2_x} \|	\Lambda_x^{\epsilon'} \bW  \|_{L^2_x} d\tau.
	\end{split}
\end{equation*}
This implies
\begin{equation}\label{nu34}
	\begin{split}
		\|	\bW  \|^2_{\dot{H}^{\epsilon'}_x} \lesssim  & \|\bw_0  \|^2_{{H}^{1+\epsilon'}} + \int^t_0 \left( \|d\bv\|_{L^\infty_x} + \|  d\bv \|_{\dot{B}^{\epsilon'}_{\infty,2}} \right) \|\bw \|^2_{{H}^{1+\epsilon'}_x}  d\tau  .
	\end{split}
\end{equation}
Combining with \eqref{nu24}, \eqref{nu26}, and \eqref{nu34}, we therefore get
\begin{equation*}\label{nu36}
	\begin{split}
		\|	\bw  \|^2_{{H}^{1+\epsilon'}_x} \lesssim  & \|\bw_0  \|^2_{{H}^{1+\epsilon'}} + \int^t_0 \left( \|d\bv\|_{L^\infty_x} + \|  d\bv \|_{\dot{B}^{\epsilon'}_{\infty,2}} \right) \|\bw \|^2_{{H}^{1+\epsilon'}_x}  d\tau  + \| \bw \cdot d\bv \|^2_{H^{\epsilon'}_x}.
	\end{split}
\end{equation*}
\end{proof}
We also need the decomposition of the rescaled velocity $\bv$. The following lemma can be directly obtained by substituting \eqref{nu00} and \eqref{nu04} into \eqref{wag0}.
\begin{Lemma}\label{Wagb}
	Let $(h,\bv)$ be a solution of \eqref{REEf} and $p(\varrho)=\varrho$. Let $\bv_{+}$ and $\bv_{-}$ be defined in \eqref{De}. Then we have
	\begin{equation}\label{wagb}
		\begin{split}
			\square \bv_{+}
			=& -2\mathrm{e}^{-2h}(v^0)^2  \mathbf{T} \mathbf{T} \bv_{-}- \bv_{-},
		\end{split}
	\end{equation}
	where $\mathbf{T}$ is stated as in \eqref{opt}.
\end{Lemma}
\subsection{Proof of Theorem \ref{thm3}}
We will use bootstrap arguments to prove the existence result stated in Theorem \ref{thm3}. Set
\begin{equation}\label{nu66}
	T_0 =\max\{  1, \frac{1}{4C^2\big( 1+\|h_0, \bv_0 \|_{H^{\frac74+\epsilon}} + \|\bw_0  \|_{H^{1+\epsilon'}}  \big)^8} \},
\end{equation}
where $C$ is a universal positive constant. Assume the estimate
\begin{equation}\label{asu}
\int^{T_0}_0 \left( \|dh, d\bv\|_{L^\infty_x} + \|  dh,d\bv \|_{\dot{B}^{\epsilon'}_{\infty,2}} \right) d\tau \leq 1,
\end{equation} 
holds. By using Lemma \ref{DW3}, for $t\leq T_0 $ we derive that
\begin{equation}\label{nu40}
	\begin{split}
		\| (h, \bv) \|_{H^{\frac74+\epsilon}_x} \leq  & C\| ( h_0, \bv_0 ) \|_{H^{\frac74+\epsilon}} \exp \left( \int^t_0   \|(dh,d\bv)\|_{L^\infty_x}  d\tau \right) 
		\\
		\leq  & C\| ( h_0, \bv_0 ) \|_{H^{\frac74+\epsilon}} .
	\end{split}
\end{equation} 
Due to Lemma \ref{Wagc} and $0<\epsilon'<\epsilon$, when $t\leq T_0 $, we obtain
\begin{equation}\label{nu42}
	\begin{split}
		\|	\bw  \|^2_{{H}^{1+\epsilon'}_x} \leq  & C\|\bw_0  \|^2_{{H}^{1+\epsilon'}} + C\int^t_0 \left( \|d\bv\|_{L^\infty_x} + \|  d\bv \|_{\dot{B}^{\epsilon'}_{\infty,2}} \right) \|\bw \|^2_{{H}^{1+\epsilon'}_x}  d\tau  + C\| \bw \cdot d\bv \|^2_{H^{\epsilon'}_x}
		\\
		\leq  & C\|\bw_0  \|^2_{{H}^{1+\epsilon'}} + C\int^t_0 \left( \|d\bv\|_{L^\infty_x} + \|  d\bv \|_{\dot{B}^{\epsilon'}_{\infty,2}} \right) \|\bw \|^2_{{H}^{1+\epsilon'}_x}  d\tau  + C\| d \bv \cdot d\bv \|^2_{H^{\epsilon'}_x}
		\\
		 \leq  & C\|\bw_0  \|^2_{{H}^{1+\epsilon'}} + C\int^t_0 \left( \|d\bv\|_{L^\infty_x} + \|  d\bv \|_{\dot{B}^{\epsilon'}_{\infty,2}} \right) \|\bw \|^2_{{H}^{1+\epsilon'}_x}  d\tau  + C\| \bv \|^4_{H^{\frac74+ \epsilon}_x}.
	\end{split}
\end{equation}
Combining \eqref{nu40} with \eqref{nu42}, we find
\begin{equation}\label{nu44}
	\begin{split}
		\|	\bw  \|^2_{{H}^{1+\epsilon'}_x} 
		\leq  & C\|\bw_0  \|^2_{{H}^{1+\epsilon'}} + C\int^t_0 \left( \|d\bv\|_{L^\infty_x} + \|  d\bv \|_{\dot{B}^{\epsilon'}_{\infty,2}} \right) \|\bw \|^2_{{H}^{1+\epsilon'}_x}  d\tau  + C^4\mathrm{e}^4\| ( h_0, \bv_0 ) \|^4_{H^{\frac74+\epsilon}}.
	\end{split}
\end{equation}
Applying Gronwall's inequality on \eqref{nu44}, it yields
\begin{equation*}
	\begin{split}
		\|	\bw  \|^2_{{H}^{1+\epsilon'}_x} 
		\leq  & C(C\|\bw_0  \|^2_{{H}^{1+\epsilon'}} + C^4\| ( h_0, \bv_0 ) \|^4_{H^{\frac74+\epsilon}} ) \exp \left\{  \int^t_0 \left( \|d\bv\|_{L^\infty_x} + \|  d\bv \|_{\dot{B}^{\epsilon'}_{\infty,2}} \right) d\tau \right\} 
		\\
			\leq  & C(\|\bw_0  \|^2_{{H}^{1+\epsilon'}} + \| ( h_0, \bv_0 ) \|^4_{H^{\frac74+\epsilon}} ).
	\end{split}
\end{equation*}
Taking square on the above inequality, we get
\begin{equation}\label{nu46}
	\begin{split}
		\|	\bw  \|_{{H}^{1+\epsilon'}_x} 
		\leq  & C(\|\bw_0  \|_{{H}^{1+\epsilon'}} + \| ( h_0, \bv_0 ) \|^2_{H^{\frac74+\epsilon}} ).
	\end{split}
\end{equation}
By using Lemma \ref{yuy}, for $0<\epsilon'<\epsilon \leq \frac18$, we can demonstrate 
\begin{equation}\label{nu48}
	\| \partial_t \bv_{-}  \|_{L^2_{[0,T_0]} H_x^{1+\epsilon'}}+ \| \bv_{-} \|_{L^2_{[0,T_0]} H_x^{2+\epsilon'}}  \leq C(1+\| ( h, \bv ) \|_{L^\infty_{[0,T_0]} H^{\frac74+\epsilon}}) \| \bw \|_{L^2_{[0,T_0]} H_x^{1+\epsilon'}},
\end{equation}	
and
\begin{equation}\label{nu50}
	\| \partial_t \mathbf{T} \bv_{-} \|_{L^2_{[0,T_0]}H_x^{\frac{3}{4}+\epsilon}} + 	\| \mathbf{T} \bv_{-} \|_{L^2_{[0,T_0]}H_x^{\frac{7}{4}+\epsilon}} \leq C\| \bw\|_{L^2_{[0,T_0]} H_x^{1+\epsilon'}} (1+\| h,\bv \|^2_{L^\infty_{[0,T_0]} H_x^{\frac{7}{4}+\epsilon}}), 
\end{equation}
Take advantage of the estimates \eqref{nu40} and \eqref{nu46}, the bounds for \eqref{nu48} and \eqref{nu50} become
\begin{equation}\label{nu52}
	\| \partial_t \bv_{-}  \|_{L^2_{[0,T_0]} H_x^{1+\epsilon'}}+ \| \bv_{-} \|_{L^2_{[0,T_0]} H_x^{2+\epsilon'}}  \leq C(1+\|h_0, \bv_0 \|_{H^{\frac74+\epsilon}} + \|\bw_0  \|_{H^{1+\epsilon'}}  )^3 ,
\end{equation}	
and
\begin{equation}\label{nu54}
	\| \partial_t \mathbf{T} \bv_{-} \|_{L^2_{[0,T_0]}H_x^{\frac{3}{4}+\epsilon}} + 	\| \mathbf{T} \bv_{-} \|_{L^2_{[0,T_0]}H_x^{\frac{7}{4}+\epsilon}} \leq C(1+\|h_0, \bv_0 \|_{H^{\frac74+\epsilon}} + \|\bw_0  \|_{H^{1+\epsilon'}}  )^4, 
\end{equation}
Due to equation \eqref{nu52} and Sobolev embeddings, we have
\begin{equation}\label{nu56}
	\begin{split}
		\|d\bv_{-}\|_{L^2_{[0,T_0]} L^\infty_x} + \|  d\bv_{-} \|_{L^2_{[0,T_0]} \dot{B}^{\epsilon'}_{\infty,2}} \leq & C\| d \bv_{-}  \|_{L^2_{[0,T_0]} H_x^{1+\epsilon'}}
		\\
		 \leq & C(1+\|h_0, \bv_0 \|_{H^{\frac74+\epsilon}} + \|\bw_0  \|_{H^{1+\epsilon'}}  )^3 .
	\end{split}
\end{equation}
By applying Strichartz estimates to \eqref{Wagb}, and since $0<\epsilon'<\epsilon \leq \frac18$, we can derive
\begin{equation}\label{nu58}
	\begin{split}
		\|d\bv_{+}\|_{L^4_{[0,T_0]} L^\infty_x}  \leq & C\big( \|\bv_0\|_{ H^{\frac{7}{4}+\epsilon} } + \| d \mathbf{T} \bv_{-} \|_{L^1_{[0,T_0]}H_x^{\frac{3}{4}+\epsilon}} +  \| \bv_{-} \|_{L^1_{[0,T_0]}H_x^{\frac{3}{4}+\epsilon}} \big)
		\\
		\leq & C\big( \|\bv_0\|_{ H^{\frac{7}{4}+\epsilon} } + \| d \mathbf{T} \bv_{-} \|_{L^1_{[0,T_0]}H_x^{\frac{3}{4}+\epsilon}} +  \| \bv_{-} \|_{L^1_{[0,T_0]}H_x^{2+\epsilon'}} \big)
		\\
		\leq & C(1+\|h_0, \bv_0 \|_{H^{\frac74+\epsilon}} + \|\bw_0  \|_{H^{1+\epsilon'}}  )^4.
	\end{split}
\end{equation}
Similarly, since $0<\epsilon'<\epsilon \leq \frac18$, we can also obtain
\begin{equation}\label{nu60}
	\begin{split}
		\|  d\bv_{+} \|_{L^4_{[0,T_0]} \dot{B}^{\epsilon'}_{\infty,2}} \leq & C\big( \|\bv_0\|_{ H^{\frac{7}{4}+\epsilon} } + \| d \mathbf{T} \bv_{-} \|_{L^1_{[0,T_0]}H_x^{\frac{3}{4}+\epsilon}} +  \| \bv_{-} \|_{L^1_{[0,T_0]}H_x^{\frac{3}{4}+\epsilon}} \big)
		\\
		\leq & C\big( \|\bv_0\|_{ H^{\frac{7}{4}+\epsilon} } + \| d \mathbf{T} \bv_{-} \|_{L^1_{[0,T_0]}H_x^{\frac{3}{4}+\epsilon}} +  \| \bv_{-} \|_{L^1_{[0,T_0]}H_x^{2+\epsilon'}} \big)
		\\
		\leq & C(1+\|h_0, \bv_0 \|_{H^{\frac74+\epsilon}} + \|\bw_0  \|_{H^{1+\epsilon'}}  )^4.
	\end{split}
\end{equation}
Applying the Strichartz estimates on the first equation in system \eqref{wrq}, it yields
\begin{equation}\label{nu62}
	\begin{split}
		\|  dh \|_{L^4_{[0,T_0]} L^\infty_x }+\|  dh \|_{L^4_{[0,T_0]} \dot{B}^{\epsilon'}_{\infty,2}}  \leq & C\big( \|h_0\|_{ H^{\frac{7}{4}+\epsilon} } + \| (d \bv, dh) \cdot (d \bv, dh) \|_{L^1_{[0,T_0]}H_x^{\frac{3}{4}+\epsilon}} \big)
		\\
		\leq & C\big( \|h_0\|_{ H^{\frac{7}{4}+\epsilon} } + \| d \bv,dh \|_{L^1_{[0,T_0]}L_x^{\infty}} \| d \bv, dh \|_{L^\infty_{[0,T_0]}H_x^{\frac{3}{4}+\epsilon}} \big)
		\\
		\leq & C(1+\|h_0, \bv_0 \|_{H^{\frac74+\epsilon}}  ),
	\end{split}
\end{equation}
where we also use \eqref{asu} and \eqref{nu40}. Due to \eqref{nu56}, \eqref{nu58}, \eqref{nu60}, \eqref{nu62}, and H\"older's inequality, we can prove
\begin{equation*}\label{nu64}
	\begin{split}
 \int^{T_0}_0 \big( \|dh, d\bv\|_{L^\infty_x} + \|dh,  d\bv \|_{\dot{B}^{\epsilon'}_{\infty,2}} \big) d\tau 
\leq &  T_0^{\frac12} \big( \|dh, d\bv\|_{L^2_{[0,T_0]} L^\infty_x} + \|  d\bv \|_{L^2_{[0,T_0]} \dot{B}^{\epsilon'}_{\infty,2}} \big) 
\\
\leq &  T_0^{\frac12} \big(  \|d\bv_{-}\|_{L^2_{[0,T_0]} L^\infty_x} + \|  d\bv_{-} \|_{L^2_{[0,T_0]} \dot{B}^{\epsilon'}_{\infty,2}} \big)
\\
& + T_0^{\frac14} \big(  \|dh, d\bv_{+}\|_{L^4_{[0,T_0]} L^\infty_x} +  \|dh,  d\bv_{+} \|_{L^4_{[0,T_0]} \dot{B}^{\epsilon'}_{\infty,2}} \big)
\\
\leq & CT_0^{\frac12} \big( 1+\|h_0, \bv_0 \|_{H^{\frac74+\epsilon}} + \|\bw_0  \|_{H^{1+\epsilon'}}  \big)^4.
	\end{split}
\end{equation*}
Noting \eqref{nu66}, we therefore get
\begin{equation*}\label{nu65}
 \int^{T_0}_0 \big( \|dh, d\bv\|_{L^\infty_x} + \|dh,  d\bv \|_{ \dot{B}^{\epsilon'}_{\infty,2}} \big) d\tau  \leq \frac12.
\end{equation*}
According to bootstrap arguments, the estimate \eqref{asu} holds. Therefore, on $[0,T_0]\times \mathbb{R}^2$, the total energy satisfies
\begin{equation*}\label{nu70}
	\begin{split}
	\|h\|_{H_x^{\frac74+\epsilon}}+\|\bv \|_{H_x^{\frac74+\epsilon}}+	\|	\bw  \|_{{H}^{1+\epsilon'}_x} 
		\leq  & C\big(1+\|\bw_0  \|_{{H}^{1+\epsilon'}} +\| h_0 \|_{H^{\frac74+\epsilon}} +\| \bv_0 \|_{H^{\frac74+\epsilon}} \big)^2.
	\end{split}
\end{equation*}
Moreover, the following Strichartz estimate holds:
\begin{equation*}\label{nu72}
	\begin{split}
		\|dh, d\bv\|_{L^4_{[0,T_0]}L^\infty_x} + \|  dh,d\bv \|_{L^4_{[0,T_0]}\dot{B}^{\epsilon'}_{\infty,2}}
		\leq  & C\big(1+\|\bw_0  \|_{{H}^{1+\epsilon'}} +\| h_0 \|_{H^{\frac74+\epsilon}} +\| \bv_0 \|_{H^{\frac74+\epsilon}} \big)^4.
	\end{split}
\end{equation*}
We thus prove the existence of solutions. For uniqueness and continuous dependence, we omit the details because the system \eqref{wrq} is semi-linear and straightforward to verify. At this stage, we have proved Theorem \ref{thm3}.

\subsection{Proof of Theorem \ref{thm4}}
For the system \eqref{wrtq} satisfying the ``wave map" null condition, we refer to Klainerman and Selberg \cite{KS} (Theorem 1) and Tataru \cite{T} (page 38, Theorem 1). Consequently, we obtain the conclusions stated in Theorem \ref{thm4}.


\section*{Conflicts of interest and Data Availability Statements}
The authors declared that this work does not have any conflicts of interest. The authors also confirm that the data supporting the findings of this study are available within the article.

\end{sloppypar}
\end{document}